\documentclass[11pt, letterpaper, reqno, final]{amsart}
\usepackage[foot]{amsaddr}

\usepackage{amsthm,amsfonts,amssymb,amsmath,amscd,latexsym}
\usepackage{graphicx,exscale,cite,epsfig,xcolor,multirow}
\usepackage{scalerel}
\usepackage{bigints}

\usepackage{cellspace}
\usepackage[colorlinks]{hyperref}
\usepackage[notref, notcite]{showkeys}

\allowdisplaybreaks

\renewcommand{\geq}{\geqslant}
\renewcommand{\leq}{\leqslant}

\def\eps{\varepsilon}
\def\s{\,{\rm s}}
\def\c{\,{\rm c}}
\def\sh{\,{\rm sh}}
\def\ch{\,{\rm ch}}
\def\th{\,{\rm th}}
\def\e{\,{\rm e}}

\usepackage{enumitem}
\setlist[itemize]{leftmargin=0.8cm}
\newcommand\ba{\begin{equation}\begin{aligned}}
\newcommand\ea{\end{aligned}\end{equation}}
\newcommand{\be}{\begin{equation}}
\newcommand{\ee}{\end{equation}}

\newtheorem{theorem}{Theorem}[section]

\newtheorem{corollary}[theorem]{Corollary}
\newtheorem{lemma}[theorem]{Lemma}
\newtheorem{definition}[theorem]{Definition}

\theoremstyle{remark}
\newtheorem{remark}[theorem]{Remark}
\newtheorem*{remark*}{Remark}
\newtheorem*{theorem*}{Theorem}
\newtheorem*{definition*}{Definition}
\newtheorem*{notation*}{Notation}

\numberwithin{equation}{section}

\begin{document}

\title[Stable and Unstable capillary-gravity waves]{Stable and Unstable capillary-gravity waves}

\author{Vera~Mikyoung~Hur}
\address{Department of Mathematics, University of Illinois Urbana-Champaign, Urbana, IL 61801 USA}
\email{verahur@illinois.edu}

\author{Zhao~Yang}
\address{State Key Laboratory of Mathematical Sciences, Academy of Mathematics and Systems Science, Chinese Academy of Sciences, Beijing 100190 China}
\email{yangzhao@amss.ac.cn}

\date{\today}


\thanks{ 
VMH has been partially supported by the NSF through DMS-2009981. ZY thanks the Department of Mathematics at the University of Illinois Urbana-Champaign for the valuable postdoc research and teaching experience. }

\begin{abstract}
We develop a rigorous theory of spectral stability and instability for traveling periodic capillary-gravity waves of sufficiently small amplitude in a fluid of finite depth, encompassing both unimodal waves and bimodal Wilton ripples. The latter arise when the fundamental harmonic and a higher harmonic possess the same phase speed in the zero-amplitude limit.
Our analysis builds upon a periodic Evans function framework developed by the authors for Stokes' gravity waves, but requires substantial extensions and new analytical tools to accommodate the considerably richer dynamics of the capillary-gravity wave problem. Particularly, we identify all spectral values of the zero-amplitude problem at which two Floquet exponents associated with perturbations differ by an integer multiple of the wave number of the unperturbed solution. Away from such resonances, we apply the Weierstrass preparation theorem to the periodic Evans function to show that the spectrum remains confined to the imaginary axis for sufficiently small amplitude. At resonances, by contrast, the spectrum may bifurcate from the imaginary axis as the amplitude increases, giving rise to spectral instability.

A key new ingredient in the analysis for unimodal waves near nonzero resonances is an application of the Weierstrass preparation theorem to the discriminant of the quadratic Weierstrass polynomial associated with the periodic Evans function. This yields another quadratic Weierstrass polynomial whose discriminant is an explicitly computable real-analytic function of the amplitude, and the sign of the lowest-order nonvanishing term in its Taylor expansion determines spectral stability or instability. We extend the analysis to repeated resonances at critical values of the zero-amplitude spectral curves. Particularly, we obtain the first rigorous spectral stability and instability criteria for traveling periodic capillary-gravity waves away from the origin of the complex plane. We further show that, to leading order in the amplitude, the unstable spectrum bifurcating from a nonzero resonance traces either an ellipse or a circle. This rigorously justifies spectral configurations obtained through formal asymptotic analysis and numerical computation. 

We then investigate modulational instability near the origin of the complex plane for unimodal waves and derive an explicit modulational instability index. Our instability criterion agrees with the prediction of formal asymptotic analysis and recovers the Benjamin--Feir instability for Stokes waves in the zero surface-tension limit.

We extend the analysis to bimodal Wilton ripples and derive explicit stability indices for resonances both at and away from the origin of the complex plane. Particularly, we obtain the first rigorous spectral stability and instability criteria for Wilton ripples away from the origin.

Combining numerical evaluations of the resulting stability indices with a Krein signature argument, we locate the region for spectral stability in the plane of the dimensionless surface-tension coefficient and the wave number of the unperturbed solution. 
Outside this region, our numerical computations indicate spectral instability either near the origin through modulational instability or near a nonzero resonance of order one or two. 
\end{abstract}

\maketitle 

\tableofcontents

\section{Introduction}\label{sec:intro}

We consider capillary-gravity waves of sufficiently small amplitude at the free surface of an incompressible inviscid fluid of finite depth in two dimensions. Suppose that in Cartesian coordinates, waves propagate in the $x$ direction while gravity acts in the negative $y$ direction. We represent the fluid surface at time $t$ by 
\[
y=h+\eta(x,t),\quad x\in\mathbb{R},
\] 
where $h>0$ is the undisturbed fluid depth and $y=0$ is the rigid bed. Physically admissible is $h+\eta(x,t)>0$ for each $x\in\mathbb{R}$. Assuming that the flow is irrotational, we introduce a velocity potential $\phi(x,y,t)$ satisfying
\begin{subequations}\label{eqn:ww;h}
\begin{align}
&\phi_{xx}+\phi_{yy}=0&& \text{for $0<y<h+\eta(x,t)$,}
\intertext{together with the impermeability condition at the rigid bed}
&\phi_y=0&&\text{at $y=0$.}
\intertext{The kinematic and dynamic boundary conditions at the free surface are}
&\left.\begin{aligned}
\label{eqq:ww;boundary}
&\eta_t-c\eta_x+\eta_x\phi_x=\phi_y\\
&\begin{aligned}\phi_t-c\phi_x&+\frac{\phi_x^2+\phi_y^2}{2}\\
&+g\eta-\frac{b\eta_{xx}}{(1+\eta_x^2)^{3/2}}=q(t)\end{aligned}\\
\end{aligned} \right\}&&\text{at $y=h+\eta(x,t)$},
\end{align}
\end{subequations}
where $q(t)$ is an arbitrary function. Here, $c\neq0,\in\mathbb{R}$ is the wave speed, $g>0$ is the gravitational acceleration, and $b\geq0$ the ratio of surface tension to fluid density. 

When the effects of surface tension are negligible, that is, when $b=0$, \eqref{eqn:ww;h} admits temporally stationary and spatially periodic solutions, namely Stokes' gravity waves \cite{Stokes1847}. The existence of Stokes waves has been established extensively through rigorous analysis and numerical computation. 

The stability of Stokes waves has attracted sustained attention since the discovery in the 1960s of the celebrated Benjamin--Feir instability, more generally referred to as modulational instability: a Stokes wave of sufficiently small amplitude becomes unstable to long wavelength perturbations if $\kappa h>1.3627\ldots$ \cite{Benjamin;BF, Whitham;BF}. Throughout, $\kappa$ denotes the wave number of the unperturbed steady periodic wave solution, so that $\kappa h$ is the dimensionless long-wavelength, or shallow-water, parameter. 
See \cite{BM;BF} for a rigorous proof of the Benjamin--Feir instability, and \cite{HY2023,Berti2023} for more recent alternative proofs. See also \cite{NS2023, Berti2022} for infinite depth, and \cite{HHY2026a,HHY2026b} for constant vorticity. 

The Benjamin--Feir instability concerns the spectrum near the origin of the complex plane. Beginning in the 1980s, however, numerical investigations indicated that Stokes waves, not necessarily of small amplitude, may exhibit spectral instability away from $0\in\mathbb{C}$ when two Floquet exponents associated with perturbations differ by an integer multiple of $\kappa$. See, for instance, \cite{McLean;finite-depth}, and \cite{DO} for more recent numerical computations. See also \cite{doi:10.1098/rspa.1978.0080,doi:10.1098/rspa.1978.0081,mclean_1982} for infinite depth. More recently, the authors \cite{HY2023} introduced a periodic Evans function framework to develop the first rigorous analysis of this kind of instability. Specifically, we proved that a Stokes wave of sufficiently small amplitude is spectrally unstable near a resonance of order $2$, away from $0\in\mathbb{C}$, if an explicitly computable index (see \cite[(6.34)]{HY2023}) is positive. Particularly, this establishes the existence of Stokes waves that are spectrally unstable even though they are insusceptible to the Benjamin--Feir instability. See also \cite{berti2025}, and \cite{berti2024isolamodulationalinstabilitystokes} for infinite depth.

Surface tension produces substantially more varied traveling waves. 
When $b/(gh^2)\geq1/3$, corresponding to strong surface tension or small fluid depth, \eqref{eqn:ww;h} admits steady \emph{unimodal} wave solutions for sufficiently small amplitude, in direct analogy with Stokes' gravity waves. When $0<b/(gh^2)<1/3$, on the other hand, two distinct harmonics can propagate at the same speed in the zero-amplitude limit. Particularly, when the fundamental harmonic and the $M$-th harmonic, $M\geq2$, have the same phase speed, the resulting \emph{bimodal} wave solutions are historically known as \emph{Wilton ripples of order $M$} \cite{doi:10.1080/14786440508635350}. See \cite{Nayfeh1970, McGoldrick_1970, CSaffmam1979} for formal asymptotic treatments of Wilton ripples for infinite depth.  

The existence of traveling periodic capillary-gravity waves of sufficiently small amplitude has been rigorously established for both unimodal waves and bimodal Wilton ripples. See, for instance, \cite{jones_1989} for finite depth and \cite{Jones1986,10.2307/2397696} for infinite depth. An excellent review of numerical results for traveling capillary-gravity waves can be found in \cite{vanden-broeck_2010}. 

Surface tension also produces a vastly more intricate stability picture. Djordjevic and Redekopp \cite{djordjevic_redekopp_1977} (see also \cite{AS1979,Ma1978,MR1979}) extended the formal Benjamin--Feir instability analysis for Stokes waves to traveling periodic capillary-gravity waves, locating several intervals of $\kappa h$ for modulational stability and instability for each $b/(gh^2)>0$. 
See also \cite{Benney1976,Hogan1985} for infinite depth, and \cite{Jones1992} for bimodal Wilton ripples. Numerical computations of the spectrum can be found in \cite{ZM1987}. See also \cite{DECONINCK201497,TDW16} for more recent computations.

Although the existence of traveling periodic capillary-gravity waves of sufficiently small amplitude is well established, and their stability has been studied through formal asymptotic analysis and numerical computation, a systematic and rigorous spectral theory has remained undeveloped. (After earlier versions of this manuscript were posted on arXiv, further rigorous treatments of the spectrum appeared in \cite{hsiao2025, sun2025}.)  Particularly, there has been no rigorous theory of spectral instability away from $0\in\mathbb{C}$, comparable to that available for Stokes' gravity waves. Surface tension introduces bimodal Wilton ripples and a rich array of resonances, together with other difficulties, that require substantial extensions of the ideas and techniques developed for gravity waves as well as new analytical tools.

The present purpose is to develop a rigorous theory of spectral stability and instability for traveling periodic capillary-gravity waves of sufficiently small amplitude, encompassing both unimodal waves and bimodal Wilton ripples. We examine the spectrum near the origin of the complex plane, corresponding to modulational instability, and away from $0\in\mathbb{C}$, for spectral instability near resonances. To the best of our knowledge, this is the first rigorous analysis that treats, within a single framework, the different kinds of spectral instability for different kinds of capillary-gravity waves. Our analysis builds upon the periodic Evans function framework developed in \cite{HY2023} for Stokes' gravity waves, but extends it in essential ways to capture the considerably richer dynamics of the capillary-gravity wave problem. 
Particularly, the analysis involves a complete classification of resonances, a separate treatment of ``repeated'' resonances, and new analytical tools that turn power series expansions of the periodic Evans function into explicitly computable spectral stability indices. Notably, we obtain the first rigorous spectral stability and instability criteria near nonzero resonances for both steady unimodal capillary-gravity waves and bimodal Wilton ripples. 

\subsection*{Overview of the analysis}

Following \cite{HY2023}, we begin by reformulating \eqref{eqn:ww;h} in dimensionless variables. Let 
\begin{equation}\label{def:dimensionless}
\begin{gathered}
x\mapsto \frac{x}{h},\quad y\mapsto \frac{y}{h},\quad t\mapsto \frac{ct}{h},\\
\eta\mapsto \frac{\eta}{h},\quad \phi\mapsto \frac{\phi}{ch}, \quad q\mapsto \frac{q}{c^2},
\end{gathered}
\end{equation}
and, correspondingly, let 
\begin{equation}\label{def:mu}
\mu=\frac{gh}{c^2}\quad\text{and}\quad \beta=\frac{b}{gh^2}   
\end{equation}
denote the inverse square of the Froude number and the Bond number, respectively. Substituting \eqref{def:dimensionless} and \eqref{def:mu} into \eqref{eqn:ww;h}, after some algebra, we arrive at:
\ba\label{eqn:ww;1}
&\left.\phi_{xx}+\phi_{yy}=0\right.&& \text{for $0<y<1+\eta(x,t)$,}\\
&\left.\phi_y=0\right.&&\text{at $y=0$,}\\
&\left.\begin{aligned}&\eta_t-\eta_x+\eta_x\phi_x=\phi_y \\ 
&\begin{aligned}\phi_t-\phi_x&+\frac{\phi_x^2+\phi_y^2}{2}\\
&+\mu\eta-\frac{\beta\mu\eta_{xx}}{(1+\eta_x^2)^{3/2}}=q(t)\end{aligned}\end{aligned}\right\}&&\text{at $y=1+\eta(x,t)$.}  
\ea 
Following \cite{HS;cg-solitary} (see also \cite{HY2023}), we introduce 
\be\label{def:u}
u=\phi_x \quad\text{and}\quad z=\frac{\beta\eta_x}{(1+\eta_x^2)^{1/2}},
\ee 
and make the change of variables
\[
y \mapsto \frac{y}{1+\eta(x,t)},
\]
thereby reformulating \eqref{eqn:ww;1} in the {\em fixed} infinite cylinder $\mathbb{R}\times (0,1)$ as a first-order dynamical system with respect to the $x$ variable:
\ba\label{eqn:ww}
&\left.\begin{aligned}&\phi_x-\frac{y z \phi_y}{(1+\eta)\sqrt{\beta^2- z ^2}}-u=0\\
&u_x-\frac{y z  u_y}{(1+\eta)\sqrt{\beta^2- z ^2}}+\frac{\phi_{yy}}{(1+\eta)^2}=0\end{aligned}\right\}&&\text{for $0<y<1$,}\\
&\left.\begin{aligned}&\eta_x-\frac{ z }{\sqrt{\beta^2- z ^2}}=0\\
&\phi_t-\mu z _x-u+\frac{u^2}{2}+\mu\eta\\
&\quad+\frac{(u-1) z \phi_y}{(1+\eta)\sqrt{\beta^2- z ^2}}-\frac{\phi_y^2}{2(1+\eta)^2}=q(t)\end{aligned}\right\}&&\text{at $y=1$},
\ea 
and
\ba 
\label{eqn:bd}
&\phi_y=0&&\text{at $y=0$},\\
&\eta_t+\frac{ z (u-1)}{\sqrt{\beta^2- z ^2}}-\frac{\phi_y}{1+\eta}=0\quad&&\text{at $y=1$}.\\
\ea 
The third equation of \eqref{eqn:ww} follows by solving the latter equation of \eqref{def:u}.

In Section \ref{sec:Stokes}, we compute the Stokes expansion for temporally stationary and spatially periodic solutions to \eqref{eqn:ww}-\eqref{eqn:bd} for sufficiently small amplitude. Throughout, $\eps\in\mathbb{R}$ denotes the dimensionless amplitude parameter. A fundamental difference from Stokes' gravity waves is that bimodal wave solutions arise for $\beta\in (0,1/3)$ if
\[
\mu_0(M_1\kappa;\beta)=\mu_0(M_2\kappa;\beta)\quad\text{for some $M_1,M_2\in\mathbb{Z}^+$ and coprime}
\]
for some $\kappa>0$, where $\mu_0(\kappa;\beta)$, in \eqref{def:mu0}, is $\mu$ in the zero-amplitude limit, related to the wave speed. When $M_1=1$ and $M_2=M\geq 2$, the resulting bimodal wave solutions are termed {\em Wilton ripples of order $M$} \cite{doi:10.1080/14786440508635350}. Throughout we restrict attention to $M_1=1$. 

In the first quadrant of the $(\beta,\kappa)$-plane, we define the region corresponding to subcritical waves, for which $\mu_0(\kappa;\beta)<1$, denoted by $S_3$, and the region corresponding to supercritical waves, for which $\mu_0(\kappa;\beta)>1$. In the region for supercritical waves, there exist $\kappa_1$ and $\kappa_2$ such that  
\[
\mu_0(\kappa_1;\beta)=\mu_0(\kappa_2;\beta)
\quad \text{and}\quad 0<\kappa_1<\kappa_2<\infty.
\]
Accordingly, we define the regions corresponding to $2\pi/\kappa_1$- and $2\pi/\kappa_2$-periodic waves, denoted by $S_1$ and $S_2$, respectively. See the right panel of Figure \ref{figure3}. 

In Sections \ref{sec:linearize} and \ref{sec:L}, we formulate the spectral stability problem for \eqref{eqn:ww}-\eqref{eqn:bd}, linearized about a steady periodic capillary-gravity wave of sufficiently small amplitude, in the operator form
\[
\mathcal{L}(\eps)\mathbf{u}:=\mathbf{u}_x-(\mathbf{L}(\lambda)+\mathbf{B}(x;\lambda,\eps))\mathbf{u}=0, \quad \lambda\in\mathbb{C},
\]
where $\frac{\partial}{\partial x}-\mathbf{L}(\lambda)$, in \eqref{def:L}, is the $O(1)$ part of $\mathcal{L}(\eps)$ and $\mathbf{B}(x;\lambda,\eps)$ is $O(\eps)$ for $\eps\in\mathbb{R}$ and $|\eps|\ll1$. Spectral stability means ${\rm spec}(\mathcal{L}(\eps)) \subset i\mathbb{R}$.  
In Section \ref{sec:eps=0}, we examine ${\rm spec}(\mathcal{L}(0))=i\mathbb{R}$, defining the generalized eigenspace of $\mathbf{L}(i\sigma)$ and selecting an ordered basis for each $\sigma\in \mathbb{R}$. See Lemma \ref{lem:eps=0} and Definition \ref{def:proj} for details. An important difference from Stokes' gravity waves is that  when $\mu_0(\kappa;\beta)>1$, each spectral branch for $\mathcal{L}(0)$ (see \eqref{def:sigma}) admits two critical values, denoted by $-\sigma_{c,1}$ and $\sigma_{c,2}$ for the $-$ branch.  

In Section \ref{sec:reduction}, when $\eps\in\mathbb{R}$ and $0<|\eps|\ll1$, let $\lambda=i\sigma+\delta$, where $\delta\in\mathbb{C}$ and $|\delta|\ll1$. We apply Mielke's center manifold reduction theorem \cite{Mielke;reduction} to 
\[
\mathbf{u}_x=\mathbf{L}(i\sigma)\mathbf{u}+\mathbf{B}(x;\sigma,\delta,\eps)\mathbf{u}
\]
in infinite dimension (see \eqref{eqn:LB;delta}), to arrive at a matrix ODE
\[
\mathbf{v}_x=\mathbf{A}(x;\sigma,\delta,\eps)\mathbf{v}
\]
in {\em finite} dimension (see \eqref{eqn:A}). 
This enables us to define the monodromy matrix and periodic Evans function for \eqref{eqn:A}, denoted by $\mathbf{V}(T;\sigma,\delta,\eps)$ and $D(i\sigma+\delta,k,\eps)$, respectively, where $T=2\pi/\kappa$ is the period and $k\in\mathbb{R}$ is the Floquet exponent. Importantly, \(\lambda\in{\rm spec}(\mathcal L(\eps))\) if and only if \(D(\lambda,k,\eps)=0\) for some \(k\in\mathbb R\), for $\lambda\in\mathbb{C}$ near $i\sigma \in i\mathbb{R}$ for $\eps\in\mathbb{R}$ sufficiently small.

In Section \ref{sec:expansion_monodromy}, we explain how to compute power series expansions of the center variable $\mathbf{v}$, in finite dimension, and the transverse variable $\mathbf{w}$ (see \eqref{def:u2}) and, hence, a power series expansion of the monodromy matrix, for $\delta\in\mathbb{C}$, $\eps\in\mathbb{R}$ and $|\delta|,|\eps|\ll1$. This is the most computationally intensive part of the analysis, and we carry it out with the aid of the MATLAB Symbolic Math Toolbox. It would be neither practical nor illuminating to record all the resulting expressions. Instead we collect the key results in Sections \ref{sec:high freq instability}-\ref{sec:wilton instability} and Appendices \ref{app:V(m,n) unimodal}, \ref{app:V(m,n) M=3}, and \ref{app:V(m,n) M=2}. We emphasize that the abstract framework outlined in Section \ref{sec:spec} applies equally well to steady unimodal capillary-gravity waves and bimodal Wilton ripples. 

Our analysis is organized around resonances along the imaginary axis. The authors' previous work \cite{HY2023}, together with earlier numerical studies (see, for instance, \cite{McLean;finite-depth, FK, DO}, and also \cite{MS}), indicates that the spectrum for the gravity wave problem may bifurcate from the imaginary axis as the amplitude increases, implying instability, only when
\[
k_j(\sigma)-k_{j'}(\sigma)=N\kappa \quad \text{for some $N\in\mathbb{Z}^+$},
\]
where $ik_j(\sigma)$ and $ik_{j'}(\sigma)$ are eigenvalues of $\mathbf{L}(i\sigma)$. We then say that $i\sigma\in i\mathbb{R}$ is {\em resonant of order $N$} associated with $k_j(\sigma)$ and $k_{j'}(\sigma)$. In Section \ref{resonances}, we identify all such resonances for steady periodic capillary-gravity waves.

Surface tension produces a rich resonance structure with no analogue in the gravity wave problem. For example, the capillary-gravity wave problem can admit nonzero resonances of order $1$, whereas all nonzero resonances for the gravity wave problem are of orders $\geq 2$. This has an important consequence: the unstable spectrum near a nonzero resonance of order $1$ has size $O(\eps)$ as $\eps\to0$ (Corollary \ref{cor:N=1}), whereas the unstable spectrum near $0\in\mathbb{C}$, corresponding to modulational instability, if present, is smaller. In other words, non-modulational instability can dominate modulational instability. Moreover, the critical values of the zero-amplitude spectral curves, $-\sigma_{c,1}$ and $\sigma_{c,2}$, can themselves be {\em repeated} resonances associated with two pairs of Floquet exponents. 

In Section \ref{sec:stability}, we apply the Weierstrass preparation theorem to the periodic Evans function and show that, away from resonances, ${\rm spec}(\mathcal{L}(\eps))$ remains confined to the imaginary axis for $\eps\in\mathbb{R}$ and $|\eps|\ll1$. At resonances, the spectrum may instead bifurcate from the imaginary axis, giving rise to instability. This new result also provides a rigorous foundation for the periodic Evans function approach in \cite{HY2023} for Stokes' gravity waves.

Since $D(i\sigma,k_j(\sigma),0)=0$ whenever $i\sigma \in i\mathbb{R}$ is resonant associated with $k_j(\sigma)$ (Lemma \ref{lem:dispersion}), we examine the roots of $D(i\sigma+\delta,k_j(\sigma)+\gamma,\eps)$ for $\delta\in\mathbb{C}$, $\gamma,\eps\in\mathbb{R}$ and $|\delta|,|\gamma|,|\eps|\ll1$. The Weierstrass polynomial associated with the periodic Evans function is quadratic for each $i\sigma\neq 0,\in i\mathbb{R}$ treated here, for both steady unimodal capillary-gravity waves and bimodal Wilton ripples. See Sections \ref{sec:N=1}-\ref{sec:N=3}, \ref{sec:sigma N=1}-\ref{sec:sigma N=2}, \ref{sec:M=2 N=2}, \ref{sec:M=3 N=1,M}-\ref{sec:M=3 1<N<M} and, particularly, \eqref{def:W(l,m,n)} and \eqref{def:W(l,m,n)c N=1}. By contrast, the Weierstrass polynomial for $i\sigma=0$, corresponding to modulational instability, has degree $4$ or $6$. 
See, for instance, \eqref{def:W} and \eqref{eqn:evans_wilton2}. 

The quadratic Weierstrass polynomial permits a complete characterization of its roots. The symmetries of the Weierstrass polynomial and its roots imply that spectral stability or instability near a nonzero resonance is determined by the sign of its discriminant, denoted by $\Delta(\gamma,\eps)$, in \eqref{def:disc1}, which is a real-valued function of $\gamma,\eps\in\mathbb{R}$ for $|\gamma|,|\eps|\ll1$, and analytic at $(0,0)$. We then apply the Weierstrass preparation theorem to $\Delta(\gamma,\eps)$ for fixed $\eps\in\mathbb{R}$ and $|\eps|\ll1$. This yields a quadratic Weierstrass polynomial in $\gamma$ (see \eqref{def:W(m,n)2}). Lemma \ref{lem:weierstrass2} establishes that the sign of \eqref{def:disc1} is opposite to the sign of \eqref{def:W(m,n)2}. The discriminant of \eqref{def:W(m,n)2}, denoted by $\Delta_2(\eps)$, in \eqref{def:disc2}, is real-valued function of $\eps\in\mathbb{R}$ for $|\eps|\ll1$, analytic at $0$, and $\Delta_2(0)=0$. Therefore, spectral stability or instability for sufficiently small amplitude is determined by the coefficient of the lowest-order nonvanishing term in the Taylor series of \eqref{def:disc2}, referred to as the {\em spectral stability index}, following the terminology introduced in \cite{HY2023} for Stokes' gravity waves. The application of the Weierstrass preparation theorem to the discriminant of the Weierstrass polynomial associated with the periodic Evans function is one of the key new ingredients of the analysis.

This complete characterization of the roots of a quadratic Weierstrass polynomial also strengthens the result in \cite{HY2023} for Stokes' gravity waves near a nonzero resonance of order $2$, where positivity of the spectral stability index (see \cite[(6.34)]{HY2023}) implies instability, whereas negativity merely implies that the spectrum remains on the imaginary axis up to order $\eps^2$ as $\eps\to0$. The possibility of instability at some higher order of $\eps$ thus remains open. The second application of the Weierstrass preparation theorem closes this gap, and negativity of the index implies spectral stability for sufficiently small amplitude. See Remark \ref{rem:Stokes} for further discussion.

The authors' previous work \cite{HY2023} for Stokes' gravity waves provided no spectral information for resonances of order $\geq3$. On the other hand, \cite{creedon_deconinck_trichtchenko_2022} carried out formal asymptotic expansions of the spectrum to third order with respect to perturbation parameters, including the Floquet exponent, while \cite{berti2025} extended the analytical framework, developed in \cite{Berti2022} for the Benjamin--Feir instability, to determine spectral stability near resonances of arbitrary order. In Section \ref{sec:N=3}, we investigate spectral stability near resonances of order $\geq 3$ within the present framework. We show that \eqref{def:disc2} vanishes at order $\eps^4$ as $\eps\to0$ for resonances of order $3$ (see \eqref{eqn:ind2 N=3}), whence stability or instability would be determined by the next nonvanishing term in its Taylor expansion. This would require higher order terms in the power series expansion of the monodromy matrix. Such computations become substantially more involved and are not pursued. (We compute $\mathbf{V}^{(m,n)}(T;\sigma)$, in \eqref{def:V(m,n)}, for $m+n\leq 2$, which is sufficient to determine spectral stability and instability near the resonances listed in Table \ref{table-indexes}.) Nonetheless, our analysis provides a rigorous justification of the formal expansion with respect to the Floquet exponent in \cite[(5.1c)]{creedon_deconinck_trichtchenko_2022}. We also formulate a conjectural expansion for \eqref{def:disc2} at resonances of arbitrary order.

The modulational instability analysis is fundamentally different. At $0\in\mathbb{C}$, which is itself a resonance, the Weierstrass polynomial associated with the periodic Evans function has degree $4$ or $6$. Polynomials with real coefficients factor over $\mathbb{R}$ into irreducible polynomials of degree $\leq 2$, suggesting a route to a tractable analysis. For Stokes' gravity waves, \cite{Berti2022,Berti2023} performed a sequence of block diagonalizations, reducing a $4\times4$ matrix to two $2\times2$ blocks, from which the relevant quadratic form emerges. Algebraically, this procedure is equivalent to the factorization suggested here. Particularly, \cite[Theorem 1.1]{Berti2023} implies that, for Stokes waves in finite depth, the potentially unstable roots of the quartic Weierstrass polynomial enjoy the asymptotic expansion in \cite[(5.27)]{HY2023}. We make the analogous expansion (see \eqref{def:lambda(gamma,eps)}) for steady unimodal capillary-gravity waves, and exploit \eqref{def:lambda(gamma,eps)} to derive the explicit modulational instability index, in \eqref{def:ind_5}, which is in agreement with the formal asymptotic calculation, for instance, in \cite{djordjevic_redekopp_1977}. The validity of \eqref{def:lambda(gamma,eps)} has subsequently been established in \cite{sun2025,hsiao2025}. 

A rigorous treatment of modulational instability for Wilton ripples would be of great interest, including an application of the methods of \cite{sun2025,hsiao2025}. 
Without a rigorous justification of the asymptotic expansions in \eqref{def:lambda(gamma,eps)_wiltonm} or \eqref{eqn:Wilton exp2}, we interpret our modulational instability results in Section \ref{sec:wilton modulational} as partly formal. We emphasize that, by contrast, our analysis near nonzero resonances is rigorous for both steady unimodal capillary-gravity waves and bimodal Wilton ripples. Other rigorous analyses of spectral stability away from $0\in\mathbb{C}$ include \cite{berti2024isolamodulationalinstabilitystokes,berti2025} for the surface water wave problem and \cite{Noble_2023} for the electronic Euler--Poisson system.

\subsection*{Summary of the main results} 

Our analysis near nonzero resonances both for steady unimodal capillary-gravity waves and bimodal Wilton ripples is organized by the resonance order, by whether the resonance is repeated at $\sigma=-\sigma_{c,1}$ or $\sigma_{c,2}$, and by the order of Wilton ripples. The spectral stability indices are in Theorems \ref{thm:N=1}, \ref{thm:N=2}, \ref{thm:sigma N=1}, \ref{thm:sigma N=2}, \ref{thm:M=2 N=2}, \ref{thm:M=3 N=1}, \ref{thm:M=3 N=M}, and \ref{thm:M=3 1<N<M}. See also Table \ref{table-indexes}. 
We further show that, to leading order of $\eps$ $(\to0)$, the unstable spectrum near a nonzero resonance traces either an ellipse or a circle, depending on whether the resonance is repeated. The center of the ellipse or circle is the resonance itself for resonances of order $1$ while the center is shifted along the imaginary axis by a distance of $O(\eps^2)$ for order $2$. See Corollaries \ref{cor:N=1}, \ref{cor:N=2}, \ref{cor:sigma N=1}, \ref{cor:sigma N=2}, \ref{cor:M=2 N=2}, \ref{cor:N=3 M=N}, and \ref{cor:M=3 1<N<M}.  

Several features different from Stokes' gravity waves must be discussed. 
The analysis distinguish unimodal waves from bimodal Wilton ripples, and modulational instability near the origin of the complex plane from spectral instability away from $0\in\mathbb{C}$ and, additionally, repeated resonances. Depending on the location of the resonance, the periodic Evans function admits six roots, counted with multiplicity, if $\sigma\in(0,\min(-\sigma_{c,1},\sigma_{c,2}))$; four roots if $\sigma\in(\min(-\sigma_{c,1},\sigma_{c,2}),\max(-\sigma_{c,1},\sigma_{c,2}))$; and two roots if $\sigma\in(\max(-\sigma_{c,1},\sigma_{c,2}),\infty)$. For Wilton ripples of order $\geq2$, additional higher harmonics enter the Stokes expansion (see, for instance, \eqref{def:wilton1 M=2} and \eqref{def:wilton1 M=3}), making the interactions between the unperturbed solutions and perturbations considerably more intricate. Together, these features contribute to more than ten spectral stability indices. Each index is an explicit symbolic expression in $\beta$, $\kappa$, $\sigma$, and $k_j(\sigma)$ and $k_{j'}(\sigma)$. Table \ref{table-indexes} provides a compact guide to these indices and the corresponding theorems.

\begin{table}[htbp]
\begin{center}
\begin{tabular}{ S{c} S{c} S{c} S{c} }
\hline 
Wave family & Resonance and order & Index & Theorem \\
\hline\hline
\multirow{6}{*}{\begin{tabular}[c]{@{}c@{}}steady unimodal\\ capillary-gravity\end{tabular}}
& $\sigma\neq-\sigma_{c,1},\sigma_{c,2}$ and $N=1$&\eqref{def:ind_1} &\ref{thm:N=1}\\ 
& $\sigma\neq-\sigma_{c,1},\sigma_{c,2}$ and $N=2$& \eqref{def:ind_2'} & \ref{thm:N=2} \\
& $\sigma=-\sigma_{c,1},\sigma_{c,2}$ and $N=1$&\eqref{def:ind_3} &\ref{thm:sigma N=1} \\ 
&$\sigma=-\sigma_{c,1},\sigma_{c,2}$ and $N=2$& \eqref{def:ind_4} &\ref{thm:sigma N=2}\\
&$\sigma\neq 0$ and $N\geq 3$&N.A.&N.A.\\
& $\sigma=0$& \eqref{def:ind_5} & \ref{thm:modulation}\\ 
\hline\multirow{4}{*}{\begin{tabular}[c]{@{}c@{}}Wilton ripples\\ of order $2$\end{tabular}}
 & $\sigma\neq-\sigma_{c,1},\sigma_{c,2}$ and $N=2$ & \eqref{def:ind_6} &\ref{thm:M=2 N=2} \\
 & $\sigma\neq-\sigma_{c,1},\sigma_{c,2}$ and $N\geq 3$ &N.A.& N.A. \\
 & $\sigma=-\sigma_{c,1},\sigma_{c,2}$ &N.A.&N.A. \\
 & $\sigma=0$ & \eqref{def:ind_11} &\ref{thm:wilton2} \\ \hline
\multirow{9}{*}{\begin{tabular}[c]{@{}c@{}}Wilton ripples\\ of order $M\geq3$\end{tabular}}  & $\sigma\neq-\sigma_{c,1},\sigma_{c,2}$ and $N= 1$ & \eqref{def:ind_7} &\ref{thm:M=3 N=1} \\  
& $\sigma\neq-\sigma_{c,1},\sigma_{c,2}$ and $N= M$ & \eqref{def:ind_7'} & \ref{thm:M=3 N=M}\\ 
& \begin{tabular}[c]{@{}c@{}}$\sigma\neq-\sigma_{c,1},\sigma_{c,2}$ and \\ $N=2,M\pm1,2M $\end{tabular} & \eqref{def:ind_8} & \ref{thm:M=3 1<N<M}\\  
& \begin{tabular}[c]{@{}c@{}}$\sigma\neq-\sigma_{c,1},\sigma_{c,2}$ and \\ $N\neq 1,2,M\pm1,M,2M $\end{tabular} &N.A.& N.A.\\ 
& $\sigma=-\sigma_{c,1},\sigma_{c,2}$ &N.A.& N.A.\\  
& $\sigma=0$ & \begin{tabular}[c]{@{}c@{}}\eqref{def:ind_9},\\ \eqref{def:ind_10}\end{tabular} &\ref{thm:wiltonm} \\ \hline
\end{tabular}
\end{center}
\caption{Spectral stability indices and theorems for steady unimodal capillary-gravity waves and bimodal Wilton ripples near resonances $i\sigma \in i\mathbb{R}$ of order $N$ at and away from $0\in\mathbb{C}$. It suffices to consider $\sigma\geq 0$. N.A. stands for ``Not Available''. See  Section \ref{sec:N=3} for further discussion.}
\label{table-indexes}
\end{table}

We numerically evaluate the spectral stability indices in Table \ref{table-indexes}, using floating-point arithmetic, to locate instability thresholds. See Sections \ref{sec:numerics N=2}, \ref{sec:numerics sigma}, \ref{sec:numerics M=2 N=2}, and also Sections \ref{sec:M=3 N=1,M}, \ref{sec:M=3 1<N<M}, \ref{sec:modulational M=3}, and \ref{sec:modulational M=2}. The numerical results for spectral stability and instability away from $0\in\mathbb{C}$ suggest a practical sign rule as detailed in Section \ref{sec:numerics N=2}. As in \cite{HY2023}, one can replace floating-point arithmetic with interval arithmetic to rigorously verify the signs of the spectral stability indices. We do not pursue this.
We combine the numerical results of Sections \ref{sec:numerics N=2}, \ref{sec:numerics sigma}, and \ref{sec:modulational} for steady unimodal capillary-gravity waves into the stability diagrams in Figure \ref{figure1}, which shows parameter regimes for either modulational instability or spectral instability near nonzero resonances of order $1$ and $2$. Throughout, shaded regions in the figures indicate numerical spectral instability. We do not construct stability diagrams for bimodal Wilton ripples because the domain of Wilton ripples for each order is a one-dimensional curve in the $(\beta,\kappa)$-plane.

\begin{figure}[htbp]
    \centering
    \includegraphics[width=\textwidth]{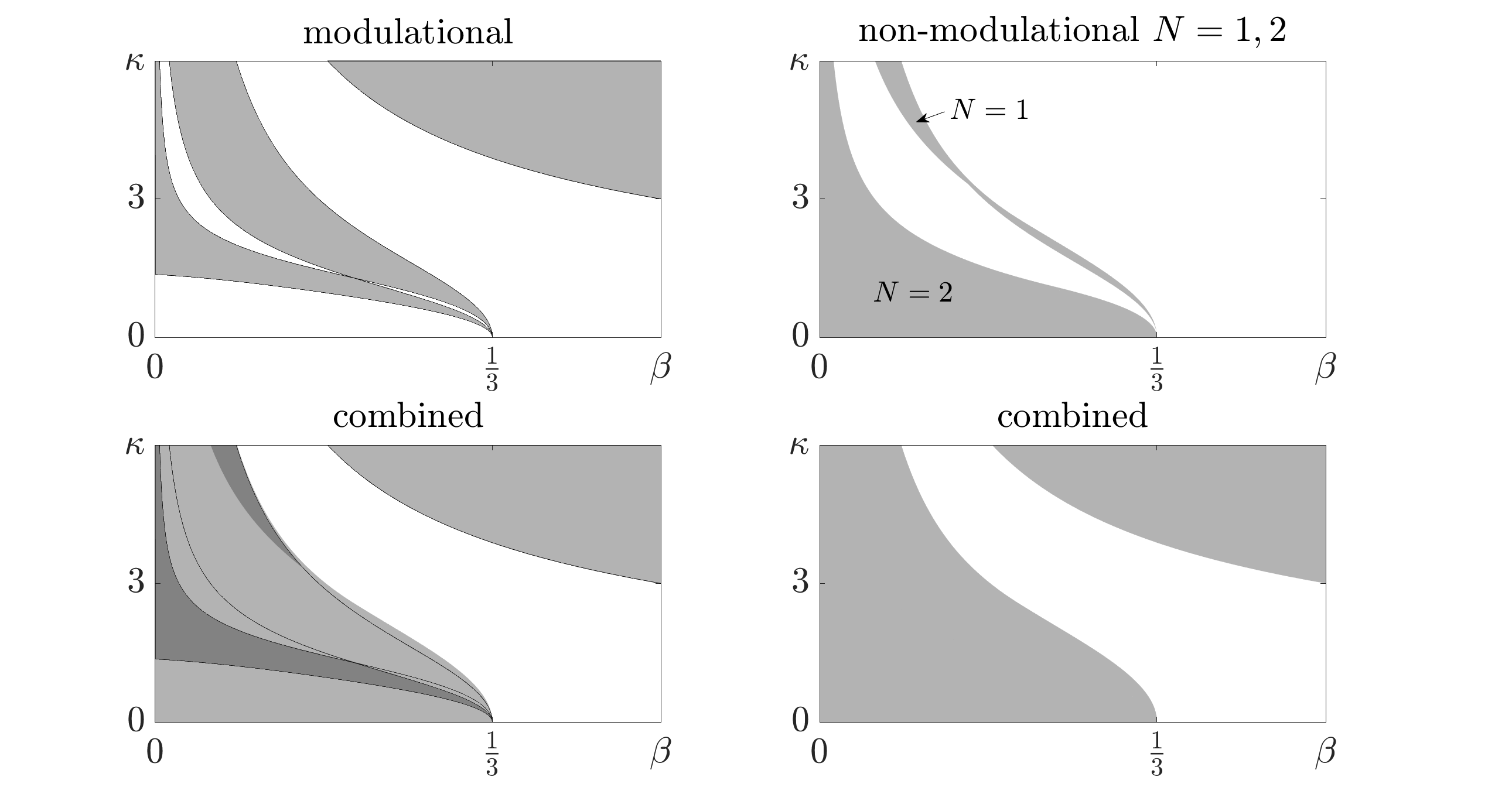}
    \caption{Spectral instability at and away from $0\in\mathbb{C}$ for steady unimodal capillary-gravity waves. \protect\\
    \protect\hspace*{11pt}Top left: The shaded region indicates numerical modulational instability. See also Figure \ref{figure15}. 
    Top right: The shaded region indicates numerical spectral instability near nonzero resonances of order $N=1$ and $2$, termed non-modulational instability. \protect\\
    \protect\hspace*{11pt}Bottom left: The shaded region indicates either modulational or non-modulational instability. The darker region highlights their intersections, meaning both instabilities. The solid curves within the shaded region make the boundaries between subregions exhibiting only modulational instability and subregions exhibiting only non-modulational instability. Bottom right: The shaded region combines modulational and non-modulational instability. The unshaded region means spectral stability both near $0\in\mathbb{C}$ and nonzero resonances of orders $1$ and $2$.}
    \label{figure1}
\end{figure}

In the unshaded region in the bottom panels of Figure \ref{figure1}, the numerical computations indicate that the spectrum remains confined to the imaginary axis near $0\in\mathbb{C}$ and near nonzero resonances of orders $1$ and $2$ for steady unimodal capillary-gravity waves, at least for $\kappa\in(0, 6]$, our computational window and the range displayed in the figures. Our index computations do not address spectral stability and instability near resonances of order $\geq 3$. On the other hand, recently, Sun and Wahlén \cite{sun2025} employed a Krein signature argument to resolve this remaining question for subcritical waves. See Remark \ref{rem:Krein}. 


More precisely, \cite{sun2025} established that when $\mu_0(\kappa;\beta)<1$ (see \eqref{def:mu0}), that is, when the Froude number $>1$, modulational stability implies spectral stability throughout $\mathbb{C}$. The proof relies on that for such subcritical waves, each nonzero resonance is associated with Floquet exponents with the same Krein signature and, consequently, the spectrum remains confined to the imaginary axis in the vicinity of each such resonance. Our numerical results corroborate this. Indeed, we do not numerically detect spectral instability near nonzero resonances in the $S_3$ region in the $(\beta,\kappa)$-plane, corresponding to subcritical waves (see the right panel of Figure \ref{figure3}). See Section \ref{sec:numerics N=2} for further discussion. 

To determine spectral stability for supercritical waves in the unshaded portion of $S_2$ lying outside the instability region, we verify in Sections \ref{S1S2region} and \ref{resonance_summary} that when $(\beta,\kappa)\in S_2$, no nonzero resonance of order $\geq 2$ is associated with Floquet exponents of opposite Krein signatures. Such resonances therefore cannot give rise to spectral instability. Consequently, establishing spectral stability throughout $\mathbb{C}$ for $2\pi/\kappa_2$-periodic supercritical waves reduces to excluding modulational instability and spectral instability near resonances of order $1$.

The explicit computations of indices, together with the Krein signature conclusions of \cite{sun2025}, are summarized in the stability diagram in Figure \ref{figure2}. The unshaded region corresponds to spectral stability throughout $\mathbb{C}$, whereas the shaded region corresponds to spectral instability either near the origin of the complex plane through modulational instability or away from $0\in\mathbb{C}$ near nonzero resonances of order $1$ or $2$. In $S_3$, corresponding to subcritical waves, the transition between spectral stability and instability occurs along the curve $\beta=\beta_5(\kappa)$, where $\beta_5(\kappa)$ is the positive real root of ${\rm ind}_{5,3}$. See Section \ref{sec:modulational supercritical} and Appendix \ref{app:ind}. In $S_2$, for $2\pi/\kappa_2$-periodic supercritical waves, on the other hand, the transition occurs along the boundary of the instability region associated with nonzero resonances of order $1$. See also the top right panel of Figure \ref{figure1}.

\begin{figure}[htbp]
    \centering
    \includegraphics[scale=0.3]{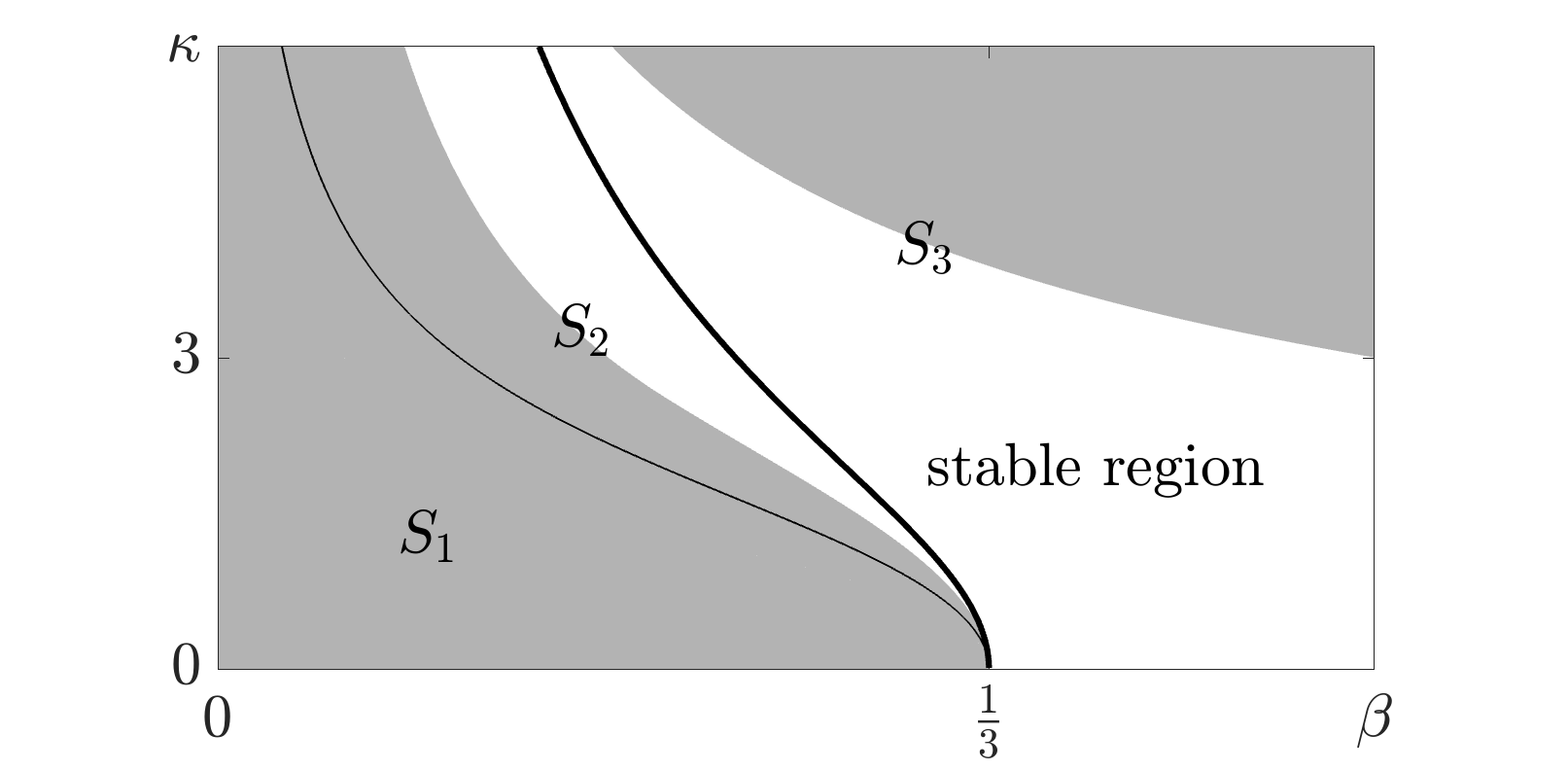}
    \caption{$\kappa$ versus $\beta$. The $S_1$ region is bounded from above by the curve $\beta=\beta_{1,2}(\kappa)$ (see \eqref{def:beta12}) while the $S_2$ region is bounded from below by $\beta=\beta_{1,2}(\kappa)$ and from above by $\beta=\beta_{2,3}(\kappa)$ (see \eqref{def:beta23}). The $S_3$ region, corresponding to subcritical waves, is bounded from below by $\beta=\beta_{2,3}(\kappa)$. In the unshaded region within $S_3$, spectral stability throughout $\mathbb{C}$ was proved in \cite{sun2025}, rather than near resonances of orders $1$ and $2$. In $S_2$, each nonzero resonance of order $\geq 2$ is associated with Floquet exponents of the same Krein signature and cannot give rise to spectral instability. The unshaded region additionally excludes modulational instability and instability near nonzero resonances of order $1$.}
    \label{figure2}
\end{figure}

We numerically track these two boundary curves up to $\kappa=100$ and find that they do not intersect, suggesting that the region for spectral stability is unbounded in the $(\beta,\kappa)$-plane, extending toward both the infinite surface-tension limit as $\beta\to \infty$ and the infinite short-wave limit as $\kappa\to \infty$. 

\newpage

\section{Steady periodic capillary-gravity waves}\label{sec:Stokes}

Let $\beta>0$. We seek temporally stationary and spatially periodic solutions to \eqref{eqn:ww}-\eqref{eqn:bd} for sufficiently small amplitude, that is, 
\begin{subequations}\label{stationary-periodic}
\begin{equation}\label{stationary-periodic:a}
\phi,\quad \eta,\quad u=\phi_x-\frac{y\eta_x\phi_y}{1+\eta},\quad\text{and}\quad 
z =\frac{\beta\eta_x}{(1+\eta_x^2)^{1/2}}, 
\end{equation}
satisfying:
\begin{align}
&u_x-\frac{y\eta_xu_y}{1+\eta}+\frac{\phi_{yy}}{(1+\eta)^2}=0&&\text{for $0<y<1$,}\\
&\phi_y=0&&\text{at $y=0$,}\\
&(u-1)\eta_x-\frac{\phi_y}{1+\eta}=0&&\text{at $y=1$},\label{stationary-periodic:K}\\
&\mu z _x+u-\frac{u^2}{2}-\mu\eta-\frac{\phi_y^2}{2(1+\eta)^2}+q=0&&\text{at $y=1$},\label{stationary-periodic:D}
\end{align}
\end{subequations}
for some $q$ and $\mu$. Here, \eqref{stationary-periodic:a} follows from the first and third equations of \eqref{eqn:ww}, and \eqref{stationary-periodic:D} follows from that
\[
\mu z _x+u-\frac{u^2}{2}-\mu\eta-\frac{(u-1)\eta_x\phi_y}{1+\eta}+\frac{\phi_y^2}{2(1+\eta)^2}+q=0\quad\text{at $y=1$},
\]
by the fourth equation of \eqref{eqn:ww}, after substituting \eqref{stationary-periodic:K}.

Throughout,
\[
\text{$\eps\in\mathbb{R}$ represents the dimensionless amplitude parameter}.
\]
Suppose:
\begin{equation}\label{eqn:stokes exp}
\begin{aligned}
&\phi(x,y;\eps)=\phi_1(x,y)\eps+\phi_2(x,y)\eps^2+\phi_3(x,y)\eps^3+O(\eps^4), \\
&\eta(x;\eps)=\eta_1(x)\eps+\eta_2(x)\eps^2+\eta_3(x)\eps^3+O(\eps^4),\\
&q(\eps)=q_1\eps+q_2\eps^2+q_3\eps^3+O(\eps^4),\\
&\mu(\eps)=\mu_0+\mu_1\eps+\mu_2\eps^2+\mu_3\eps^3+O(\eps^4),
\end{aligned}
\end{equation}
as $\eps\to 0$ for some $\phi_1$, $\phi_2$, $\phi_3$, $\dots$, $\eta_1$, $\eta_2$, $\eta_3$, $\dots$, $q_1$, $q_2$, $q_3$, $\dots$, and $\mu_0$, $\mu_1$, $\mu_2$, $\dots$ and, hence,
\[
\begin{aligned}
&u(x,y;\eps)=u_1(x,y)\eps+u_2(x,y)\eps^2+u_3(x,y)\eps^3+O(\eps^4),\\
&z(x;\eps)= z _1(x)\eps+ z _2(x)\eps^2+ z _3(x)\eps^3+O(\eps^4),\\
\end{aligned}
\]
as $\eps\to 0$, where $u_1$, $u_2$, $u_3$, $\dots$ and $z _1$, $z _2$, $z _3$, $\dots$ can be determined in terms of $\phi_1$, $\phi_2$, $\phi_3$, $\dots$ and $\eta_1$, $\eta_2$, $\eta_3$, $\dots$ using the third and fourth equations of \eqref{stationary-periodic:a}. We assume that $\phi_1$, $\phi_2$, $\phi_3$, $\dots$, $\eta_1$, $\eta_2$, $\eta_3$, $\dots$, $u_1$, $u_2$, $u_3$, $\dots$, and $z _1$, $z _2$, $z _3$, $\dots$ are $T$-periodic in $x$, where
\[
\text{$T=2\pi/\kappa$ and $\kappa>0$ is the wave number}.
\]
Additionally, we assume that $\phi_1$, $\phi_2$, $\phi_3$, $\dots$ are odd in $x$ and that $\eta_1$, $\eta_2$, $\eta_3$, $\dots$ are even and of mean zero over one period. 

The existence of small amplitude solutions to \eqref{stationary-periodic}-\eqref{eqn:stokes exp} is well established through rigorous analysis. See, for instance, \cite{jones_1989}. See also \cite{Jones1986, 10.2307/2397696} for infinite depth. Particularly, \eqref{eqn:stokes exp} converges as $\eps\to0$, for instance, in $H^{3+}(\mathbb{R}/T\mathbb{Z}\times (0,1))\times H^{7/2+}(\mathbb{R}/T\mathbb{Z})\times\mathbb{R}\times\mathbb{R}$ \cite{NR;anal}. Consequently, $\phi(\eps)$, $\eta(\eps)$, $q(\eps)$, $\mu(\eps)$ and, hence, $u(\eps)$ and $z(\eps)$ depend analytically on $\eps\in\mathbb{R}$ for $|\eps|\ll1$.

We proceed to compute $\phi_1$, $\phi_2$, $\phi_3$, $\dots$, $\eta_1$, $\eta_2$, $\eta_3$, $\dots$, $q_1$, $q_2$, $q_3$, $\dots$, and $\mu_0$, $\mu_1$, $\mu_2$, $\dots$. 

\begin{notation*}
Throughout we adopt the shorthand notation:
\begin{align*}
&\s(\cdot)=\sin(\cdot),& &\c(\cdot)=\cos(\cdot), & & \\
&\sh(\cdot)=\sinh(\cdot),& &\ch(\cdot)=\cosh(\cdot),& &\th(\cdot)=\tanh(\cdot).
\end{align*}
\end{notation*}

\subsection{Unimodal waves and parameter regimes}\label{sec:unimodal}

Substituting \eqref{eqn:stokes exp} into \eqref{stationary-periodic} and collecting terms of order $\eps$, we gather:
\begin{align*} 
&{\phi_1}_{xx}+{\phi_1}_{yy}=0&&\text{for $0<y<1$},\\
&{\phi_1}_y=0&&\text{at $y=0$},\\
&{\eta_1}_x+{\phi_1}_y=0&&\text{at $y=1$},\\
&\mu_0\beta{\eta_1}_{xx}+{\phi_1}_x-\mu_0\eta_1+q_1=0&&\text{at $y=1$}.
\end{align*}
Recall that $\phi_1$ and $\eta_1$ are $T(=2\pi/\kappa)$-periodic in $x$. Recall that $\phi_1$ is odd in $x$ and that $\eta_1$ is even and of mean zero over the period. We solve this, for instance, by separation of variables to obtain: 
\begin{gather}
\phi_1(x,y)=\s(\kappa x)\ch(\kappa y), 
\quad \eta_1(x)=\sh(\kappa)\c(\kappa x),\quad q_1=0,\label{def:stokes1} 
\intertext{provided}
\mu_0(\beta \kappa^2+1)\th(\kappa)-\kappa=0.\label{eqn:muo} 
\end{gather}
We define
\be\label{def:mu0}
\mu_0(\kappa;\beta)=\frac{\kappa}{(\beta \kappa^2+1)\th(\kappa)},
\ee 
so that \eqref{eqn:muo} becomes 
\[
\mu_0=\mu_0(\kappa;\beta).
\]
Clearly, $\mu_0(\kappa;\beta)\to1$ as $\kappa\to 0^+$ and $\mu_0(\kappa;\beta)\to0$ as $\kappa\to\infty$.

\begin{lemma}[$\mu_0(\cdot\,;\beta)$ for $\beta\geq1/3$ versus $0<\beta<1/3$]\label{lem:mu0}
If $\beta\geq1/3$ then \eqref{def:mu0} is a strictly decreasing function for $\kappa\in(0,\infty)$. 

If $0<\beta<1/3$, on the other hand, \eqref{def:mu0} has a critical point, denoted by $\kappa_c$, and $\mu_0(\kappa;\beta)$ is strictly increasing for $\kappa\in(0,\kappa_c)$ while strictly decreasing for $\kappa\in(\kappa_c,\infty)$.
\end{lemma}

\begin{proof}
Since 
\[
\frac{d\mu_0}{d\kappa}(\kappa;\beta)=-\frac{\kappa^3(1-\th(\kappa)^2)+\kappa^2\th(\kappa)(\beta-\beta_{1,2}(\kappa))}{(\beta \kappa^2+1)^2\th(\kappa)^2},
\]
where
\be \label{def:beta12}
\beta_{1,2}(\kappa):=\frac{\sh(2\kappa)-2\kappa}{\kappa^2(\sh(2\kappa)+2\kappa)},
\ee 
the sign of $\frac{d\mu_0}{d\kappa}(\kappa;\beta)$ is opposite to the sign of $\beta-\beta_{1,2}(\kappa)$. Clearly, $\beta_{1,2}(\kappa)\to1/3$ as $\kappa\to0^+$ and $\beta_{1,2}(\kappa)\to0$ as $\kappa\to\infty$. It suffices to show that \eqref{def:beta12} is strictly decreasing throughout the interval $(0,\infty)$. 

Since
\[
\frac{d\beta_{1,2}}{d\kappa}(\kappa)=-\frac{2(e^{2\kappa} + 1)}{\kappa^3(e^{4\kappa} + 4\kappa e^{2\kappa} - 1)^2}f(\kappa),
\]
where  
\[
f(\kappa)=e^{6\kappa} +(-8\kappa^2+ 4\kappa  -1) e^{4\kappa} +( - 8\kappa^2- 4\kappa -1)e^{2\kappa} + 1,
\]
and since 
\[
f,\frac{df}{d\kappa},
\frac{d^2f}{d\kappa^2},
\frac{d^3f}{d\kappa^3},
\frac{d^4f}{d\kappa^4},
\frac{d^5f}{d\kappa^5}(0)=0
\quad\text{and}\quad\frac{d^6f}{d\kappa^6}(0)=1024,
\]
the mean value theorem implies: when $\kappa>0$ is sufficiently small,
\[
f(\kappa)=\kappa^{(1)}\kappa^{(2)}\kappa^{(3)}\kappa^{(4)}\kappa^{(5)}\frac{d^6f}{d\kappa^6}(\kappa^{(6)})\kappa
\] 
for some $\kappa^{(\ell)}$, $\ell=1,2,\ldots,6$, such that $0<\kappa^{(6)}<\kappa^{(5)}<\kappa^{(4)}<\kappa^{(3)}<\kappa^{(2)}<\kappa^{(1)}<\kappa$. An INTLAB computation \cite{Ru99a} demonstrates
\begin{equation}\label{eqn:f6}
\frac{d^6f}{d\kappa^6}(\text{infsup}(0,0.001))
=[ 0.48576425206934,1.56320790772994]\times 10^{3},
\end{equation}
where $\text{infsup}(0,0.001)$ expresses an interval rigorously enclosing $[0,0.001]$, taking into account the rounding error, and the right side of \eqref{eqn:f6} is an interval rigorously enclosing ${\rm range}(\frac{d^6f}{d\kappa^6}[0,0.001])$. Consequently, $\frac{d^6f}{d\kappa^6}(\kappa)>0$ for $\kappa\in[0,0.001]$ and, hence, $f(\kappa)>0$ for $\kappa\in(0,0.001]$.

When $\kappa$ is large, for instance, $\kappa>2$, on the other hand, a direct calculation reveals
\begin{align*}
f(\kappa)=(e^{2\kappa}-8\kappa^2)e^{4\kappa}+((4\kappa-1)e^{2\kappa}-8\kappa^2-4\kappa-1)e^{2\kappa}+1>0.
\end{align*}

It remains to show that $f(\kappa)>0$ for $\kappa\in [0.001,2]$. We partition the interval into finitely many subintervals, denoted by $I$, and perform validated numerics, also referred to as rigorous computation, similar to that for \eqref{eqn:f6}, to verify $\inf(f(I))>0$ for each subinterval $I$. We omit the details.

Therefore we conclude that $f(\kappa)>0$ and, hence, $\frac{d\beta_{1,2}}{d\kappa}(\kappa)<0$ for $\kappa\in (0,\infty)$. This completes the proof.
\end{proof}

Lemma \ref{lem:mu0} implies:

\begin{itemize}
    \item[$\bullet$] If $\beta\geq 1/3$ and $0<\mu_0<1$, then $\mu_0(\kappa;\beta)=\mu_0$ has one solution in the interval $(0,\infty)$.
    \item[$\bullet$] If $\beta\geq 1/3$ and $\mu_0\geq 1$, then $\mu_0(\kappa;\beta)=\mu_0$ has no solution.
    \item[$\bullet$] If $0<\beta<1/3$ and $0<\mu_0\leq 1$, then $\mu_0(\kappa;\beta)=\mu_0$ has one solution in the interval $(0,\infty)$.
    \item[$\bullet$] If $0<\beta<1/3$ and $1<\mu_0\leq\mu_0(\kappa_c;\beta)=\max_{\kappa\in(0,\infty)}\mu_0(\kappa;\beta)$, then $\mu_0(\kappa;\beta)=\mu_0$ admits two solutions, denoted by $\kappa_1$ and $\kappa_2$, such that $0<\kappa_1\leq \kappa_c\leq \kappa_2<\infty$.
    \item[$\bullet$] If $0<\beta<1/3$ and $\mu_0>\mu_0(\kappa_c;\beta)$, then $\mu_0(\kappa;\beta)=\mu_0$ has no solution.
\end{itemize}

The left and middle panels of Figure \ref{figure3} provide examples.  

\begin{figure}[htbp]
    \centering
    \includegraphics[scale=0.3]{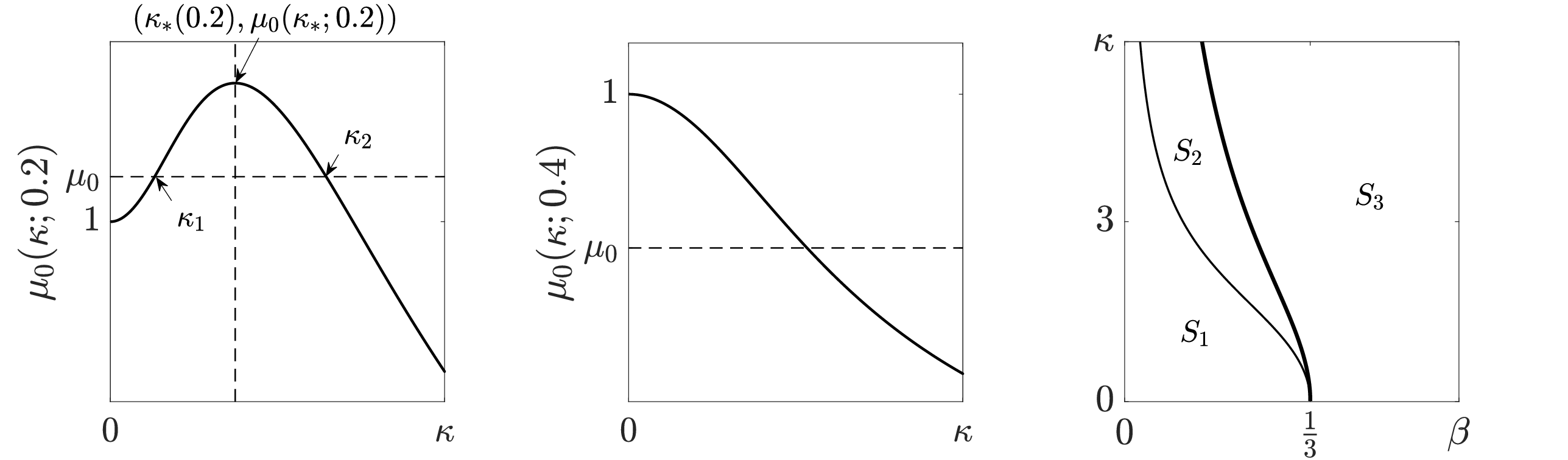}
    \caption{Left: The graph of \eqref{def:mu0} for $\beta=0.2  \,(<1/3)$ has a unique maximum point in the interval $(0,\infty)$. Middle: The graph of \eqref{def:mu0}, $\beta=0.4\,(>1/3)$, is strictly decreasing throughout $(0,\infty)$.\protect\\ 
    \protect\hspace*{11pt} Right: $\kappa$ versus $\beta$ for steady periodic capillary-gravity waves. The region $S_3$, where $0<\mu_0(\kappa;\beta)<1$, corresponding to subcritical waves, is bounded by the curve $\beta=\beta_{2,3}(\kappa)$ along which $\mu_0(\kappa;\beta)=1$. 
    The region for supercritical waves, where $\mu_0(\kappa;\beta)>1$, is divided into the regions $S_1$ and $S_2$, corresponding to $2\pi/\kappa_1$- and $2\pi/\kappa_2$-periodic waves, where $\kappa_1 < \kappa_2$ and $\mu_0(\kappa_1;\beta)=\mu_0(\kappa_2;\beta)$. The regions $S_1$ and $S_2$ are separated by the curve $\beta=\beta_{1,2}(\kappa)$ along which $\frac{d\mu_0}{d\kappa}(\kappa;\beta)=0$.}
    \label{figure3}
\end{figure}

\begin{definition*}
In the first quadrant of the $(\beta,\kappa)$-plane, we define the region corresponding to {\em subcritical} capillary-gravity waves, where $0<\mu_0(\kappa;\beta)<1$, denoted by $S_3$, bounded by the curve
\be \label{def:beta23}
\beta=\beta_{2,3}(\kappa):=\frac{\kappa-\th(\kappa)}{\kappa^2\th(\kappa)}
\ee 
along which $\mu_0(\kappa;\beta)=1$. See the right panel of Figure \ref{figure3}. If $(\beta,\kappa)\in S_3$ then there exists a family of {\em unimodal} wave solutions to \eqref{stationary-periodic}-\eqref{eqn:stokes exp} for $\eps\in\mathbb{R}$ and $|\eps|\ll1$, for which $\phi_1$ and $\eta_1$ are in \eqref{def:stokes1} up to a rescaling of $\eps$, and $\mu_0$ in \eqref{def:mu0}. See, for instance, \cite{jones_1989}. See also \cite{Jones1986, 10.2307/2397696} for infinite depth.

Considering {\em supercritical} capillary-gravity waves, for which  $\mu_0(\kappa;\beta)>1$, on the other hand, $\mu_0(\kappa_1;\beta)=\mu_0(\kappa_2;\beta)$ for $\kappa=\kappa_1$ or $\kappa_2$, satisfying $0<\kappa_1<\kappa_2<\infty$. Accordingly, we define the regions corresponding to $2\pi/\kappa_1$- and $2\pi/\kappa_2$-periodic waves, denoted by $S_1$ and $S_2$, respectively, separated by the curve $\beta=\beta_{1,2}(\kappa)$, in \eqref{def:beta12}, along which $\frac{d\mu_0}{d\kappa}(\kappa;\beta)=0$. See the right panel of Figure \ref{figure3}.

Suppose $(\beta,\kappa_1)\in S_1$, $(\beta,\kappa_2)\in S_2$, and $\mu_0(\kappa_1;\beta)=\mu_0(\kappa_2;\beta)$. If $\kappa_1/\kappa_2\notin\mathbb{Q}$ then there exist two families of unimodal wave solutions to \eqref{stationary-periodic}-\eqref{eqn:stokes exp} for $\eps\in\mathbb{R}$ and $|\eps|\ll1$, for which $\phi_1$ and $\eta_1$ are in \eqref{def:stokes1} up to a rescaling of $\eps$, where $\kappa=\kappa_\ell$, and $\mu_0=\mu_0(\kappa_\ell;\beta)$, $\ell=1,2$. 

If $\kappa_1/\kappa_2 \in \mathbb{Q}$ and $\kappa_1\neq\kappa_2$, on the other hand, we write 
\[
\kappa_1=M_1\kappa\quad\text{and}\quad \kappa_2=M_2\kappa 
\quad\text{for some $M_1, M_2\in \mathbb{Z}^+$ and coprime}
\]
for some $\kappa>0$. There instead exist {\em bimodal} wave solutions to \eqref{stationary-periodic}-\eqref{eqn:stokes exp} for $\eps\in\mathbb{R}$ and $|\eps|\ll1$, for which
\ba \label{def:stokes1 M1M2}
&\phi_1(x,y)=\alpha_1\s(M_1\kappa x)\ch(M_1\kappa y)+\alpha_2\s(M_2\kappa x)\ch(M_2\kappa y), \\
&\eta_1(x)= \alpha_1\sh(M_1\kappa)\c(M_1\kappa x)+\alpha_2\sh(M_2\kappa)\c(M_2\kappa x),\\
&q_1=0,
\ea 
for some $\alpha_1$ and $\alpha_2$, and $\mu_0=\mu_0(M_\ell\kappa;\beta)$, $\ell=1,2$. See, for instance, \cite{jones_1989,Reeder1981part1,Reeder1981part2}. Particularly, when $M_1=1$ so that 
\[
\kappa_1=\kappa \quad\text{and} \quad \kappa_2=M\kappa 
\quad \text{for some $M\geq2,\in\mathbb{Z}^+$},
\]
the corresponding bimodal wave solutions are referred to as {\em Wilton ripples} of order $M$  \cite{doi:10.1080/14786440508635350}. See Section \ref{sec:bimodal} for further discussion.
\end{definition*}

In the remainder of this subsection, we assume:  
\begin{itemize}
\item[$\bullet$] either $\mu_0(\kappa;\beta)<1$,
\item[$\bullet$] or $\mu_0(\kappa_1;\beta)=\mu_0(\kappa_2;\beta)$ where $\kappa_1/\kappa_2\notin \mathbb{Q}$, 
\end{itemize}
so that there exist unimodal wave solutions to \eqref{stationary-periodic}-\eqref{eqn:stokes exp} for $\eps\in\mathbb{R}$ and $|\eps|\ll1$, for which $\phi_1$ and $\eta_1$ take the form \eqref{def:stokes1}. We proceed to compute $\phi_2,\phi_3,\dots$, $\eta_2,\eta_3,\dots$, $q_2,q_3,\dots$, and $\mu_1,\mu_2,\dots$.

Substituting \eqref{eqn:stokes exp} into \eqref{stationary-periodic} and collecting terms of order $\eps^2$, after some algebra, we arrive at:
\begin{equation*}
\begin{aligned}
&{\phi_2}_{xx}+{\phi _{2}}_{yy}=2\eta_1{\phi_1}_{yy}+y{\eta_1}_{xx}{\phi_1}_y+2y{\eta_1}_x{\phi_1}_{xy}&&\text{for $0<y<1$}, \\
&{\phi_2}_y=0&&\text{at $y=0$}, \\
&{\eta_2}_x+{\phi_2}_y={\eta_1}_x{\phi_1}_x+\eta_1{\phi_1}_y &&\text{at $y=1$},\\
&\left.\begin{aligned}&\mu_0\beta{\eta_2}_{xx}+{\phi_2}_x-\mu_0\eta_2+q_{2}\\&=-\mu_1\beta{\eta_1}_{xx}
+{\eta_1}_x{\phi_1}_y+\tfrac12({\phi_1}_x^2+{\phi_1}_y^2)+\mu_1\eta_1\end{aligned}\right\}&&\text{at $y=1$},
\end{aligned}
\end{equation*}
where $\phi_1$ and $\eta_1$ are in \eqref{def:stokes1}, and $\mu_0$ in \eqref{def:mu0}. Recall that $\phi_2$ and $\eta_2$ are $2\pi/\kappa$-periodic in $x$, $\phi_2$ is odd in $x$, and $\eta_2$ is even and of mean zero over the period. We solve this, for instance, by the method of undetermined coefficients, to obtain:
\ba\label{def:stokes2}
&\begin{aligned}\phi_2(x,y)=&-\frac{\kappa}{4}\frac{2 \kappa +\mu _{0} \sh(2\kappa)-\kappa  \ch(2\kappa)+4 \mu_0\beta\kappa^2 \sh(2\kappa)}{\mu_0\sh(2\kappa)-2\kappa\ch(2\kappa)+4\mu_0\beta\kappa^2\sh(2\kappa)} \\ 
&\quad\cdot\s(2\kappa x)\ch(2\kappa y)\\
&+\frac{\kappa\sh(\kappa)}{2}y\s(2\kappa x)\sh(\kappa y),\end{aligned}\\
&\eta_2(x)=-\frac{\kappa^2}{4}\frac{\sh(2\kappa)(\ch(2\kappa)+2)}{\mu_0\sh(2\kappa)-2\kappa\ch(2\kappa)+4\mu_0\beta\kappa^2\sh(2\kappa)}\c(2\kappa x),\\
&q_{2}=\frac{\kappa^2}{4}, \quad\text{and}\quad \mu_1=0,
\ea
provided 
\[
\mu_0\sh(2\kappa)-2\kappa\ch(2\kappa)+4\mu_0\beta\kappa^2\sh(2\kappa)\neq0
\]
or, equivalently,
\begin{equation}\label{A:wilton2}
\frac{\kappa}{(\beta\kappa^2+1)\th(\kappa)}\neq\frac{2\kappa}{(4\beta\kappa^2+1)\th(2\kappa)}. 
\end{equation}
We pause to remark that if \eqref{A:wilton2} is not valid and $\mu_0(\kappa;\beta)=\mu_0(2\kappa;\beta)$ holds, then instead emerge Wilton ripples of order $2$. See Section \ref{sec:bimodal} for further discussion.

At the order of $\eps^3$, we arrive at:
\begin{align*}
&\left.\begin{aligned}{\phi_3}_{xx}&+{\phi_3}_{yy}\\
=&2y{\eta_1}_x{\phi_2}_{xy}+2\eta_1{\phi_2}_{yy}+y{\eta_1}_{xx}{\phi_2}_y\\
&+2y({\eta_2}_x-\eta_1{\eta_1}_x){\phi_1}_{xy}+(2\eta_2-3\eta_1^2-y^2{\eta_1}_x^2){\phi_1}_{yy}\\
&+y({\eta_2}_{xx}-2{\eta_1}_x^2-\eta_1{\eta_1}_{xx}){\phi_1}_{y}\end{aligned}\right\}&&\text{for }0<y<1,\\
&\,{\phi_3}_y=0 &&\text{at }y=0,\\ 
&\,{\eta_3}_x+{\phi_3}_y={\eta_1}_x{\phi_2}_x+\eta_1{\phi_2}_y+{\eta_2}_x{\phi_1}_x+(\eta_2-{\eta_1}_x^2-\eta_1^2){\phi_1}_y&&\text{at }y=1,\\
&\left.\begin{aligned}\mu_0\beta&{\eta_3}_{xx}+{\phi_3}_x-\mu_0\eta_3+q_3\\
=&{\phi_1}_y({\phi_2}_y+{\eta_2}_x-\eta_1{\eta_1}_x-{\phi_1}_x{\eta_1}_x)\\&-\eta_1{\phi_1}_y^2+{\phi_1}_x{\phi_2}_x+{\eta_1}_x{\phi_2}_y\\
&+\tfrac32\mu_0\beta{\eta_1}_x^2{\eta_1}_{xx}+\mu_2\eta_1-\mu_2\beta{\eta_1}_{xx}\end{aligned}\right\}&&\text{at }y=1,
\end{align*}
where $\phi_1$ and $\eta_1$ are in \eqref{def:stokes1}, $\mu_0$ in \eqref{def:mu0}, and $\phi_2$, $\eta_2$, and $\mu_1$ in \eqref{def:stokes2}. We solve this by the same procedure as for order $\eps^2$ to obtain
\ba\label{def:mu2}
\mu_2=&\frac{\kappa ^3\ch(\kappa)}{32  \sh(\kappa) {(\beta  \kappa ^2+1)}^2 (2 \beta  \kappa ^2 { \sh(\kappa)}^2-{ \sh(\kappa)}^2+3 \beta  \kappa ^2)} \\
&\cdot (2 \beta ^2 \kappa ^4 { \ch(2\kappa)}^2+\beta  \kappa ^2 { \ch(2\kappa)}^2+8 { \ch(2\kappa)}^2 \\
&\quad +30 \beta ^2 \kappa ^4  \ch(2\kappa)+30 \beta  \kappa ^2  \ch(2\kappa)+52 \beta ^2 \kappa ^4+89 \beta  \kappa ^2+28).
\ea
The formulae for $\phi_3$, $\eta_3$, $q_3$, $\phi_4$, $\eta_4$, $q_4$, $\mu_3, \dots$ are omitted. 

\subsection{Bimodal Wilton ripples}\label{sec:bimodal}

When \eqref{A:wilton2} fails so that $\mu_0(\kappa;\beta)=\mu_0(2\kappa;\beta)$ holds instead, \eqref{def:stokes2} becomes singular at the order of $\eps^2$. More generally, suppose 
\[
\mu_0(M_1\kappa;\beta)=\mu_0(M_2\kappa;\beta)
\quad \text{for some $M_1,M_2\in \mathbb{Z}^+$ and coprime}
\] 
for some $\kappa>0$. If $\phi_1$ and $\eta_1$ were to take the form \eqref{def:stokes1}, then \eqref{eqn:stokes exp} would become singular at some higher order of $\eps$. See, for instance, \cite{doi:10.1029/JB073i020p06545, Nayfeh1970}. See also \cite{CSaffmam1979,Hogan2,Hogan3} for infinite depth. 
Consequently, there cannot exist a unimodal wave solution to \eqref{stationary-periodic}-\eqref{eqn:stokes exp}, for which $\phi_1$ and $\eta_1$ are of the form \eqref{def:stokes1}. Nonetheless, there exist bimodal wave solutions for which $\phi_1$ and $\eta_1$ take the form \eqref{def:stokes1 M1M2} instead. See, for instance, \cite{jones_1989}. See also \cite{Reeder1981part1,Reeder1981part2,10.2307/2397696,Jones1986} for infinite depth. 

Particularly, when $M_1=1$ so that
\begin{equation}\label{A:wilton}
\mu_0(\kappa;\beta)=\mu_0(M\kappa;\beta)\quad \text{for some $M\geq 2, \in \mathbb{Z}^+$}\quad \text{for some $\kappa>0$},
\end{equation}
the corresponding bimodal wave solutions to \eqref{stationary-periodic}-\eqref{eqn:stokes exp} are referred to as Wilton ripples of order $M$ \cite{doi:10.1080/14786440508635350}. The existence and multiplicity of small-amplitude Wilton ripples is well studied through rigorous analysis. See, for instance, \cite{jones_1989}. See also \cite{Reeder1981part1,Reeder1981part2,10.2307/2397696,Jones1986} for infinite depth. Throughout we restrict our attention to $M_1=1$\footnote{More generally, when $M_1\neq 1$, ``the algebra would seem too complicated to make this worthwhile.''\cite{jones_1989}}. 

We emphasize that in the first quadrant of the $(\beta,\kappa)$-plane, Wilton ripples of order $M$ take place along the curve 
\ba\label{def:beta_W}
\beta=\beta_W(\kappa;M):=\frac{\th(M\kappa) - M\th(\kappa)}{M\kappa^2(\th(\kappa) - M\th(M\kappa))}.
\ea 
See, for instance, the right panel of Figure \ref{figure7} for $M=2$ and $3$. Recall Section \ref{sec:unimodal} and, clearly, $(\beta_W(\kappa;M),\kappa) \in S_1$ for each $M\geq 2$. 

In the remainder of this subsection, we assume \eqref{A:wilton}. Proceeding along the lines of the argument in Section \ref{sec:unimodal}, we solve \eqref{stationary-periodic}-\eqref{eqn:stokes exp} at orders $\eps$, $\eps^2$, $\eps^3, \dots$ for $\eps\in\mathbb{R}$ and $|\eps|\ll1$, to determine $\phi_1$, $\phi_2$, $\phi_3$, $\dots$, $\eta_1$, $\eta_2$, $\eta_3$, $\dots$, $q_1$, $q_2$, $q_3$, $\dots$, and $\mu_0$, $\mu_1$, $\mu_2$, $\dots$.

\subsubsection*{Wilton ripples of order $2$} 

We begin by setting $M=2$. At the order of $\eps$, recall that $\phi_1$, $\eta_1$, and $q_1$ are in \eqref{def:stokes1 M1M2} for $M_1=1$ and $M_2=2$. We assume without loss of generality $\alpha_1=1$ and $\alpha_2=\alpha$, so that \eqref{def:stokes1 M1M2} becomes:
\ba  \label{def:wilton1 M=2}
&\phi_1(x,y)=\s(\kappa  x)\ch(\kappa  y)+\alpha\s(2\kappa  x)\ch(2\kappa  y),\\ 
&\eta_1(x)=\sh(\kappa )\c(\kappa x)+\alpha\sh(2\kappa )\c(2\kappa  x),\\ &q_1=0.
\ea  

At the order of $\eps^2$, we carry out the same procedure as in Section \ref{sec:unimodal}, where $\phi_1$ and $\eta_1$ are in \eqref{def:wilton1 M=2} and $\mu_0$ in \eqref{A:wilton}. We then simplify the resulting expressions, exploiting that \eqref{def:beta_W} and \eqref{def:mu0} can be rewritten as
\[
\beta=\frac{\ch(2\kappa) - 1}{2\kappa^2(\ch(2\kappa) + 2)}\quad \text{and} \quad \mu_0=\frac{\kappa(2\ch(\kappa)^2 + 1)}{3\ch(\kappa)\sh(\kappa)},
\] 
to obtain:
\ba \label{def:alpha M=2}
q_2=\left(\alpha^2+\frac{1}{4}\right)\kappa^2,\quad \mu_1=-\frac{\kappa^2(\ch(2\kappa) + 2)^2}{6\alpha\sh(4\kappa)}, \quad \alpha^2=\frac{1}{8\ch(2\kappa)}.
\ea 
Notably, there are two different kinds of Wilton ripples of order $2$, depending on the sign of $\alpha$. The formulae for $\phi_2$ and $\eta_2$ are in Appendix \ref{app:Wilton exp}. 

At the order of $\eps^3$, we follow the same procedure to determine $\phi_3$, $\eta_3$, $q_3$, and $\mu_2$. The formula for $\mu_2$ is in Appendix \ref{app:Wilton exp}. The formulae for $\phi_3$, $\eta_3$, $q_3$, $\phi_4$, $\eta_4$, $q_4$, $\mu_3, \dots$ are omitted.

\subsubsection*{Wilton ripples of order $\geq 3$}

Turning to $M\geq 3, \in \mathbb{Z}^+$, at the order of $\eps$, we follow the same argument as for $M=2$ to obtain:
\ba \label{def:wilton1 M=3}
&\phi_1(x,y)=\s(\kappa  x)\ch(\kappa  y)+\alpha\s(M\kappa  x)\ch(M\kappa  y),\\ 
&\eta_1(x)=\sh(\kappa )\c(\kappa x)+\alpha\sh(M\kappa )\c(M\kappa  x),\\
&q_1=0,
\ea
for some $\alpha$ to be determined.

At the order of $\eps^2$, we follow the same procedure as for $M=2$, and simplify the resulting expressions using \eqref{def:beta_W} and \eqref{def:mu0} and, more precisely, 
\[
\mu_0=-\frac{M\kappa (\th(\kappa) - M\th(M\kappa))}{\th(\kappa)\th(M\kappa)(M^2 - 1)},
\]
to obtain: 
\[
q_2=\tfrac{1}{4}(\alpha^2M^2+1)\kappa^2 \quad\text{and}\quad \mu_1=0.
\]
The formulae for $\phi_2$ and $\eta_2$ are in Appendix \ref{app:Wilton exp}. We pause to remark that $\alpha$ remains undetermined at order $\eps^2$. 

\begin{remark*}\rm
We emphasize that $\mu_1=0$ for steady unimodal capillary-gravity waves and for bimodal Wilton ripples of order $\geq 3$. By contrast, $\mu_1\neq 0$ for Wilton ripples of order $2$. See \eqref{def:alpha M=2}. Compare, for instance, Lemma \ref{lem:V(m,n) M=2 N=2} with Lemmas \ref{lem:V(m,n) N=1}, \ref{lem:V(m,n) N=2}, \ref{lem:V(m,n) N=3}, \ref{lem:V(m,n) M=3 N=1,M} and \ref{lem:V(m,n) M=3 N=1,M2}.
\end{remark*}

Let $M=3$. At the order of $\eps^3$, a straightforward calculation leads to:
\ba \label{eqn:alpha M=3}
&\begin{aligned}\alpha^3(&468+7962\ch(\kappa)^2 - 38316\ch(\kappa)^4 \\
&- 155343\ch(\kappa)^6 + 910974\ch(\kappa)^8 - 1514544\ch(\kappa)^{10}\\
&+ 1591680 \ch(\kappa)^{12} - 2078208 \ch(\kappa)^{14} + 1460736 \ch(\kappa)^{16})\end{aligned}\\
&\begin{aligned}+\alpha^2(&54\ch(\kappa)^2 + 684\ch(\kappa)^4 - 8541\ch(\kappa)^6+ 1566\ch(\kappa)^8 \\
&+ 236376\ch(\kappa)^{10} - 831456\ch(\kappa)^{12} + 708480\ch(\kappa)^{14})\end{aligned}\\
&\begin{aligned}+\alpha(&92+2102\ch(\kappa)^2 + 3692\ch(\kappa)^4 - 28841\ch(\kappa)^6\\
& + 40306\ch(\kappa)^8 - 40736\ch(\kappa)^{10} + 14880\ch(\kappa)^{12}\big)\end{aligned}\\
&\begin{aligned}-\ch(\kappa)^2(1+36\ch(\kappa)^2 - 135\ch(\kappa)^4
-854\ch(\kappa)^6+ 4920\ch(\kappa)^8)=0.\end{aligned}
\ea
We adapt a similar argument to that in the proof of Lemma \ref{lem:mu0}, particularly, performing validated numerics, to verify that the left side of \eqref{eqn:alpha M=3} is a cubic polynomial in $\alpha$ with positive discriminant for each $\kappa\in (0,\infty)$. We omit the details. Consequently, there correspond three different kinds of Wilton ripples of order $3$, depending on the three roots of \eqref{eqn:alpha M=3}. See Sections \ref{sec:M=3 N=1,M}, \ref{sec:M=3 1<N<M}, and \ref{sec:modulational M=3} for further discussion.

More generally, if $M\geq 4$ then $\alpha$ would be determined upon solving \eqref{stationary-periodic}-\eqref{eqn:stokes exp} at higher orders of $\eps$. As noted in \cite{vanden-broeck_2010}, however, such computations quickly become ``intractable.''  

The formulae for $\phi_3$, $\eta_3$, $q_3$, $\phi_4$, $\eta_4$, $q_4$, $\mu_3, \dots$ are omitted.

\newpage

\section{Formulation of the spectral stability problem}\label{sec:spec}

Let $\beta>0$ and $\kappa>0$. Recall Section \ref{sec:Stokes}, and suppose that $\phi(\eps)$, $\eta(\eps)$, $q(\eps)$, and $\mu(\eps)$ make a $T(=2\pi/\kappa)$-periodic solution to \eqref{stationary-periodic}-\eqref{eqn:stokes exp} for $\eps\in\mathbb{R}$ and $|\eps|\ll1$. We are interested in the spectral stability and instability of such steady periodic capillary-gravity wave of sufficiently small amplitude. 

\subsection{Linearized equations}\label{sec:linearize}

Linearizing \eqref{eqn:ww} and \eqref{eqn:bd} about $\phi(\eps)$, $\eta(\eps)$, $u(\eps)$, and $z(\eps)$, and evaluating $q=q(\eps)$ and $\mu=\mu(\eps)$, after some algebra, we arrive at:
\begin{align*} 
&\left.\begin{aligned}\phi_x&-\frac{y\eta_x(\eps)}{1+\eta(\eps)}\phi_y +\frac{y(\eta_x\phi_y)(\eps)}{(1+\eta(\eps))^2}\eta\\
&-\frac{y\phi_y(\eps)(1+\eta_x(\eps)^2)^{3/2}}{\beta(1+\eta(\eps))} z-u=0\end{aligned}\right\}&&\text{for  $0<y<1$,}\\
&\left.\begin{aligned}u_x&-\frac{y\eta_x(\eps)}{1+\eta(\eps)}u_y\\
&+\frac{y(\eta_xu_y)(\eps)}{(1+\eta(\eps))^2}\eta
-\frac{yu_y(\eps)(1+\eta_x(\eps)^2)^{3/2}}{\beta(1+\eta(\eps))} z \\
&+\frac{1}{(1+\eta(\eps))^2}\phi_{yy}-\frac{2\phi_{yy}(\eps)}{(1+\eta(\eps))^3}\eta=0\end{aligned}\right\}&&\text{for $0<y<1$,}\\
&\,\eta_x-\frac{1}{\beta}(1+\eta_x(\eps)^2)^{3/2}z =0&&\text{at $y=1$},\\
\intertext{and}
&\left.\begin{aligned}
\phi_t&-\mu(\eps) z _x-u
+\frac{(u(\eps)-1)\phi_y(\eps)(1+\eta_x(\eps)^2)^{3/2}}{\beta(1+\eta(\eps))} z \\ &+\frac{(\phi_y\eta_x)(\eps)}{1+\eta(\eps)}u+u(\eps)u+\mu(\eps)\eta=0\end{aligned}\right\}&&\text{at $y=1$},\\
&\left.\begin{aligned}\eta_t&+ \frac{1}{\beta}(u(\eps)-1)(1+\eta_x(\eps)^2)^{3/2}z \\&+\eta_x(\eps)u-\frac{1}{1+\eta(\eps)}\phi_y+\frac{\phi_y(\eps)}{(1+\eta(\eps))^2}\eta=0\end{aligned}\right\}&&\text{at $y=1$},
\intertext{subject to}
&\,\phi_y=0&&\text{at $y=0$}.
\end{align*} 
Here the fifth equation follows from linearizing the fourth equation of \eqref{eqn:ww} and substituting \eqref{stationary-periodic:K}. 

Seeking a solution of the form 
\[
\begin{pmatrix}\phi(x,y,t) \\ u(x,y,t) \\ \eta(x,t)\\ z (x,t) \end{pmatrix}
=e^{\lambda t}\begin{pmatrix}\phi(x,y) \\ u(x,y) \\ \eta(x)\\  z (x) \end{pmatrix},
\quad\lambda\in\mathbb{C},
\] 
we gather:
\begin{subequations}\label{eqn:spec}
\begin{align}
&\left.\begin{aligned}\phi_x&-\frac{y\eta_x(\eps)}{1+\eta(\eps)}\phi_y +\frac{y(\eta_x\phi_y)(\eps)}{(1+\eta(\eps))^2}\eta\\
&-\frac{y\phi_y(\eps)(1+\eta_x(\eps)^2)^{3/2}}{\beta(1+\eta(\eps))} z-u=0\end{aligned}\right\}&&\text{for $0<y<1$,} \label{eqn:spec;Phi}\\
&\left.\begin{aligned}\label{eqn:spec;u}
u_x&-\frac{y\eta_x(\eps)}{1+\eta(\eps)}u_y +\frac{y(\eta_xu_y)(\eps)}{(1+\eta(\eps))^2}\eta\\
&-\frac{yu_y(\eps)(1+\eta_x(\eps)^2)^{3/2}}{\beta(1+\eta(\eps))} z\\
&+\frac{1}{(1+\eta(\eps))^2}\phi_{yy}-\frac{2\phi_{yy}(\eps)}{(1+\eta(\eps))^3}\eta=0\end{aligned}\right\}&&\text{for $0<y<1$,}\\
&\,\eta_x-\frac{1}{\beta}(1+\eta_x(\eps)^2)^{3/2} z=0&&\text{at $y=1$,}\label{eqn:spec;eta}\\
&\left.\begin{aligned}\lambda\phi&-\mu(\eps) z_x-u\\
&+\frac{\phi_y(\eps)(u(\eps)-1)(1+\eta_x(\eps)^2)^{3/2}}{\beta(1+\eta(\eps))} z \\
&+\frac{(\phi_y\eta_x)(\eps)}{1+\eta(\eps)}u+u(\eps)u+\mu(\eps)\eta=0\end{aligned}\right\}&&\text{at $y=1$,}\label{eqn:spec;z}\\
&\left.\begin{aligned}\lambda\eta&+ \frac{1}{\beta}(u(\eps)-1)(1+\eta_x(\eps)^2)^{3/2}z \\ &+\eta_x(\eps)u-\frac{1}{1+\eta(\eps)}\phi_y+\frac{\phi_y(\eps)}{(1+\eta(\eps))^2}\eta=0\end{aligned}\right\}&&\text{at $y=1$},\label{eqn:spec;bdry1}\\
&\,\phi_y=0&&\text{at $y=0$}.\label{eqn:spec;bdry2}
\end{align}
\end{subequations}
Informally, $\lambda\in\mathbb{C}$ belongs to the spectrum of the linearized capillary-gravity wave problem if \eqref{eqn:spec} admits a nontrivial bounded solution in an appropriate function space, and $\phi(\eps)$, $\eta(\eps)$, $q(\eps)$, and $\mu(\eps)$, for $\eps\in\mathbb{R}$ and $|\eps|\ll1$, make a spectrally stable solution to the capillary-gravity wave problem  for sufficiently small amplitude if the spectrum does not intersect the right half plane of $\mathbb{C}$. We will make this rigorous in Section \ref{sec:L}.

We emphasize that \eqref{eqn:spec;bdry1} is not autonomous. Following \cite{HS;cg-solitary} (see also \cite{HY2023}), we introduce
\ba\label{def:varphi}
\varphi(x,y)=&((1-u(x,1;\eps))(1+\eta_x(x;\eps)^2)^{3/2})^{-1}\\
&\begin{aligned}\cdot\Big(\frac{\phi(x,y)}{1+\eta(x;\eps)}
-\frac12\Big(\lambda&+\frac{\phi_y(x,1;\eps)}{(1+\eta(x;\eps))^2}\Big)y^2\eta(x)\\
&-\eta_x(x;\eps)\int_0^yy'u(x,y')~dy'\Big)\end{aligned}
\ea
so that \eqref{eqn:spec;bdry1} and \eqref{eqn:spec;bdry2} become 
\be\label{eqn:bdry;new}
\varphi_y+\frac{z}{\beta}=0\quad\text{at $y=1$}\quad\text{and}\quad \varphi_y=0\quad\text{at $y=0$}.
\ee
Clearly, \eqref{def:varphi} is well-defined provided $1+\eta(\eps)\neq0$ and $1-u(\cdot,1;\eps)\neq0$ and, particularly, for $\eps\in\mathbb{R}$ and $|\eps|\ll1$. On the other hand,
\begin{subequations} 
\begin{equation}\label{def:phiinverse}
\phi(x,y)=f_1(x;\eps)\varphi(x,y)+f_2(x;\eps)y^2\eta(x)+f_3(x;\eps)\int_0^yy' u(x,y')~dy',
\end{equation}
where
\begin{align}
f_1(x;\eps)=&(1-u(x,1;\eps))(1+\eta(x;\eps))(1+\eta_x(x;\eps)^2)^{3/2},\label{def:f1(eps)}\\
f_2(x;\eps)=&\frac{1}{2}\left(\frac{\phi_y(x,1;\eps)}{1+\eta(\eps)}+\lambda(1+\eta(x;\eps))\right),\notag\\
f_3(x;\eps)=&(1+\eta(x;\eps))\eta_x(x;\eps).\notag
\end{align}
\end{subequations}
Clearly, \eqref{def:phiinverse} is well-defined for $\eps\in\mathbb{R}$ and $|\eps|\ll1$.

We rewrite \eqref{eqn:spec;Phi}-\eqref{eqn:spec;z} in terms of $\varphi$, $\eta$, $u$, and $z$. We differentiate \eqref{def:phiinverse} with respect to $y$, and substitute   
\begin{align*}
&\,\phi_y=f_1(\eps)\varphi_y+2f_2(\eps)y\eta+f_3(\eps)yu
\intertext{into \eqref{eqn:spec;Phi}. We substitute}
&\,\phi_{yy}=f_1(\eps)\varphi_{yy}+2f_2(\eps)\eta+f_3(\eps)(u+yu_y)
\intertext{into \eqref{eqn:spec;u}, and}
&\,\phi=f_1(\eps)\varphi+f_2(\eps)\eta+f_3(\eps)\int_0^1yu(\cdot,y)~dy\quad \text{at $y=1$}
\intertext{into \eqref{eqn:spec;z}. We differentiate \eqref{def:phiinverse} with respect to $x$, and substitute}
&\begin{aligned}
\phi_x=&f_1(\eps)\varphi_x+{f_1}_x(\eps)\varphi+f_2(\eps)y^2\eta_x+{f_2}_x(\eps)y^2\eta\\
&+f_3(\eps)\int_0^yy'u_x(\cdot,y')~dy'+{f_3}_x(\eps)\int_0^yy'u(\cdot,y')~dy'
\end{aligned} 
\end{align*}
into \eqref{eqn:spec;Phi}.

\subsection{Operator formulation}\label{sec:L}

We write \eqref{eqn:spec} as
\be \label{eqn:LB}
\mathbf{u}_x=\mathbf{L}(\lambda)\mathbf{u}+\mathbf{B}(x;\lambda,\eps)\mathbf{u}.
\ee  
Here $\lambda\in\mathbb{C}$, $\eps\in\mathbb{R}$ and $|\eps|\ll1$, 
\begin{equation*}
\mathbf{L}(\lambda): {\rm dom}(\mathbf{L}) \subset Y \to Y,\quad\text{and}\quad
\mathbf{B}(x;\lambda,\eps): 
\mathbb{R}\times {\rm dom}(\mathbf{L}) \subset \mathbb{R}\times Y \to Y,
\end{equation*}
where
\be \label{def:L}
\mathbf{L}(\lambda)\mathbf{u}=\begin{pmatrix} 
u-\dfrac{\lambda y^2z}{2\beta}  \\ -\varphi_{yy}-\lambda\eta\\
\dfrac{z}{\beta}\\ \dfrac{\lambda\varphi(1)}{\mu_0}- \dfrac{u(1)}{\mu_0}+\eta+\dfrac{\lambda^2\eta}{2\mu_0} \end{pmatrix},\quad \mathbf{u}=\begin{pmatrix} \varphi \\ u \\ \eta \\ z \end{pmatrix},
\ee  
and $\mu_0=\mu_0(\kappa;\beta)$ is in \eqref{def:mu0}. Additionally,
\begin{equation}\label{def:Y}
Y=H^1(0,1)\times L^2(0,1)\times \mathbb{C}\times  \mathbb{C}
\end{equation}
and
\begin{equation}\label{def:dom}
{\rm dom}(\mathbf{L}) \\
=\left\{\mathbf{u}\in H^2(0,1)\times H^1(0,1)\times \mathbb{C}\times \mathbb{C}:\varphi_y(1)=-\frac{z}{\beta}, \varphi_y(0)=0\right\}.
\end{equation}

Upon inspection, $\mathbf{L}(\lambda)$ is the $O(1)$ part of \eqref{eqn:spec;Phi}-\eqref{eqn:spec;z} while $\mathbf{B}(x;\lambda,\eps)$ is $O(\eps)$, so that $\mathbf{B}(x;\lambda,0)=\mathbf{0}$. Consequently, when $\eps=0$, \eqref{eqn:LB} becomes $\mathbf{u}_x=\mathbf{L}(\lambda)\mathbf{u}$. The analysis does not require all the details of $\mathbf{B}(x;\lambda,\eps)$, 
whence we omit its formulae. Clearly, \eqref{eqn:bdry;new} holds for $\mathbf{u}\in{\rm dom}(\mathbf{L})$ and ${\rm dom}(\mathbf{L})$ is dense in $Y$. Additionally, $\mathbf{L}(\lambda)$ and $\mathbf{B}(x;\lambda,\eps)$ depend analytically on $\lambda\in\mathbb{C}$. It follows from Section \ref{sec:Stokes} that $\mathbf{B}(x;\lambda,\eps)$ is $T(=2\pi/\kappa)$-periodic and infinitely differentiable in $x$, and depends analytically on $\eps\in\mathbb{R}$ for $|\eps|\ll1$. 

We rewrite \eqref{eqn:LB} as 
\[
\mathcal{L}(\lambda,\eps)\mathbf{u}=\boldsymbol{0},
\]
where
\begin{equation}\label{def:LL}
\mathcal{L}(\lambda,\eps):{\rm dom}(\mathcal{L})\subset X \to X
\end{equation}
is defined as
\begin{equation}\label{def:operator}
\mathcal{L}(\lambda,\eps)\mathbf{u}=\mathbf{u}_x-(\mathbf{L}(\lambda)+\mathbf{B}(x;\lambda,\eps))\mathbf{u}.
\end{equation}
Additionally,
\begin{equation}\label{def:X}
X=L^2(\mathbb{R};Y)\quad\text{and}\quad 
{\rm dom}(\mathcal{L})=H^1(\mathbb{R};Y)\bigcap L^2(\mathbb{R};{\rm dom}(\mathbf{L})).
\end{equation}
Clearly, ${\rm dom}(\mathcal{L})$ is dense in $X$. 

\begin{notation*}
Throughout, $\mathcal{L}(\eps)$ means $\mathcal{L}(\lambda,\eps)$, parametrized by $\lambda\in\mathbb{C}$.   
\end{notation*}
 
\begin{definition}[The spectrum of $\mathcal{L}(\eps)$]\label{def:spec}\rm 
We define 
\[
{\rm spec}(\mathcal{L}(\eps))=\{\lambda\in\mathbb{C}: \text{\eqref{def:LL}-\eqref{def:X} is not invertible}\}
\]
for $\eps\in\mathbb{R}$ and $|\eps|\ll1$.
\end{definition}

It is noteworthy that \eqref{def:operator} is well-defined for $\eps\in\mathbb{R}$ and $|\eps|\ll1$. 

\begin{definition}[Spectral stability and instability]\label{def:stability}\rm
Let $\beta>0$ and $\kappa>0$. Suppose that $\phi(\eps)$, $\eta(\eps)$, $q(\eps)$, and $\mu(\eps)$ make a $T(=2\pi/\kappa)$-periodic solution to \eqref{stationary-periodic}-\eqref{eqn:stokes exp} for $\eps\in\mathbb{R}$ and $|\eps|\ll1$. Such steady periodic capillary-gravity wave of sufficiently small amplitude is said to be a {\em spectrally stable} solution to \eqref{eqn:ww}-\eqref{eqn:bd} if 
\[
{\rm spec}(\mathcal{L}(\eps)) \subset \{\lambda\in\mathbb{C}: {\rm Re}\, \lambda\leq 0\},
\]
and {\em spectrally unstable} otherwise.
\end{definition}

Since \eqref{eqn:LB} and, hence, \eqref{def:operator} are $T$-periodic in $x$, it follows from Floquet theory (see, for instance, \cite{Chicone,KK2013}) that the point spectrum of $\mathcal{L}(\eps)$ is empty and $\lambda\in\mathbb{C}$ belongs to the essential spectrum of $\mathcal{L}(\eps)$ if and only if \eqref{eqn:LB} admits a nontrivial solution $\mathbf{u}\in L^\infty(\mathbb{R};Y)$, satisfying
\[
\mathbf{u}(x+T)=e^{ik T}\mathbf{u}(x)
\quad \text{for some $k\in\mathbb{R}$, the Floquet exponent}.
\]
See, for instance, \cite{Gardner;evans1, Gardner;evans2, OS;evans}. See also \cite{HY2023} for Stokes' gravity waves.

\begin{lemma}[Symmetry of ${\rm spec}(\mathcal{L}(\eps))$]\label{lem:symm} 
If $\lambda\in{\rm spec}(\mathcal{L}(\eps))$ then $\lambda^*,-\lambda\in{\rm spec}(\mathcal{L}(\eps))$ and, hence, $-\lambda^*\in{\rm spec}(\mathcal{L}(\eps))$. In other words, ${\rm spec}(\mathcal{L}(\eps))$ is symmetric with respect to reflections across the real and imaginary axes.
\end{lemma}

Throughout, the superscript $*$ means complex conjugation.

The proof of Lemma \ref{lem:symm} mirrors the proof of \cite[Lemma 4.3]{HY2023}, and is therefore omitted. 

\begin{remark*}[Spectral stability and instability]\rm
A steady periodic capillary-gravity wave of sufficiently small amplitude is spectrally stable if and only if ${\rm spec}(\mathcal{L}(\eps))\subset i\mathbb{R}$ for $\eps\in\mathbb{R}$ and $|\eps|\ll1$.
\end{remark*}

\subsection{The spectrum of \texorpdfstring{$\mathcal{L}(0)$}{Lg}}\label{sec:eps=0}

Suppose $\eps=0$. It follows from Section \ref{sec:L} that $\lambda\in{\rm spec}(\mathcal{L}(0))$ if and only if 
\[
ik\mathbf{u}=\mathbf{L}(\lambda)\mathbf{u},
\quad\text{where}\quad \mathbf{L}(\lambda):{\rm dom}(\mathbf{L})\subset Y\to Y,
\]
admits a nontrivial solution for some $k\in\mathbb{R}$. In other words, $ik$ belongs to the spectrum of $\mathbf{L}(\lambda)$. We pause to remark that $\mathbf{L}(\lambda)$ has compact resolvent and its spectrum consists of discrete eigenvalues of finite multiplicity. A straightforward calculation reveals $\lambda=i\sigma \in i\mathbb{R}$, satisfying
\begin{equation}\label{eqn:sigma}
(\sigma-k)^2=\mu_0(\kappa;\beta)k\th(k)(\beta k^2+1)\quad 
\text{for some $k\in\mathbb{R}$},
\end{equation}
where $\mu_0(\kappa;\beta)$ is in \eqref{def:mu0}. We define
\begin{equation}\label{def:sigma}
\sigma_\pm(k;\beta,\mu_0)=k\pm\sqrt{\mu_0 k\th(k)(\beta k^2+1)},\quad k\in \mathbb{R}.
\end{equation}

\begin{remark*}\rm
The range of \eqref{def:sigma} is $\mathbb{R}$. Consequently, ${\rm spec}(\mathcal{L}(0))=i\mathbb{R}$, implying the spectral stability of the zero amplitude solution to \eqref{eqn:ww}-\eqref{eqn:bd}.
\end{remark*}

Our objective of this subsection is to determine the eigenvalues and eigenfunctions of $\mathbf{L}(i\sigma)$, $\sigma\in\mathbb{R}$. We gather preliminary results.

\begin{lemma}[$\sigma_\pm(\cdot;\beta,\mu_0)$ for supercritical waves]\label{lem:sigmaS12}
Suppose $0<\beta<1/3$ and $1<\mu_0<\max_{\kappa\in(0,\infty)}\mu_0(\kappa;\beta)$. Recall from Section \ref{sec:unimodal} that $\mu_0(\kappa;\beta)=\mu_0$ has two solutions, $\kappa_1$ and $\kappa_2$, where $0<\kappa_1<\kappa_2<\infty$. 

Then $\sigma_-(\kappa_1;\beta,\mu_0)=\sigma_-(\kappa_2;\beta,\mu_0)=0$ and $\sigma_-(\cdot;\beta,\mu_0)$ has two critical points, denoted by $k_{c,1}$ and $k_{c,2}$, such that 
\[
0<k_{c,1}<\kappa_1<k_{c,2}<\kappa_2.
\]
Moreover, $\sigma_-(k)$ is strictly decreasing for $k\in (0,k_{c,1}) \bigcup (k_{c,2},\infty)$ while strictly increasing for $k\in(k_{c,1},k_{c,2})$. 
Furthermore, $\sigma_-(\cdot;\beta,\mu_0)$ has an inflection point $k_i \in (k_{c,1},k_{c,2})$, and $\sigma_-(k)$ is convex for $k\in (0,k_i)$ while concave for $k\in(k_i,\infty)$.
\end{lemma}

\begin{proof}
Clearly, $\sigma_-(\kappa;\beta,\mu_0)=0$ if and only if $\mu_0(\kappa;\beta)=\mu_0$, so that
\[
\sigma_-(\kappa_1;\beta,\mu_0)=\sigma_-(\kappa_2;\beta,\mu_0)=0
\]
by hypothesis. Since $\sigma_-(0;\beta,\mu_0)=0$, consequently, $\sigma_-(\cdot;\beta,\mu_0)$ has critical points in the intervals $(0,\kappa_1)$ and $(\kappa_1,\kappa_2)$. We introduce 
\begin{align}
&\frac{d\sigma_-}{d k}(k;\beta,\mu_0) \notag \\
&\begin{aligned}&=1-\sqrt{\mu_0}\frac{\beta(k^3 - k^3\th(k)^2+ 3k^2\th(k))- k\th(k)^2 + \th(k) + k}{2\sqrt{k\th(k)(\beta k^2+1)}} \\
&=:1-\sqrt{\mu_0}f(k;\beta).
\end{aligned}\label{def:fS12}
\end{align}
Clearly, $f(k;\beta)\to1$ as $k\to 0^+$ and $f(k;\beta)\to\infty$ as $k\to \infty$. 
We aim to show that there exists a critical point of $f(\cdot;\beta)$ or, equivalently, an inflection point of $\sigma_-(\cdot;\beta,\mu_0)$, denoted by $k_i$, and $\frac{df}{dk}(k;\beta)<0$ for $k\in (0,k_i)$ while $\frac{df}{dk}(k;\beta)>0$ for $k\in (k_i,\infty)$. 

A direct calculation yields
\[
\frac{df}{d k}(k;\beta)=\frac{a(k)\beta^2+b(k)\beta+c(k)}{2(k\th(k)(\beta k^2+1))^{3/2}},
\]
where
\begin{align*}
a(k)&=k^6(3\th(k)^4 - 2\th(k)^2 - 1) -6k^5\th(k)(\th(k)^2 - 1) +3k^4\th(k)^2,\\
b(k)&=2k^4(3\th(k)^4 - 2\th(k)^2 - 1) -8k^3\th(k)(\th(k)^2 - 1)+ 6k^2\th(k)^2,\\
c(k)&=k^2(3\th(k)^4 - 2\th(k)^2 - 1) -2k\th(k)(\th(k)^2 - 1) -\th(k)^2.
\end{align*}
We implement a similar argument to that in the proof of Lemma \ref{lem:mu0} and, particularly, perform validated numerics, to verify
\[
a(k), b(k)>0\quad \text{and}\quad c(k)<0\quad\text{for $k\in(0,\infty)$}. 
\]
We omit the details. Consequently, $a(k)\beta^2+b(k)\beta+c(k)$ has one positive root, denoted by $\beta_+(k)=\frac{-b+\sqrt{b^2-4ac}}{2a}(k)$, and one negative root. Clearly, $a(k)\beta^2+b(k)\beta+c(k)>0$ for $\beta\in(\beta_+(k),\infty)$ and $a(k)\beta^2+b(k)\beta+c(k)<0$ for $\beta\in(0,\beta_+(k))$. A straightforward calculation reveals that $\beta_+(k)\to1/3$ as $k\to 0^+$ and $\beta_+(k)\to0$ as $k\to\infty$.

Furthermore, 
\[
\beta_+'(k)=\frac{(ba'-ab')\sqrt{b^2 - 4ac}+abb' - b^2a' + 2caa' - 2a^2c'}{2a^2\sqrt{b^2 - 4ac}}(k),
\]
where the prime means differentiation with respect to $k$. We adapt the same procedure as in the proof of Lemma \ref{lem:mu0}, performing validated numerics, to verify:
\begin{align*}
&(ba'-ab')(k)>0,\\ 
&(abb' - b^2a' + 2caa' - 2a^2c')(k)<0,\\
&((ba'-ab')\sqrt{b^2 - 4ac})^2(k)-(abb' - b^2a' + 2caa' - 2a^2c')^2(k)<0,
\end{align*}
for $k\in(0,\infty)$. We omit the details. Therefore, $\beta_+(k)$ is strictly decreasing for $k\in(0,\infty)$, implying that $\beta_+(k)=\beta$ has a unique solution in the interval $(0,\infty)$. This completes the proof.
\end{proof}

\begin{definition*}\rm 
Recall Section \ref{sec:unimodal}, and suppose $\mu_0(\kappa;\beta)>1$, corresponding to supercritical capillary-gravity waves. Lemma \ref{lem:sigmaS12} motivates us to define 
\begin{equation}\label{def:sigma_c12}
\sigma_{c,\ell}=\sigma_-(k_{c,\ell};\beta,\mu_0),\quad \text{where}\quad
\frac{d\sigma_-}{dk}(k_{c,\ell};\beta,\mu_0)=0, \quad \ell=1,2.
\end{equation}
In other words, $\sigma_{c,1}$ and $\sigma_{c,2}$ are the critical values of $\sigma_-(\cdot;\beta,\mu_0)$. It then follows by symmetry that $-\sigma_{c,\ell}=\sigma_+(-k_{c,\ell};\beta,\mu_0)$, $\ell=1,2$, are the critical values of $\sigma_+(\cdot;\beta,\mu_0)$. 

Suppose $0<-\sigma_{c,1}<\sigma_{c,2}$. Below we write $\sigma_\pm(k)$ instead of $\sigma_\pm(k;\beta,\mu_0)$ where $\beta$ and $\mu_0$ are suppressed for simplicity of notation. Lemma \ref{lem:sigmaS12} implies:

\begin{itemize}
\item[$\bullet$] If $\sigma=0$ then $\sigma_+(k)=0$ admits two simple roots $k_5(0):=-\kappa_2$ and $k_1(0):=-\kappa_1$, and one repeated root $k_3(0),k_4(0):=0$, while $\sigma_-(k)=0$ has two simple roots $k_2(0):=\kappa_1$ and $k_6(0):=\kappa_2$. 

\item[$\bullet$] If $0<\sigma<-\sigma_{c,1}$ then $\sigma_+(k)=\sigma$ has four simple solutions, denoted by $k_5(\sigma)$, $k_1(\sigma)$, $k_3(\sigma)$, and $k_4(\sigma)$, such that $k_5(\sigma)<-\kappa_2$ and $ -\kappa_1<k_1(\sigma)<k_3(\sigma)< 0< k_4(\sigma)$, while $\sigma_-(k)=\sigma$ has two simple solutions $k_2(\sigma)$ and $k_6(\sigma)$ such that $\kappa_1<k_2(\sigma)<k_6(\sigma)<\kappa_2$. 

\item[$\bullet$] If $\sigma=-\sigma_{c,1}$ then $\sigma_+(k)=-\sigma_{c,1}$ has two simple solutions $k_5(-\sigma_{c,1})$ and $k_4(-\sigma_{c,1})$ such that $k_5(-\sigma_{c,1})<-\kappa_2$ and $0<k_4(-\sigma_{c,1})$, and one repeated solution $k_1(-\sigma_{c,1}),k_3(-\sigma_{c,1})=-k_{c,1}$, while $\sigma_-(k)=-\sigma_{c,1}$ has two simple solutions $k_2(-\sigma_{c,1})$ and $k_6(-\sigma_{c,1})$ such that $\kappa_1<k_2(-\sigma_{c,1})<k_6(-\sigma_{c,1})<\kappa_2$.

\item[$\bullet$] If $-\sigma_{c,1}<\sigma<\sigma_{c,2}$ then $\sigma_+(k)=\sigma$ has two simple solutions $k_5(\sigma)$ and $k_4(\sigma)$ such that $k_5(\sigma)<-\kappa_2$ and $0<k_4(\sigma)$, while $\sigma_-(k)=\sigma$ has two simple solutions $k_2(\sigma)$ and $k_6(\sigma)$ such that $\kappa_1<k_2(\sigma)<k_6(\sigma)<\kappa_2$.

\item[$\bullet$] If $\sigma=\sigma_{c,2}$ then $\sigma_+(k)=\sigma_{c,2}$ has two simple solutions $k_5(\sigma_{c,2})$ and $k_4(\sigma_{c,2})$ such that $k_5(\sigma_{c,2})<-\kappa_2<0<k_4(\sigma_{c,2})$, while $\sigma_-(k)=\sigma_{c,2}$ has one repeated solution $k_2(\sigma_{c,2}),k_6(\sigma_{c,2})=k_{c,2}$.

\item[$\bullet$] If $\sigma>\sigma_{c,2}$ then $\sigma_+(k)=\sigma$ has two simple solutions $k_5(\sigma)$ and $k_4(\sigma)$ such that $k_5(\sigma)<-\kappa_2<0<k_4(\sigma)$, and $\sigma_-(k)=\sigma$ has no solution. 
\end{itemize} 
\end{definition*}

Figure \ref{figure4} provides an example. 

\begin{figure}[htbp]
    \centering
    \includegraphics[scale=0.3]{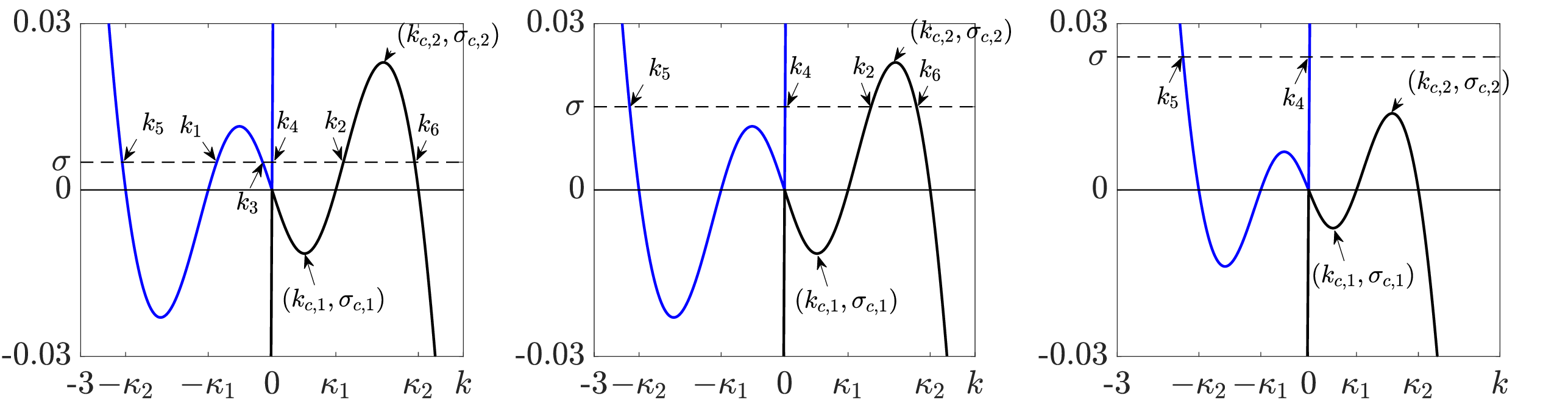}
    \caption{Graphs of $\sigma_+(k;\beta,\mu_0)$ (blue) and $\sigma_-(k;\beta,\mu_0)$ (black) for $0<\beta<1/3$, $\mu_0>1$, and $-\sigma_{c,1}<\sigma_{c,2}$. 
    Left: If $0<\sigma<-\sigma_{c,1}$ then $\sigma_\pm(k)=\sigma$ admit six simple solutions $k_j(\sigma)$, $j=6,2,4,3,1,5$. Middle: If $-\sigma_{c,1}<\sigma<\sigma_{c,2}$ then $\sigma_\pm(k)=\sigma$ have four simple solutions $k_j(\sigma)$, $j=6,2,4,5$. Right: If $\sigma>\sigma_{c,2}$ then $\sigma_+(k)=\sigma$ has two simple solutions $k_j(\sigma)$, $j=4,5$. }
    \label{figure4}
\end{figure}

Suppose, on the other hand, $-\sigma_{c,1}>\sigma_{c,2}>0$. We analogously characterize the solutions to $\sigma_\pm(k;\beta,\mu_0)=\sigma$. We omit the details. See Figure \ref{figure5} for an example. 

\begin{figure}[htbp]
    \centering
    \includegraphics[scale=0.3]{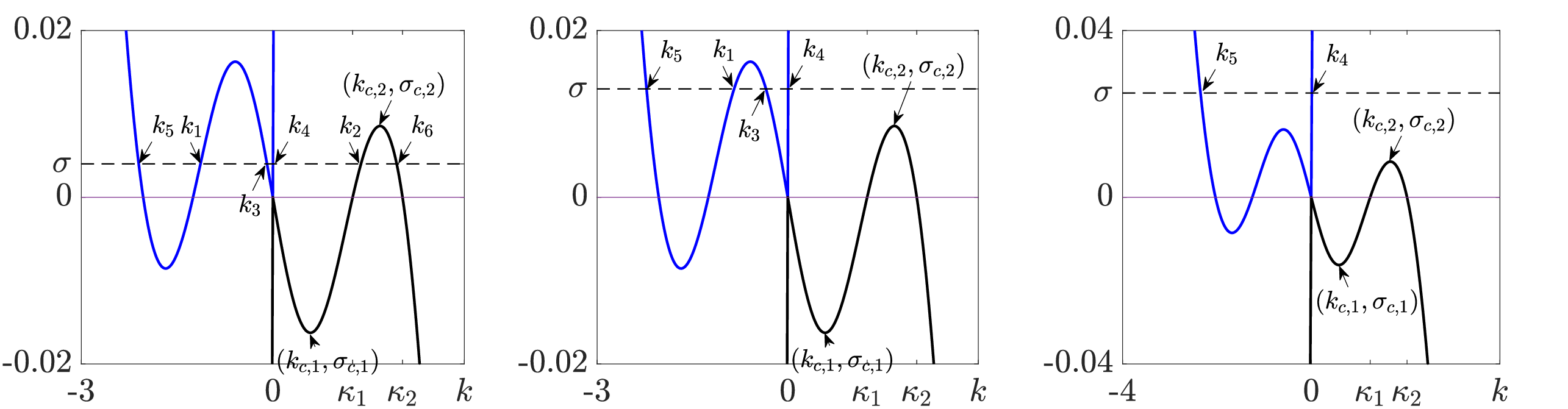}
    \caption{Graphs of $\sigma_+(k;\beta,\mu_0)$ (blue) and $\sigma_-(k;\beta,\mu_0)$ (black) for $0<\beta<1/3$, $\mu_0>1$ and $-\sigma_{c,1}>\sigma_{c,2}$. Left: If $0<\sigma<\sigma_{c,2}$ then $\sigma_\pm(k)=\sigma$ admit six simple solutions $k_j(\sigma)$, $j=6,2,4,3,1,5$. Middle: If $\sigma_{c,2}<\sigma<-\sigma_{c,1}$ then  $\sigma_\pm(k)=\sigma$ have four simple solutions $k_j(\sigma)$, $j=4,3,1,5$. Right: If $\sigma>-\sigma_{c,1}$ then $\sigma_+(k)=\sigma$ has two simple solutions $k_j(\sigma)$, $j=4,5$.}
    \label{figure5}
\end{figure}

The region for supercritical capillary-gravity waves, namely $S_1\bigcup S_2$, in the $(\beta,\kappa)$-plane can be divided into subregions where $-\sigma_{c,1}>\sigma_{c,2}$ and $-\sigma_{c,1}<\sigma_{c,2}$. See the left panel of Figure \ref{figure7}.

\begin{lemma}[$\sigma_\pm(\cdot;\beta,\mu_0)$ for subcritical waves]\label{lem:sigmaS3}
Let $0<\mu_0<1$. Recall from Section \ref{sec:unimodal} that $\mu_0(\cdot;\beta)=\mu_0$ has a unique solution, say, $\kappa\in(0,\infty)$. 

Then $\sigma_-(\kappa;\beta,\mu_0)=0$ and $\sigma_-(\cdot;\beta,\mu_0)$ has a critical point, denoted by $k_{c}$, in the interval $(0,\kappa)$. Moreover, $\sigma_-(k)$ is strictly increasing for $k\in (0,k_{c})$ while strictly decreasing for $k\in (k_{c},\infty)$. 

If $0<\beta<1/3$ then $\sigma_-(\cdot;\beta,\mu_0)$ has an inflection point $k_i \in (0,k_{c})$, and $\sigma_-(k)$ is convex for $k\in (0,k_i)$ while concave for $k\in (k_i,\infty)$. If $\beta\geq 1/3$, on the other hand, $\sigma_-(k)$ is concave for $k\in(0,\infty)$.
\end{lemma}

\begin{proof}
Clearly, $\sigma_-(0;\beta,\mu_0),\sigma_-(\kappa;\beta,\mu_0)=0$ and, consequently, $\sigma_-(\cdot;\beta,\mu_0)$ has a critical point $k_c\in(0,\kappa)$. We follow the lines of the argument in the proof of Lemma \ref{lem:sigmaS12} and, particularly, make validated numerics, to verify $\frac{d\sigma_-}{dk}(k;\beta,\mu_0)>0$ for $k\in(0,k_c)$ while $\frac{d\sigma_-}{dk}(k;\beta,\mu_0)<0$ for $k\in(k_c,\infty)$. We additionally verify the convexity of $\sigma_-(\cdot;\beta,\mu_0)$. We omit the details.
\end{proof}

\begin{definition*}
Recall Section \ref{sec:unimodal}. Suppose $0<\mu_0(\kappa;\beta)<1$, corresponding to subcritical capillary-gravity waves. Recall Lemma \ref{lem:sigmaS3} and we define
\begin{equation}\label{def:sigma_c3}
\sigma_c=\sigma_-(k_c;\beta,\mu_0), \quad \text{where} \quad 
\frac{d\sigma_-}{dk}(k_c;\beta,\mu_0)=0.
\end{equation}
That is, $\sigma_c$ is the critical value of $\sigma_-(\cdot;\beta,\mu_0)$. It follows by symmetry that $-\sigma_c=\sigma_+(-k_c;\beta,\mu_0)$ is the critical value of $\sigma_+(\cdot;\beta,\mu_0)$. Below we write $\sigma_\pm(k)$ instead of $\sigma_\pm(k;\beta,\mu_0)$ for brevity. Lemma \ref{lem:sigmaS3} implies:

\begin{itemize}
\item[$\bullet$] If $\sigma=0$ then $\sigma_+(k)=0$ admits one simple root $k_5(\sigma):=-\kappa$ and one repeated root $k_4(\sigma), k_2(\sigma):=0$, while $\sigma_-(k)=0$ has one simple root $k_6(\sigma):=\kappa$. 

\item[$\bullet$] If $0<\sigma<\sigma_c$ then $\sigma_+(k)=\sigma$ has two simple solutions, denoted by $k_5(\sigma)$ and $k_4(\sigma)$, such that $k_5(\sigma)\leq-\kappa<0\leq k_4(\sigma)$, while $\sigma_-(k)=\sigma$ has two simple solutions $k_2(\sigma)$ and $k_6(\sigma)$ such that $0\leq k_2(\sigma)<k_6(\sigma)\leq\kappa$.

\item[$\bullet$] If $\sigma=\sigma_c$ then $\sigma_-(k)=\sigma_c$ has one repeated solution $k_2(\sigma_c),k_6(\sigma_c)=k_c$, while $\sigma_+(k)=\sigma_c$ has two simple solutions $k_5(\sigma)$ and $k_4(\sigma)$ such that $k_5(\sigma_c)<-\kappa<0<k_4(\sigma_c)$. 

\item[$\bullet$] If $\sigma>\sigma_c$ then $\sigma_+(k)=\sigma$ has two simple solutions $k_5(\sigma)$ and $k_4(\sigma)$ such that $k_5(\sigma)<-\kappa<0<k_4(\sigma)$, and $\sigma_-(k)=\sigma$ has no solution. 
\end{itemize} 
\end{definition*}

See Figure \ref{figure6} for an example. 

\begin{figure}[htbp]
    \centering
    \includegraphics[scale=0.3]{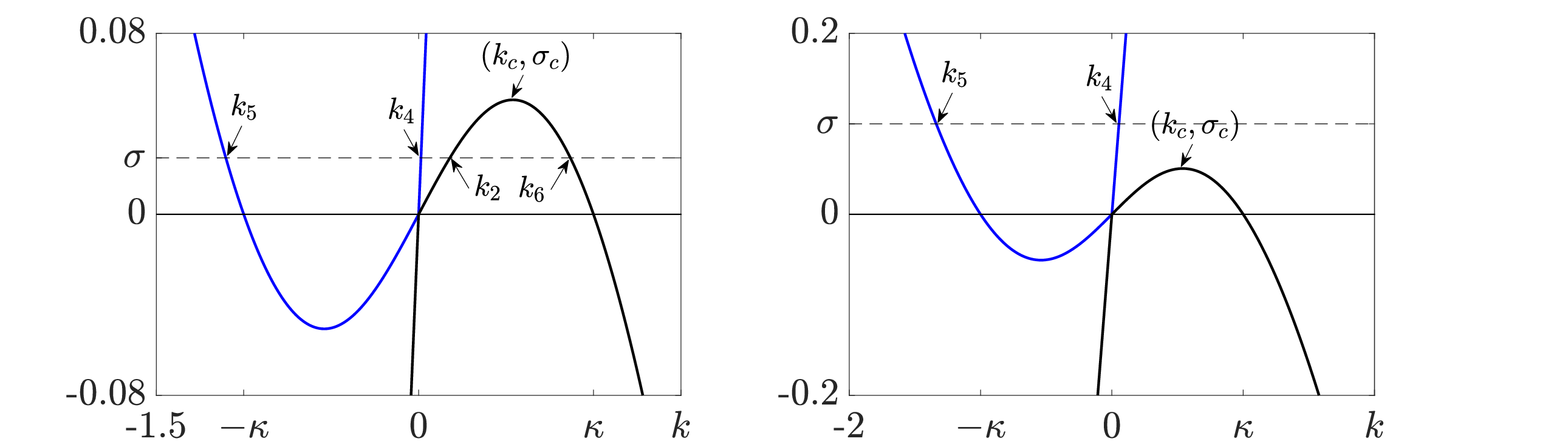}
    \caption{Graphs of $\sigma_+(k;\beta,\mu_0)$ (blue) and $\sigma_-(k;\beta,\mu_0)$ (black) for $0<\mu_0<1$. Left: If $0<\sigma<\sigma_c$ then $\sigma_\pm(k)=\sigma$ admit four simple solutions $k_j(\sigma)$, $j=6,2,4,5$. Right: If $\sigma>\sigma_c$, on the other hand, $\sigma_\pm(k)=\sigma$ have two simple solutions $k_j(\sigma)$, $j=4,5$.}
    \label{figure6}
\end{figure}

We state the main result of this subsection.

\begin{lemma}[The spectrum of $\mathbf{L}(i\sigma)$]\label{lem:eps=0} 
Let $\beta>0$ and $\kappa>0$. Recall Section \ref{sec:L}, and consider
\[
\mathbf{L}(i\sigma):{\rm dom}(\mathbf{L})\subset Y\to Y.
\] 

Recall Section \ref{sec:unimodal}, and suppose $0<\mu_0(\kappa;\beta)<1$, corresponding to subcritical capillary-gravity waves. Recall \eqref{def:sigma_c3}.

\begin{itemize}
\item[$\bullet$] If $\sigma=0$ then $ik_j(0)=(-1)^ji \kappa$, $j=6,5$, are simple eigenvalues of $\mathbf{L}(i0)$, and
\ba\label{def:eigen0}
\boldsymbol{\phi}_j(0)&=\begin{pmatrix}
\ch(k_j(0)y)\\ik_j(0)\ch(k_j(0)y)\\i\sh(k_j(0))\\-\beta k_j(0)\sh(k_j(0))\end{pmatrix},
\ea
$j=6,5$, are the corresponding eigenfunctions. 

Additionally, $ik_j(0)=0$, $j=4,2$, is an eigenvalue of $\mathbf{L}(i0)$ with algebraic multiplicity $2$ and geometric multiplicity $1$, and 
\begin{equation}\label{def:eigen0'}
\boldsymbol{\phi}_4(0)=\begin{pmatrix} 1 \\ 0 \\ 0\\0 \end{pmatrix}
\quad\text{and}\quad
\boldsymbol{\phi}_2(0)=\begin{pmatrix} 0 \\ 1 \\ \frac{1}{\mu_0}\\ 0\end{pmatrix}
\end{equation}
are the corresponding generalized eigenfunctions, satisfying 
\[
\mathbf{L}(i0)\boldsymbol{\phi}_4(0)=\boldsymbol{0}\quad \text{and} \quad \mathbf{L}(i0)\boldsymbol{\phi}_2(0)=\boldsymbol{\phi}_4(0).
\]

\item[$\bullet$] If $0<\sigma<\sigma_c$ then $ik_j(\sigma)$, $j=6,2,4,5$, are simple eigenvalues of $\mathbf{L}(i\sigma)$, and 
\be\label{def:eigen}
\boldsymbol{\phi}_j(\sigma)=\begin{pmatrix}
\ch(k_j(\sigma)y)+\dfrac{k_j(\sigma)\sigma\sh(k_j(\sigma))y^2}{2(k_j(\sigma)-\sigma)}\\
ik_j(\sigma)\ch(k_j(\sigma)y)\\
\dfrac{ik_j(\sigma)\sh(k_j(\sigma))}{k_j(\sigma)-\sigma}\\
-\dfrac{\beta k_j(\sigma)^2\sh(k_j(\sigma))}{k_j(\sigma)-\sigma}\end{pmatrix},
\ee
$j=6,2,4,5$, are the corresponding eigenfunctions.

\item[$\bullet$] If $\sigma=\sigma_c$ then $ik_j(\sigma_c)$, $j=4,5$, are simple eigenvalues of $\mathbf{L}(i\sigma_c)$, and \eqref{def:eigen}, $j=4,5$, are the corresponding eigenfunctions. 

Additionally, $ik_j(\sigma_c)=ik_c$, $j=6,2$, is an eigenvalue of $\mathbf{L}(i\sigma_c)$ with algebraic multiplicity $2$ and geometric multiplicity $1$,
and $\boldsymbol{\phi}_6(\sigma_c)$, in \eqref{def:eigen}, and
\ba \label{def:eigen'}
\boldsymbol{\phi}_2(\sigma_c)
=\begin{pmatrix}
-iy\sh(k_c y)+\dfrac{i\sigma_c(\sigma_c\sh(k_c) - k_c^2\ch(k_c) + k_c \sigma_c\ch(k_c))y^2}{2(k_c - \sigma_c)^2}\\
\ch(k_c y) + k_cy\sh(k_cy)\\
-\dfrac{\sigma_c\sh(k_c) - k_c^2\ch(k_c) + k_c\sigma_c \ch(k_c)}{(k_c - \sigma_c)^2}\\
\dfrac{i\beta k_c(k_c\sh(k_c) - 2\sigma_c\sh(k_c) + k_c^2\ch(k_c) - k_c\sigma_c\ch(k_c))}{(k_c - \sigma_c)^2}\end{pmatrix}
\ea 
are the corresponding generalized eigenfunctions, satisfying
\[
(\mathbf{L}(i\sigma_c)-ik_c\boldsymbol{1})\boldsymbol{\phi}_6(\sigma_c)=\boldsymbol{0}\quad\text{and}\quad
(\mathbf{L}(i\sigma_c)-ik_c\boldsymbol{1})\boldsymbol{\phi}_2(\sigma_c)=\boldsymbol{\phi}_6(\sigma_c).
\]
Here $\mathbf{1}$ means the identity operator and is suppressed where there is no confusion. 
\item[$\bullet$] If $\sigma>\sigma_c$ then $ik_j(\sigma)$, $j=4,5$, are simple eigenvalues of $\mathbf{L}(i\sigma)$, and \eqref{def:eigen}, $j=4,5$, are the corresponding eigenfunctions. 
\end{itemize}

Suppose, on the other hand, $\mu_0(\kappa;\beta)>1$, corresponding to supercritical capillary-gravity waves. Recall \eqref{def:sigma_c12}, and suppose $0<-\sigma_{c,1}<\sigma_{c,2}$. 

\begin{itemize}
\item[$\bullet$] If $\sigma=0$ then $ik_j(0)=(-1)^ji \kappa_1$, $j=2,1$, and $ik_j(0)=(-1)^ji \kappa_2$, $j=6,5$, are simple eigenvalues of $\mathbf{L}(i0)$, and \eqref{def:eigen0}, $j=6,2,1,5$, are the corresponding eigenfunctions. 

Additionally, $ik_j(0)=0$, $j=4,3$, is an eigenvalue of $\mathbf{L}(i0)$ with algebraic multiplicity $2$ and geometric multiplicity $1$, 
and
\begin{equation}\label{def:eigen0''}
\boldsymbol{\phi}_4(0)=\begin{pmatrix} 1 \\ 0 \\ 0\\0 \end{pmatrix}
\quad\text{and}\quad
\boldsymbol{\phi}_3(0)=\begin{pmatrix} 0 \\ 1 \\ \frac{1}{\mu_0} \\ 0 \end{pmatrix}
\end{equation}
are the corresponding generalized eigenfunctions, satisfying 
\[
\mathbf{L}(i0)\boldsymbol{\phi}_4(0)=\boldsymbol{0}\quad \text{and} \quad \mathbf{L}(i0)\boldsymbol{\phi}_3(0)=\boldsymbol{\phi}_4(0).
\]

\item[$\bullet$] If $0<\sigma<-\sigma_{c,1}$ then $ik_j(\sigma)$, $j=6,2,4,3,1,5$, are simple eigenvalues of $\mathbf{L}(i\sigma)$, and \eqref{def:eigen}, $j=6,2,4,3,1,5$, are the corresponding eigenfunctions.

\item[$\bullet$] If $\sigma=-\sigma_{c,1}$ then $ik_j(-\sigma_{c,1})$, $j=6,2,4,5$, are simple eigenvalues of $\mathbf{L}(-i\sigma_{c,1})$, and \eqref{def:eigen}, $j=6,2,4,5$, are the corresponding eigenfunctions. 

Additionally, $ik_j(-\sigma_{c,1})=-ik_{c,1}$, $j=1,3$, is an eigenvalue of $\mathbf{L}(-i\sigma_{c,1})$ with algebraic multiplicity $2$ and geometric multiplicity $1$, and $\boldsymbol{\phi}_1(-\sigma_{c,1})$ (see \eqref{def:eigen}) and $\boldsymbol{\phi}_3(-\sigma_{c,1})$ (see \eqref{def:eigen'}) are the corresponding generalized eigenfunctions, satisfying 
\[
(\mathbf{L}(-i\sigma_{c,1})+ik_{c,1})\boldsymbol{\phi}_1(-\sigma_{c,1})=\boldsymbol{0},\quad (\mathbf{L}(-i\sigma_{c,1})+ik_{c,1})\boldsymbol{\phi}_3(-\sigma_{c,1})=\boldsymbol{\phi}_1(-\sigma_{c,1}).
\]
Here $\boldsymbol{\phi}_3(-\sigma_{c,1})$ is defined by replacing $k_c$ with $-k_{c,1}$ and $\sigma_c$ with $-\sigma_{c,1}$ in \eqref{def:eigen'}.

\item[$\bullet$] If $-\sigma_{c,1}<\sigma<\sigma_{c,2}$ then $ik_j(\sigma)$, $j=6,2,4,5$, are simple eigenvalues of $\mathbf{L}(i\sigma)$, and \eqref{def:eigen}, $j=6,2,4,5$, are the corresponding eigenfunctions.

\item[$\bullet$] If $\sigma=\sigma_{c,2}$ then $ik_j(\sigma_{c,2})$, $j=4,5$, are simple eigenvalues $\mathbf{L}(i\sigma_{c,2})$, and \eqref{def:eigen}, $j=4,5$, are the corresponding eigenfunctions. 

Additionally, $ik_j(\sigma_{c,2})=ik_{c,2}$, $j=6,2$, is an eigenvalue $\mathbf{L}(i\sigma_{c,2})$ with algebraic multiplicity $2$ and geometric multiplicity $1$, and $\boldsymbol{\phi}_6(\sigma_{c,2})$ (see \eqref{def:eigen}) and $\boldsymbol{\phi}_2(\sigma_{c,2})$ (see \eqref{def:eigen'}) are the corresponding generalized eigenfunctions, satisfying 
\[(\mathbf{L}(i\sigma_{c,2})-ik_{c,2})\boldsymbol{\phi}_6(\sigma_{c,2})=\boldsymbol{0},\quad 
(\mathbf{L}(i\sigma_{c,2})-ik_{c,2})\boldsymbol{\phi}_2(\sigma_{c,2})=\boldsymbol{\phi}_6(\sigma_{c,2}).
\]
Here $\boldsymbol{\phi}_2(\sigma_{c,2})$ is defined by replacing $k_c$ with $k_{c,2}$ and $\sigma_c$ with $\sigma_{c,2}$ in \eqref{def:eigen'}. 

\item[$\bullet$] If $\sigma>\sigma_{c,2}$ then $ik_j(\sigma_c)$, $j=4,5$, are simple eigenvalues $\mathbf{L}(i\sigma)$, and \eqref{def:eigen}, $j=4,5$, are the corresponding eigenfunctions. 
\end{itemize}
\end{lemma}

The proof of Lemma \ref{lem:eps=0} mirrors the proof of \cite[Lemma~4.4]{HY2023}, and is therefore omitted. 

Suppose $\mu_0(\kappa;\beta)>1$ and $-\sigma_{c,1}>\sigma_{c,2}$ (see Figure \ref{figure5} for the solutions to $\sigma_\pm(\cdot;\beta,\mu_0)=\sigma$). We analogously determine the eigenvalues and eigenfunctions of $\mathbf{L}(i\sigma)$. We do not record the result. 

\begin{definition}[Generalized eigenspace, projection, and ordered basis]${}$\label{def:proj}\rm

\noindent Let $\beta>0$ and $\kappa>0$. Recall Lemma \ref{lem:eps=0}, and 
let $Y(\sigma)$ denote the {\em generalized eigenspace} of $\mathbf{L}(i\sigma): {\rm dom}(\mathbf{L})\subset Y\to Y$, associated with its finitely many and purely imaginary eigenvalues. Let  
\[
\boldsymbol{\Pi}(\sigma): {\rm dom}(\mathbf{L})\subset Y \to Y(\sigma) 
\]
denote the {\em projection} of ${\rm dom}(\mathbf{L})$ onto $Y(\sigma)$.
Clearly, $\boldsymbol{\Pi}(\sigma)$ commutes with $\mathbf{L}(i\sigma)$. 

Recall Section \ref{sec:unimodal}, and suppose $0<\mu_0(\kappa;\beta)<1$, corresponding to subcritical capillary-gravity waves. We utilize the eigenfunctions and generalized eigenfunctions in Lemma \ref{lem:eps=0}, to select an {\em ordered basis} $Y_0(\sigma)$ of $Y(\sigma)$.

\begin{itemize}
\item[$\bullet$] If $\sigma=0$ then $Y_0(0)=\{ \boldsymbol{\phi}_6(0), \boldsymbol{\phi}_2(0), \boldsymbol{\phi}_4(0), \boldsymbol{\phi}_5(0)\}$, where $\boldsymbol{\phi}_j(0)$, $j=6,5$, are in \eqref{def:eigen0} and $\boldsymbol{\phi}_j(0)$, $j=2,4$, in \eqref{def:eigen0'}.

\item[$\bullet$] If $0<\sigma<\sigma_c$ (see \eqref{def:sigma_c3}) then $Y_0(\sigma)=\{ \boldsymbol{\phi}_6(\sigma), \boldsymbol{\phi}_2(\sigma), \boldsymbol{\phi}_4(\sigma), \boldsymbol{\phi}_5(\sigma)\}$, where $\boldsymbol{\phi}_j(\sigma)$, $j=6,2,4,5$, are in \eqref{def:eigen}.

\item[$\bullet$] If $\sigma=\sigma_c$ then $Y_0(\sigma_c)=\{ \boldsymbol{\phi}_6(\sigma_c), \boldsymbol{\phi}_2(\sigma_c), \boldsymbol{\phi}_4(\sigma_c), \boldsymbol{\phi}_5(\sigma_c)\}$, where $\boldsymbol{\phi}_{j}(\sigma_c)$, $j=6,4,5$, are in \eqref{def:eigen} and $\boldsymbol{\phi}_{2}(\sigma_c)$ in \eqref{def:eigen'}.

\item[$\bullet$] If $\sigma>\sigma_c$ then $Y_0(\sigma)=\{ \boldsymbol{\phi}_4(\sigma), \boldsymbol{\phi}_5(\sigma)\}$, where $\boldsymbol{\phi}_j(\sigma)$, $j=4,5$, are in \eqref{def:eigen}.
\end{itemize}

Suppose, on the other hand, $\mu_0(\kappa;\beta)>1$, corresponding to supercritical waves. Recall \eqref{def:sigma_c12}, and suppose $0<-\sigma_{c,1}<\sigma_{c,2}$. We utilize the eigenfunctions and generalized eigenfunctions in Lemma \ref{lem:eps=0}, to analogously select an {\em ordered basis} $Y_0(\sigma)$ of $Y(\sigma)$.

\begin{itemize}
\item[$\bullet$] If $\sigma=0$ then $Y_0(0)=\{\boldsymbol{\phi}_6(0), \boldsymbol{\phi}_2(0), \boldsymbol{\phi}_4(0), \boldsymbol{\phi}_3(0), \boldsymbol{\phi}_1(0), \boldsymbol{\phi}_5(0)\}$, where $\boldsymbol{\phi}_j(0)$, $j=6,2,1,5$, are in \eqref{def:eigen0} and $\boldsymbol{\phi}_j(0)$, $j=4,3$, in \eqref{def:eigen0''}.

\item[$\bullet$] If $0<\sigma<-\sigma_{c,1}$ then $Y_0(\sigma)=\{\boldsymbol{\phi}_6(\sigma), \boldsymbol{\phi}_2(\sigma), \boldsymbol{\phi}_4(\sigma), \boldsymbol{\phi}_3(\sigma), \boldsymbol{\phi}_1(\sigma), \boldsymbol{\phi}_5(\sigma)\}$, where $\boldsymbol{\phi}_j(\sigma)$, $j=6,2,1,5,4,3$, are in \eqref{def:eigen}.

\item[$\bullet$] If $\sigma=-\sigma_{c,1}$ then 
$Y_0(-\sigma_{c,1})=\{\boldsymbol{\phi}_6(-\sigma_{c,1}), \boldsymbol{\phi}_2(-\sigma_{c,1}), \boldsymbol{\phi}_4(-\sigma_{c,1}),\boldsymbol{\phi}_3(-\sigma_{c,1}), \protect\linebreak\boldsymbol{\phi}_1(-\sigma_{c,1}), \boldsymbol{\phi}_5(-\sigma_{c,1})\}$,
where $\boldsymbol{\phi}_j(-\sigma_{c,1})$, $j=6,2,1, 5,4$, are in \eqref{def:eigen} and $\boldsymbol{\phi}_3(-\sigma_{c,1})$ in \eqref{def:eigen'}.

\item[$\bullet$] If $-\sigma_{c,1}<\sigma<\sigma_{c,2}$ then $Y_0(\sigma)=\{ \boldsymbol{\phi}_6(\sigma), \boldsymbol{\phi}_2(\sigma), \boldsymbol{\phi}_4(\sigma),\boldsymbol{\phi}_5(\sigma)\}$, where $\boldsymbol{\phi}_j(\sigma)$, $j=6,2,4,5$, are in \eqref{def:eigen}.

\item[$\bullet$] If $\sigma=\sigma_{c,2}$ then $Y_0(\sigma_{c,2})=\{ \boldsymbol{\phi}_6(\sigma_{c,2}), \boldsymbol{\phi}_2(\sigma_{c,2}), \boldsymbol{\phi}_4(\sigma_{c,2}),\boldsymbol{\phi}_5(\sigma_{c,2})\}$, where $\boldsymbol{\phi}_j(\sigma_{c,2})$, $j=6,4,5$, are in \eqref{def:eigen}, and $\boldsymbol{\phi}_2(\sigma_{c,2})$ in \eqref{def:eigen'}.

\item[$\bullet$] If $\sigma>\sigma_{c,2}$ then $Y_0(\sigma)=\{\boldsymbol{\phi}_4(\sigma),\boldsymbol{\phi}_5(\sigma)\}$, where $\boldsymbol{\phi}_j(\sigma)$, $j=4,5$, are in \eqref{def:eigen}.
\end{itemize}
\end{definition}

Suppose $-\sigma_{c,1}>\sigma_{c,2}>0$ (see Figure \ref{figure5} for the solutions to $\sigma_\pm(\cdot;\beta,\mu_0)=\sigma$). We analogously select an ordered basis. We omit the details.

\begin{remark*}\rm
The statements of Lemma \ref{lem:eps=0} and Definition \ref{def:proj} extend to $\sigma<0$ by symmetry. 
\end{remark*}

\subsection{Formulae for \texorpdfstring{$\boldsymbol{\Pi}(\sigma)$}{Lg}}\label{sec:proj}

Writing $\mathbf{u}_1=\begin{pmatrix}\varphi_1\\ u_1 \\ \eta_1\\ z_1\end{pmatrix}$ and $\mathbf{u}_2=\begin{pmatrix}\varphi_2\\ u_2\\ \eta_2 \\ z_2\end{pmatrix}$, we define an inner product for $Y$, in \eqref{def:Y}, as 
\begin{equation}\label{def:inner}
\langle \mathbf{u}_1, \mathbf{u}_2\rangle_Y
=\int^1_0(\varphi_1\varphi_2^*+{\varphi_1}_y{\varphi_2}_y^*)~dy+\int^1_0 u_1u_2^*~dy+\eta_1\eta_2^*+z_1z_2^*,
\end{equation}
where the superscript $*$ means complex conjugation. Let $\beta>0$ and $\kappa>0$. Recall \eqref{def:L}, \eqref{def:Y}, and \eqref{def:dom}, and a straightforward calculation leads to
\[
\langle \mathbf{L}(\lambda)\mathbf{u}_1,\mathbf{u}_2\rangle_Y
=\langle \mathbf{u}_1,\mathbf{L}^\dag(\lambda)\mathbf{u}_2\rangle_Y,
\]
where
\[
\mathbf{L}^\dag(\lambda):{\rm dom}(\mathbf{L}^\dag) \subset Y\to Y
\]
is defined as
\[
\mathbf{L}^\dag(\lambda)\mathbf{u}
=\begin{pmatrix} u+\varphi^\dag \\ 
-\varphi_{yy}+\varphi \\
{\displaystyle -\lambda^*\int_0^1u}+\left(1+\dfrac{{\lambda^*}^2}{2\mu_0}\right)z\\
{\displaystyle -\frac{\lambda^*}{2\beta}\int_0^1(y^2\varphi+2y{\varphi}_y)+\frac{1}{\beta}(\eta+u(1))} \end{pmatrix},
\quad\mathbf{u}=\begin{pmatrix}\varphi \\ u \\ \eta \\ z \end{pmatrix}.
\]
Here, $\mu_0=\mu_0(\kappa;\beta)$ is in \eqref{def:mu0} and
\begin{align*}
\varphi^\dag(x,y)=&\frac{\ch(y)}{\sh(1)}\left(-\int_0^1\ch(1-y)u(x,y)~dy+\frac{\lambda^*}{\mu_0}z(x)\right) \\
&+\int_0^y\sh(y-y')u(x,y')~dy'.
\end{align*}
Additionally,  
\[
{\rm dom}(\mathbf{L}^\dag)=\left\{\mathbf{u}\in H^2(0,1)\times H^1(0,1)\times \mathbb{C}\times \mathbb{C}:
\varphi_y(1)=\frac{z}{\mu_0}, \varphi_y(0)=0\right\}.
\]
Clearly, ${\rm dom}(\mathbf{L}^\dag)$ is dense in $Y$. See \cite[Section 4.5]{HY2023} for an analogous calculation for Stokes' gravity waves.

We proceed to determine the eigenvalues and eigenfunctions of $\mathbf{L}^\dag(i\sigma)$ and, in turn, derive formulae for $\boldsymbol{\Pi}(\sigma)$. 

\begin{notation*}
Below we may write $k_j$ instead of $k_j(\sigma)$ for simplicity of notation.
\end{notation*}

Recall Section \ref{sec:unimodal}, and suppose $0<\mu_0(\kappa;\beta)<1$, corresponding to subcritical capillary-gravity waves. Recall \eqref{def:sigma_c3}.

\begin{itemize}
\item[$\bullet$] Setting $\sigma=0$, we utilize the eigenvalues and eigenfunctions in Lemma \ref{lem:eps=0}, and make a straightforward calculation to show that: $-ik_j(0)=(-1)^{j+1}i\kappa$, $j=6,5$, are simple eigenvalues of $\mathbf{L}^\dag(i0)$, and 
\be \label{def:adj0}
\boldsymbol{\psi}_j(0)=\begin{pmatrix} \dfrac{ k_j\sh(k_j)  \ch(y)  -  k_j^2 \sh(1) \ch(k_j y)}{\sh(1) ( k_j^2 - 1) ( k_j^2 - \mu_0  \sh(k_j)^2 +  \beta  k_j^2 \mu_0  \sh(k_j)^2)}\\
\dfrac{ik_j \ch(k_j y)}{k_j^2 - \mu_0  \sh(k_j)^2 +  \beta  k_j^2 \mu_0  \sh(k_j)^2}\\
-\dfrac{i\mu_0  \sh(k_j)}{ k_j^2 - \mu_0  \sh(k_j)^2 +  \beta  k_j^2 \mu_0  \sh(k_j)^2}\\
-\dfrac{ k_j \mu_0  \sh(k_j)}{ k_j^2 - \mu_0  \sh(k_j)^2 +  \beta  k_j^2 \mu_0  \sh(k_j)^2}
\end{pmatrix}, 
\ee 
$j=6,5$, are the corresponding eigenfunctions, satisfying $\langle \boldsymbol{\phi}_j(0),\boldsymbol{\psi}_{j'}(0)\rangle_Y=\delta_{jj'}$, $j,j'=6,5$, where $\boldsymbol{\phi}_j(0)$, $j=6,5$, are in \eqref{def:eigen0}. 

Additionally, $-ik_j(0)=0$, $j=4,2$, is an eigenvalue of $\mathbf{L}^\dag(i0)$ with algebraic multiplicity $2$ and geometric multiplicity $1$, and 
\be 
\boldsymbol{\psi}_4(0)=\frac{1}{\mu_0-1}\begin{pmatrix}\mu_0 - \dfrac{\ch(y)}{\sh(1)}\\0\\0\\-\mu_0\end{pmatrix}\quad\text{and}\quad
\boldsymbol{\psi}_2(0)=\frac{\mu_0}{\mu_0-1}\begin{pmatrix}0\\ 1\\-1\\0\end{pmatrix}
\ee 
are the corresponding generalized eigenfunctions, satisfying $\langle \boldsymbol{\phi}_j(0),\boldsymbol{\psi}_{j'}(0)\rangle_Y=\delta_{jj'}$, $j,j'=4,2$, where $\boldsymbol{\phi}_j(0)$, $j=4,2$, are in \eqref{def:eigen0'}. Furthermore, 
\[
\langle \boldsymbol{\phi}_j(0),\boldsymbol{\psi}_{j'}(0)\rangle_Y=\delta_{jj'},\quad j,j'=6,2,4,5.
\]
Recall Definition \ref{def:proj} and we therefore define $\boldsymbol{\Pi}(0):{\rm dom}(\mathbf{L}) \subset Y\to Y(\sigma)$ as
\ba\label{def:Pi0} 
\boldsymbol{\Pi}(0)\mathbf{u}=\sum_{j=6,2,4,5}\langle \mathbf{u},\boldsymbol{\psi}_j(0)\rangle_Y\boldsymbol{\phi}_j(0).
\ea 

\begin{remark*}
Suppose $\kappa=1$. The first component of \eqref{def:adj0} appears to be not defined. However,
\begin{align*}
\frac{\displaystyle \sh(\kappa)\ch(y)-\kappa\sh(1)\ch(\kappa y)}{1-\kappa}\to y\sh(1)\sh(y) +(\sh(1) - \ch(1))\ch(y)
\end{align*}
as $\kappa\to1$. We therefore define $\boldsymbol{\psi}_j(0)=\lim_{\kappa\to1}\boldsymbol{\psi}_j(0)$, $j=6,5$, which satisfy $\langle \boldsymbol{\phi}_{j}(0),\boldsymbol{\psi}_{j'}(0)\rangle=\delta_{jj'}$, $j=6,2,4,5$ and $j'=6,5$.
\end{remark*}

\item[$\bullet$] Let $0<\sigma<\sigma_c$ or $\sigma>\sigma_c$. Utilizing the result of Lemma \ref{lem:eps=0}, we make an explicit calculation to show that: $-ik_j(\sigma)$, $j=6,2,4,5$ (if $\sigma<\sigma_c$) or $j=4,5$ (if $\sigma>\sigma_c)$, are simple eigenvalues of $\mathbf{L}^\dag(i\sigma)$, and 
\be \label{def:adj}
\boldsymbol{\psi}_j(\sigma)=\begin{pmatrix} 
p_{j,1}\ch(k_jy)+p_{j,2} \ch(y)\\
i(1-k_j^2)k_j^{-1}p_{j,1}\ch(k_jy)\\
p_{j,3}\\\mu_0(k_jp_{j,1}\sh(k_j)+p_{j,2}\sh(1))\end{pmatrix},
\ee 
$j=6,2,4,5$ or $j=4,5$, are the corresponding eigenfunctions, satisfying
\[
\langle \boldsymbol{\phi}_{j}(\sigma),\boldsymbol{\psi}_{j'}(\sigma)\rangle_Y=\delta_{jj'}\quad \text{$j=5,6,4,2$ or $j=5,4$},
\]
where $\boldsymbol{\phi}_j(\sigma)$ are in \eqref{def:eigen}. Here
\begin{align*} 
p_{j,1}=&\frac{k_j \sh(k_j) (k_j - \sigma) (\beta k_j^2 + 1)}{k_j^2 - 1} \\ 
& \cdot (\sigma \ch(k_j)^3- k_j^2 \sh(k_j) - k_j \ch(k_j) - \sigma \ch(k_j)  + k_j \ch(k_j)^3\\
&\quad + k_j \sigma \sh(k_j) + \beta k_j^3 \ch(k_j) - \beta k_j^4 \sh(k_j) - \beta k_j^3 \ch(k_j)^3 \\
& \quad + 3 \beta k_j^2 \sigma \ch(k_j)^3 - 3 \beta k_j^2 \sigma \ch(k_j) + \beta k_j^3 \sigma \sh(k_j))^{-1},\\
p_{j,2}=&\frac{k_j \sh(k_j)^2 (k_j \sigma - 1) (\beta k_j^2 + 1)}{\sh(1) (k_j^2 - 1)} \\
& \cdot (\sigma \ch(k_j)^3 - k_j^2 \sh(k_j) - k_j \ch(k_j) - \sigma \ch(k_j) + k_j \ch(k_j)^3 \\
&\quad + k_j \sigma \sh(k_j) + \beta k_j^3 \ch(k_j) - \beta k_j^4 \sh(k_j) - \beta k_j^3 \ch(k_j)^3 \\
& \quad + 3 \beta k_j^2 \sigma \ch(k_j)^3 - 3 \beta k_j^2 \sigma \ch(k_j) + \beta k_j^3 \sigma \sh(k_j))^{-1},\\
p_{j,3}=&-i\sh(k_j)(2 k_j^2)^{-1} \\ 
&\cdot (2\sigma^2 \sh(k_j)  - 2k_j^3 \ch(k_j)  - 2k_j \sigma \sh(k_j) \\
& \quad - 2k_j \sigma^2 \ch(k_j)  + 4k_j^2 \sigma \ch(k_j)  + k_j^2 \sigma^2 \sh(k_j) \\
& \quad + 2\beta k_j^2 \sigma^2 \sh(k_j)  + \beta k_j^4 \sigma^2 \sh(k_j)  - 2\beta k_j^3 \sigma \sh(k_j) )\\
&\cdot(\sigma \ch(k_j)^3 - k_j^2 \sh(k_j) - k_j \ch(k_j) - \sigma \ch(k_j) + k_j \ch(k_j)^3 \\
&\quad + k_j \sigma \sh(k_j) + \beta k_j^3 \ch(k_j) - \beta k_j^4 \sh(k_j) - \beta k_j^3 \ch(k_j)^3\\ 
&\quad + 3 \beta k_j^2 \sigma \ch(k_j)^3 - 3 \beta k_j^2 \sigma \ch(k_j) + \beta k_j^3 \sigma \sh(k_j))^{-1}.
\end{align*}
Recall Definition \ref{def:proj} and we define
\begin{equation}\label{def:Pi;high}
\boldsymbol{\Pi}(\sigma)\mathbf{u}=\sum_{\boldsymbol{\phi}_j(\sigma)\in Y_0(\sigma)}\langle \mathbf{u},\boldsymbol{\psi}_j(\sigma)\rangle_Y\boldsymbol{\phi}_j(\sigma).
\end{equation}

\item[$\bullet$] Let $\sigma=\sigma_c$. Utilizing the result of Lemma \ref{lem:eps=0}, we make an explicit calculation to show that $-ik_{j}(\sigma_c)$, $j=4,5$, are simple eigenvalues of $\mathbf{L}^\dag(i\sigma_c)$, and \eqref{def:adj}, $j=4,5$, are the corresponding eigenfunctions. 

Additionally, $-ik_{j}(\sigma_c)=-ik_c$, $j=6,2$, is an eigenvalue of $\mathbf{L}^\dag(i\sigma_c)$ with algebraic multiplicity $2$ and geometric multiplicity $1$, and  
\begin{align}
&\boldsymbol{\psi}_{6}(\sigma_c)=\begin{pmatrix}q_{6,1}\ch(k_{c}y)+q_{6,2}\ch(y)\\
i(1-k_c^2)k_c^{-1}q_{6,1}\ch(k_{c}y)\\ q_{6,3}\\ 
\mu _{0}(k_{c}q_{6,1}\sh(k_{c})+q_{6,2}\sh(1))\end{pmatrix} \notag \\
\intertext{(see also \eqref{def:adj}) and}
&\boldsymbol{\psi}_{2}(\sigma_c)=\begin{pmatrix}q_{2,1}\ch(k_{c}y)+q_{2,2}\ch(y)+q_{2,3}y\sh(k_{c}y)\\ 
q_{2,4}\ch(k_{c}y)+q_{2,5}y\sh(k_{c}y)\\ q_{2,6} \\ 
\mu _0(q_{2,1}k_{c}\sh(k_{c})+q_{2,2}\sh(1)+q_{2,3}k_{c}\ch(k_{c})+q_{2,3}\sh(k_{c})) \end{pmatrix}\label{def:adj0'}
\end{align}
are the corresponding generalized eigenfunctions, satisfying 
\[
\langle \boldsymbol{\phi}_{j}(\sigma_c),\boldsymbol{\psi}_{j'}(\sigma_c)\rangle_Y=\delta_{jj'}, \quad j,j'=6,2,4,5,
\]
where $\boldsymbol{\phi}_j(\sigma_c)$, $j=4,5,6$, are in \eqref{def:eigen}, and $\boldsymbol{\phi}_2(\sigma_c)$ in \eqref{def:eigen'}. The formulae for $q_{6,\ell}$, $\ell=1,2,3$, and $q_{2,\ell}$, $\ell=1,2,\dots,6$, are omitted. We then define
\[
\boldsymbol{\Pi}(\sigma_c)\mathbf{u}=\sum_{\boldsymbol{\phi}_j(\sigma_c)\in Y_0(\sigma_c)}\langle \mathbf{u},\boldsymbol{\psi}_j(\sigma_c)\rangle_Y\boldsymbol{\phi}_j(\sigma_c).
\]
\end{itemize}

See \cite[Section 4.5]{HY2023} for an analogous calculation for Stokes' gravity waves.

Suppose, on the other hand, $\mu_0(\kappa;\beta)>1$, corresponding to supercritical capillary-gravity waves. We utilize the eigenvalues and eigenfunctions in Lemma \ref{lem:eps=0}, and make an explicit calculation to analogously determine the eigenvalues and eigenfunctions of $\mathbf{L}^\dag(\sigma)$ and, in turn, derive formulae for $\boldsymbol{\Pi}(\sigma)$. We do not record the result. See also Sections \ref{sec:sigma N=1} and \ref{sec:sigma N=2}.

\subsection{The spectrum of \texorpdfstring{$\mathcal{L}(\eps)$}{Lg} and the periodic Evans function}\label{sec:reduction}

We turn our attention to $|\eps|\neq0,\ll1$. Let
\[
\lambda=i\sigma+\delta,\quad\text{where}\quad \sigma\in\mathbb{R}, 
\quad\delta\in \mathbb{C}\quad\text{and}\quad |\delta|\ll1,
\] 
and we rewrite \eqref{eqn:LB} as 
\ba \label{eqn:LB;delta}
\mathbf{u}_x&=
\mathbf{L}(i\sigma)\mathbf{u}+(\mathbf{L}(i\sigma+\delta)-\mathbf{L}(i\sigma))\mathbf{u}
+\mathbf{B}(x;i\sigma+\delta,\eps)\mathbf{u}
\\
&=:\mathbf{L}(i\sigma)\mathbf{u}+\mathbf{B}(x;\sigma,\delta,\eps)\mathbf{u},
\ea 
where $\mathbf{L}(i\sigma):{\rm dom}(\mathbf{L})\subset Y\to Y$ is in \eqref{def:L}, \eqref{def:Y}, and \eqref{def:dom}, and 
\[
\mathbf{B}(x;\sigma,\delta,\eps):\mathbb{R}\times {\rm dom}(\mathbf{L})\subset \mathbb{R}\times Y\to Y.
\]
Clearly, $\mathbf{B}(x;\sigma,0,0)=\mathbf{0}$. It follows from Section \ref{sec:L} that $\mathbf{B}(x;\sigma,\delta,\eps)$ is $T(=2\pi/\kappa)$-periodic and infinitely differentiable in $x$, and, after complexifying $\eps$, jointly analytic in $(\delta,\eps)$ at $(0,0)$, for each $\sigma\in\mathbb{R}$ and locally uniformly for $x\in[0,T]$. Our analysis involves leading order terms in a power series expansion of $\mathbf{B}(x;\sigma,\delta,\eps)$ about $(\delta,\eps)=(0,0)$, but not all higher order terms. See \eqref{def:Bexp} and Appendix \ref{app:Bexp} for more details. 

Recall Definition \ref{def:proj} and Section \ref{sec:proj}. Let $\mathbf{u}\in {\rm dom}(\mathbf{L})$, and we introduce
\begin{equation*}
\mathbf{v}=\boldsymbol{\Pi}(\sigma)\mathbf{u}\quad\text{and}\quad 
\mathbf{w}=(\mathbf{1}-\boldsymbol{\Pi}(\sigma))\mathbf{u},
\end{equation*}
where $\mathbf{1}$ means the identity operator, so that \eqref{eqn:LB;delta} becomes
\ba\label{eqn:LB;u12}
&{\mathbf{v}}_x=\mathbf{L}(i\sigma)\mathbf{v}
+\boldsymbol{\Pi}(\sigma)\mathbf{B}(x;\sigma,\delta,\eps)(\mathbf{v}+\mathbf{w}),\\
&{\mathbf{w}}_x=\mathbf{L}(i\sigma)\mathbf{w}
+(\mathbf{1}-\boldsymbol{\Pi}(\sigma))\mathbf{B}(x;\sigma,\delta,\eps)(\mathbf{v}+\mathbf{w}).
\ea 
We emphasize that $\mathbf{v}\in Y(\sigma)$ is represented in the ordered basis $Y_0(\sigma)$ of $Y(\sigma)$. 

Crucial to our analysis is applying Mielke's center manifold reduction theorem and, specifically, \cite[Theorem 1]{Mielke;reduction} to \eqref{eqn:LB;u12}. Hypotheses \mbox{(A1)} and \mbox{(A2)} in \cite{Mielke;reduction} clearly hold, and it remains to verify (A3). Taking into account \cite[Lemma 2.1]{ik1992}, \cite[Proposition 3.2]{BGT}, \cite[Proposition~2]{MR1720395}, and \cite[Lemma 3.4]{GM}, among others, it suffices to establish the following resolvent estimate. 

\begin{lemma}[Resolvent estimate for $\mathbf{L}(i\sigma)$]\label{lem:resolvent}
Recall Section \ref{sec:eps=0}. Let $\sigma\in\mathbb{R}$. If $k\in\mathbb{R}$ and $|k|$ is sufficiently large, then $ik$ belongs to the resolvent set of $\mathbf{L}(i\sigma):{\rm dom}(\mathbf{L}) \subset Y \to Y$. Moreover, 
\[
\|(ik -\mathbf{L}(i\sigma))^{-1}\|_{Y\rightarrow Y}\leq \frac{C}{|k|}
\] 
for some $C>0$. 
\end{lemma}

\begin{proof}
Since the spectrum of $\mathbf{L}(i\sigma)$ consists of finitely many and purely imaginary eigenvalues, by Lemma \ref{lem:eps=0}, there exists $k_0(\sigma)>0$ such that $ik$ belongs to the resolvent set of $\mathbf{L}(i\sigma)$ whenever $|k|>k_0(\sigma)$.

Fix $\beta>0$ and $\kappa>0$. Recall Section \ref{sec:L}, and consider
\[
(ik -\mathbf{L}(i\sigma))\mathbf{u}=\mathbf{f}, \quad \mathbf{u}\in \text{dom}(\mathbf{L})\subset Y\quad \text{and} \quad \mathbf{f}\in Y,
\]
or, equivalently, 
\begin{equation}\label{eqn:resolvent}
\begin{aligned}
&ik \varphi-u+\dfrac{i\sigma zy^2}{2\beta}=\varphi_f &&\text{for $0<y<1$}, \\
& ik u+\varphi_{yy}+i\sigma\eta=u_f &&\text{for $0<y<1$}, \\
&ik\eta-\dfrac{z}{\beta}=\eta_f && \text{at $y=1$},\\
&ik z-\eta+\dfrac{\sigma^2\eta}{2\mu_0}-\dfrac{i\sigma}{\mu_0}\varphi(1)+\dfrac{u(1)}{\mu_0}=z_f && \text{at $y=1$}, 
\end{aligned}
\end{equation}
subject to 
\begin{equation}\label{eqn:resolventBC}
\varphi_y(1)=-\frac{z}{\beta}\quad\text{and}\quad \varphi_y(0)=0,
\end{equation}
where $\mu_0=\mu_0(\kappa;\beta)$ is in \eqref{def:mu0}. 

The first equation of \eqref{eqn:resolvent} yields
\[
\|u-ik\varphi\|_{H^1(0,1)}^2=\left\|\varphi_f-\frac{i\sigma zy^2}{2\beta}\right\|_{H^1(0,1)}^2,
\]
where 
\begin{align*}
&\|u-ik\varphi\|_{H^1}^2=\|u\|_{H^1}^2+k^2\|\varphi\|_{H^1}^2-2k\text{Im}\int_0^1 u\varphi^*-2k\text{Im}\int_0^1 u_y\varphi_y^*,\\
&\Big\|\varphi_f-\frac{i\sigma z}{2\beta}y^2\Big\|_{H^1}^2
\leq 2\|\varphi_f\|_{H^1}^2+2\Big\|\frac{i\sigma z}{2\beta}y^2\Big\|_{H^1}^2=2\|\varphi_f\|_{H^1}^2+\frac{23\sigma^2}{15\beta^2}|z|^2.
\end{align*}
Throughout the proof, the superscript $*$ means complex conjugation. Putting these together, 
\be\label{estimate1}
\begin{aligned}
\|u\|_{H^1}^2+k^2\|\varphi\|_{H^1}^2 \leq & 2\|\varphi_f\|_{H^1}^2+\frac{23\sigma^2}{15\beta^2}|z|^2 \\
&+2k\text{Im}\int_0^1 u\varphi^*+2k\text{Im}\int_0^1 u_y\varphi_y^*.
\end{aligned}
\ee 
Applying Young's inequality to the third term on the right side of \eqref{estimate1},
\be \label{estimate2}
2k\text{Im}\int_0^1 u\varphi^*\leq \frac{1}{C_1}\|\varphi\|_{L^2}^2+C_1k^2 \|u\|_{L^2}^2
\ee 
for any $C_1>0$. The fourth term on the right side of \eqref{estimate1} becomes
\begin{align*} 
2k\text{Im}\int_0^1 u_y\varphi_y^*
=&2k\text{Im}\left(u(1)\varphi_y(1)^*-\int_0^1 u\varphi_{yy}^*\right)\\
=&-2k\text{Im}\left(\mu_0z_f-i\mu_0kz+\mu_0\eta+i\sigma\varphi(1)-\frac{\sigma^2\eta}{2}\right)\frac{z^*}{\beta}\\
&-2k\text{Im}\int_0^1 u(u_f-iku-i\sigma\eta)^*\\
=&-2k\text{Im} \left(\mu_0z_f+\mu_0\eta+i\sigma\varphi(1)-\frac{\sigma^2\eta}{2}\right)\frac{z^*}{\beta}-\frac{2\mu_0k^2}{\beta}|z|^2\\
&-2k\text{Im}\int_0^1 u(u_f-i\sigma\eta)^*-2k^2\|u\|_{L^2}^2.
\end{align*}
Here, the first equality follows after an integration by parts using the second equation of \eqref{eqn:resolventBC}, and the second equality follows from substituting $u(1)$ with the fourth equation of \eqref{eqn:resolvent}, and $\varphi_y(1)$ with the first equation of \eqref{eqn:resolventBC}, and substituting $\varphi_{yy}$ with the second equation of \eqref{eqn:resolvent}. Applying Young's inequality to the right side, 
\begin{align} 
&2k\text{Im}\int_0^1 u_y\varphi_y^*\notag\\
&\begin{aligned}
\leq & \frac{\mu_0^2}{C_2\beta^2}|z_f|^2+C_2 k^2|z|^2
+\frac{(\mu_0-\frac{\sigma^2}{2})^2}{C_3\beta^2}|\eta|^2+C_3 k^2|z|^2 \\
&+\frac{\sigma^2}{C_4\beta^2}|\varphi(1)|^2+C_4 k^2|z|^2-\frac{2\mu_0k^2}{\beta}|z|^2 \\
& +\frac{1}{C_5}\|u_f\|_{L^2}^2+C_5k^2\|u\|_{L^2}^2
+\frac{\sigma^2}{C_6}|\eta|^2+C_6k^2\|u\|_{L^2}^2-2k^2\|u\|_{L^2}^2 \end{aligned} \notag \\
&\begin{aligned}\label{estimate4}
\leq & \frac{1}{C_5}\|u_f\|_{L^2}^2+\frac{\mu_0^2}{C_2\beta^2}|z_f|^2+\frac{C'\sigma^2}{C_4\beta^2}\|\varphi\|_{H^1}^2
+\left(\frac{(\mu_0-\frac{\sigma^2}{2})^2}{C_3\beta^2}+\frac{\sigma^2}{C_6}\right)|\eta|^2\\
&+(C_5+C_6-2)k^2\|u\|_{L^2}^2
+\left(C_2 +C_3+C_4-\frac{2\mu_0}{\beta}\right)k^2|z|^2
\end{aligned}
\end{align}
for any $C_1,C_2,\dots, C_6>0$ for some $C'>0$, where the second inequality follows because 
\[
|\varphi(1)|\leq \|\varphi\|_{C[0,1]}\leq C'\|\varphi\|_{H^1(0,1)}.
\]

Substituting \eqref{estimate2} and \eqref{estimate4} into \eqref{estimate1}, 
\ba\label{estimate5}
&\left(k^2-\frac{1}{C_1}-\frac{C'\sigma^2}{C_4\beta^2}\right)\|\varphi\|_{H^1}^2+\|u\|_{H^1}^2+(2-C_1-C_5-C_6)k^2\|u\|_{L^2}^2\\
&+\left(\left(\frac{2\mu_0}{\beta}-C_2-C_3-C_4\right)k^2-\frac{23\sigma^2}{15\beta^2}\right)|z|^2\\
&\leq 2\|\varphi_f\|_{H^1}^2+\frac{1}{C_5}\|u_f\|_{L^2}^2+\frac{\mu_0^2}{C_2\beta^2}|z_f|^2+\left(\frac{(\mu_0-\frac{\sigma^2}{2})^2}{C_3\beta^2}+\frac{\sigma^2}{C_6}\right)|\eta|^2.
\ea 
On the other hand, the third equation of \eqref{eqn:resolvent} yields 
\begin{equation}\label{estimate6}
\begin{aligned}
\frac{1}{\beta^2}|z|^2+k^2|\eta|^2 
&=|\eta_f|^2+2k\text{Im}\Big(\frac{z\eta^*}{\beta}\Big) \\
&\leq |\eta_f|^2+\frac{1}{C_7\beta^2}|z|^2+C_7k^2|\eta|^2
\end{aligned}
\end{equation}
for any $C_7>0$. Adding \eqref{estimate5} and \eqref{estimate6}, therefore, 
\ba\label{estimate7}
&\left(k^2-\frac{1}{C_1}-\frac{C'\sigma^2}{C_4\beta^2}\right)\|\varphi\|_{H^1}^2+(2-C_1-C_5-C_6)k^2\|u\|_{L^2}^2\\
&\quad +\left((1-C_7)k^2-\frac{(\mu_0-\frac{\sigma^2}{2})^2}{C_3\beta^2}-\frac{\sigma^2}{C_6}\right)|\eta|^2\\
&\quad +\left(\left(\frac{2\mu_0}{\beta}-C_2-C_3-C_4\right)k^2-\frac{23\sigma^2}{15\beta^2}+\frac{1}{\beta^2}-\frac{1}{C_7\beta^2}\right)|z|^2\\
&\leq 2\|\varphi_f\|_{H^1}^2+\frac{1}{C_5}\|u_f\|_{L^2}^2+|\eta_f|^2+\frac{\mu_0^2}{C_2\beta^2}|z_f|^2.
\ea  
We choose 
\[
C_1,C_5,C_6=\frac{1}{3}, \quad C_2,C_3,C_4=\frac{\mu_0}{3\beta}, \quad C_7=\frac{1}{2},
\]
so that \eqref{estimate7} becomes 
\begin{align*}
k^2&(\|\varphi\|_{H^1}^2+\|u\|_{L^2}^2+|\eta|^2+|z|^2)\\
&\leq C(\sigma,\beta,\mu_0,C')(\|\varphi_f\|_{H^1}^2+\|u_f\|_{L^2}^2+|\eta_f|^2+|z_f|^2)
\end{align*}
for some positive constant $C$ for $|k|$ is sufficiently large. 
This completes the proof. 
\end{proof}

Therefore, it follows from Mielke's center manifold reduction theorem and, specifically, \cite[Theorem 1]{Mielke;reduction} that: 
\[
\mathbf{u}(x)=\mathbf{v}(x)+\mathbf{w}(x,\mathbf{v}(x);\sigma,\delta,\eps)
\]
makes a bounded solution to \eqref{eqn:LB;delta} and, hence, \eqref{eqn:LB}, where $\mathbf{v}(x)\in Y(\sigma)$ and
\begin{equation}\label{def:u2}
\mathbf{w}(x,\mathbf{v}(x);\sigma,\delta,\eps): \mathbb{R}\times Y(\sigma) \to {\rm dom}(\mathbf{L}),
\end{equation}
for $\sigma\in\mathbb{R}$ for $\delta\in\mathbb{C}$ and $|\delta|\ll1$ for $\eps\in\mathbb{R}$ and $|\eps|\ll1$ {\em if and only if} $\mathbf{v}\in Y(\sigma)$ makes a bounded solution to 
\be \label{eqn:LB;u1}
{\mathbf{v}}_x=\mathbf{L}(i\sigma)\mathbf{v}
+\boldsymbol{\Pi}(\sigma)\mathbf{B}(x;\sigma,\delta,\eps)
(\mathbf{v}+\mathbf{w}(x,\mathbf{v};\sigma,\delta,\eps)).
\ee 
Notably, we turn \eqref{eqn:LB;delta}, for which $\mathbf{u}\in{\rm dom}(\mathbf{L})$ and $\dim({\rm dom}(\mathbf{L}))=\infty$, into \eqref{eqn:LB;u1}, for which $\mathbf{v}\in Y(\sigma)$ and $\dim(Y(\sigma))<\infty$. Since $\mathbf{B}(x;\sigma,\delta,\eps)$ is jointly analytic in $(\delta,\eps)$ at $(0,0)$, it follows from \cite{Mielke;reduction} and, particularly, hypothesis (A3) that $\mathbf{w}(x,\mathbf{v};\sigma,\delta,\eps)$ is jointly analytic in $(\mathbf{v};\delta,\eps)$ at $(\mathbf{0};0,0)$, after complexifying $\eps$, for fixed $\delta\in\mathbb{R}$ and locally uniformly for $x\in [0,T]$. 

We further rewrite \eqref{eqn:LB;u1} as a matrix ODE in the ordered basis $Y_0(\sigma)$ of $Y(\sigma)$ as 
\be\label{eqn:A}
\mathbf{v}_x=\mathbf{A}(x;\sigma,\delta,\eps)\mathbf{v}.
\ee 
Clearly, $\mathbf{A}(x;\sigma,\delta,\eps)$ is $T(=2\pi/\kappa)$-periodic and infinitely differentiable in $x$, and jointly analytic in $(\delta,\eps)$ at $(0,0)$, after complexifying $\eps$, for fixed $\sigma\in\mathbb{R}$ and locally uniformly for $x\in[0,T]$. Our analysis involves leading order terms in a power series expansion of $\mathbf{A}(x;\sigma,\delta,\eps)$ about $(\delta,\eps)=(0,0)$, but not all higher order terms. See \eqref{eqn:v_k(0)}, \eqref{def:f(0)}, \eqref{eqn:v_k(0):sigma}, and \eqref{def:f(sigma)} for more details. We pause to remark that $\mathbf{A}(x;\sigma,\delta,\eps)$ is a square matrix of order $\dim(Y(\sigma))$. Recall from Definition \ref{def:proj} that $\dim(Y(\sigma))=2$ or $4$ or $6$, depending on $\mu_0$ and $\sigma$.

It follows from Floquet theory \cite{Chicone, KK2013} that if $\mathbf{v}$ is a bounded solution of \eqref{eqn:A} then, necessarily, 
\[
\mathbf{v}(x+T)=e^{ik T}\mathbf{v}(x)\quad\text{for some $k\in\mathbb{R}$, the Floquet exponent}.
\]
Adapting the approach, for instance, of \cite{Gardner;evans1,Gardner;evans2}, we thus introduce Gardner's periodic Evans function. See also \cite{HY2023} for Stokes' gravity waves.

\begin{definition}[The periodic Evans function]\label{def:Evans}\rm
Let $\lambda=i\sigma+\delta$, $\sigma\in \mathbb{R}$, $\delta\in\mathbb{C}$ for $|\delta|\ll1$, and $\eps\in\mathbb{R}$ for $|\eps|\ll1$, and let $\mathbf{V}(x;\sigma,\delta,\eps)$ denote the fundamental solution to \eqref{eqn:A} satisfying $\mathbf{V}(0;\sigma,\delta,\eps)=\mathbf{I}$, where $\mathbf{I}$ is the identity matrix. 

We define the {\em monodromy matrix} for \eqref{eqn:A} as 
\begin{equation}\label{def:monodromy}
\text{$\mathbf{V}(T;\sigma,\delta,\eps)$,\quad where $T>0$ is the period of $\mathbf{A}(x;\sigma,\delta,\eps)$},
\end{equation}
and the {\em periodic Evans function} for \eqref{eqn:A} as
\be\label{def:Delta}
D(\lambda,k,\eps)=\det(e^{ikT}\mathbf{I}-\mathbf{V}(T;\sigma,\delta,\eps)), \quad k\in\mathbb{R}.
\ee 
\end{definition}

It is evident from the theory of ODEs that $\mathbf{V}(T;\sigma,\delta,\eps)$ is jointly analytic in $(\delta,\eps)$ at $(0,0)$, after complexifying $\eps$, for fixed $\sigma\in\mathbb{R}$ and locally uniformly for $x\in[0,T]$. 
Consequently, $D(i\sigma+\delta,k,\eps)$ is jointly analytic in $(\delta,k,\eps)$ at $(0,k_0,0)$, after complexifying $k$ and $\eps$, for fixed $\sigma\in\mathbb{R}$ and $k_0\in\mathbb{R}$. 

\begin{notation*}
Since
\[
D(\lambda,K\kappa+k,\eps)=D(\lambda,k,\eps)
\quad\text{for each $K\in\mathbb{Z}$},
\]
by periodicity, we identify $D(\lambda,k+K\kappa,\eps)$ with $D(\lambda,k,\eps)$ for convenience. 
\end{notation*}

Recall Section \ref{sec:L} and, particularly, Definition \ref{def:spec}. It follows from Mielke's center manifold reduction theorem \cite{Mielke;reduction} and Floquet theory \cite{Chicone,KK2013} that $\lambda\in{\rm spec}(\mathcal{L}(\eps))$ if and only if
\[
D(\lambda,k,\eps)=0\quad\text{for some $k\in\mathbb{R}$}
\]
for $\sigma\in\mathbb{R}$ and $|\lambda-i\sigma|\ll1$ for $\eps\in\mathbb{R}$ and $|\eps|\ll1$. See also \cite{Gardner;evans1,Gardner;evans2}. Therefore, we identify ${\rm spec}(\mathcal{L}(\eps))$ with the roots of $D(\cdot,k,\eps)$ in the vicinity of $i\sigma$.

\begin{remark*}[Spectral stability and instability]\rm
Recall Definition \ref{def:stability} and Lemma \ref{lem:symm}. A steady periodic capillary-gravity wave of sufficiently small amplitude is spectrally stable if and only if 
\begin{equation*}
\{\lambda\in\mathbb{C}:D(\lambda,k,\eps)=0\text{ for some $k\in\mathbb{R}$}\}\subset i\mathbb{R}
\end{equation*}
for $\eps\in\mathbb{R}$ and $|\eps|\ll1$.
\end{remark*}

\begin{lemma}[The periodic Evans function for $\eps=0$]\label{lem:dispersion}
If $\eps,\delta=0$ then \eqref{def:Delta} satisfies $D(i\sigma,k_j(\sigma),0)=0$ for each $\sigma\in\mathbb{R}$, where $k_j(\sigma)$ is in Lemma \ref{lem:eps=0}.
\end{lemma}

\begin{proof}
If $\eps,\delta=0$ then \eqref{eqn:LB;u1} becomes ${\mathbf{v}}_x=\mathbf{L}(i\sigma)\mathbf{v}$. We utilize the eigenvalues and eigenfunctions in Lemma \ref{lem:eps=0}, and make an explicit calculation of \eqref{def:monodromy} and \eqref{def:Delta}, to evaluate $\mathbf{V}(T;\sigma,0,0)$, where $T>0$ is the period, and, in turn, $D(i\sigma, k_j(\sigma),0)$. 
We omit the details.
\end{proof}

We are interested in the roots of $D(\lambda,k,\eps)$ for $(\lambda,k,\eps)$ in the vicinity of $(i\sigma,k_j(\sigma),0)$. Since \eqref{def:Delta} is jointly analytic in its arguments, we make a power series expansion of $D(i\sigma+\delta,k_j(\sigma)+\gamma,\eps)$ about $(\delta,\gamma,\eps)=(0,0,0)$, and investigate the roots of its leading order terms in $\delta$, $\gamma$, and $\eps$. To this end, we prepare a power series expansion of $\mathbf{V}(T;\sigma,\delta,\eps)$ about $(\delta,\eps)=(0,0)$.

\subsection{Expansion of the monodromy matrix}\label{sec:expansion_monodromy}

Recall Section \ref{sec:reduction}. We begin by writing \eqref{def:monodromy} as
\begin{equation}\label{def:Vexp}
\mathbf{V}(x;\sigma,\delta,\eps)=\sum_{m+n=0}^\infty \mathbf{V}^{(m,n)}(x;\sigma)\delta^m\eps^n
\end{equation}
for $\delta\in\mathbb{C}$ and $|\delta|\ll1$ for $\eps\in\mathbb{R}$ and $|\eps|\ll 1$, where 
\begin{subequations}\label{def:V(m,n)}
\begin{equation}
\mathbf{V}^{(m,n)}(x;\sigma)=(v_{jk}^{(m,n)}(x;\sigma)),
\end{equation}
and $v_{jk}^{(m,n)}(x;\sigma)$ are to be determined
\begin{equation}
\text{for $j$, $k$ such that $\boldsymbol{\phi}_{j}(\sigma), \boldsymbol{\phi}_{k}(\sigma)\in Y_0(\sigma)$} \quad
\text{for $m,n=0,1,2,\dots$}.
\end{equation}
\end{subequations}
We emphasize that \eqref{def:V(m,n)} is in the ordered basis $Y_0(\sigma)$ of $Y(\sigma)$. Recall Definition \ref{def:Evans}. Since $\mathbf{V}(x;\sigma,\delta,\eps)$ depends analytically on $\delta\in\mathbb{C}$ for $|\delta|\ll1$ and $\eps\in\mathbb{R}$ for $|\eps|\ll1$, consequently, \eqref{def:Vexp} converges as $\delta$, $\eps \to 0$. Since $\mathbf{V}(0;\sigma,\delta,\eps)=\mathbf{I}$, it is readily seen that
\be \label{cond:V(0)}
\mathbf{V}^{(0,0)}(0;\sigma)=\mathbf{I}\quad\text{and}\quad 
\mathbf{V}^{(m,n)}(0;\sigma)=\mathbf{0}\quad \text{for $m+n\geq 1$}.
\ee
Crucial to our analysis is evaluating $\mathbf{V}^{(m,n)}(T;\sigma)$, $m,n=0,1,2,\dots$. We pause to remark that $\mathbf{V}^{(m,n)}(x;\sigma)$, $m,n=0,1,2,\dots$ are square matrices of order $\dim(Y(\sigma))$, where  $\dim(Y(\sigma))=2$ or $4$ or $6$.

We write \eqref{def:monodromy} as
\[
\mathbf{V}(x;\sigma,\delta,\eps)=\begin{pmatrix}\mathbf{v}_k & \cdots & \mathbf{v}_{k'}\end{pmatrix}(x;\sigma,\delta,\eps),
\]
where $k,\dots, k'$ are such that $\boldsymbol{\phi}_k(\sigma),\dots, \boldsymbol{\phi}_{k'}(\sigma)$ make the ordered basis $Y_0(\sigma)$ of $Y(\sigma)$. Correspondingly, \eqref{def:Vexp} becomes
\begin{equation}\label{def:v_k}
\mathbf{v}_k(x;\sigma,\delta,\eps)=\sum_{\boldsymbol{\phi}_j(\sigma)\in Y_0(\sigma)}\Big(\sum_{m+n=0}^{\infty}v^{(m,n)}_{jk}(x;\sigma)\delta^m\eps^n\Big)\boldsymbol{\phi}_j(\sigma)
\end{equation}
as $\delta$, $\eps\to0$ where $k$ is such that $\boldsymbol{\phi}_k(\sigma)\in Y_0(\sigma)$. We write \eqref{def:u2} as
\be\label{def:w_k}
\mathbf{w}(x,\mathbf{v}_k(x);\sigma,\delta,\eps)=\sum_{m+n=1}^\infty \mathbf{w}_{k}^{(m,n)}(x;\sigma)\delta^m\eps^n
\ee
as $\delta,\eps\to0$ for $k$ such that $\boldsymbol{\phi}_k(\sigma)\in Y_0(\sigma)$. It follows from Section \ref{sec:reduction} that $\mathbf{w}_{k}^{(0,0)}(x;\sigma)=\mathbf{0}$. 
Additionally, we write 
\be\label{def:Bexp}
\mathbf{B}(x;\sigma,\delta,\eps)=\sum_{m+n=1}^\infty\mathbf{B}^{(m,n)}(x;\sigma)\delta^m\eps^n
\ee
as $\delta,\eps\to0$. It follows from Section \ref{sec:reduction} that $\mathbf{B}^{(0,0)}(x;\sigma)=\mathbf{0}$ while $\mathbf{B}^{(m,n)}(x;\sigma)$, $1\leq m+n\leq2$, are in Appendix \ref{app:Bexp}. Our proofs do not involve $\mathbf{B}^{(m,n)}(x;\sigma)$, $m+n\geq3$, whence their formulae are omitted. Clearly, $\mathbf{B}^{(m,0)}(x;\sigma)$, $m\geq1$, are independent of $x$. 

Since \eqref{eqn:A} is equivalent to \eqref{eqn:LB;u1} and, hence, to \eqref{eqn:LB;u12}, we proceed to substitute \eqref{def:v_k}-\eqref{def:Bexp} into the former equation of \eqref{eqn:LB;u12}, and collect terms of orders $\delta^m\eps^n$, $m,n=0,1,2,\dots$. 

Let $\beta>0$ and $\kappa>0$. Recall Section \ref{sec:unimodal}, and suppose $0<\mu_0(\kappa;\beta)<1$, corresponding to subcritical capillary-gravity waves. 

\begin{itemize}
\item[$\bullet$] Let $\sigma=0$. At the order of $\delta^0\eps^0$, Lemma \ref{lem:eps=0} establishes
\[\mathbf{L}(i0)\boldsymbol{\phi}_j(0)=ik_j(0)\boldsymbol{\phi}_j(0),\quad k_j(0)=(-1)^j\kappa, \quad j=6,5,\]
$\mathbf{L}(i0)\boldsymbol{\phi}_4(0)=\mathbf{0}$, and $ \mathbf{L}(i0)\boldsymbol{\phi}_2(0)=\boldsymbol{\phi}_4(0)$.
Recall Definition \ref{def:proj}, and an explicit calculation then reveals
\be \label{def:V(0,0)0}
\mathbf{V}^{(0,0)}(x;0)=\begin{pmatrix} 
e^{i\kappa x}&0&0&0 \\0&1&0&0\\0&x&1&0\\ 0&0&0&e^{-i\kappa x}\end{pmatrix}.
\ee 

At the order of $\delta^m\eps^n$, $m+n\geq1$, the former equation of \eqref{eqn:LB;u12} becomes, after some algebra,
\ba \label{eqn:v_k(0)}
\sum_{j=6,2,4,5}\Big(\frac{d}{dx}v_{jk}^{(m,n)}\Big)\boldsymbol{\phi}_j(0)=&
i\kappa v_{6k}^{(m,n)}\boldsymbol{\phi}_6(0)-i\kappa v_{5k}^{(m,n)}\boldsymbol{\phi}_5(0)
\\&+v_{2k}^{(m,n)}\boldsymbol{\phi}_4(0)+\boldsymbol{\Pi}(0)\mathbf{f}_{k}^{(m,n)}(x;0).
\ea 
Here $\boldsymbol{\phi}_j(0)$, $j=6,2,4,5$, are in Lemma \ref{lem:eps=0}, $\boldsymbol{\Pi}(0)$ in \eqref{def:Pi0}, and
\ba \label{def:f(0)}
\mathbf{f}_{k}^{(m,n)}(x;0)=\sum_{\substack{0\leq m'\leq m\\ 0\leq n' \leq n}}
\mathbf{B}^{(m',n')}(x;0)\Big(&\mathbf{w}_{k}^{(m-m',n-n')}(x;0)
\\
&+\sum_{j=6,2,4,5}v^{(m-m',n-n')}_{jk}(x;0)\boldsymbol{\phi}_j(0)\Big),
\ea 
where $\mathbf{B}^{(m',n')}(x;0)$, in \eqref{def:Bexp}, is in Appendix \ref{app:Bexp} for $m'+n'=1,2$, and $\mathbf{w}_k^{(m-m',n-n')}(x;0)$, in \eqref{def:w_k}, can be determined by solving \eqref{eqn:w_k}-\eqref{eqn:w_k bd}. See Lemma \ref{lem:V(m,n)0} for the solutions for $m+n=1,2$.

\item[$\bullet$] Let $0<\sigma<\sigma_c$ or $\sigma>\sigma_c$. At the order of $\delta^0\eps^0$, we utilize the eigenvalues and eigenfunctions in Lemma \ref{lem:eps=0} and the ordered basis in Definition \ref{def:proj}, to obtain
\be \label{def:V(0,0)sigma}
\mathbf{V}^{(0,0)}(x;\sigma)=\begin{pmatrix}
e^{ik_j(\sigma)x} & \dots & 0 \\
\vdots & \ddots & \vdots \\
0 & \dots & e^{ik_{j'}(\sigma)x}
\end{pmatrix},
\ee 
where $k_j(\sigma)$, \dots, $k_{j'}(\sigma)$ are in Lemma \ref{lem:eps=0} for $j$, \dots, $j'$ such that $Y_0(\sigma)=\{\boldsymbol{\phi}_j(\sigma),\dots, \boldsymbol{\phi}_{j'}(\sigma)\}$. More precisely, $j,\dots, j'=6,2,4,5$ if $\sigma<\sigma_c$ and $j,j'=4,5$ if $\sigma>\sigma_c$.

At the order of $\delta^m\eps^n$, $m+n\geq 1$, on the other hand,
\ba\label{eqn:v_k(0):sigma}
\sum_{\boldsymbol{\phi}_j(\sigma)\in Y_0(\sigma)} \Big(\frac{d}{dx}v^{(m,n)}_{jk}\Big)\boldsymbol{\phi}_j(\sigma)
=\sum_{\boldsymbol{\phi}_j(\sigma)\in Y_0(\sigma)}
&ik_j(\sigma)v^{(m,n)}_{jk}\boldsymbol{\phi}_j(\sigma)\\
 &+\boldsymbol{\Pi}(\sigma)\mathbf{f}_{k}^{(m,n)}(x;\sigma),
\ea
where $\boldsymbol{\phi}_j(\sigma)$, $j=6,2,4,5$ or $j=4,5$, are in Lemma \ref{lem:eps=0}, $\boldsymbol{\Pi}(\sigma)$ in \eqref{def:Pi;high}, 
\ba \label{def:f(sigma)}
\mathbf{f}_{k}^{(m,n)}(x;\sigma)=\sum_{\substack{0\leq m'\leq m\\ 0\leq n' \leq n}}
\mathbf{B}^{(m',n')}(x;\sigma)\Big(&\mathbf{w}_{k}^{(m-m',n-n')}(x;\sigma)\\
&+\sum_{\boldsymbol{\phi}_j(\sigma)\in Y_0(\sigma)} v^{(m-m',n-n')}_{jk}(x;\sigma)\boldsymbol{\phi}_j(\sigma)\Big),
\ea 
and $\mathbf{w}_k^{(m-m',n-n')}(x;\sigma)$ can be determined by solving \eqref{eqn:w_k}-\eqref{eqn:w_k bd}. 
See, for instance, Lemmas \ref{lem:V(m,n) N=2} and \ref{lem:V(m,n) N=3} for some solutions.

\item[$\bullet$] Let $\sigma=\sigma_c$, and we can analogously determine $\mathbf{V}^{(m,n)}(x;\sigma_c)$, $m,n=0,1,2,\dots$. We omit the details.
\end{itemize}

Suppose $\mu_0(\kappa;\beta)>1$, corresponding to supercritical capillary-gravity waves. We can analogously determine \eqref{def:V(m,n)}. We omit the details. 

Turning to the latter equation of \eqref{eqn:LB;u12}, we abuse notation and write $\mathbf{w}_{k}^{(m,n)}=\begin{pmatrix}\varphi\\u\\ \eta\\z\end{pmatrix}$. Substituting \eqref{def:v_k}-\eqref{def:Bexp} and collecting terms  of order $\delta^m\eps^n$, $m+n\geq1$, after some algebra, we arrive at:
\begin{equation}\label{eqn:w_k}
\frac{d}{dx}\begin{pmatrix}
\varphi\\u\\ \eta\\z
\end{pmatrix}=
\begin{pmatrix}
u-\dfrac{i\sigma y^2}{2\beta}z\\-\varphi_{yy}-i\sigma\eta\\ \dfrac{z}{\beta}\\ \eta-\dfrac{\sigma^2}{2\mu_0}\eta+\dfrac{i\sigma}{\mu_0}\varphi(1)-\dfrac{u(1)}{\mu_0}
\end{pmatrix}+(\mathbf{1}-\boldsymbol{\Pi}(\sigma))\mathbf{f}_{k}^{(m,n)}(x;\sigma)
\end{equation}
and
\begin{equation}\label{eqn:w_k bd}
\varphi_y(1)=-\frac{z}{\beta},\qquad \varphi_y(0)=0.
\end{equation}
Here, $\mathbf{f}_{k}^{(m,n)}(x;\sigma)$ is in \eqref{def:f(0)} or \eqref{def:f(sigma)} when $0<\mu_0(\kappa;\beta)<1$ and $\sigma\neq \sigma_c$, and $\boldsymbol{\Pi}(\sigma)$ in Section \ref{sec:proj}. At first glance, the right side of \eqref{eqn:w_k} appears to depend on $\mathbf{w}_k^{(m,n)}(x;\sigma)$ itself. However, since $\mathbf{B}^{(0,0)}(x;\sigma)=\mathbf{0}$, the right sides of \eqref{def:f(0)} and \eqref{def:f(sigma)} are independent of the highest order term and, rather, involve lower order terms $\mathbf{w}_{k}^{(m',n')}(x;\sigma)$, $0< m'+n'< m+n$. Consequently, \eqref{eqn:w_k} is well-defined as an ODE for $\mathbf{w}_k^{(m,n)}$. Recall \eqref{def:dom} and, clearly, \eqref{eqn:w_k bd} ensures that $\mathbf{w}_{k}^{(m,n)}\in {\rm dom}(\mathbf{L})$. 

An inspection of the result in Appendix \ref{app:Bexp} reveals that $\mathbf{B}^{(m,n)}(x;\sigma)$, $m+n=1,2,\dots$, depend linearly on $\s(\ell \kappa x)$ and $\c(\ell\kappa x)$, $\ell\in\mathbb{Z}^+$, with respect to $x$, for steady unimodal capillary-gravity waves, for which $\phi_1,\phi_2$, $\eta_1,\eta_2$, and $\mu_1, \mu_2$ are in Section \ref{sec:unimodal}. Consequently, we can solve \eqref{eqn:w_k}-\eqref{eqn:w_k bd}, for instance, by the method of undetermined coefficients, subject to $\mathbf{w}_{k}^{(m,n)}\in(\mathbf{1}-\boldsymbol{\Pi}(\sigma))Y$. The expressions for $\mathbf{w}_{k}^{(m,n)}$ are lengthy and not presented here. We pause to remark that since $\mu_1\neq 0$ for Wilton ripples of order $2$ (see \eqref{def:alpha M=2}), $\mathbf{B}^{(0,1)}(x;\sigma)$ additionally depends on a constant term with respect to $x$. See, for instance, Lemma \ref{lem:V(m,n) M=2 N=2}. 

 See \cite[Section 5.1]{HY2023} for an analogous calculation for Stokes' gravity waves.

\newpage

\section{Resonances}\label{resonances}

Let $\beta>0$ and $\kappa>0$. Recall from Section \ref{sec:eps=0} that ${\rm spec}(\mathcal{L}(0))=i\mathbb{R}$. More precisely,
\[
{\rm spec}(\mathcal{L}(0))=\{i\sigma \in i\mathbb{R}:\sigma=\sigma_\pm(k;\beta,\mu_0), k\in\mathbb{R} \},
\]
where $\sigma_\pm(k;\beta,\mu_0)$ is in \eqref{def:sigma} and  $\mu_0=\mu_0(\kappa;\beta)$ in \eqref{def:mu0}. Lemma \ref{lem:eps=0} characterizes finitely many solutions to $\sigma_\pm(k)=\sigma$, $\sigma\in\mathbb{R}$.

\begin{definition}[Resonances]\label{def:resonance}\rm
We say that $i\sigma\in {\rm spec}(\mathcal{L}(0))(= i\mathbb{R})$ is {\em resonant} of order $N$ associated with $k_j(\sigma)$ and $k_{j'}(\sigma)$ if
\[
k_j(\sigma)-k_{j'}(\sigma)=N\kappa 
\]
where $N\in\mathbb{Z}^+$, and $k_j(\sigma)$ and $k_{j'}(\sigma)$ are in Lemma \ref{lem:eps=0} for $j$ and $j'$ such that $\boldsymbol{\phi}_j(\sigma), \boldsymbol{\phi}_{j'}(\sigma) \in Y_0(\sigma)$, in Definition \ref{def:proj}. 
We say that $k_j(\sigma)$ and $k_{j'}(\sigma)$ have {\em opposite Krein signatures} if they belong to opposite branches of \eqref{def:sigma}. 
We say that $i\sigma$ is {\em non-resonant} if
\[
k_j(\sigma)-k_{j'}(\sigma)\neq N\kappa \quad \text{for any $N\in\mathbb{Z}^+$}  \quad \text{for any $j\neq j'$}.
\]
\end{definition}

Resonances in the sense of Definition \ref{def:resonance} play a fundamental role in spectral stability or instability in the surface water wave problem and, more generally, Hamiltonian PDEs. Concerning Stokes' gravity waves, numerical studies have suggested the onset of spectral instability near resonances along the imaginary axis. See, for instance, \cite{McLean;finite-depth, FK, DO}. See also \cite{mclean_1982, McLean;infinite2} for infinite depth. The authors' previous work \cite{HY2023} showed that Stokes waves of sufficiently small amplitude can become spectrally unstable near resonances, and analytically constructed a spectral instability criterion for a resonance of order $2$. 

Here we identify all resonances for steady periodic capillary-gravity waves, and determine their implications for spectral stability and instability for sufficiently small amplitude. Specifically, Section \ref{sec:stability} demonstrates spectral stability near non-resonances while Sections \ref{sec:high freq instability}-\ref{sec:wilton instability} investigate spectral stability and instability near resonances for steady unimodal capillary-gravity waves and bimodal Wilton ripples. 

It is readily seen from \eqref{def:mu0} and \eqref{eqn:sigma} that $i0$ is resonant. Accordingly, we restrict our attention to nonzero resonances. 

\begin{notation*}
Throughout this section, we write $\sigma_\pm(k)$ instead of $\sigma_\pm(k;\beta,\mu_0)$, where $\beta$ and $\mu_0$ are suppressed for simplicity of notation unless otherwise specified.
\end{notation*}

We gather preliminary results. 

\begin{lemma}[Monotonicity of $(k_j-k_{j'})(\sigma)$]\label{monotonykikj}
Let $\beta>0$ and $\kappa>0$. Recall Lemma \ref{lem:eps=0}.

If $\mu_0(\kappa;\beta)>1$ then $(k_6-k_5)(\sigma)$ is strictly decreasing for $\sigma \in [0,\sigma_{c,2}]$. 
If $0<\mu_0(\kappa;\beta)<1$ then $(k_6-k_5)(\sigma)$ is strictly decreasing for $\sigma \in [0,\sigma_{c}]$.

If $\mu_0(\kappa;\beta)>1$ then $(k_4-k_1)(\sigma)$ is strictly decreasing for $\sigma \in [0,-\sigma_{c,1}]$ while $(k_2-k_4)(\sigma)$ is strictly increasing for $\sigma \in [0,\sigma_{c,2}]$.
\end{lemma}

\begin{proof}
We begin by observing
\[
\frac{d(k_6-k_5)}{d\sigma}(\sigma)=\frac{dk_6}{d\sigma}(\sigma)-\frac{dk_6}{d\sigma}(-\sigma)=\int_{-\sigma}^{\sigma}\frac{d^2k_6}{d\sigma^2}(\sigma')~d\sigma',
\]
where the first equality follows from Lemma \ref{lem:eps=0} by symmetry. Suppose $\mu_0(\kappa;\beta)>1$. Since $\sigma_-(k;\beta,\mu_0)$ is concave for $k\in (k_{c,2},\infty)$ by Lemma \ref{lem:sigmaS12}, the right side $<0$ for $\sigma\in[0,\sigma_{c,2}]$. If $0<\mu_0(\kappa;\beta)<1$, correspondingly, by Lemma \ref{lem:sigmaS3}, the right side $<0$ over the interval $[0,\sigma_c]$. 

Suppose $\mu_0(\kappa;\beta)>1$, and consider
\[
\frac{d}{d\sigma}(k_4-k_1)(\sigma)=
\frac{\frac{d\sigma_-}{dk}(-k_1(\sigma))-\frac{d\sigma_+}{dk}(k_4(\sigma))}{\frac{d\sigma_-}{dk}(-k_1(\sigma))\frac{d\sigma_+}{dk}(k_4(\sigma))},
\]
where $\frac{d\sigma_-}{dk}(-k_1(\sigma)),\frac{d\sigma_+}{dk}(k_4(\sigma))>0$  for $\sigma\in[0,-\sigma_{c,1}]$ by Lemmas \ref{lem:sigmaS12} and \ref{lem:eps=0}. Since
\begin{align*}
\frac{d\sigma_-}{dk}(-k_1(\sigma))-\frac{d\sigma_+}{dk}(k_4(\sigma))=&1-\sqrt{\mu_0}f(-k_1(\sigma))-(1+\sqrt{\mu_0}f(k_4(\sigma)))\\
=&-\sqrt{\mu_0}(f(-k_1(\sigma))+f(k_4(\sigma))),
\end{align*}
where $f(k)$ is in \eqref{def:fS12}, and since $f(k)>0$ for $k\in(0,\infty)$ by the proof of Lemma \ref{lem:sigmaS12}, consequently, the right side $<0$. Therefore, $\frac{d}{d\sigma}(k_4-k_1)(\sigma)<0$ for $\sigma\in[0,-\sigma_{c,1}]$.

On the other hand, since
\[
\frac{dk_2}{d\sigma}(\sigma)=\frac{1}{\dfrac{d\sigma_-}{dk}(k_2(\sigma))}>0\quad \text{and}\quad\frac{dk_4}{d\sigma}(\sigma)=\frac{1}{\dfrac{d\sigma_+}{dk}(k_4(\sigma))}>0
\]
for $\sigma\in[0,\sigma_{c,2}]$ by Lemmas \ref{lem:sigmaS12} and \ref{lem:eps=0}, and since
\begin{align*}
\frac{d\sigma_+}{dk}(k_4(\sigma))-\frac{d\sigma_-}{dk}(k_2(\sigma))
&=1+\sqrt{\mu_0}f(k_4(\sigma))-(1-\sqrt{\mu_0}f(k_2(\sigma)))\\
&=\sqrt{\mu_0}(f(k_4(\sigma))+f(k_2(\sigma)))>0
\end{align*}
by the proof of Lemma \ref{lem:sigmaS12}, $\frac{d}{d\sigma}(k_2-k_4)(\sigma)>0$ for $\sigma\in[0,\sigma_{c,2}]$. This completes the proof.
\end{proof}

\subsection{Resonances for subcritical waves}\label{S3region}

Let $\beta>0$ and $\kappa>0$. Recall Section \ref{sec:unimodal}, and suppose $0<\mu_0(\kappa;\beta)<1$, corresponding to subcritical capillary-gravity waves. Recall Section \ref{sec:eps=0}. 

If $0<\sigma\leq\sigma_c$ then $\sigma_\pm(k;\beta,\mu_0)=\sigma$ have four solutions such that $k_5(\sigma)\leq -\kappa <0\leq k_4(\sigma)$ and $0\leq k_2(\sigma) <k_6(\sigma)\leq \kappa$. It is straightforward to verify:
\begin{itemize}
\item[$\bullet$] $\kappa<(k_4-k_5)(\sigma)<(k_2-k_5)(\sigma)\leq (k_6-k_5)(\sigma)<2\kappa$, where the last inequality follows from Lemma \ref{monotonykikj}.
\item[$\bullet$] $0<(k_2-k_4)(\sigma)<(k_6-k_4)(\sigma)<\kappa$.
\item[$\bullet$] $0<(k_6-k_2)(\sigma)<\kappa$.
\end{itemize}
Together, no resonance occurs. 

If $\sigma>\sigma_c$, on the other hand, an explicit calculation reveals that $(k_4-k_5)(\sigma)$ is strictly increasing and $(k_4-k_5)(\sigma) \to \infty$ as $\sigma\to \infty$. Therefore, $k_4(\sigma)-k_5(\sigma)=N\kappa$ has a unique solution for each $N\geq 2,\in \mathbb{Z}^+$.

To summarize, 
\[
k_4(\sigma)-k_5(\sigma)=N\kappa\quad \text{for each $N\geq2,\in\mathbb{Z}^+$} \quad \text{for some $\sigma>\sigma_c$},
\]
depending on $N$, and no other resonances occur. 
We pause to remark that $k_4(\sigma)$ and $k_5(\sigma)$ belong to the $+$ branch of \eqref{def:sigma}. That is, they have the same Krein signature.

\subsection{Resonances for supercritical waves}\label{S1S2region}

Recall Section \ref{sec:unimodal}. 
Suppose $\mu_0(\kappa;\beta)>1$, corresponding to supercritical capillary-gravity waves, so that $\mu_0(\kappa_1;\beta)=\mu_0(\kappa_2;\beta)$ and $0<\kappa_1< \kappa_2<\infty$ for each $\beta (\in (0,1/3))$. In the region for supercritical waves in the $(\beta,\kappa)$-plane, we distinguish the regions for $2\pi/\kappa_1$- and $2\pi/\kappa_2$-periodic waves, denoted by $S_1$ and $S_2$, respectively. See the right panel of Figure \ref{figure3}. 

Recall from Section \ref{sec:eps=0} that $\sigma_-(\cdot;\beta,\mu_0)$ has two critical values 
\[
\sigma_{c,\ell}=\sigma_-(k_{c,\ell};\beta,\mu_0), \quad \ell=1,2.
\]
Within $S_1\bigcup S_2$, we further distinguish the regions where 
\[
-\sigma_{c,1}>\sigma_{c,2}>0\quad \text{and} \quad 0<-\sigma_{c,1}<\sigma_{c,2},
\]
separated by two curves along which $-\sigma_{c,1}=\sigma_{c,2}$. See the left panel of Figure \ref{figure7}.

\begin{figure}[htbp]
\begin{center}
    \includegraphics[scale=0.3]{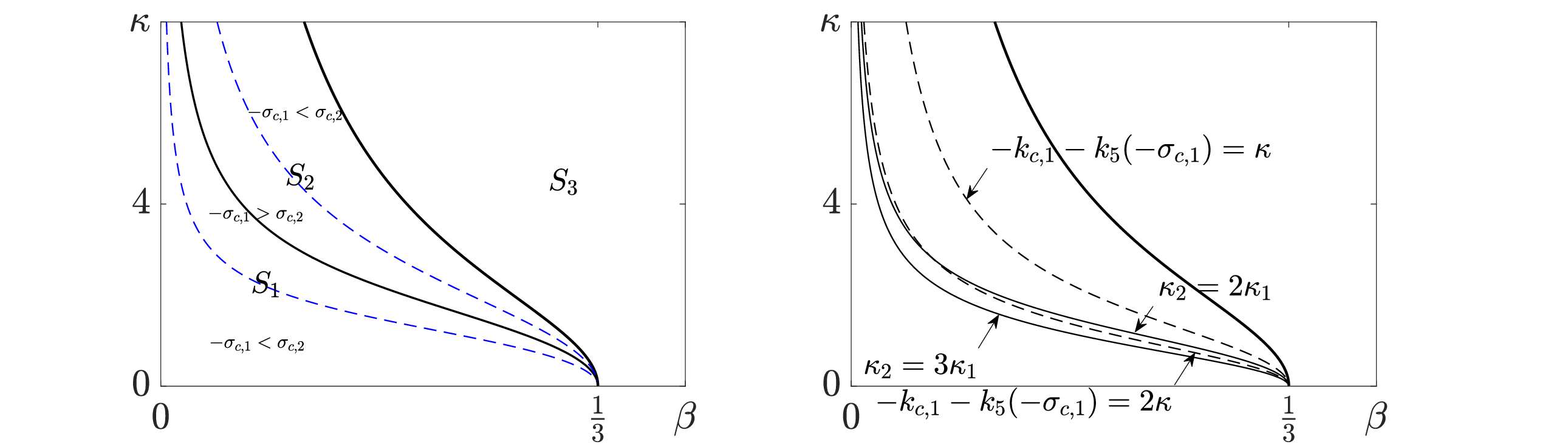}
    \end{center}
    \caption{$\kappa$ versus $\beta$. Left: Two dashed blue curves along which $-\sigma_{c,1}=\sigma_{c,2}$. In the region bounded by these curves, $-\sigma_{c,1}>\sigma_{c,2}$. In the region above the upper curve and below the lower one, on the other hand, $-\sigma_{c,1}<\sigma_{c,2}$. The upper dashed curve lies in $S_2$, and the lower curve lies in $S_1$. \protect\\ 
    \protect\hspace*{11pt} Right: The resonance region for (R1) of order $1$ is bounded from above by the dashed curve along which $-k_{c,1}-k_5(-\sigma_{c,1})=\kappa$, and from below by the solid curve along which $\kappa_2=2\kappa_1$, corresponding to Wilton ripples of order $2$.
    The resonance region for (R1) and $N=2$ is bounded from above by the dashed curve along which $-k_{c,1}-k_5(-\sigma_{c,1})=2\kappa$, and from below by the solid curve along which $\kappa_2=3\kappa_1$, corresponding to Wilton ripples of order $3$. \protect\\
    \protect\hspace*{11pt} The resonance region for (R2) and $N=1$ is bounded from above by the solid curve $\kappa_2=2\kappa_1$. The region for (R2) and $N=2$ is bounded from above by the solid curve $\kappa_2=3\kappa_1$. \protect\\
    \protect\hspace*{11pt} Within $S_1\bigcup S_2$, the resonance region for (R4) and $N=2$ is bounded from below by the solid curve $\kappa_2=2\kappa_1$, and the region for (R4) and $N=3$ is bounded from below by the solid curve $\kappa_2=3\kappa_1$.}
    \label{figure7}
\end{figure}

Below we utilize the results of Section \ref{sec:eps=0} to classify all nonzero resonances for supercritical waves.

\subsubsection*{\rm (R1) $k_1(\sigma)-k_5(\sigma)=N\kappa$ for some $N\in\mathbb{Z}^+$} 

Observe that $(k_1-k_5)(\sigma)$ is a strictly increasing function for $\sigma\in[0,-\sigma_{c,1}]$ and, hence,
\[
\kappa_2-\kappa_1\leq (k_1-k_5)(\sigma)\leq-k_{c,1}-k_5(-\sigma_{c,1}).
\]
Therefore, within $S_1\bigcup S_2$, the resonance region corresponding to (R1) of order $N$ is bounded from below by the curve 
\[
\kappa_2=(N+1)\kappa_1,
\]
along which Wilton ripples of order $N+1$ occur, and bounded from above by the curve 
\[
-k_{c,1}-k_5(-\sigma_{c,1})=N\kappa.
\]
The right panel of Figure \ref{figure7} depicts these boundary curves for $N=1$ and $2$.

\subsubsection*{\rm (R2) $k_6(\sigma)-k_2(\sigma)=N\kappa$ for some $N\in\mathbb{Z}^+$} 

Observe that $(k_6-k_2)(\sigma)$ is strictly decreasing for $\sigma\in [0,\sigma_{c,2}]$ and, hence,
\[
0\leq(k_6-k_2)(\sigma)\leq\kappa_2-\kappa_1.
\]
Therefore, in $S_1\bigcup S_2$, the resonance region for (R2) of order $N$ is bounded from above by the curve $\kappa_2=(N+1)\kappa_1$, corresponding to Wilton ripples of order $N+1$. See the right panel of Figure \ref{figure7} for $N=1$ and $2$.

\subsubsection*{\rm (R3) $k_2(\sigma)-k_5(\sigma)=N\kappa$ for some $N\in\mathbb{Z}^+$} 

Observe that $(k_2-k_5)(\sigma)$ is strictly increasing for $\sigma\in[0,\sigma_{c,2}]$. Consequently, 
\[
2\kappa_1<\kappa_1+\kappa_2\leq(k_2-k_5)(\sigma)\leq k_{c,2}-k_5(\sigma_{c,2})<2\kappa_2,
\]
where the last inequality follows from Lemma \ref{monotonykikj}. Therefore, resonance (R3) can occur in $S_1$ but only for orders $\geq 3$. On the other hand, (R3) cannot occur in $S_2$ for any $N\geq 1$.

\subsubsection*{\rm (R4) $k_4(\sigma)-k_5(\sigma)=N\kappa$ for some $N\in\mathbb{Z}^+$} 

Observe that $(k_4-k_5)(\sigma)$ is strictly increasing, 
\[
\text{$(k_4-k_5)(\sigma)\geq\kappa_2$ for $\sigma\in [0,\infty)$},\quad\text{and}\quad 
\text{$(k_4-k_5)(\sigma)\to \infty$ as $\sigma\to\infty$}.
\]
Therefore, in $S_1\bigcup S_2$, the resonance region for (R4) of order $N$ is bounded from below by the curve $\kappa_2=N\kappa_1$, corresponding to Wilton ripples of order $N$. See the right panel of Figure \ref{figure7} for $N=2$ and $3$.

\subsubsection*{\rm (R5) $k_3(\sigma)-k_1(\sigma)=N\kappa$ for some $N\in\mathbb{Z}^+$}  

Since $(k_3-k_1)(\sigma)$ is strictly decreasing for $\sigma\in[0,-\sigma_{c,1}]$, consequently, $0\leq(k_3-k_1)(\sigma)\leq \kappa_1<\kappa_2$. Therefore, resonance (R5) cannot occur for any $N\geq1$.

\subsubsection*{\rm (R6) $k_6(\sigma)-k_1(\sigma)=N\kappa$ for some $N\in\mathbb{Z}^+$} 

Clearly, $(k_6-k_1)(\sigma)$ is strictly decreasing wherever defined. Consequently,
\[
\kappa_1<k_6(-\sigma_{c,1})+k_{c,1}\leq(k_6-k_1)(\sigma)<\kappa_1+\kappa_2<2\kappa_2
\]
if $0<-\sigma_{c,1}<\sigma_{c,2}$, while 
\[
\kappa_1<k_{c,2}-k_1(\sigma_{c,2})\leq(k_6-k_1)(\sigma)<\kappa_1+\kappa_2<2\kappa_2
\]
if $-\sigma_{c,1}> \sigma_{c,2}$.

Therefore, in $S_1$, the resonance region for (R6) of order $N\geq2$ is bounded by the curve 
\[
\begin{cases}
k_6(-\sigma_{c,1})+k_{c,1}=N\kappa_1 & \text{if $-\sigma_{c,1}< \sigma_{c,2}$}, \\
k_{c,2}-k_1(\sigma_{c,2})=N\kappa_1 & \text{if $-\sigma_{c,1}> \sigma_{c,2}$}.
\end{cases}
\]
See the left panel of Figure \ref{figure8} for $N=2$. 

On the other hand, in $S_2$, the resonance region for (R6) of order $1$ is bounded by the curve
\[
\begin{cases}
k_6(-\sigma_{c,1})+k_{c,1}=\kappa_2 & \text{if $-\sigma_{c,1}< \sigma_{c,2}$}, \\
k_{c,2}-k_1(\sigma_{c,2})=\kappa_2 & \text{if $-\sigma_{c,1}> \sigma_{c,2}$}.
\end{cases}
\]
See Figure \ref{figure8}. However, (R6) cannot occur in $S_2$ for any $N\geq 2$.

\begin{figure}[htbp]
    \centering
    \includegraphics[scale=0.3]{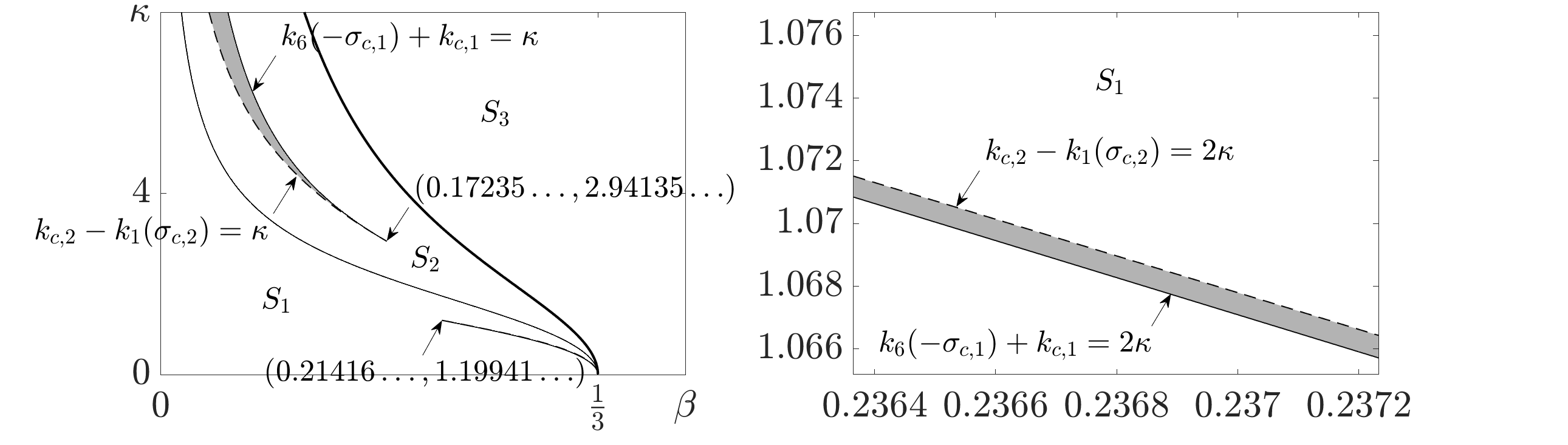}
    \caption{$\kappa$ versus $\beta$ for (R6). In $S_1$, the resonance region for $N=2$ is bounded from above by the dashed curve along which $k_{c,2}-k_1(\sigma_{c,2})=2\kappa$, and from below by the solid curve along which $k_6(-\sigma_{c,1})+k_{c,1}=2\kappa$. See the right panel for a close-up. The dashed and solid curves intersect at $(\beta,\kappa)=(0.21416\dots, 1.19941\dots)$, where $-\sigma_{c,1}=\sigma_{c,2}$ and $k_{c,1}+k_{c,2}=2\kappa$. In $S_2$, the resonance region for $N=1$ is bounded from above by the solid curve $k_6(-\sigma_{c,1})+k_{c,1}=\kappa$, and from below by the dashed curve $k_{c,2}-k_1(\sigma_{c,2})=\kappa$. These boundary curves intersect at $(0.17235\dots, 2.94135\dots)$, where $-\sigma_{c,1}=\sigma_{c,2}$ and $k_{c,1}+k_{c,2}=\kappa$.}
    \label{figure8}
\end{figure}

\subsubsection*{\rm (R7) $k_2(\sigma)-k_3(\sigma)=N\kappa$ for some $N\in\mathbb{Z}^+$} 
Clearly, $(k_2-k_3)(\sigma)$ is strictly increasing wherever defined. Consequently,
\[
\kappa_1<(k_2-k_3)(\sigma)\leq k_2(-\sigma_{c,1})+k_{c,1}\leq k_6(-\sigma_{c,1})+k_{c,1} <\kappa_1+\kappa_2<2\kappa_2
\]
if $-\sigma_{c,1}<\sigma_{c,2}$, while
\[
\kappa_1<(k_2-k_3)(\sigma)\leq k_{c,2}-k_3(\sigma_{c,2})< k_{c,2}-k_1(\sigma_{c,2})<\kappa_1+\kappa_2<2\kappa_2
\]
if $-\sigma_{c,1}> \sigma_{c,2}$. 

Therefore, in $S_1$, the resonance region for (R7) of order $N\geq 2$ is bounded by the curve
\[
\begin{cases}
k_2(-\sigma_{c,1})+k_{c,1}=N\kappa_1 &\text{if $-\sigma_{c,1}< \sigma_{c,2}$},\\
k_{c,2}-k_3(\sigma_{c,2})=N\kappa_1 & \text{if $-\sigma_{c,1}> \sigma_{c,2}$}.
\end{cases}
\]
See the left panel of Figure \ref{figure9} for $N=2$. 

On the other hand, in $S_2$, the resonance region for (R7) of order $1$ is bounded by the curve
\[
\begin{cases}
k_2(-\sigma_{c,1})+k_{c,1}=\kappa_2 & \text{if $-\sigma_{c,1}< \sigma_{c,2}$},\\
k_{c,2}-k_3(\sigma_{c,2})=\kappa_2 & \text{if $-\sigma_{c,1}> \sigma_{c,2}$}.
\end{cases}
\]
See Figure \ref{figure9}. However, (R7) cannot occur in $S_2$ for any $N\geq 2$.

\begin{figure}[htbp]
    \centering
    \includegraphics[scale=0.3]{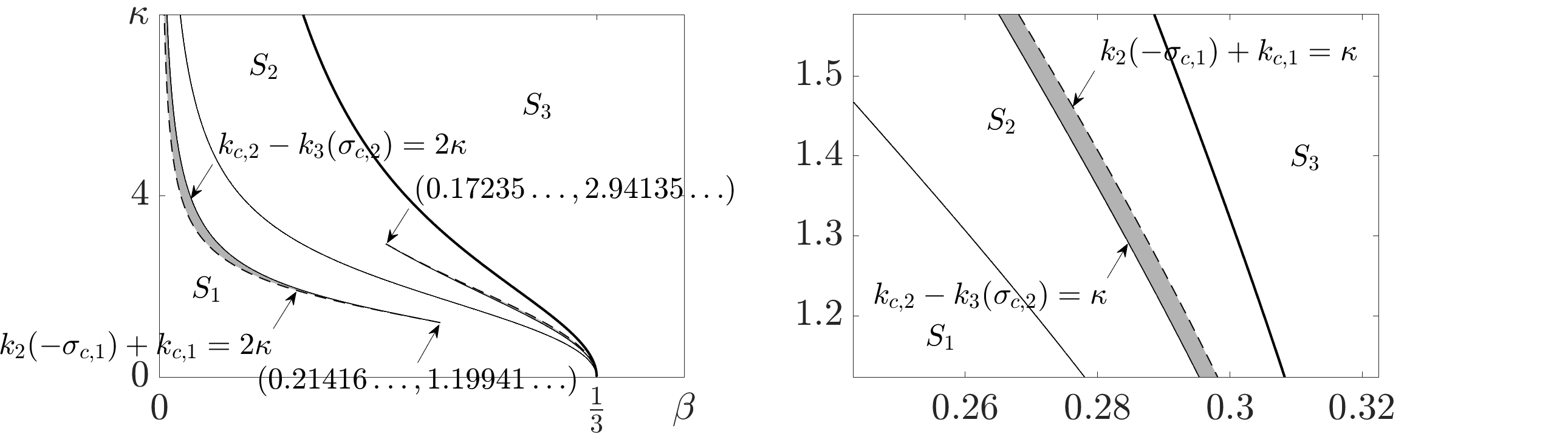}
    \caption{$\kappa$ versus $\beta$ for (R7). 
    In $S_1$, the resonance region for $N=2$ is bounded from above by the solid curve along which $k_{c,2}-k_3(\sigma_{c,2})=2\kappa$, and from below by the dashed curve along which $k_2(-\sigma_{c,1})+k_{c,1}=2\kappa$. The solid and dashed curves intersect at $(\beta,\kappa)=(0.21416\dots, 1.19941\dots)$, where $-\sigma_{c,1}=\sigma_{c,2}$ and $k_{c,1}+k_{c,2}=2\kappa$. In $S_2$, the resonance region for $N=1$ is bounded from above by the dashed curve $k_2(-\sigma_{c,1})+k_{c,1}=\kappa$, and from below by the solid curve $k_{c,2}-k_3(\sigma_{c,2})=\kappa$. See the right panel for a close-up. These boundary curves intersect at $(0.17235\dots, 2.94135\dots)$, where $-\sigma_{c,1}=\sigma_{c,2}$ and $k_{c,1}+k_{c,2}=\kappa$. }
    \label{figure9}
\end{figure}

\subsubsection*{\rm (R8) $k_4(\sigma)-k_3(\sigma)=N\kappa$ for some $N\in\mathbb{Z}^+$} 

Since $(k_4-k_3)(\sigma)$ is strictly increasing for $\sigma\in [0,-\sigma_{c,1}]$, consequently, 
\[
0\leq (k_4-k_3)(\sigma)\leq k_4(-\sigma_{c,1})+k_{c,1}<\kappa_1,
\]
where the last inequality follows from Lemma \ref{monotonykikj}. Therefore, resonance (R8) cannot occur for any $N\geq 1$.

\subsubsection*{\rm (R9) $k_6(\sigma)-k_4(\sigma)=N\kappa$ for some $N\in\mathbb{Z}^+$}  

Observe that $(k_6-k_4)(\sigma)$ is strictly decreasing for $\sigma\in[0,\sigma_{c,2}]$. Consequently, 
\[
\kappa_1<k_{c,2}-k_4(\sigma_{c,2})\leq (k_6-k_4)(\sigma)\leq\kappa_2,
\]
where the first inequality follows from Lemma \ref{monotonykikj}. Therefore, in $S_1$, the resonance region for (R9) of order $N \geq 2$ is bounded from above by the curve $\kappa_2=N\kappa_1$, corresponding to Wilton ripples of order $N$, and from below by the curve 
\[
k_{c,2}-k_4(\sigma_{c,2})=N\kappa_1.
\]
See the left panel of Figure \ref{figure10} for $N=2$. On the other hand, (R9) cannot occur in $S_2$ for any $N\geq1$.

\begin{figure}[htbp]
    \centering
    \includegraphics[scale=0.3]{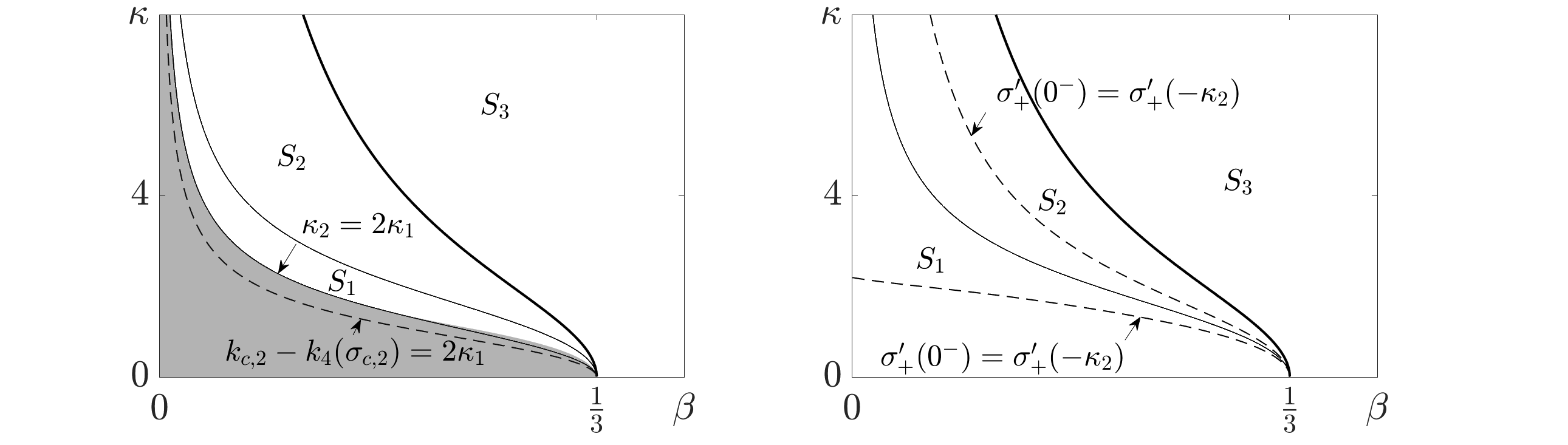}
    \caption{$\kappa$ versus $\beta$. 
    Left: The resonance region for (R9) and $N=2$ is bounded from above by the solid curve along which $\kappa_2=2\kappa_1$, and from below by the dashed curve along which $k_{c,2}-k_4(\sigma_{c,2})=2\kappa_1$. 
    The resonance region for (R12) and $N=2$ is bounded from above by the dashed curve $k_{c,2}-k_4(\sigma_{c,2})=2\kappa_1$. \protect\\
    \protect\hspace*{11pt} Right: Two dashed curves along which $\frac{d\sigma_+}{dk}(0^-)=\frac{d\sigma_+}{dk}(-\kappa_2)$. In the region bounded by these curves, $\frac{d\sigma_+}{dk}(0^-)<\frac{d\sigma_+}{dk}(-\kappa_2)$ so that $k_3-k_5$ increases and then decreases over the interval $[0,-\sigma_{c,1}]$. In the region above the upper curve and below the lower one, $k_3-k_5$ is strictly decreasing throughout $[0,-\sigma_{c,1}]$. The upper dashed curve lies within $S_2$, and the lower curve lies in $S_1$. }
    \label{figure10}
\end{figure}

\subsubsection*{\rm (R10) $k_6(\sigma)-k_5(\sigma)=N\kappa$ for some $N\in\mathbb{Z}^+$}  

Since $(k_6-k_5)(\sigma)$ is strictly decreasing for $\sigma\in [0,\sigma_{c,2}]$ by Lemma \ref{monotonykikj},
\[
2\kappa_1<\kappa_1+\kappa_2<k_{c,2}-k_5(\sigma_{c,2})\leq (k_6-k_5)(\sigma)\leq 2\kappa_2.
\]
Therefore, resonance (R10) can occur in $S_1$ but only for orders $\geq 3$. On the other hand, (R10) cannot occur in $S_2$ for any $N\geq 1$.

\subsubsection*{\rm (R11) $k_4(\sigma)-k_1(\sigma)=N\kappa$ for some $N\in\mathbb{Z}^+$}  

Since $(k_4-k_1)(\sigma)$ is strictly decreasing for $\sigma\in [0,-\sigma_{c,1}]$ by Lemma \ref{monotonykikj}, $(k_4-k_1)(\sigma)<\kappa_1$. Therefore, resonance (R11) cannot occur for any $N\geq 1$.

\subsubsection*{\rm (R12) $k_2(\sigma)-k_4(\sigma)=N\kappa$ for some $N\in\mathbb{Z}^+$}  

Since $(k_2-k_4)(\sigma)$ is strictly increasing for $\sigma \in [0,\sigma_{c,2}]$ by Lemma \ref{monotonykikj}, 
\[
\kappa_1\leq(k_2-k_4)(\sigma)\leq k_{c,2}-k_{4}(\sigma_{c,2})<\kappa_2.
\]
Therefore, in $S_1$, the resonance region for (R12) of order $N$ is bounded from above by the curve $k_{c,2}-k_{4}(\sigma_{c,2})=N\kappa_1$. See the left panel of Figure \ref{figure10} for $N=2$. However, (R12) cannot occur in $S_2$ for any $N\geq 1$. 

\subsubsection*{\rm (R13) $k_3(\sigma)-k_5(\sigma)=N\kappa$ for some $N\in\mathbb{Z}^+$} 

We observe that $k_3-k_5$ is not necessarily monotonically decreasing. 

\begin{lemma}[The monotonicity of $(k_3-k_5)(\sigma)$]\label{nonmonotonekikj}
If 
\[
\frac{d\sigma_+}{dk}(0^-;\beta,\mu_0)<\frac{d\sigma_+}{dk}(-\kappa_2;\beta,\mu_0)
\]
then $k_3-k_5$ is increasing and then decreasing over the interval $[0,-\sigma_{c,1}]$. Otherwise, $k_3-k_5$ is strictly decreasing throughout $[0,-\sigma_{c,1}]$.
\end{lemma}

\begin{proof}
Throughout the proof, we write $\sigma_\pm(k)$ instead of $\sigma_\pm(k;\beta,\mu_0)$ for the sake of notation. 

Since
\[
\frac{d(k_3-k_5)}{d\sigma}(\sigma)=\frac{\frac{d\sigma_+}{dk}(k_5(\sigma))-\frac{d\sigma_+}{dk}(k_3(\sigma))}{\frac{d\sigma_+}{dk}(k_5(\sigma))\frac{d\sigma_+}{dk}(k_3(\sigma))} 
\]
and since 
\[
\frac{d\sigma_+}{dk}(k_5(\sigma)), \frac{d\sigma_+}{dk}(k_3(\sigma))<0 \quad \text{for $\sigma\in[0,-\sigma_{c,1}]$}
\]
by Lemmas \ref{lem:sigmaS12} and \ref{lem:eps=0}, the sign of $\frac{d}{d\sigma}(k_3-k_5)(\sigma)$ coincides with the sign of $\frac{d\sigma_+}{dk}(k_5(\sigma))-\frac{d\sigma_+}{dk}(k_3(\sigma))$.
Since
\[
\frac{d^2\sigma_+}{dk^2}(k_5(\sigma))=-\frac{d^2\sigma_-}{dk^2}(k_6(-\sigma))>0\,\text{and}\,
\frac{d^2\sigma_+}{dk^2}(k_3(\sigma))=\frac{d^2\sigma_-}{dk^2}(k_2(-\sigma))<0 
\]
by Lemmas \ref{lem:sigmaS12} and \ref{lem:eps=0}, and by symmetry, consequently,
\[
\frac{d^2\sigma_+}{dk^2}(k_5(\sigma))\frac{dk_5}{d\sigma}(\sigma)-\frac{d^2\sigma_+}{dk^2}(k_3(\sigma))\frac{dk_3}{d\sigma}(\sigma)<0\quad \text{for $\sigma\in[0,-\sigma_{c,1}]$}.
\]
Therefore, $\frac{d\sigma_+}{dk}(k_5(\sigma))-\frac{d\sigma_+}{dk}(k_3(\sigma))$ is strictly decreasing for $\sigma\in[0,-\sigma_{c,1}]$. The assertions then follow upon observing
\[
\frac{d\sigma_+}{dk}(k_5(-\sigma_{c,1}))-\frac{d\sigma_+}{dk}(-k_{c,1})=\frac{d\sigma_+}{dk}(k_5(-\sigma_{c,1}))<0.
\]
This completes the proof.
\end{proof}

The right panel of Figure \ref{figure10} depicts two dashed curves along which $\frac{d\sigma_+}{dk}(0^-)=\frac{d\sigma_+}{dk}(-\kappa_2)$. In the region bounded by these curves, $\frac{d\sigma_+}{dk}(0^-)<\frac{d\sigma_+}{dk}(-\kappa_2)$ so that $k_3-k_5$ is increasing and then decreasing over the interval $[0,-\sigma_{c,1}]$ by Lemma \ref{nonmonotonekikj}. In the regions above the upper dashed curve and below the lower curve, on the other hand, $\frac{d\sigma_+}{dk}(0^-)>\frac{d\sigma_+}{dk}(-\kappa_2)$ so that $k_3-k_5$ is strictly decreasing throughout $[0,-\sigma_{c,1}]$.

Therefore, the resonance region for (R13) of order $1$ is bounded from above by the upper dashed curve along which $\frac{d\sigma_+}{dk}(0^-)=\frac{d\sigma_+}{dk}(-\kappa_2)$, and from below by the curve along which $-k_{c,1}-k_5(-\sigma_{c,1})=\kappa$. We pause to remark that the lower boundary curve intersects the curve separating $S_1$ and $S_2$, where $\frac{d\mu_0}{d\kappa}(\kappa;\beta)=0$, implying that the underlying steady periodic capillary-gravity waves are not necessarily $2\pi/\kappa_2$-periodic. See Figure \ref{figure11} for more details.

In the subregion of $S_1$ where $\frac{d\sigma_+}{dk}(0^-)<\frac{d\sigma_+}{dk}(-\kappa_2)$, so that $(k_3-k_5)(\sigma)$ is increasing for $\sigma\in [0,\sigma_{3,5}]$ and decreasing for $\sigma \in [\sigma_{3,5},-\sigma_{c,1}]$ by Lemma \ref{nonmonotonekikj} ($\sigma_{3,5}$ is the critical point of $k_3-k_5$), the resonance region for (R13) of order $2$ is bounded from above by the curve along which $(k_3-k_5)(\sigma_{3,5})=2\kappa$, and from below by the curve along which $\min(\kappa_2,-k_{c,1}-k_5(-\sigma_{c,1}))=2\kappa$. 

On the other hand, in the subregion of $S_1$ where $\frac{d\sigma_+}{dk}(0^-)> \frac{d\sigma_+}{dk}(-\kappa_2)$, so that $k_3-k_5$ is strictly decreasing  throughout $[0,-\sigma_{c,1}]$, the resonance region for (R13) of order $2$ is bounded from above by the curve $\kappa_2=2\kappa_1$, corresponding to Wilton ripples of order $2$, and from below by the curve along which $-k_{c,1}-k_5(-\sigma_{c,1})=2\kappa$. See Figure \ref{figure11}. 

\begin{figure}[htbp]
    \centering
    \includegraphics[scale=0.3]{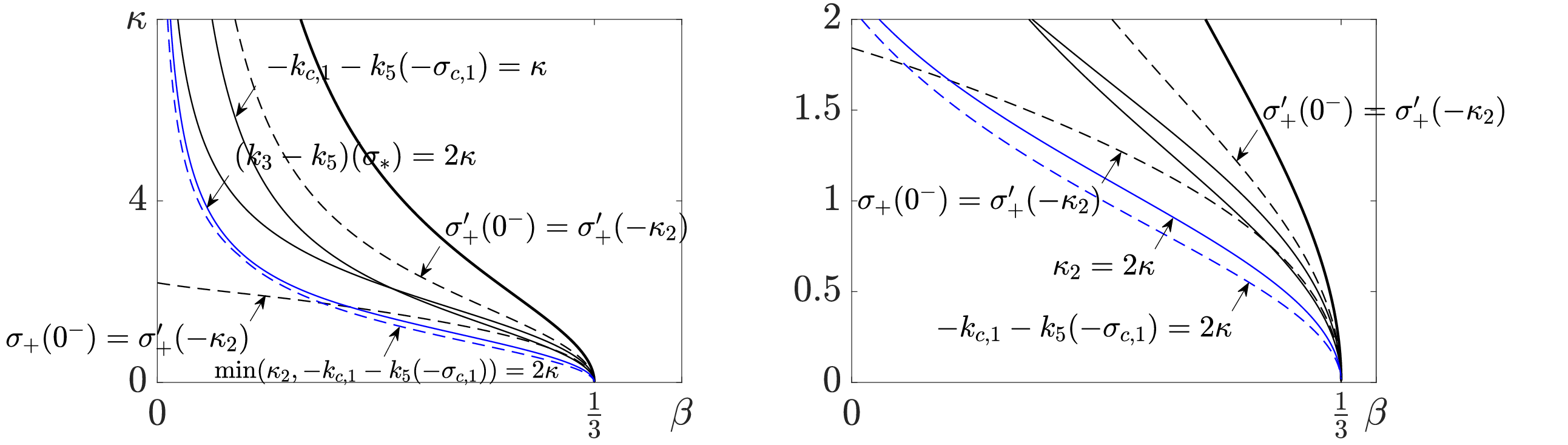}
    \caption{$\kappa$ versus $\beta$ for (R13). 
    Left: The resonance region for $N=1$ is bounded from above by the upper dashed black curve along which $\frac{d\sigma_+}{dk}(0^-)=\frac{d\sigma_+}{dk}(-\kappa_2)$, and from below by the solid black curve along which $-k_{c,1}-k_5(-\sigma_{c,1})=\kappa$. The lower boundary curve intersects the curve separating $S_1$ and $S_2$.  In the subregion of $S_1$ where $\frac{d\sigma_+}{dk}(0^-)<\frac{d\sigma_+}{dk}(-\kappa_2)$, so that $(k_3-k_5)(\sigma)$ increases and then decreases for $\sigma\in[0,-\sigma_{c,1}]$, the resonance region for $N=2$ is bounded from above by the solid blue curve along which $\max_{\sigma\in[0,-\sigma_{c,1}]}(k_3-k_5)(\sigma)=2\kappa_1$, and from below by the dashed blue curve along which $\min(\kappa_2,-k_{c,1}-k_5(-\sigma_{c,1}))=2\kappa$.\protect\\
    \protect\hspace*{11pt} Right: In the region where $\frac{d\sigma_+}{dk}(0^-)>\frac{d\sigma_+}{dk}(-\kappa_2)$, so that $(k_3-k_5)(\sigma)$ is strictly decreasing for $\sigma\in [0,-\sigma_{c,1}]$, the region for $N=2$ is bounded from above by the solid blue curve $\kappa_2=2\kappa_1$, and from below by the dashed blue curve $-k_{c,1}-k_5(-\sigma_{c,1})=2\kappa$.}
    \label{figure11}
\end{figure}


\subsubsection*{\rm (R14) $k_2(\sigma)-k_1(\sigma)=N\kappa$ for some $N\in\mathbb{Z}^+$} 

Recall from (R7) and (R10) that $(k_2-k_3)(\sigma)\geq\kappa_1$ for $\sigma\in [0,-\sigma_{c,1}]$ or $[0,\sigma_{c,2}]$, depending on whether $-\sigma_{c,1}<\sigma_{c,2}$ or $-\sigma_{c,1}>\sigma_{c,2}$, while $(k_6-k_5)(\sigma)\leq 2\kappa_2$ for $\sigma\in [0,\sigma_{c,2}]$. Consequently,
\[
\kappa_1\leq (k_2-k_3)(\sigma)< (k_2-k_1)(\sigma)<(k_6-k_5)(\sigma)\leq 2\kappa_2
\]
for $\sigma\in[0,\min(-\sigma_{c,1},\sigma_{c,2})]$. Therefore, resonance (R14) can occur in $S_1$ but only for $N\geq 2$, and in $S_2$ but only for $N=1$.

We inspect the monotonicity of $k_2-k_1$. Since
\[
\frac{d(k_2-k_1)}{d\sigma}(\sigma)=\frac{\frac{d\sigma_+}{dk}(k_1(\sigma))-\frac{d\sigma_-}{dk}(k_2(\sigma))}{\frac{d\sigma_+}{dk}(k_1(\sigma))\frac{d\sigma_-}{dk}(k_2(\sigma))}
\]
and since $\frac{d\sigma_+}{dk}(k_1(\sigma))$, $\frac{d\sigma_-}{dk}(k_2(\sigma))>0$ for $\sigma\in[0,\min(-\sigma_{c,1},\sigma_{c,2})]$ by Lemmas \ref{lem:sigmaS12} and \ref{lem:eps=0}, the sign of $\frac{d}{d\sigma}(k_2-k_1)(\sigma)$ coincides with the sign of $\frac{d\sigma_+}{dk}(k_1(\sigma))-\frac{d\sigma_-}{dk}(k_2(\sigma))$. If $-\sigma_{c,1}>\sigma_{c,2}$ then
\[
\frac{d\sigma_+}{dk}(k_1(\sigma_{c,2}))-\frac{d\sigma_-}{dk}(k_{c,2})=\frac{d\sigma_+}{dk}(k_1(\sigma_{c,2}))>0,
\]
implying that $(k_2-k_1)(\sigma)$ is increasing for $\sigma$ near $\sigma_{c,2}$. If $-\sigma_{c,1}<\sigma_{c,2}$, on the other hand,
\[
\frac{d\sigma_+}{dk}(-k_{c,1})-\frac{d\sigma_-}{dk}(k_2(-\sigma_{c,1}))=-\frac{d\sigma_-}{dk}(k_2(-\sigma_{c,1}))<0,
\]
implying that $(k_2-k_1)(\sigma)$ is decreasing near $\sigma=-\sigma_{c,1}$. Additionally,
\begin{align*}
\mathrm{sgn}\left(\frac{d(k_2-k_1)}{d\sigma}(\sigma)\right)=&
\mathrm{sgn}\left(\frac{d^2\sigma_+}{dk^2}(k_1(0))\frac{dk_1}{d\sigma}(0)-\frac{d^2\sigma_-}{dk^2}(k_2(0))\frac{dk_2}{d\sigma}(0)\right)\\
=&\mathrm{sgn}\left(\frac{d^2\sigma_+}{dk^2}(-\kappa_1)\frac{dk_1}{d\sigma}(0)\right)=\mathrm{sgn}\left(\frac{d^2\sigma_+}{dk^2}(-\kappa_1)\right)
\end{align*}
for $\sigma$ near $0$. Here the first equality follows because $\frac{d\sigma_+}{dk}(k_1(0))=\frac{d\sigma_-}{dk}(k_2(0))$.

The top left panel of Figure \ref{figure12} depicts two solid curves along which 
\[
\frac{d^2\sigma_+}{dk^2}(-\kappa_1)=0.
\]
In the region bounded by these curves, $\frac{d^2\sigma_+}{dk^2}(-\kappa_1)>0$ so that $(k_2-k_1)(\sigma)$ is increasing for $\sigma$ near $0$. In the region above the upper curve and below the lower curve, $(k_2-k_1)(\sigma)$ is decreasing near $\sigma=0$. The lower solid curve lies within $S_1$ and intersects the lower dashed curve, where $-\sigma_{c,1}=\sigma_{c,2}$, at $\beta=0.222492\dots$. See the top middle panel of Figure \ref{figure12}. The upper solid curve lies in $S_2$ and intersects the upper dashed curve $-\sigma_{c,1}=\sigma_{c,2}$ at the same value of $\beta$. See the top right panel of Figure \ref{figure12}. 

In the subregion of $S_1$, bounded by the lower solid curve  $\frac{d^2\sigma_+}{dk^2}(-\kappa_1)=0$ and the lower dashed curve $-\sigma_{c,1}=\sigma_{c,2}$ and, correspondingly, in the subregion of $S_2$ bounded by the upper solid and dashed curves, $(k_2-k_1)(\sigma)$ increases and then decreases for $\sigma\in [0,\min(-\sigma_{c,1},\sigma_{c,2})]$ if $\beta<0.222492\dots$, while decreasing and then increasing if $\beta>0.222492\dots$.

In the subregion of $S_1$, above the lower solid curve $\frac{d^2\sigma_+}{dk^2}(-\kappa_1)=0$ and above the lower dashed curve $-\sigma_{c,1}=\sigma_{c,2}$, and, correspondingly, in the subregion of $S_2$ below the upper solid and dashed curves, on the other hand, $k_2-k_1$ increases, decreases, and then increases over the interval $[0,\min(-\sigma_{c,1},\sigma_{c,2})]$. Each of these regions is bounded from above by a curve along which $k_2-k_1$ has only one critical point in the interval $[0,\sigma_{c,2}]$, rather than two. The top middle and right panels of Figure \ref{figure12} show these boundary curves in $S_1$ and $S_2$, respectively.

The monotonicity of $k_2-k_1$ implies that (R14) occurs in $S_1$ for $N=2$, and in $S_2$ for $N=1$. See Figure \ref{figure12} for further details.

\begin{figure}[htbp]
\centering
    \includegraphics[scale=0.3]{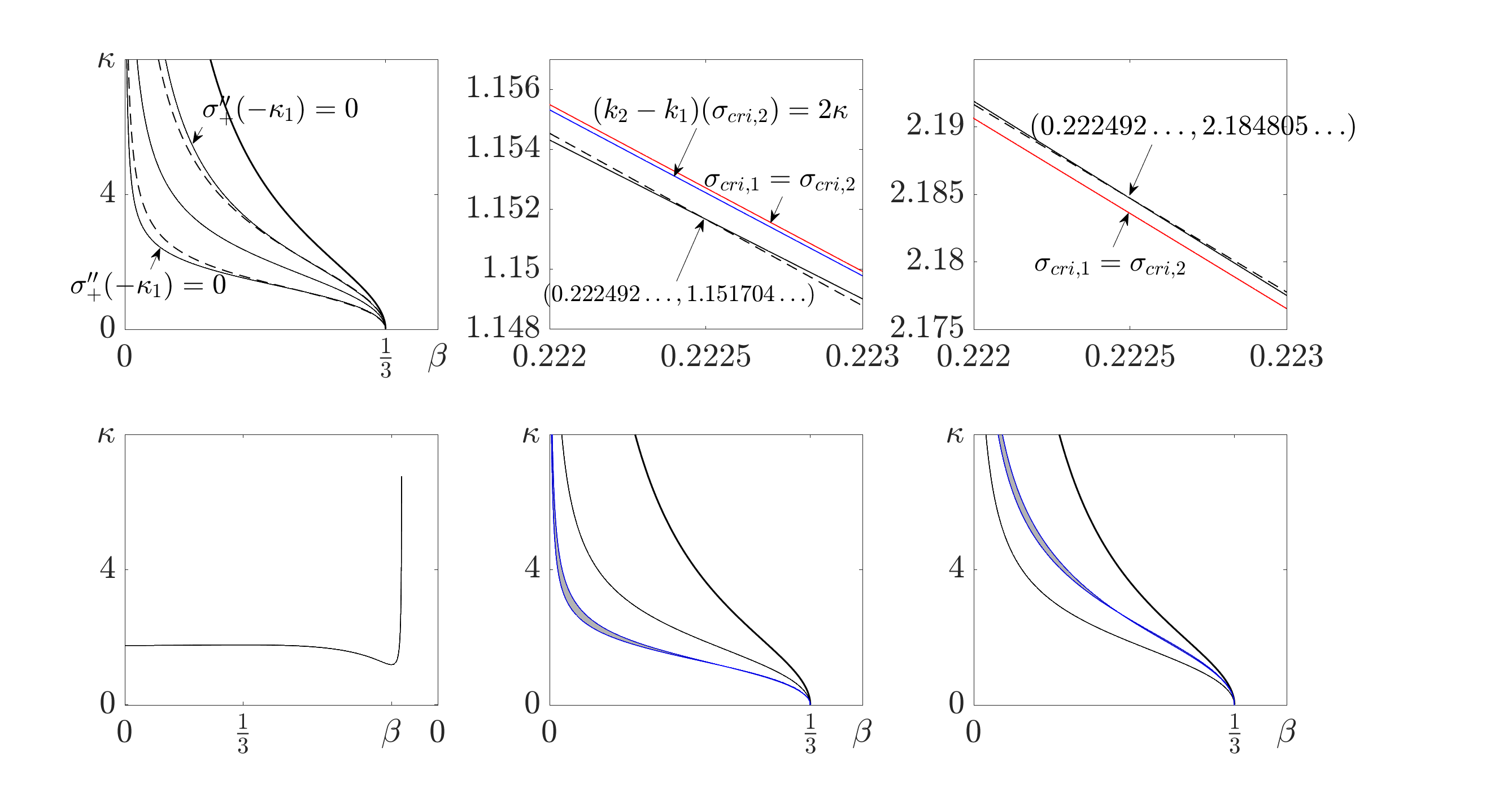}
    \caption{Resonance (R14). The top left panel shows two solid curves along which $\frac{d^2\sigma_+}{dk^2}(-\kappa_1)=0$, and two dashed curves where $-\sigma_{c,1}=\sigma_{c,2}$. In the region bounded by the two solid curves, $\frac{d^2\sigma_+}{dk^2}(-\kappa_1)>0$, so that $(k_2-k_1)(\sigma)$ increases for $\sigma$ near $0$. The lower solid curve lies in $S_1$ and intersects the lower dashed curve at $(0.222492\dots,1.151704\dots)$. The upper solid curve lies in $S_2$ and intersects the upper dashed curve at $(0.222492\dots,2.184805\dots)$. See the top middle and right panels for close-ups. \protect\\
    \protect\hspace*{11pt} In the subregion in $S_1$ bounded by the lower solid and dashed curves, and in the subregion in $S_2$ bounded by the upper solid and dashed curves, $k_2-k_1$ increases and then decreases if $\beta<0.222492\dots$, while it decreases and then increases otherwise. In the subregion in $S_1$ above the lower solid and dashed curves, and in the subregion in $S_2$ below the upper solid and dashed curves, on the other hand, $(k_2-k_1)(\sigma)$ increases, decreases, and then increases. The top middle and right panels show that each of these regions is bounded by a solid red curve along which  $k_2-k_1$ has only one critical point. A representative graph of $(k_2-k_1)(\sigma)$ is in the bottom left panel. \protect\\
    \protect\hspace*{11pt}Bottom middle: In $S_1$, the resonance region for $N=2$ is bounded from above and below by the solid blue curves. The upper boundary curve crosses the dashed curve $-\sigma_{c,1}=\sigma_{c,2}$ transversely from the left as $\beta$ increases through $0.214164\dots$. \protect\linebreak}
    \label{figure12}
\end{figure}
\addtocounter{figure}{-1}
\begin{figure}[htbp]
  \caption{(Continued) Along the upper boundary curve, $k_2(-\sigma_{c,1})+k_{c,1}=2\kappa$ if $\beta<0.214164\dots$ while $(k_2-k_1)(\sigma_{c,2})=2\kappa$ if $\beta>0.214164\dots$. Along the lower boundary curve, $\frac{d^2\sigma_+}{dk^2}(-\kappa_1)=0$ if $\beta<0.222603\dots$ while $k_{c,2}-k_1(\sigma_{c,2})=2\kappa$ otherwise. Bottom right: In $S_2$, the resonance region for $N=1$ is bounded from above and below by the solid blue curves. We omit the formulae for the boundary curves.}
  \end{figure}

\subsubsection*{\rm (R15) $k_6(\sigma)-k_3(\sigma)=N\kappa$ for some $N\in \mathbb{Z}^+$} 

Recall from (R7) and (R10) that $(k_2-k_3)(\sigma)\geq\kappa_1$ wherever defined, while $(k_6-k_5)(\sigma)\leq 2\kappa_2$ for $\sigma\in[0,\sigma_{c,2}]$. Consequently, 
\[
\kappa_1\leq(k_2-k_3)(\sigma)<(k_6-k_3)(\sigma)<(k_6-k_5)(\sigma)\leq 2\kappa_2
\]
for $\sigma\in[0,\min(-\sigma_{c,1},\sigma_{c,2})]$. Therefore, resonance (R15) can occur in $S_1$ but only for $N\geq 2$, and in $S_2$ but only for $N=1$.

Since
\[
\frac{d(k_6-k_3)}{d\sigma}(\sigma)=\frac{\frac{d\sigma_+}{dk}(k_3(\sigma))-\frac{d\sigma_-}{dk}(k_6(\sigma))}{\frac{d\sigma_+}{dk}(k_3(\sigma))\frac{d\sigma_-}{dk}(k_6(\sigma))}
\]
and since $\frac{d\sigma_+}{dk}(k_3(\sigma)),\frac{d\sigma_-}{dk}(k_6(\sigma))<0$ for $\sigma\in[0,\min(-\sigma_{c,1},\sigma_{c,2})]$ by Lemmas \ref{lem:sigmaS12} and \ref{lem:eps=0}, the sign of $\frac{d}{d\sigma}(k_6-k_3)(\sigma)$ is the same as the sign of $\frac{d\sigma_+}{dk}(k_3(\sigma))-\frac{d\sigma_-}{dk}(k_6(\sigma))$. If $-\sigma_{c,1}>\sigma_{c,2}$ then 
\[
\frac{d\sigma_+}{dk}(k_3(\sigma_{c,2}))-\frac{d\sigma_-}{dk}(k_{c,2})=\frac{d\sigma_+}{dk}(k_3(\sigma_{c,2}))<0,
\]
implying that $(k_6-k_3)(\sigma)$ is decreasing for $\sigma$ near $\sigma_{c,2}$. If $-\sigma_{c,1}<\sigma_{c,2}$, on the other hand,
\[
\frac{d\sigma_+}{dk}(-k_{c,1})-\frac{d\sigma_-}{dk}(k_6(-\sigma_{c,1}))=-\frac{d\sigma_-}{dk}(k_6(-\sigma_{c,1}))>0,
\]
implying that $(k_6-k_3)(\sigma)$ is increasing near $\sigma=-\sigma_{c,1}$. 
Moreover, the sign of $\frac{d(k_6-k_3)}{d\sigma}(\sigma)$ coincides with the sign of $\frac{d\sigma_+}{dk}(0^-)-\frac{d\sigma_-}{dk}(\kappa_2)$ for $\sigma$ near $0$. 

The top left panel of Figure \ref{figure13} depicts two solid black curves along which 
\[
\frac{d\sigma_+}{dk}(0^-)=\frac{d\sigma_-}{dk}(\kappa_2),
\]
and two dashed curves where $-\sigma_{c,1}=\sigma_{c,2}$. 
In the region bounded by these solid black curves, $\frac{d\sigma_+}{dk}(0^-)<\frac{d\sigma_-}{dk}(\kappa_2)$ so that $(k_6-k_3)(\sigma)$ is decreasing for $\sigma$ near $0$. In the region above the upper solid black curve and below the lower curve, $(k_6-k_3)(\sigma)$ is increasing near $\sigma=0$. The lower solid black curve lies within $S_1$ and the upper curve lies in $S_2$. 

In the subregion of $S_1$, bounded by the lower solid black curve $\frac{d\sigma_+}{dk}(0^-)=\frac{d\sigma_-}{dk}(\kappa_2)$ and the lower dashed curve $-\sigma_{c,1}=\sigma_{c,2}$, and, correspondingly, in the subregion of $S_2$ bounded by the upper solid and dashed black curves, $(k_6-k_3)(\sigma)$ decreases and then increases for $\sigma \in [0,\min(-\sigma_{c,1},\sigma_{c,2})]$ if $\beta<0.127595\dots$ while increasing and then decreasing if $\beta>0.127595\dots$. See the bottom left panel of Figure \ref{figure13}.

On the other hand, in the subregion of $S_1$, above the lower solid black curve $\frac{d\sigma_+}{dk}(0^-)=\frac{d\sigma_-}{dk}(\kappa_2)$ and the lower dashed curve $-\sigma_{c,1}=\sigma_{c,2}$, and in the subregion of $S_2$ below the upper solid and dashed black curves, $k_6-k_3$ increases, decreases, and then increases over the interval $[0,\min(-\sigma_{c,1},\sigma_{c,2})]$. The top left panel of Figure \ref{figure13} shows that each of these regions is bounded from above by a solid red curve along which $k_6-k_3$ has only one critical point in the interval $[0,-\sigma_{c,1}]$.  

The monotonicity of $k_6-k_3$ implies that (R15) occurs in $S_1$ for $N=2$, and in $S_2$ for $N=1$. See Figure \ref{figure13} for further details.

\begin{figure}[htbp]
    \centering
    \includegraphics[scale=0.3]{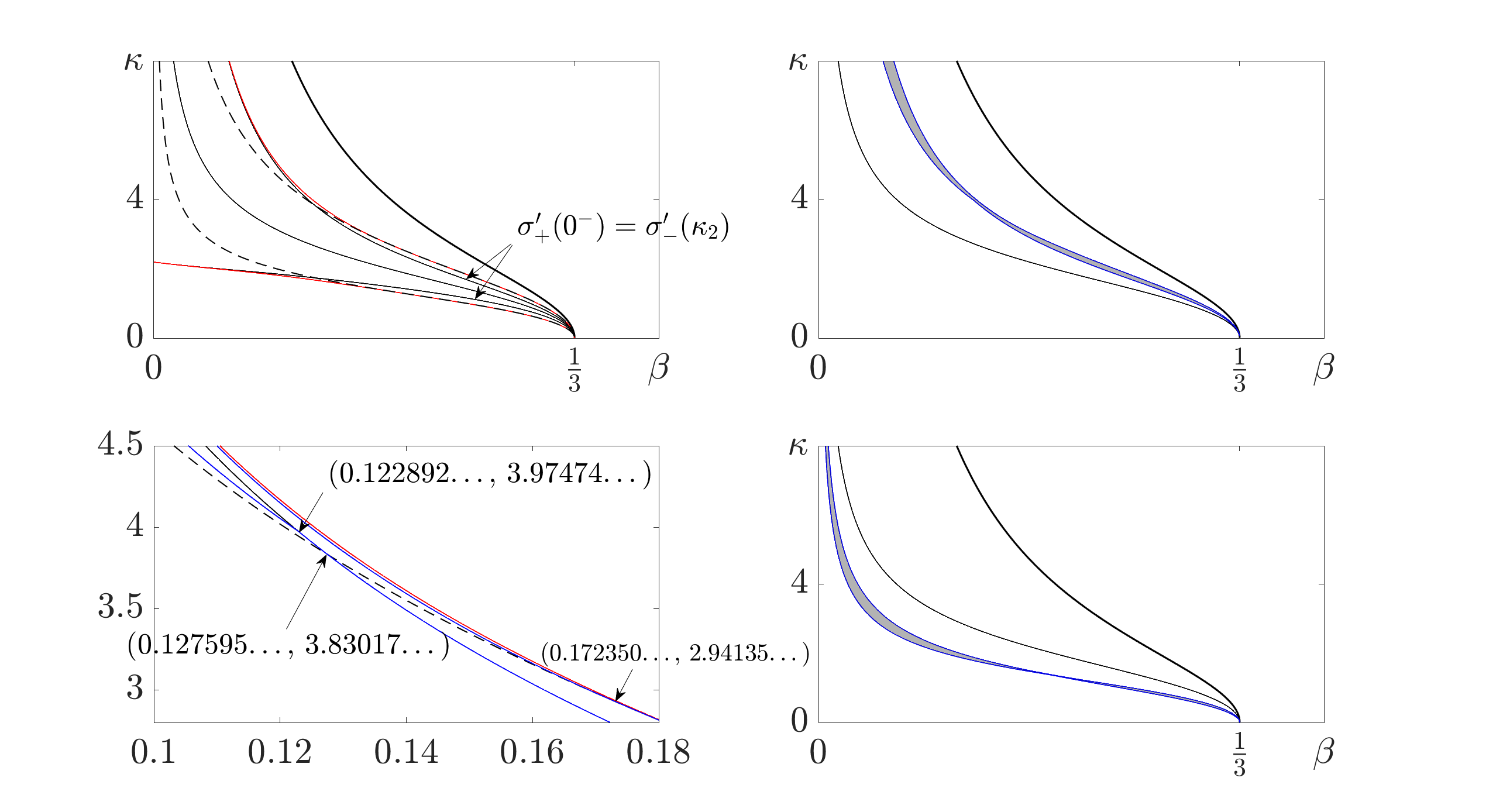}
    \caption{$\kappa$ versus $\beta$ for (R15). The top left panel shows two solid black curves along which $\frac{d\sigma_+}{dk}(0^-)=\frac{d\sigma_-}{dk}(\kappa_2)$, and two dashed curves where $-\sigma_{c,1}=\sigma_{c,2}$. In the region bounded by the two solid black curves, $\frac{d\sigma_+}{dk}(0^-)<\frac{d\sigma_-}{dk}(\kappa_2)$, so that $(k_6-k_3)(\sigma)$ decreases for $\sigma$ near $0$. The lower solid black curve lies in $S_1$ and the upper one lies in $S_2$. \protect\\
    \protect\hspace*{11pt}In the subregion in $S_1$ bounded by the lower solid and dashed black curves, and in the subregion in $S_2$ bounded by the upper solid and dashed black curves, $k_6-k_3$ decreases and then increases if $\beta<0.127595\dots$, while it increases and then decreases otherwise. 
    In the subregion in $S_1$ above the lower solid and dashed black curves, and in the subregion in $S_2$ below the upper solid and dashed black curves, $k_6-k_3$ increases, decreases, and then increases. Each such region is bounded by a solid red curve along which $k_6-k_3$ has only one critical point. \protect\\
    \protect\hspace*{11pt}Top right: In $S_2$, the resonance region for $N=1$ is bounded from above and below by the solid blue curves. The upper boundary curve crosses the dashed curve $-\sigma_{c,1}=\sigma_{c,2}$ transversely  from the left as $\beta$ increases through $0.172350\dots$. Along the upper boundary curve,  $k_6-k_3=\kappa$ at the critical point if $\beta<0.172350\dots$ while $k_{c,2}-k_3=\kappa$ if $\beta>0.172350\dots$. Along the lower boundary curve, $\frac{d\sigma_+}{dk}(0^-)=\frac{d\sigma_-}{dk}(\kappa_2)$ if $\beta>0.122892\dots$ while $k_6(-\sigma_{c,1})+k_{c,1}=\kappa$ otherwise. See the bottom left panel for a close-up. Bottom right: In $S_1$, the resonance region for $N=2$ is bounded from above and below by the solid blue curves. We omit the formulae for the boundary curves.}
    \label{figure13}
\end{figure}

It remains to treat resonances at $-i\sigma_{c,1}$ and $i\sigma_{c,2}$. Recall from Lemma \ref{lem:eps=0} that 
\[
k_1(-\sigma_{c,1})=k_3(-\sigma_{c,1})=-k_{c,1}\quad\text{and}\quad k_2(\sigma_{c,2})=k_6(\sigma_{c,2})=k_{c,2}.
\]

\subsubsection*{\rm (R16) $-k_{c,1}-k_5(-\sigma_{c,1})=N\kappa$ for some $N\in\mathbb{Z}^+$} Along the resonance curve $-k_{c,1}-k_5(-\sigma_{c,1})=N\kappa$ (see the right panel of Figure \ref{figure7} for $N=1,2$), 
\[
k_1(-\sigma_{c,1})-k_5(-\sigma_{c,1})=N\kappa
\quad\text{and}\quad 
k_3(-\sigma_{c,1})-k_5(-\sigma_{c,1})=N\kappa.
\]

\subsubsection*{\rm (R17) $k_6(-\sigma_{c,1})+k_{c,1}=N\kappa$ or $k_{c,2}-k_1(\sigma_{c,2})=N\kappa$ for some $N\in\mathbb{Z}^+$} 

Along the resonance curve $k_6(-\sigma_{c,1})+k_{c,1}=N\kappa$ (see Figure \ref{figure8} for $N=1,2$), 
\[
k_6(-\sigma_{c,1})-k_1(-\sigma_{c,1})=N\kappa\quad\text{and}\quad
k_6(-\sigma_{c,1})-k_3(-\sigma_{c,1})=N\kappa.
\]
Along the resonance curve $k_{c,2}-k_1(\sigma_{c,2})=N\kappa$ (see Figure \ref{figure8} for $N=1,2$),
\[
k_2(\sigma_{c,2})-k_1(\sigma_{c,2})=N\kappa\quad\text{and}\quad
k_6(\sigma_{c,2})-k_1(\sigma_{c,2})=N\kappa.
\]

\subsubsection*{\rm (R18) $k_2(-\sigma_{c,1})+k_{c,1}=N\kappa$ or $k_{c,2}-k_3(\sigma_{c,2})=N\kappa$ for some $N\in\mathbb{Z}^+$} 

Along the resonance curve $k_2(-\sigma_{c,1})+k_{c,1}=N\kappa$ (see Figure \ref{figure9} for $N=1,2$), 
\[
k_2(-\sigma_{c,1})-k_1(-\sigma_{c,1})=N\kappa\quad\text{and}\quad
k_2(-\sigma_{c,1})-k_3(-\sigma_{c,1})=N\kappa.
\]
Along the resonance curve $k_{c,2}-k_3(\sigma_{c,2})=N\kappa$ (see Figure \ref{figure9} for $N=1,2$),
\[
k_2(\sigma_{c,2})-k_3(\sigma_{c,2})=N\kappa\quad\text{and}\quad 
k_6(\sigma_{c,2})-k_3(\sigma_{c,2})=N\kappa.
\]

\subsubsection*{\rm (R19) $k_{c,2}-k_4(\sigma_{c,2})=N\kappa$ for some $N\in\mathbb{Z}^+$} 
Along the resonance curve $k_{c,2}-k_4(\sigma_{c,2})=N\kappa$, $N\geq 2$, (see the left panel of Figure \ref{figure10} for $N=2$), 
\[
k_2(\sigma_{c,2})-k_4(\sigma_{c,2})=N\kappa\quad\text{and}\quad
k_6(\sigma_{c,2})-k_4(\sigma_{c,2})=N\kappa.
\]

For Stokes' gravity waves, repeated resonances of this kind do not occur. See \cite{HY2023}. 

\subsection{Classification of resonances}\label{resonance_summary} 

We summarize nonzero resonances for steady periodic capillary-gravity waves associated with Floquet exponents of opposite Krein signatures. That is, $k_j(\sigma)$ and $k_{j'}(\sigma)$ belong to opposite branches of \eqref{def:sigma}. 

\begin{itemize}
\item[$\bullet$] In $S_3$, corresponding to subcritical capillary-gravity waves, no such resonances occur at any order. This plays a key role in the spectral stability result obtained in \cite{sun2025}.

\item[$\bullet$] In $S_1$, corresponding to $2\pi/\kappa_1$-periodic supercritical waves, where $\kappa_1<\kappa_2$, such resonances occur for 
\begin{itemize}
\item[$\circ$] orders $\geq 3$ in (R3), (R6), (R7), (R9), (R10), (R12), (R14), and (R15), and
\item[$\circ$] order $2$ in (R6), (R7), (R9), (R12), (R14), and (R15).
\end{itemize}
No such resonances occur for order $1$.

\item[$\bullet$] In $S_2$, for $2\pi/\kappa_2$-periodic supercritical waves, such resonances occur for order $1$ in (R6), (R7), (R14), (R15). They do not occur for orders $\geq 2$.
\end{itemize}



\newpage

\section{Spectral stability away from resonances}\label{sec:stability} 

Recall Section \ref{sec:reduction}. Lemma \ref{lem:dispersion} establishes 
\[
D(i\sigma,k_j(\sigma),0)=0\quad\text{for each $\sigma\in \mathbb{R}$},
\]
where $k_j(\sigma)$ is in Lemma \ref{lem:eps=0}. We are thus interested in the roots of $D(\lambda,k,\eps)$ near $(i\sigma,k_j(\sigma),0)$. Let
\[
\text{$\lambda=i\sigma+\delta$, $\delta\in\mathbb{C}$ and  $|\delta|\ll1$},\quad
\text{$k=k_j(\sigma)+\gamma$, $\gamma\in\mathbb{R}$ and $|\gamma|\ll1$}.
\]
Recall that \eqref{def:Delta} is jointly analytic in its arguments. Suppose
\[
\frac{\partial D}{\partial\lambda}(i\sigma,k_j(\sigma),0),\dots,\frac{\partial^{p-1}D}{\partial\lambda^{p-1}}(i\sigma,k_j(\sigma),0)=0,
\]
and
\[
\frac{\partial^pD}{\partial\lambda^p}(i\sigma,k_j(\sigma),0)\neq0
\quad\text{for some $p\in \mathbb{Z}^+$},
\]
so that 
\ba \label{eqn:D(l,m,n)p}
D&(i\sigma+\delta,k_j(\sigma)+\gamma,\eps)=\frac{\partial^{p}D}{\partial\lambda^{p}}(i\sigma,k_j(\sigma),0)\frac{\delta^p}{p!}+O(|\delta|^{p+1}+|\gamma|+|\eps|)
\ea 
as $\delta$, $\gamma$, $\eps\to 0$. The Weierstrass preparation theorem (see, for instance, \cite{Hormander}) asserts
\[
D(i\sigma+\delta,k_{j}(\sigma)+\gamma,\eps)=W(\delta,\gamma,\eps)h(\delta,\gamma,\eps)
\]
for $\delta\in\mathbb{C}$, $\gamma,\eps\in\mathbb{R}$ and $|\delta|,|\gamma|,|\eps|\ll1$, where
\ba \label{def:W(l,m,n)p}
W(\delta,\gamma,\eps)=&\delta^p+a_{p-1}(\gamma,\eps)\delta^{p-1}+\cdots+a_0(\gamma,\eps)
\ea 
is a Weierstrass polynomial, $a_{p-1}(\gamma,\eps),\ldots, a_0(\gamma,\eps)$ are analytic at $(0,0)$, and $a_{p-1}(0,0),\ldots, a_0(0,0)=0$, while $h(\delta,\gamma,\eps)$ is analytic at $(0,0,0)$ and
\[
h(0,0,0)=\frac{1}{p!}\frac{\partial^{p}D}{\partial\lambda^{p}}(i\sigma,k_j(\sigma),0)\neq 0.
\]
The roots of \eqref{eqn:D(l,m,n)p} coincide with the roots of \eqref{def:W(l,m,n)p}, counting multiplicity, for $\gamma,\eps\in\mathbb{R}$ and $|\gamma|,|\eps|\ll1$.

\begin{lemma}[Symmetry of $W(\delta,\gamma,\eps)$]\label{lem:Wsymm}
Let \eqref{def:W(l,m,n)p} denote the Weierstrass polynomial associated with  \eqref{eqn:D(l,m,n)p} about $(\delta,\gamma,\eps)=(0,0,0)$. Then
\be \label{eqn:Wsymm}
W(\delta,\gamma,\eps)=0\quad\text{if and only if}\quad W(-\delta^*,\gamma,\eps)=0
\ee 
for $\gamma,\eps\in\mathbb{R}$ and $|\gamma|,|\eps|\ll1$. In other words, the roots of \eqref{def:W(l,m,n)p} are symmetric with respect to reflection across the imaginary axis. Moreover, 
\be \label{eqn:a(m,n)p}
i^{\ell}a_{p-\ell}(\gamma,\eps)\in \mathbb{R},\quad \ell=1,2,\ldots,p,
\ee 
for $\gamma,\eps\in\mathbb{R}$ and $|\gamma|,|\eps|\ll1$.
\end{lemma}

Here and in the sequel, the superscript $*$ means complex conjugation.

\begin{proof}
It follows from Section \ref{sec:reduction} and Lemma \ref{lem:symm} that
\[
D(i\sigma+\delta,k_j(\sigma)+\gamma,\eps)=0\text{ if and only if }D(i\sigma-\delta^*,k_j(\sigma)+\gamma,\eps)=0.
\]
Moreover, the corresponding roots have the same multiplicity. 
Since the roots of \eqref{eqn:D(l,m,n)p} coincide with the roots of \eqref{def:W(l,m,n)p}, counting multiplicity, \eqref{eqn:Wsymm} follows immediately. 

To prove \eqref{eqn:a(m,n)p}, we introduce
\[
W^\dag(\delta,\gamma,\eps)
:=i^pW(\delta/i,\gamma,\eps)
=\delta^p+\sum_{\ell=1}^p
i^\ell a_{p-\ell}(\gamma,\eps)\delta^{p-\ell}.
\]
Clearly, if $z$ is a root of $W^\dag(\cdot,\gamma,\eps)$, then $z/i$ is a root of $W(\cdot,\gamma,\eps)$, and $-(z/i)^*=z^*/i$
is also a root of $W(\cdot,\gamma,\eps)$ with the same multiplicity, by \eqref{eqn:Wsymm}. Consequently, $z^*$ is a root of $W^\dag(\cdot,\gamma,\eps)$ with the same multiplicity. In other words, the roots of $W^\dag(\cdot,\gamma,\eps)$ are invariant under complex conjugation, counting multiplicity. Therefore, its coefficients are all real. Comparing coefficients then yields
\[
i^\ell a_{p-\ell}(\gamma,\eps)\in\mathbb{R},
\qquad \ell=1,\ldots,p.
\]
This completes the proof.
\end{proof}
Recall Definition \ref{def:resonance}. That is, $i\sigma \in i\mathbb{R}$ is {\rm resonant} if 
\[
k_j(\sigma)-k_{j'}(\sigma)=N\kappa \quad \text{for some $N\in\mathbb{Z}^+$}\quad \text{for some $j$ and $j'$},
\]
where $k_j(\sigma)$ and $k_{j'}(\sigma)$ are in Lemma \ref{lem:eps=0}. 

\begin{lemma}[Weierstrass polynomial at non-resonances]\label{lem:W(p=1)}
Let \eqref{def:W(l,m,n)p} be the Weierstrass polynomial associated with \eqref{eqn:D(l,m,n)p} about $(i\sigma,k_j(\sigma),0)$ or, equivalently, $(\delta,\gamma,\eps)=(0,0,0)$. If $i\sigma\in i\mathbb{R}$ is non-resonant, then $W(\cdot,\gamma,\eps)$ is of degree $1$ and has a unique root, purely imaginary, for $\gamma,\eps \in\mathbb{R}$ and $|\gamma|, |\eps|\ll1$.
\end{lemma}

\begin{proof}
Below we write $k_j$ and $k_{j'}$ instead of $k_j(\sigma)$ and $k_{j'}(\sigma)$ for brevity. 

Recall Sections \ref{sec:reduction} and \ref{sec:expansion_monodromy}. We solve \eqref{def:V(0,0)sigma}-\eqref{def:f(sigma)} and \eqref{cond:V(0)}, $m=1$ and $n=0$, and make an explicit calculation of the solution, to arrive at
\[
\frac{\partial D}{\partial \lambda}(i\sigma,k_j(\sigma),0)=-v^{(1,0)}_{jj} \prod_{j\neq j'}(e^{ik_jT}-e^{ik_{j'}T}),
\]
where 
\ba\label{def:V(1,0)sigma}
v^{(1,0)}_{jj}=&2k_j(\beta k_j^2+1)\sh(2k_j)Te^{ik_jT} \\ 
&\cdot((2k_j + \sh(2k_j) + 2\beta k_j^3+ 3\beta k_j^2\sh(2k_j))\sigma \\
&\quad + k_j \sh(2k_j)- 2\beta k_j^4 - 2k_j^2 - \beta k_j^3\sh(2k_j))^{-1}\neq 0
\ea
and $T=2\pi/\kappa$ the period. See the proof of \cite[Lemma 6.3]{HY2023} for an analogous calculation for Stokes' gravity waves. Since 
\[
e^{ik_jT}-e^{ik_{j'}T}\neq 0\quad \text{for any $j\neq j'$}
\]
by hypothesis, $\frac{\partial D}{\partial \lambda}(i\sigma,k_j(\sigma),0)\neq 0$. Consequently, \eqref{eqn:D(l,m,n)p} and, hence, \eqref{def:W(l,m,n)p} hold for $p=1$. Lemma \ref{lem:Wsymm} then dictates that \eqref{def:W(l,m,n)p} has a unique root $\in i\mathbb{R}$ for $\gamma,\eps\in\mathbb{R}$ and $|\gamma|,|\eps|\ll1$. This completes the proof.
\end{proof}

It is noteworthy that the set of resonances 
\[
\bigcup_{j,j'}\,\{i\sigma \in i\mathbb{R}: k_j(\sigma)-k_{j'}(\sigma)=N\kappa \quad \text{for some $N\in\mathbb{Z}^+$}\}
\]
is closed in $i\mathbb{R}$. Conversely, the set of non-resonances is open.

\begin{theorem}[Stability near non-resonances]\label{thm:S non-res}
Suppose that $i\sigma\in i\mathbb{R}$ is non-resonant, so that $[i\sigma-i\sigma_0, i\sigma+i\sigma_0] \subset i\mathbb{R}$ contains no resonances for some $\sigma_0\neq 0, \ll1$. Then 
\[
{\rm spec}(\mathcal{L}(\eps)) \bigcap 
\{\lambda\in\mathbb{C}\setminus i\mathbb{R}: |\lambda-i\sigma|<\sigma_0 \}=\emptyset
\] 
for $\eps\in\mathbb{R}$ and $|\eps|<\eps_0$ for some $\eps_0\neq 0,\ll1$. In other words, ${\rm spec}(\mathcal{L}(\eps))$ near $i\sigma$ remains confined to the imaginary axis for $|\eps|$ sufficiently small. 
\end{theorem}

\begin{proof}
Suppose, to the contrary, that for each $n\in\mathbb{Z}^+$, there exist 
\[
\lambda_n\in {\rm spec}(\mathcal{L}(\eps_n))\setminus i\mathbb{R}
\quad \text{such that}\quad 
|\lambda_n-i\sigma|<\sigma_0,
\]
and $\eps_n\in\mathbb{R}$ satisfying $0<|\eps_n|<1/n$. Decreasing $\sigma_0$, if necessary, we may assume that $D(\lambda,k,\eps)$ is defined for $\lambda\in\mathbb{C}$ and $|\lambda-i\sigma|\leq\sigma_0$ for $k\in\mathbb{R}$ for $\eps\in\mathbb{R}$ $|\eps|<\eps_0 (\ll1)$. Clearly, $D(\lambda_n,k_n,\eps_n)=0$ for some $k_n\in\mathbb{R}$ for $n$ sufficiently large. We identify $k_n$ with $k_n+K\kappa$ for some $K\in\mathbb{Z}$, if necessary, by periodicity. 

Passing to a subsequence, if necessary, we may assume $\lambda_n\to\lambda_\infty$ and $k_n\to k_\infty$ as $n\to\infty$, where $|\lambda_\infty-i\sigma|\leq\sigma_0$ by continuity, and $D(\lambda_\infty,k_\infty,0)=0$, by the joint analyticity of $D(\lambda,k,\eps)$. Since ${\rm spec}(\mathcal{L}(0))=i\mathbb{R}$, it follows that $\lambda_\infty=i\sigma_\infty$ for some $\sigma_\infty\in\mathbb{R}$. Since $i\sigma_\infty\in[i\sigma-i\sigma_0,i\sigma+i\sigma_0]$, it is non-resonant by hypothesis. 
We may assume $k_\infty=k_j(\sigma_\infty)$ for some $j$ as in Definition \ref{def:resonance}, after replacing $k_n$ with $k_n+K\kappa$ for some $K\in\mathbb{Z}$, if necessary.

Let \eqref{def:W(l,m,n)p} denote the Weierstrass polynomial associated with $D(i\sigma_\infty+\delta,k_j(\sigma_\infty)+\gamma,\eps)$ about $(i\sigma_\infty,k_j(\sigma_\infty),0)$. Lemma \ref{lem:W(p=1)} asserts that its degree is $1$, with a unique purely imaginary root for $\gamma,\eps\in\mathbb{R}$ and $|\gamma|,|\eps|\ll1$. Since $\lambda_n\to i\sigma_\infty$, $k_n\to k_j(\sigma_\infty)$, and $\eps_n\to0$ as $n\to\infty$, it follows that $\lambda_n\in i\mathbb{R}$ for all $n$ sufficiently large. A contradiction then completes the proof.
\end{proof}

\begin{theorem}[Potential instability near resonances]\label{thm:I res}
Suppose
\[
k_j(\sigma)-k_{j'}(\sigma)=N\kappa
\quad\text{for some $N\in\mathbb{Z}^+$
and some $j,j'$}.
\]
Then the Weierstrass polynomial associated with \eqref{def:W(l,m,n)p} about $(i\sigma,k_j(\sigma),0)$ has degree
$\geq2$. Furthermore, any root $\delta$ of the polynomial satisfying ${\rm Re}\,\delta\neq0$ for $\gamma,\eps\in\mathbb{R}$ and $|\gamma|,|\eps|\ll1$ yields $i\sigma+\delta, i\sigma-\delta^* \in {\rm spec}(\mathcal{L}(\eps))$. In other words, the Weierstrass polynomial may have roots near $i\sigma \in i\mathbb{R}$ that leave the imaginary axis, implying spectral instability.
\end{theorem}

\begin{proof}
Consider $D(i\sigma+\delta,k_\ell(\sigma)+\gamma,\eps)$ for $\delta\in\mathbb{C}$, $\gamma,\eps\in\mathbb{R}$ and $|\delta|,|\gamma|,|\eps|\ll1$ for each $k_\ell(\sigma)$ in Lemma \ref{lem:eps=0}.

If $k_\ell(\sigma)-k_{\ell'}(\sigma)\neq N\kappa$ for any $N\in\mathbb{Z}$ for any $\ell\neq \ell'$, then the same kind of argument as in the proof of Lemma \ref{lem:W(p=1)} establishes that the Weierstrass polynomial associated with $D(i\sigma+\delta,k_\ell(\sigma)+\gamma,\eps)$ about $(i\sigma,k_\ell(\sigma),0)$ has a unique purely imaginary root for $\gamma,\eps\in\mathbb{R}$ and $|\gamma|,|\eps|\ll1$.

On the other hand, since $e^{ik_jT}=e^{ik_{j'}T}$ by hypothesis, the same kind of argument as in the proof of Lemma \ref{lem:W(p=1)} yields
\[
\frac{\partial D}{\partial\lambda}
(i\sigma,k_j(\sigma),0)=0.
\]
Consequently, \eqref{eqn:D(l,m,n)p} and \eqref{def:W(l,m,n)p} hold for some $p\geq 2$. The assertion then follows. 
\end{proof}

\begin{remark*}[Spectral stability and instability]
Theorems \ref{thm:S non-res} and \ref{thm:I res} together imply that to determine spectral stability or instability, it suffices to examine the roots of \eqref{def:Delta} for $(\lambda,k,\eps)$ near $(i\sigma,k_j(\sigma),0)$, where $i\sigma$ is resonant and $k_j(\sigma)$ is as in Definition \ref{def:resonance}. 
Moreover, a non-resonance cannot give rise to spectral instability. 
Theorems \ref{thm:S non-res} and \ref{thm:I res} also hold for Stokes' gravity waves, thereby providing additional justification for the approach in the authors' previous work \cite{HY2023}. 
\end{remark*}

\begin{remark}[Krein signature]\label{rem:Krein}
Suppose that $i\sigma \neq 0,\in i\mathbb{R}$ is resonant. Specifically, $k_j(\sigma)-k_{j'}(\sigma)=N\kappa$ for some $N\in\mathbb{Z}^+$  for some $j,j'$. It was shown in \cite{sun2025} that if there exists $\lambda\in{\rm spec}(\mathcal{L}(\eps))$ near $i\sigma$ such that ${\rm Re}\,\lambda\neq 0$ for some $\eps\neq 0,\in\mathbb{R}$ and $|\eps|\ll1$, then $k_j(\sigma)$ and $k_{j'}(\sigma)$ must have opposite Krein signatures, that is, $k_j(\sigma)$ and $k_{j'}(\sigma)$ belong to opposite branches of \eqref{def:sigma}. Our stability results corroborate this.  
\end{remark}


The authors' previous work \cite{HY2023} developed a rigorous spectral analysis for Stokes' gravity waves near a nonzero resonance of order $2$. See also Remark \ref{rem:Stokes}. Extending such analysis to higher order resonances would require computing higher order terms in the power series expansion of the periodic Evans function, a task that is technically highly demanding. See Section \ref{sec:N=3} for further discussion.  

Computing higher order terms in \eqref{def:Vexp} and, particularly, $m+n\geq 3$, is even more challenging for the capillary-gravity wave problem. Here we restrict our effort to $\mathbf{V}^{(m,n)}(x,\sigma)$ for $m+n\leq 2$. They are sufficient to determine spectral stability or instability near resonances of orders $1$ and $2$ for steady unimodal capillary-gravity waves and for bimodal Wilton ripples of order $\geq 2, \in \mathbb{Z}^+$. See Sections \ref{sec:N=1}, \ref{sec:N=2}, \ref{sec:sigma N=1}, \ref{sec:sigma N=2}, \ref{sec:M=2 N=2} for more details. They are sufficient to determine stability or instability for Wilton ripples of order $M\geq 3$ near resonances of orders $M-1$, $M$, $M+1$, and $2M$. See Sections \ref{sec:M=3 N=1,M} and \ref{sec:M=3 1<N<M} for more details.

\newpage

\section{Spectral stability and instability of unimodal waves \protect\\ near nonzero resonances}\label{sec:high freq instability} 

Let $\beta>0$ and $\kappa>0$. Recall Section \ref{sec:unimodal} and we assume
\begin{itemize}
\item[$\bullet$] either $\mu_0(\kappa;\beta)<1$,
\item[$\bullet$] or $\mu_0(\kappa_1;\beta)=\mu_0(\kappa_2;\beta)$ and $\kappa_1/\kappa_2\notin \mathbb{Q}$,
\end{itemize}
so that there exist steady unimodal capillary-gravity waves of sufficiently small amplitude, solving \eqref{stationary-periodic}-\eqref{eqn:stokes exp} for $\eps\neq0,\in\mathbb{R}$ and $|\eps|\ll1$, and satisfying \eqref{def:stokes1}. We address their spectral stability and instability. Recall Section \ref{sec:L} and we seek $\lambda \in {\rm spec}(\mathcal{L}(\eps))$ such that ${\rm Re}\, \lambda \neq 0$. Recall Section \ref{sec:reduction}, and this amounts to investigating the roots of \eqref{def:Delta}. Section \ref{sec:wilton instability} is devoted to bimodal Wilton ripples satisfying \eqref{def:wilton1 M=2} or \eqref{def:wilton1 M=3} instead of \eqref{def:stokes1}.

Throughout this section, we assume that $i\sigma \neq0,\in i\mathbb{R}$ is resonant in the sense of Definition \ref{def:resonance}. That is,
\[
k_j(\sigma)-k_{j'}(\sigma)=N\kappa \quad \text{for some $N\in\mathbb{Z}^+$}\quad\text{for some $j$ and $j'$}.
\]
Clearly, $i0$ is resonant. Section \ref{sec:modulational} treats modulational instability near $i0$. 

Lemma \ref{lem:dispersion} establishes $D(i\sigma,k_j(\sigma),0)=0$. Recall Theorems \ref{thm:S non-res} and \ref{thm:I res}, and it suffices to investigate the roots of 
\[
D(i\sigma+\delta,k_j(\sigma)+\gamma,\eps) \text{ for $\delta\in\mathbb{C}$ and $|\delta|\ll1$ for $\gamma, \eps\in\mathbb{R}$ and $|\gamma|,|\eps|\ll1$}.
\]
Since \eqref{def:Delta} is jointly analytic in its arguments, we make a power series expansion of $D(i\sigma+\delta,k_j(\sigma)+\gamma,\eps)$ about $(\delta,\gamma,\eps)=(0,0,0)$, and examine the roots of its leading order terms in $\delta$, $\gamma$, and $\eps$.

Section \ref{resonances} classifies all nonzero resonances for steady unimodal capillary-gravity waves. Rather than treating each resonance individually, we focus on several representatives that illustrate the underlying mechanisms and demonstrate how the spectral analysis can be carried out rigorously.

\begin{notation*}
Throughout this section, we may write, for instance, $k_j$ and $k_{j'}$ instead of $k_j(\sigma)$ and $k_{j'}(\sigma)$ where $\sigma$ is suppressed for simplicity of notation.
\end{notation*}

\subsection{Resonances of order \texorpdfstring{$1$}{Lg}} \label{sec:N=1}

We assume   
\begin{itemize}
\item[(A1)]$k_j(\sigma)-k_{j'}(\sigma)=\kappa$ for some $\sigma>0$ for some $j$ and $j'$, where $k_j(\sigma)$ and $k_{j'}(\sigma)$ are as in Definition \ref{def:resonance}. 
\end{itemize}
Section \ref{S1S2region} identifies all nonzero resonances of order $1$ for supercritical capillary-gravity waves. Such resonances do not occur for subcritical waves (see Section \ref{S3region}) and Stokes' gravity waves (see, for instance, \cite{HY2023}). 

Recall Definition \ref{def:proj}. After reordering $Y_0(\sigma)$, if necessary, we assume 
\begin{itemize}
\item[(A2)] $\boldsymbol{\phi}_j(\sigma)$ and $\boldsymbol{\phi}_{j'}(\sigma)$ are the first and second basis elements of $Y_0(\sigma)$, where $j$ and $j'$ are in (A1). 
\end{itemize}

Recall Section \ref{sec:expansion_monodromy} and we begin by evaluating $\mathbf{V}^{(m,n)}(T;\sigma)$ for $m+n=0$ and $1$. 
Additionally, we assume  
\begin{itemize}
\item[(A3)] 
$k_j(\sigma)-k_\ell(\sigma)\neq N\kappa$ for any $N\in\mathbb{Z}$ for any $\ell\neq j,j'$ for which $\boldsymbol{\phi}_\ell(\sigma)\in Y_0(\sigma)$. 
\end{itemize}
In other words, $k_j(\sigma)$ (or $k_{j'}(\sigma)$) is not resonant with $k_\ell(\sigma)$. Consequently, it suffices to concentrate on the $2\times 2$ leading principal submatrices of \eqref{def:V(m,n)}. We pause to remark that $\mathbf{V}^{(m,n)}(x;\sigma)$, $m,n=0,1,2,\dots$, are square matrices of order $\dim(Y(\sigma))$, where
\[
\dim(Y(\sigma))=\begin{cases}
6 & \text{if $0<\sigma\leq\min(-\sigma_{c,1},\sigma_{c,2})$,}\\
4 & \text{if $\min(-\sigma_{c,1},\sigma_{c,2})<\sigma\leq\max(-\sigma_{c,1},\sigma_{c,2})$}.
\end{cases}
\] 
See Sections \ref{sec:sigma N=1} and \ref{sec:sigma N=2} when (A3) fails at $\sigma=-\sigma_{c,1}$ or $\sigma_{c,2}$. 

\begin{lemma}\label{lem:V(m,n) N=1}
Assume {\rm (A1)}, {\rm (A2)}, and {\rm (A3)}. Then
\begin{align*}
&\mathbf{V}^{(0,0)}(T;\sigma)=\begin{pmatrix}
e^{ik_jT} & 0 & \cdots  \\
0 & e^{ik_{j'}T} & \cdots \\
\vdots & \vdots & \ddots \end{pmatrix}, \\
&\mathbf{V}^{(1,0)}(T;\sigma)=\begin{pmatrix}
v^{(1,0)}_{jj} & 0 & \cdots  \\
0 & v^{(1,0)}_{j'j'} & \cdots \\
\vdots & \vdots & \ddots \end{pmatrix}, \\
&\mathbf{V}^{(0,1)}(T;\sigma)=\begin{pmatrix}
0 & v^{(0,1)}_{jj'} & \cdots \\
v^{(0,1)}_{j'j} & 0 & \cdots \\
\vdots & \vdots & \ddots \end{pmatrix},
\end{align*}
where $v^{(1,0)}_{jj}, v^{(1,0)}_{j'j'}\neq 0$ are in \eqref{def:V(1,0)sigma}, and $v^{(0,1)}_{jj'}$ and $v^{(0,1)}_{j'j}$ are not necessarily zero. 
\end{lemma}

\begin{proof} 
See \eqref{def:V(0,0)sigma} and \eqref{def:V(1,0)sigma} for $\mathbf{V}^{(0,0)}(T;\sigma)$ and $\mathbf{V}^{(1,0)}(T;\sigma)$, respectively. We solve \eqref{eqn:v_k(0):sigma}-\eqref{def:f(sigma)} and \eqref{cond:V(0)} for $m=0$ and $n=1$, and write the solutions as:
\begin{align*}
&v^{(0,1)}_{jj}=e^{ik_{j}T}\Big\langle\int_0^T\mathbf{B}^{(0,1)}(x;\sigma)\boldsymbol{\phi}_{j}(\sigma)~dx,\boldsymbol{\psi}_{j}(\sigma)\Big\rangle_Y,
\intertext{and $v^{(0,1)}_{j'j'}$ by replacing $j$ with $j'$;}
&v^{(0,1)}_{jj'}=e^{ik_{j}T}\Big\langle\int_0^Te^{i(k_{j'}-k_j)x}\mathbf{B}^{(0,1)}(x;\sigma)\boldsymbol{\phi}_{j'}(\sigma)~dx,\boldsymbol{\psi}_{j}(\sigma)\Big\rangle_Y,
\end{align*}
and $v^{(0,1)}_{j'j}$ by interchanging $j$ and $j'$. Here $\mathbf{B}^{(0,1)}(x;\sigma)$ is in Appendix \ref{app:Bexp}, where $\phi_1$, $\eta_1$, and $\mu_1$ are in \eqref{def:stokes1}, $\boldsymbol{\phi}_j(\sigma)$ and $\boldsymbol{\phi}_{j'}(\sigma)$ are in Lemma \ref{lem:eps=0}, $\boldsymbol{\psi}_j(\sigma)$ and $\boldsymbol{\psi}_{j'}(\sigma)$ in Section \ref{sec:proj}, and $\langle \cdot \,, \cdot \rangle_Y$ in \eqref{def:inner}. 

Upon inspection, $\mathbf{B}^{(0,1)}(x;\sigma)$ depends linearly on $\s(\kappa x)$ and $\c(\kappa x)$ with respect to $x$. Consequently, $v^{(0,1)}_{jj},v^{(0,1)}_{j'j'}=0$. On the other hand, since 
\[
\int_0^T e^{\pm i\kappa x}\s(\kappa x)~dx, \int_0^T e^{\pm i\kappa x}\c(\kappa x)~dx\neq 0,
\]
$v^{(0,1)}_{jj'}$ and $v^{(0,1)}_{j'j}$ are not necessarily zero. Their formulae are excessively long and therefore omitted.
\end{proof}

See \cite[Lemma 6.3]{HY2023} for an analogous calculation for Stokes' gravity waves. The analysis in this subsection does not involve $\mathbf{V}^{(m,n)}(T;\sigma)$ for $m+n\geq 2$. See, for instance, Section \ref{sec:N=2} for resonances of order $2$ where the analysis additionally involves $\mathbf{V}^{(0,2)}(T;\sigma)$.

Recall Sections \ref{sec:reduction} and \ref{sec:expansion_monodromy}, and we proceed to a power series expansion of  $D(i\sigma+\delta,k_j(\sigma)+\gamma,\eps)$ about $(\delta,\gamma,\eps)=(0,0,0)$. 
Substituting the result of Lemma \ref{lem:V(m,n) N=1} into \eqref{def:Delta}, \eqref{def:Vexp}, and \eqref{def:V(m,n)}, we make a straightforward calculation to arrive at
\ba  \label{eqn:D(l,m,n) N=1}
&D(i\sigma+\delta,k_j(\sigma)+\gamma,\eps)\\
&=d^{(2,0,0)}\delta^2+d^{(0,2,0)}\gamma^2+d^{(0,0,2)}\eps^2+d^{(1,1,0)}\delta\gamma+o((|\delta|+|\gamma|+|\eps|)^2)
\ea
as $\delta,\gamma,\eps\to 0$, where 
\ba \label{def:d(l,m,n) N=1}
d^{(2,0,0)}=&v_{jj}^{(1,0)}v_{j'j'}^{(1,0)}\prod_{\ell\neq j,j'}(e^{ik_jT}-e^{ik_\ell T}),\\
d^{(0,2,0)}=&-T^2e^{2ik_jT}\prod_{\ell\neq j,j'}(e^{ik_jT}-e^{ik_\ell T}),\\
d^{(0,0,2)}=&-v_{jj'}^{(0,1)}v_{j'j}^{(0,1)}\prod_{\ell\neq j,j'}(e^{ik_jT}-e^{ik_\ell T}),\\
d^{(1,1,0)}=&-iT(v_{jj}^{(1,0)} + v_{j'j'}^{(1,0)})e^{ik_jT}\prod_{\ell\neq j,j'}(e^{ik_jT}-e^{ik_\ell T}).
\ea 
Lemma \ref{lem:V(m,n) N=1} ensures $d^{(2,0,0)}\neq 0$. The Weierstrass preparation theorem then asserts 
\[
D(i\sigma+\delta,k_{j}(\sigma)+\gamma,\eps)=W(\delta,\gamma,\eps)h(\delta,\gamma,\eps)
\]
for $\delta\in\mathbb{C}$ and $|\delta|\ll1$ for $\gamma,\eps\in\mathbb{R}$ and $|\gamma|,|\eps|\ll1$. Here 
\begin{equation}\label{def:W(l,m,n)}
W(\delta,\gamma,\eps)=\delta^2+a_1(\gamma,\eps)\delta+a_0(\gamma,\eps)
\end{equation}
is a Weierstrass polynomial, where 
\ba\label{def:aj N=1}
a_1(\gamma,\eps)&=\frac{d^{(1,1,0)}}{d^{(2,0,0)}}\gamma+o(|\gamma|+|\eps|), \\
a_0(\gamma,\eps)&=\frac{d^{(0,2,0)}}{d^{(2,0,0)}}\gamma^2+\frac{d^{(0,0,2)}}{d^{(2,0,0)}}\eps^2+o((|\gamma|+|\eps|)^2),
\ea
as $\gamma,\eps\to 0$, $a_1(\gamma,\eps)$ and $a_0(\gamma,\eps)$ are analytic at $(0,0)$, and $a_1(0,0), a_0(0,0)=0$, while $h(\delta,\gamma,\eps)$ is analytic at $(0,0,0)$ and $h(0,0,0)=d^{(2,0,0)}\neq 0$. The roots of \eqref{eqn:D(l,m,n) N=1} coincide with the roots of \eqref{def:W(l,m,n)} for $\gamma,\eps\in\mathbb{R}$ and $|\gamma|, |\eps|\ll1$, denoted by
\be \label{def:delta N=1}
\delta_\pm(\gamma,\eps)=-\frac{a_1(\gamma,\eps)}{2}\pm\frac{\sqrt{a_1(\gamma,\eps)^2-4a_0(\gamma,\eps)}}{2}.
\ee
Here and in what follows, whenever $i\sigma\neq 0\in i\mathbb{R}$, we additionally assume
\begin{equation}\label{A:transversality}
\frac{\partial\delta_+}{\partial\gamma}(0,0)\neq
\frac{\partial\delta_-}{\partial\gamma}(0,0).
\end{equation}
In other words, the two branches of \eqref{def:sigma} intersect at $\sigma$ with different slopes. 

Lemma \ref{lem:Wsymm} implies
\begin{equation}\label{eqn:aj sign N=1}
a_1(\gamma,\eps)\in i\mathbb{R},\quad a_0(\gamma,\eps)\in\mathbb{R} \quad\text{and, hence,}\quad 
a_1(\gamma,\eps)^2-4a_0(\gamma,\eps) \in \mathbb{R}
\end{equation}
for $\gamma,\eps\in\mathbb{R}$ and $|\gamma|, |\eps|\ll1$. We then introduce
\be\label{def:disc1}
\Delta(\gamma,\eps)=a_1(\gamma,\eps)^2-4a_0(\gamma,\eps), 
\ee 
so that \eqref{def:delta N=1} becomes 
\[
\delta_\pm(\gamma,\eps)=-\frac{a_1(\gamma,\eps)}{2}\pm\frac{\sqrt{\Delta(\gamma,\eps)}}{2}, 
\]
where the sign of \eqref{def:disc1} determines spectral stability or instability. 

Since $a_1(\gamma,\eps)$ and $a_0(\gamma,\eps)$ are analytic at $(0,0)$, and $a_1(0,0)$, $a_0(0,0)=0$, consequently, \eqref{def:disc1} is analytic at $(0,0)$ and $\Delta(0,0)=0$. Suppose 
\[
\frac{\partial\Delta}{\partial\gamma}(0,0),\dots,\frac{\partial^{q-1}\Delta}{\partial\gamma^{q-1}}(0,0)=0,\quad \text{and}\quad \frac{\partial^q\Delta}{\partial\gamma^q}(0,0)\neq 0
\]
for some $q\in\mathbb{Z}^+$. The Weierstrass preparation theorem asserts
\be\label{eqn:W(m,n)}
\Delta(\gamma,\eps)=W(\gamma,\eps)h(\gamma,\eps)
\ee 
for $\gamma,\eps\in\mathbb{R}$ and $|\gamma|,|\eps|\ll1$, where 
\be \label{def:W(m,n)}
W(\gamma,\eps)=\gamma^q+b_{q-1}(\eps)\gamma^{q-1}+\dots+b_0(\eps)
\ee
is a Weierstrass polynomial, $b_{q-1}(\eps)$, \dots, $b_0(\eps)$ are real-valued, analytic at $0$, and $b_{q-1}(0)$, \dots, $b_0(0)=0$, while $h(\gamma,\eps)$ is real-valued, analytic at $(0,0)$, and $h(0,0)\neq0$. 

\begin{lemma}[Weierstrass polynomial for $\Delta(\gamma,\eps)$]\label{lem:weierstrass2}
Let \eqref{def:W(m,n)} denote the Weierstrass polynomial associated with \eqref{def:disc1} about $(0,0)$ so that \eqref{eqn:W(m,n)} holds. If \eqref{A:transversality} holds, then
\be\label{def:W(m,n)2} 
W(\gamma,\eps)=\gamma^2+b_1(\eps)\gamma+b_0(\eps)
\ee 
for $\gamma,\eps\in\mathbb{R}$ and $|\gamma|,|\eps|\ll1$, and $h(0,0)<0$.
\end{lemma}

\begin{proof}
It follows from \eqref{def:delta N=1} that
\[
\Delta(\gamma,0)=(\delta_+(\gamma,0)-\delta_-(\gamma,0))^2.
\]
Since $h:=\frac{\partial\delta_+}{\partial\gamma}(0,0)-\frac{\partial\delta_-}{\partial\gamma}(0,0)\neq0$ by \eqref{A:transversality}, consequently,
\[
\Delta(\gamma,0)=h^2\gamma^2+o(\gamma^2)\quad \text{as $\gamma\to0$}.
\]
Since $W(\gamma,0)=\gamma^q$ by \eqref{def:W(m,n)}, it follows from \eqref{eqn:W(m,n)} that $q=2$ and $h(0,0)=h^2$. Since $\delta_\pm(\gamma,0)\in i\mathbb{R}$ for $\gamma\in\mathbb{R}$ and $|\gamma|\ll1$, moreover, $h \neq 0, \in i\mathbb{R}$ and, hence, $h(0,0)=h^2<0$. This completes the proof.
\end{proof}

Recall that the sign of \eqref{def:disc1} determines spectral stability or instability. Lemma \ref{lem:weierstrass2} establishes that the sign of \eqref{def:disc1} is opposite to the sign of \eqref{def:W(m,n)2}. We then introduce 
\be\label{def:disc2}
\Delta_2(\eps)=b_1(\eps)^2-4b_0(\eps).
\ee
Clearly, \eqref{def:disc2} is real-valued, analytic at $0$, and $\Delta_2(0)=0$. If $\Delta_2(\eps)\leq 0$ for $\eps\in\mathbb{R}$ and $|\eps|\ll 1$, then $W(\gamma,\eps)\geq 0$ for $\gamma,\eps\in\mathbb{R}$ and $|\gamma|,|\eps|\ll1$. If $\Delta_2(\eps)>0$ for $\eps\neq 0,\in\mathbb{R}$ and $|\eps|\ll1$, on the other hand, $W(\gamma,\eps)<0$ for $\eps\neq0,\in\mathbb{R}$ and $|\eps|\ll1$ for 
\be \label{def:Gamma}
\gamma\in \varGamma(\eps):=\left(\tfrac{1}{2}(-b_1(\eps)-\sqrt{\Delta_2(\eps)}),\tfrac{1}{2}(-b_1(\eps)+\sqrt{\Delta_2(\eps)})\right) (\neq \emptyset).
\ee 
Therefore, the sign of \eqref{def:disc2} determines spectral stability or instability. 

\begin{remark*}\rm 
Lemma \ref{lem:weierstrass2} is an important new analytical ingredient of our analysis, enabling us to completely determine the sign of \eqref{def:disc1}. It applies to nonzero resonances of other kinds. See Sections \ref{sec:N=2}, \ref{sec:N=3}, \ref{sec:sigma N=1}, \ref{sec:sigma N=2}, \ref{sec:M=2 N=2}, \ref{sec:M=3 N=1,M}, and \ref{sec:M=3 1<N<M}.	
\end{remark*}

We state the main result of this subsection.

\begin{theorem}[Unimodal wave solutions and resonances of order $1$]${}$\label{thm:N=1}

\noindent Let $\beta>0$ and $\kappa>0$. Suppose that a steady unimodal wave solution exists to \eqref{stationary-periodic}-\eqref{eqn:stokes exp} for $\eps\neq 0,\in\mathbb{R}$ and $|\eps|\ll1$, satisfying \eqref{def:stokes1}. Assume {\rm (A1)}, {\rm (A2)}, {\rm (A3)}, and \eqref{A:transversality}. Recall Section \ref{sec:L}, and consider ${\rm spec}(\mathcal{L}(\eps))$ near $i\sigma \neq0, \in i\mathbb{R}$. We define 
\be \label{def:ind_1}
{\rm ind}_1(\beta,\kappa,\sigma;k_j,k_{j'})=\frac{v^{(0,1)}_{jj'}v^{(0,1)}_{j'j}}{v_{jj}^{(1,0)}v_{j'j'}^{(1,0)}},
\ee 
where the right side is as in Lemma \ref{lem:V(m,n) N=1}.

\begin{itemize}
\item[$\bullet$] If ${\rm ind}_1(\beta,\kappa,\sigma;k_j,k_{j'})>0$ then there exists $\lambda \in {\rm spec}(\mathcal{L}(\eps))$ near $i\sigma$ such that ${\rm Re}\,\lambda\neq0$, implying spectral instability.

\item[$\bullet$] If ${\rm ind}_1<0$, on the other hand, ${\rm spec}(\mathcal{L}(\eps))$ near $i\sigma$ remains confined to the imaginary axis. 

\item[$\bullet$] If ${\rm ind}_1=0$ then spectral stability near $i\sigma$ would be determined by higher order terms in \eqref{def:Vexp}.
\end{itemize}
\end{theorem}

Therefore, a steady unimodal capillary-gravity wave of sufficiently small amplitude in water of unit depth is spectrally unstable near a nonzero resonance of order $1$ provided that \eqref{def:ind_1} is positive. We emphasize that \eqref{def:ind_1} is an explicit symbolic expression in $\beta$, $\kappa$, $\sigma$, $k_j(\sigma)$, and $k_{j'}(\sigma)$. Its formula is omitted, however, due to its excessive length. 

\begin{proof}
Substituting \eqref{def:aj N=1} into \eqref{def:disc1}, 
$$
\Delta(\gamma,\eps)=\frac{(d^{(1,1,0)})^2-4d^{(2,0,0)}d^{(0,2,0)}}{(d^{(2,0,0)})^2}\gamma^2-4\frac{d^{(0,0,2)}}{d^{(2,0,0)}}\eps^2+o((|\gamma|+|\eps|)^2)
$$
as $\gamma,\eps\to0$. Lemma \ref{lem:weierstrass2} implies  
\begin{align*}
b_1(\eps)&=O(\eps^2),\\
b_0(\eps)&=-4\frac{d^{(0,0,2)}}{d^{(2,0,0)}}\left(\frac{(d^{(1,1,0)})^2-4d^{(2,0,0)}d^{(0,2,0)}}{(d^{(2,0,0)})^2}\right)^{-1}\eps^2+O(\eps^3)
\end{align*}
as $\eps\to0$, and
\begin{equation}\label{def:h(0,0)}
h(0,0)=\frac{(d^{(1,1,0)})^2-4d^{(2,0,0)}d^{(0,2,0)}}{(d^{(2,0,0)})^2}<0
\end{equation}
Therefore, \eqref{def:disc2} becomes
$$
\Delta_2(\eps)=-4\left(-4\frac{d^{(0,0,2)}}{d^{(2,0,0)}}\frac{1}{h(0,0)}\right)\eps^2+O(\eps^3)
=\frac{-16}{h(0,0)}{\rm ind}_1\eps^2+O(\eps^3)
$$
as $\eps\to0$, where the second equality follows from \eqref{def:d(l,m,n) N=1} and \eqref{def:ind_1}. 

If ${\rm ind}_1>0$ then $\Delta_2(\eps)>0$ for $\eps\neq0,\in\mathbb{R}$ and $|\eps|\ll1$. It then follows from \eqref{eqn:W(m,n)}, \eqref{def:W(m,n)2}, and \eqref{def:disc2} that $W(\gamma,\eps)<0$ and, hence, $\Delta(\gamma,\eps)>0$ for $\eps\neq0,\in\mathbb{R}$ and $|\eps|\ll1$ for $\gamma \in \varGamma(\eps)$, where $\varGamma(\eps)$ is in \eqref{def:Gamma}. Recall \eqref{def:disc1} and, therefore, the real parts of \eqref{def:delta N=1} are nonzero for $\eps\neq0,\in\mathbb{R}$ and $|\eps|\ll1$ for $\gamma\in\mathbb{R}$ and $|\gamma|\ll1$, depending on $\eps$, implying spectral instability.

If ${\rm ind}_1<0$, on the other hand, $\Delta_2(\eps)<0$ and, hence, $W(\gamma,\eps)>0$ for $\eps\in\mathbb{R}$ and $|\eps|\ll1$ for $\gamma\in\mathbb{R}$ and $|\gamma|\ll1$. It then follows from \eqref{def:disc1}, \eqref{eqn:W(m,n)}, \eqref{def:W(m,n)2}, and \eqref{def:disc2} that \eqref{def:delta N=1} remains confined to the imaginary axis. Additionally, it follows from the proof of Theorem \ref{thm:I res} that the Weierstrass polynomial associated with $D(i\sigma+\delta, k_\ell(\sigma)+\gamma,\eps)$ about $(\delta,\gamma,\eps)=(0,0,0)$, where $\ell$ is in (A3), has one purely imaginary root.  

If ${\rm ind}_1=0$ so that \eqref{def:disc2} vanishes to the order of $\eps^2$ as $\eps\to0$, then its sign would be determined at some higher order of $\eps$, necessitating higher order terms in \eqref{def:Vexp}.

This completes the proof.
\end{proof}

\begin{remark*}\rm
It is noteworthy that \eqref{def:disc2} is $O(\eps^2)$ as $\eps\to0$. More generally, Lemma \ref{lem:disc2even} establishes that \eqref{def:disc2} is an even function for $\eps\in\mathbb{R}$ and $|\eps|\ll1$ for steady unimodal capillary-gravity waves and for bimodal Wilton ripples of order $2M+1$, $M\in\mathbb{Z}^+$. 
\end{remark*}

\begin{corollary}[Unstable spectrum near resonance of order $1$]\label{cor:N=1}
Recall Theorem \ref{thm:N=1}. If ${\rm ind}_1>0$ then $\lambda \in {\rm spec}(\mathcal{L}(\eps))$ near $i\sigma\neq 0,\in i\mathbb{R}$ satisfies 
\be \label{def:ellipse N=1}
({\rm Re}\,\lambda)^2+\left(\frac{v^{(1,0)}_{jj}-v^{(1,0)}_{j'j'}}{v^{(1,0)}_{jj}+v^{(1,0)}_{j'j'}}\right)^2({\rm Im}\,\lambda-\sigma)^2={\rm ind}_1\eps^2
\ee
up to order $\eps$ as $\eps\to0$. In other words, the unstable spectrum forms an ellipse to leading order, whose center is at $i\sigma$ and the semiaxes are $O(\eps)$ for $\eps\in\mathbb{R}$ and $|\eps|\ll1$.
\end{corollary}

\begin{proof}
We write $\delta=\delta_r+i\delta_i$, where $\delta_r$, $\delta_i\in\mathbb{R}$. Substituting \eqref{def:aj N=1}, the roots of \eqref{def:W(l,m,n)} satisfy  
\[
(\delta_r+i\delta_i)^2+\frac{d^{(1,1,0)}}{d^{(2,0,0)}}\gamma(\delta_r+i\delta_i)+\frac{d^{(0,2,0)}\gamma^2+d^{(0,0,2)}\eps^2}{d^{(2,0,0)}}=0
\]
to leading order, where \eqref{eqn:aj sign N=1} asserts
$$
\frac{d^{(1,1,0)}}{d^{(2,0,0)}}\in i\mathbb{R}\quad \text{and} \quad \frac{d^{(0,2,0)}}{d^{(2,0,0)}}, \frac{d^{(0,0,2)}}{d^{(2,0,0)}}\in\mathbb{R}.
$$
Taking the imaginary part, we arrive at $\gamma=-\frac{2id^{(2,0,0)}}{d^{(1,1,0)}}\delta_i$ to leading order. Inserting this, recall $\lambda=i\sigma+\delta$ and we arrive at \eqref{def:ellipse N=1} to leading order. This completes the proof.
\end{proof}

\subsection{Resonances of order \texorpdfstring{$2$}{Lg} }\label{sec:N=2}

We turn our attention to resonances of order $2$. For concreteness, we assume   
\begin{itemize}
\item[(A1')] $k_4(\sigma)-k_5(\sigma)=2\kappa$ for some $\sigma>0$, where $k_4(\sigma)$ and $k_5(\sigma)$ are in Lemma \ref{lem:eps=0}. 
\end{itemize}
It follows from Section \ref{S3region} that for subcritical capillary-gravity waves, $k_4(\sigma)-k_5(\sigma)=2\kappa$ for some $\sigma>\sigma_c$, and no other nonzero resonances occur. For supercritical waves, see Section \ref{S1S2region} (R4). 

We pause to remark that $k_4(\sigma)$ and $k_5(\sigma)$ have the same Krein signature. Consequently, such resonance is not expected to lead to spectral instability \cite{sun2025}. See Remark \ref{rem:Krein}. Our numerical result in Section \ref{sec:numerics N=2} bears this out. Here we construct a spectral stability index nonetheless because the same analysis extends without essential modification to other nonzero resonances of order $2$. 

Recall Definition \ref{def:proj} and we assume 
\begin{itemize}
\item[(A2')] $Y_0(\sigma)=\{\boldsymbol{\phi}_4(\sigma), \boldsymbol{\phi}_5(\sigma)\}$, where $\boldsymbol{\phi}_4(\sigma)$ and $\boldsymbol{\phi}_5(\sigma)$ are in Lemma \ref{lem:eps=0}.
\end{itemize}
Clearly, (A2') holds for subcritical capillary-gravity waves for $\sigma>\sigma_c$, and for supercritical waves for $\sigma>\max(-\sigma_{c,1}, \sigma_{c,2})$. 

More generally, we assume $k_j(\sigma)-k_{j'}(\sigma)=2\kappa$ for some $\sigma>0$ for some $j$ and $j'$, and $Y_0(\sigma)=\{ \boldsymbol{\phi}_j(\sigma), \boldsymbol{\phi}_{j'}(\sigma),\dots \},$ 
where $k_{j}(\sigma)-k_\ell(\sigma)\neq N\kappa$ for any $N\in\mathbb{Z}$ for any $\ell\neq j,j'$ for which $\boldsymbol{\phi}_{\ell}(\sigma)\in Y_0(\sigma)$. 

Proceeding along the lines of the argument in Section \ref{sec:N=1}, we solve \eqref{eqn:v_k(0):sigma}-\eqref{def:f(sigma)},  \eqref{cond:V(0)}, and \eqref{eqn:w_k}-\eqref{eqn:w_k bd}, utilizing the results of Appendix \ref{app:Bexp} and Section \ref{sec:unimodal}, and we make an explicit calculation of the solution, to evaluate $\mathbf{V}^{(m,n)}(T;\sigma)$ for $m+n=0,1$ and $2$. 

\begin{lemma}\label{lem:V(m,n) N=2}
Assume {\rm (A1')} and {\rm (A2')}. Then
\begin{align*}
&\mathbf{V}^{(0,0)}(T;\sigma)=e^{ik_4T}\begin{pmatrix}1&0\\0&1\end{pmatrix},
& &\mathbf{V}^{(1,0)}(T;\sigma)=\begin{pmatrix}v_{44}^{(1,0)}&0\\0&v_{55}^{(1,0)}\end{pmatrix},\\
&\mathbf{V}^{(0,1)}(T;\sigma)=\begin{pmatrix}0&0\\0&0\end{pmatrix},
& &\mathbf{V}^{(0,2)}(T;\sigma)=\begin{pmatrix}v_{44}^{(0,2)}&v_{45}^{(0,2)}\\v_{54}^{(0,2)}&v_{55}^{(0,2)}\end{pmatrix},
\end{align*}
where $v_{44}^{(1,0)}, v_{55}^{(1,0)}\neq 0$ are in \eqref{def:V(1,0)sigma}, and $v_{jk}^{(0,2)}\in ie^{ik_4T}\mathbb{R}$, $j,k=4,5$, are not necessarily zero.
\end{lemma}

The proof of Lemma \ref{lem:V(m,n) N=2} is the same as the proofs of \cite[Lemmas 6.3 and 6.4]{HY2023} for Stokes' gravity waves, and is therefore omitted. 

Furthermore, the statements of Lemma \ref{lem:V(m,n) N=2} are the same as those of \cite[Lemmas 6.3 and 6.4]{HY2023}. We then adapt the argument from \cite[Section 6.2]{HY2023} to derive a spectral stability index. Additionally, we apply Lemma \ref{lem:weierstrass2}, exploiting the Weierstrass preparation theorem to the discriminant of the Weierstrass polynomial associated with the periodic Evans function (see \eqref{def:disc1} for resonances of order $1$), establishing a stronger stability conclusion compared to \cite{HY2023}. See also Remark \ref{rem:Stokes}.

We state the main result of this subsection. 

\begin{theorem}[Unimodal wave solutions and resonances of order $2$]${}$\label{thm:N=2}

\noindent Let $\beta>0$ and $\kappa>0$. Suppose that a steady unimodal wave solution exists to \eqref{stationary-periodic}-\eqref{eqn:stokes exp} for $\eps\neq 0,\in\mathbb{R}$ and $|\eps|\ll1$, satisfying \eqref{def:stokes1}. Assume {\rm (A1')}, {\rm (A2')}, and \eqref{A:transversality}. Recall Section \ref{sec:L}, and consider ${\rm spec}(\mathcal{L}(\eps))$ near $i\sigma \neq0,\in i\mathbb{R}$. We define 
\be \label{def:ind_2}
{\rm ind}_2(\beta,\kappa,\sigma;k_4,k_5)=\frac{v_{45}^{(0,2)}v_{54}^{(0,2)}}{v_{44}^{(1,0)}v_{55}^{(1,0)}},
\ee 
where the right side is as in Lemma \ref{lem:V(m,n) N=2}.

\begin{itemize}
\item[$\bullet$] If ${\rm ind}_2(\beta,\kappa,\sigma;k_4,k_5)>0$ then there exists $\lambda \in {\rm spec}(\mathcal{L}(\eps))$ near $i\sigma$ such that ${\rm Re}\,\lambda\neq0$. 

\item[$\bullet$] If ${\rm ind}_2<0$, on the other hand, ${\rm spec}(\mathcal{L}(\eps))\subset i\mathbb{R}$ near $i\sigma$.

\item[$\bullet$] If ${\rm ind}_2=0$ then spectral stability near $i\sigma$ would be determined by higher order terms in \eqref{def:Vexp}. 
\end{itemize}
\end{theorem}

Therefore, a steady unimodal capillary-gravity wave of sufficiently small amplitude in water of unit depth is spectrally unstable near $i\sigma\neq 0,\in i\mathbb{R}$, for which $k_j(\sigma)-k_{j'}(\sigma)=2\kappa$ for some $j$ and $j'$, provided 
\begin{equation}\label{def:ind_2'}
{\rm ind}_2(\beta,\kappa,\sigma;k_j,k_{j'}):=\frac{v_{jj'}^{(0,2)}v_{j'j}^{(0,2)}}{v_{jj}^{(1,0)}v_{j'j'}^{(1,0)}}>0,
\end{equation}
where $v_{jj}^{(1,0)}$, $v_{j'j'}^{(1,0)}$, $v_{jj'}^{(0,2)}$, and $v_{j'j}^{(0,2)}$ can be defined in the same manner as in Lemma \ref{lem:V(m,n) N=2}. See \cite[Section 6.1]{HY2023} for further details.  
\begin{proof}
We make the same kind of calculation as that in \cite[Section 6.2]{HY2023} to arrive at 
\begin{align*}
D(i\sigma+\delta,k_4(\sigma)+\gamma,\eps)
=&d^{(2,0,0)}\delta^2+d^{(0,2,0)}\gamma^2+d^{(1,1,0)}\delta\gamma\\
&+d^{(1,0,2)}\delta\eps^2+d^{(0,1,2)}\gamma\eps^2+d^{(0,0,4)}\eps^4+\cdots
\end{align*}
as $\delta,\gamma,\eps\to0$, where 
\begin{align*}
&d^{(2,0,0)}=v^{(1,0)}_{44}v^{(1,0)}_{55}, & &d^{(0,2,0)}=-T^2e^{2ik_4T},\\
&d^{(1,1,0)}=-iTe^{ik_4T}(v^{(1,0)}_{44}+v^{(1,0)}_{55}),
& &d^{(1,0,2)}=v^{(0,2)}_{44}v^{(1,0)}_{55}+v^{(0,2)}_{55}v^{(1,0)}_{44},\\ 
&d^{(0,1,2)}=-iTe^{ik_4T}(v^{(0,2)}_{44}+v^{(0,2)}_{55}),
& &d^{(0,0,4)}=-v^{(0,2)}_{45}v^{(0,2)}_{54}.
\end{align*}
See also \eqref{eqn:D(l,m,n) N=1} and \eqref{def:d(l,m,n) N=1} for resonances of order $1$. The Weierstrass polynomial associated with $D(i\sigma+\delta,k_4(\sigma)+\gamma,\eps)$ about $(\delta,\gamma,\eps)=(0,0,0)$ takes the form \eqref{def:W(l,m,n)}, and \eqref{def:disc1} becomes
\begin{align*}
\Delta(\gamma,\eps)=&\frac{(d^{(1,1,0)})^2-4d^{(2,0,0)}d^{(0,2,0)}}{(d^{(2,0,0)})^2}\gamma^2+\frac{(d^{(1,0,2)})^2-4d^{(2,0,0)}d^{(0,0,4)}}{(d^{(2,0,0)})^2}\eps^4\\
&+\frac{2d^{(1,1,0)}d^{(1,0,2)}-4d^{(2,0,0)}d^{(0,1,2)}}{(d^{(2,0,0)})^2}\gamma\eps^2+\cdots
\end{align*}
as $\gamma,\eps\to0$. It then follows from Lemma \ref{lem:weierstrass2} that the Weierstrass polynomial associated with $\Delta(\gamma,\eps)$ about $(0,0)$ takes the form \eqref{def:W(m,n)2}, where
\begin{align*}
b_1(\eps)=&\frac{1}{h(0,0)}\frac{2d^{(1,1,0)}d^{(1,0,2)}-4d^{(2,0,0)}d^{(0,1,2)}}{(d^{(2,0,0)})^2}\eps^2+O(\eps^3),\\
b_0(\eps)=&\frac{1}{h(0,0)}\frac{(d^{(1,0,2)})^2-4d^{(2,0,0)}d^{(0,0,4)}}{(d^{(2,0,0)})^2}\eps^4+O(\eps^5)
\end{align*}
as $\eps\to 0$, and $h(0,0)$ is in \eqref{def:h(0,0)}. Consequently, \eqref{def:disc2} becomes
\begin{equation}\label{eqn:disc2 N=2}
\Delta_2(\eps)=\frac{-16 }{h(0,0)}{\rm ind}_2\eps^4+O(\eps^5)
\end{equation}
as $\eps \to 0$, where ${\rm ind}_2$ is in \eqref{def:ind_2} and $h(0,0)<0$. 

We follow the lines of the argument in the proof of Theorem \ref{thm:N=1}, determining spectral stability or instability by the sign of \eqref{def:ind_2}. This completes the proof.
\end{proof}

It is noteworthy that $\Delta_2(\eps)=O(\eps^4)$ as $\eps\to0$. 

\begin{corollary}[Unstable spectrum near resonance of order $2$]\label{cor:N=2}
Recall Theorem \ref{thm:N=2}. If ${\rm ind}_2>0$ then $\lambda \in {\rm spec}(\mathcal{L}(\eps))$ near $i\sigma\neq 0,\in i\mathbb{R}$ satisfies 
\[ 
({\rm Re}\,\lambda)^2+\left(\frac{v^{(1,0)}_{44}-v^{(1,0)}_{55}}{v^{(1,0)}_{44}+v^{(1,0)}_{55}}\right)^2\left({\rm Im}\,\lambda-\sigma-\eps^2\frac{v^{(0,2)}_{44} -v^{(0,2)}_{55}}{v^{(1,0)}_{44} - v^{(1,0)}_{55}}\right)^2={\rm ind}_2\eps^4
\]
up to order $\eps^2$ as $\eps\to0$. In other words, the unstable spectrum forms an ellipse to leading order, whose center is displaced from $i\sigma$ by $O(\eps^2)$ along the imaginary axis for $\eps\in\mathbb{R}$ and $|\eps|\ll1$.
\end{corollary}

The proof of Corollary \ref{cor:N=2} mirrors the proof of Corollary \ref{cor:N=1}, and is therefore omitted. 

\begin{remark*}\rm 
Comparing Corollaries \ref{cor:N=1} and \ref{cor:N=2}, we observe that the unstable spectrum forms an ellipse centered at $i\sigma$ for resonances of order $1$, while the center is shifted along the imaginary axis by a distance of $O(\eps^2)$ for resonances of order $2$. See also \cite{creedon_deconinck_trichtchenko_2022} for Stokes' gravity waves.
\end{remark*}

\begin{remark}[Remarks on Stokes' gravity waves]\label{rem:Stokes}
Considering Stokes waves, which exhibit the same kind of resonance as in (A1'), recall from \cite[Section 6]{HY2023} that the spectral stability index takes the same expression as in \eqref{def:ind_2} or \eqref{def:ind_2'}. The main result of \cite[Section 6]{HY2023} states that: if this index $\leq 0$, then the spectrum near the resonance remains confined to the imaginary axis up to order $\eps^2$ as $\eps \to 0$. Consequently, determining spectral stability or instability requires an analysis of higher order terms in the power series expansion of the associated periodic Evans function. 

Our developments in Sections \ref{sec:N=1} and \ref{sec:N=2}, and particularly, Lemmas \ref{lem:Wsymm} and \ref{lem:weierstrass2}, apply equally well to Stokes waves, to strengthen the stability conclusion in \cite{HY2023}. More precisely, one can show that spectral stability near the resonance is determined by the sign of \eqref{def:disc1} and, hence, the sign of \eqref{def:disc2}.

We state our improvement of \cite[Theorem 6.5]{HY2023}.

\begin{theorem*}
A $2\pi/\kappa$-periodic Stokes wave of sufficiently small amplitude in water of unit depth is spectrally unstable near $i\sigma \neq 0,\in i\mathbb{R}$, for which $k_j(\sigma)-k_{j'}(\sigma)=2\kappa$ for some $j$ and $j'$, provided 
$$
{\rm ind}_2(\kappa):=\frac{a_{12}^{(0,2)}a_{21}^{(0,2)}}{a_{11}^{(1,0)}a_{22}^{(1,0)}}(T)>0,
$$ 
where $a_{11}^{(1,0)}(T)$, $a_{22}^{(1,0)}(T)$, $a_{12}^{(0,2)}(T)$, and $a_{21}^{(0,2)}(T)$ are in \cite[(6.9) and (6.13)]{HY2023} and $T=2\pi/\kappa$. The spectrum near $i\sigma$ remains confined to the imaginary axis if ${\rm ind}_2(\kappa)<0$. 

The spectrum remains along $i\mathbb{R}$ up to order $\eps^2$ as $\eps\to0$ near $i\sigma\in i\mathbb{R}$ for which $k_j(\sigma)-k_{j'}(\sigma)=N\kappa$ for some $N\geq3,\in\mathbb{Z}^+$ for some $j$ and $j'$.
\end{theorem*}

Additionally, Corollary \ref{cor:N=2} extends without essential modification to Stokes waves, describing the unstable spectrum near the resonance of order $2$ up to order $\eps^2$ as $\eps\to 0$. The proof utilizes the results in \cite[pages 38-39]{HY2023}. 
\end{remark}

\subsection{Higher-order resonances}\label{sec:N=3} 

Moving to resonances of order $\geq 3$, for concreteness, we assume: 
\begin{itemize}
\item[(A1'')] $k_4(\sigma)-k_5(\sigma)=N\kappa$ for some $\sigma>0$ for some $N\geq 3,\in\mathbb{Z}^+$, where $k_4(\sigma)$ and $k_5(\sigma)$ are in Lemma \ref{lem:eps=0}. 
\item[(A2'')] $Y_0(\sigma)=\{\boldsymbol{\phi}_4(\sigma), \boldsymbol{\phi}_5(\sigma)\}$, where $\boldsymbol{\phi}_4(\sigma)$ and $\boldsymbol{\phi}_5(\sigma)$ are in Lemma \ref{lem:eps=0}.  
\end{itemize}
Section \ref{S3region} shows that (A1'') and (A2'') hold for subcritical capillary-gravity waves for some $\sigma>\sigma_c$, depending on $N$. See Section \ref{S1S2region} (R4) for supercritical waves. Incidentally, $k_4(\sigma)$ and $k_5(\sigma)$ have the same Krein signature, and thus such resonance would not give rise to spectral instability \cite{sun2025}. See Remark \ref{rem:Krein}. 

Considering Stokes' gravity waves, which exhibit the same kind of resonances, the authors' previous work \cite{HY2023} established a preliminary result towards spectral stability and instability near nonzero resonances. Additionally, \cite{creedon_deconinck_trichtchenko_2022} made a formal asymptotic expansion of the spectrum to third order in the perturbation parameters, including the Floquet exponent.    
The analysis in this subsection extends without essential modification to Stokes waves, strengthening the conclusion in \cite{HY2023} (see Remark \ref{rem:Stokes}) while rigorously justifying the formal asymptotic expansion in \cite{creedon_deconinck_trichtchenko_2022}, particularly, those with respect to the Floquet exponent. 

\begin{lemma}\label{lem:V(m,n) N=3}
Assume {\rm (A1'')} and {\rm (A2'')}. Then
\begin{align*}
&\mathbf{V}^{(0,0)}(T;\sigma)=e^{ik_4T}\begin{pmatrix}1&0\\0&1\end{pmatrix},\quad 
& &\mathbf{V}^{(1,0)}(T;\sigma)=\begin{pmatrix}v_{44}^{(1,0)}&0\\0&v_{55}^{(1,0)}\end{pmatrix}, \\
&\mathbf{V}^{(0,1)}(T;\sigma)=\begin{pmatrix}0&0\\0&0\end{pmatrix},\quad
& &\mathbf{V}^{(1,1)}(T;\sigma), \mathbf{V}^{(0,2)}(T;\sigma)=\begin{pmatrix}*&0\\0&*\end{pmatrix},
\end{align*}
where $v_{44}^{(1,0)}, v_{55}^{(1,0)}\neq 0$ are in \eqref{def:V(1,0)sigma}, and the entries marked with an asterisk are not necessarily zero.  
\end{lemma}

The proof of Lemma \ref{lem:V(m,n) N=3} is the same as the proofs of \cite[Lemmas 6.3 and 6.4]{HY2023}, and is therefore omitted. One main difference between Lemmas \ref{lem:V(m,n) N=2} and \ref{lem:V(m,n) N=3} is that the off-diagonal entries of $\mathbf{V}^{(0,2)}(T;\sigma)$ are not necessarily zero for resonances of order $2$ while zero for resonances of order $\geq 3$. 

Moreover, the statements of Lemma \ref{lem:V(m,n) N=3} are the same as those of \cite[Lemmas 6.3 and 6.4]{HY2023}. Accordingly, we proceed along the lines of the argument in \cite[Section 6.2]{HY2023}. See also Section \ref{sec:N=2}. More specifically, we make a power series expansion of $D(i\sigma+\delta, k_4(\sigma)+\gamma,\eps)$ about $(\delta,\gamma,\eps)=(0,0,0)$, and utilize the result of Lemma \ref{lem:V(m,n) N=3} to compute the discriminants of the Weierstrass polynomials (see \eqref{def:disc1} and \eqref{def:disc2} for resonances of orders $1$ and $2$). Ultimately, \eqref{def:disc2} becomes
\begin{equation}\label{eqn:ind2 N=3}
\Delta_2(\eps)=\frac{-16}{h(0,0)}\frac{v_{45}^{(0,2)}v_{54}^{(0,2)}}{v_{44}^{(1,0)}v_{55}^{(1,0)}}\eps^4+O(\eps^5)
\end{equation}
for $\eps\in\mathbb{R}$ and $|\eps|\ll1$, where $h(0,0)<0$. See \eqref{def:ind_2} and \eqref{eqn:disc2 N=2}. We omit the details. 

Lemma \ref{lem:V(m,n) N=3} establishes $v^{(0,2)}_{45}, v^{(0,2)}_{54}=0$, whence \eqref{eqn:ind2 N=3} vanishes at order $\eps^4$ as $\eps \to 0$. 
We can show that \eqref{eqn:ind2 N=3} also vanishes at order $\eps^5$. More generally, $\Delta_2(\eps)$ is even for $\eps\in\mathbb{R}$ and $|\eps|\ll1$.

\begin{lemma}[Evenness of $\Delta_2(\eps)$]\label{lem:disc2even}
Considering steady unimodal capillary-\protect\linebreak gravity waves or bimodal Wilton ripples of order $2M+1$, where $M\in\mathbb{Z}^+$, of sufficiently small amplitude, \eqref{def:disc2} is a real-valued, analytic, and even function for $\eps\in\mathbb{R}$ and $|\eps|\ll1$.
\end{lemma}

\begin{proof}
Recall Section \ref{sec:Stokes} and we fix a solution to \eqref{stationary-periodic}-\eqref{eqn:stokes exp} for $\eps\in\mathbb{R}$ and $|\eps|\ll1$, for which: if $\mu_0(\kappa_1;\beta)=\mu_0(\kappa_2;\beta)$ for $\kappa_1<\kappa_2$, then 
\begin{itemize}
\item[$\bullet$] either $\kappa_2/\kappa_1\notin \mathbb{Q}$, corresponding to a steady unimodal capillary-gravity wave, satisfying \eqref{def:stokes1}, 
\item[$\bullet$] or $\kappa_2=(2M+1)\kappa_1$, $M\in \mathbb{Z}^+$, corresponding to bimodal Wilton ripples of order $2M+1$, satisfying \eqref{def:wilton1 M=3} instead.
\end{itemize}
An inspection reveals that $\eps\mapsto-\eps$ is the same as $x\mapsto x+\pi/\kappa$, that is, the spatial translation by half a period. Since ${\rm spec}(\mathcal{L}(\eps))$ remains invariant under such a spatial translation, consequently, $\Delta_2(\eps)=\Delta_2(-\eps)$. This completes the proof.
\end{proof}

We pause to remark that $\eps\mapsto -\eps$ is not the same as $x \mapsto x+\pi/\kappa$ for Wilton ripples of order $2M$, $M\in\mathbb{Z}^+$. Indeed, ${\rm spec}(\mathcal{L}(\eps))$ is not the same as ${\rm spec}(\mathcal{L}(-\eps))$. See Corollary \ref{cor:M=2 N=2} for further discussion. 

\begin{corollary}[Analyticity of the endpoints of $\varGamma(\eps)$]\label{cor:analyticity}
Considering steady unimodal capillary-gravity waves or bimodal Wilton ripples of order $2M+1$, where $M\in\mathbb{Z}^+$, of sufficiently small amplitude, the two boundary branches of \eqref{def:Gamma} are analytic on each one-sided interval on which $\varGamma(\eps)\neq \emptyset$.
\end{corollary}

\begin{proof}
Lemma \ref{lem:disc2even} implies that either \eqref{def:disc2} vanishes identically near $\eps=0$, or 
\[
\Delta_2(\eps)=\eps^{2n}g(\eps) \quad \text{for some $n\in\mathbb{Z}^+$} \quad \text{for some $g$},
\]
where $g$ real-valued, even, and analytic such that $g(0)\neq0$. When the former holds, the two boundary branches of \eqref{def:Gamma} coincide and are analytic. For the latter, if $\varGamma(\eps)\neq \emptyset$ for $|\eps|\neq0,\ll1$, then $g(0)>0$ and, hence, 
\[
\sqrt{\Delta_2(\eps)}=|\eps|^n\sqrt{g(\eps)}
\]
is analytic on each one-sided interval. Consequently, both boundary branches of \eqref{def:Gamma} are analytic on each such interval. This completes the proof. 
\end{proof}

Corollary \ref{cor:analyticity} applies equally to Stokes' gravity waves, thereby rigorously justifying the formal asymptotic expansion in \cite[(5.1c)]{creedon_deconinck_trichtchenko_2022} with respect to the Floquet exponent.

Returning to the spectral stability and instability of steady unimodal capillary-gravity waves of sufficiently small amplitude near resonances of order $\geq 3$, we assume (A1'') and (A2''). Suppose $N=3$. We proceed as in Sections \ref{sec:N=1} and \ref{sec:N=2}, and make an explicit calculation of \eqref{def:W(l,m,n)}, \eqref{def:disc1}, \eqref{eqn:W(m,n)}, \eqref{def:W(m,n)2}, and \eqref{def:disc2}, to arrive at
$$
\Delta_2(\eps)=\frac{-16}{h(0,0)}\frac{v_{45}^{(0,3)}v_{54}^{(0,3)}}{v_{44}^{(1,0)}v_{55}^{(1,0)}}\eps^6+O(\eps^8)
$$
as $\eps\to 0$, where $h(0,0)<0$. We omit the details. An argument analogous to that in the proof of Lemma \ref{lem:V(m,n) N=2} shows that $v_{45}^{(0,3)}$ and $v_{54}^{(0,3)}$ are not necessarily zero. Consequently, the signs of $v_{45}^{(0,3)}v_{54}^{(0,3)}$ and $v_{44}^{(1,0)}v_{55}^{(1,0)}$ determine spectral stability or instability.

More generally, let $N\geq 3,\mathbb{Z}^+$. We can show that $v_{45}^{(0,N')}, v_{54}^{(0,N')}=0$ when $0\leq N'\leq N-1$ while $v_{45}^{(0,N)}$ and $v_{54}^{(0,N)}$ are not necessarily zero. This leads us to conjecture that \eqref{def:disc2} takes the form 
\be\label{conjecture_disc2}
\Delta_2(\eps)=\frac{-16}{h(0,0)}\frac{v_{45}^{(0,N)}v_{54}^{(0,N)}}{v_{44}^{(1,0)}v_{55}^{(1,0)}}\eps^{2N}+O(\eps^{2N+2})
\ee
as $\eps\to 0$, where $h(0,0)<0$. Consequently, the signs of $v_{45}^{(0,N)}v_{54}^{(0,N)}$ and $v_{44}^{(1,0)}v_{55}^{(1,0)}$ would determine spectral stability or instability.

\subsection{Numerical studies}\label{sec:numerics N=2} 

We numerically investigate the results of Sections \ref{sec:N=1} and  \ref{sec:N=2}. 

\subsubsection*{Numerical results for subcritical waves}

Recall Section \ref{sec:unimodal}, and suppose $(\beta,\kappa)\in S_3$, the region corresponding to subcritical capillary-gravity waves, where $0<\mu_0(\kappa;\beta)<1$. See the right panel of Figure \ref{figure3}. Recall Section \ref{S3region} and, clearly, Section \ref{sec:N=2} (A1') holds. That is, $k_4(\sigma)-k_5(\sigma)=2\kappa$ for some $\sigma>\sigma_c$. We numerically evaluate \eqref{def:ind_2} to obtain 
\[
{\rm ind}_2(\beta,\kappa,\sigma;k_4,k_5)\leq 0
\]
with equality along a distinguished curve. Figure \ref{figure14} depicts the solid blue curve in $S_3$ along which 
\[
{\rm ind}_2(\beta,\kappa,\sigma;k_4,k_5)=0.
\]
Throughout this subsection, the numerical computations are carried out for $\beta\in(0,1)$ and $\kappa\in(0,8)$. 

Theorem \ref{thm:N=2} then predicts that ${\rm spec}(\mathcal{L}(\eps))$ near $i\sigma\neq 0,\in i\mathbb{R}$ remains confined to the imaginary axis for $\eps\in\mathbb{R}$ and $|\eps|\ll1$, implying spectral stability near $i\sigma$, except along the solid blue curve in Figure \ref{figure14} where ${\rm ind}_2(\beta,\kappa,\sigma;k_4,k_5)=0$. Along this curve, the leading-order stability criterion becomes inconclusive, and higher order terms in \eqref{def:Vexp} are required to determine spectral stability or instability. We emphasize that the zeros of \eqref{def:ind_2} are not transversal. 

\begin{figure}[htbp]
    \centering
    \includegraphics[scale=0.3]{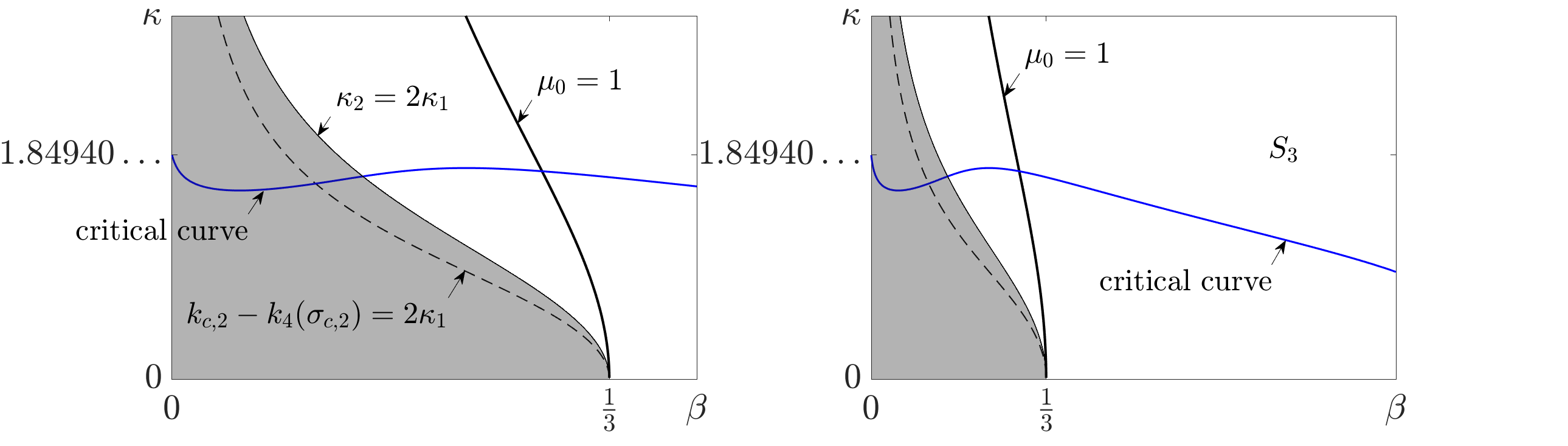}
    \caption{$\kappa$ versus $\beta$ for spectral stability and instability near nonzero resonances of order $2$. The numerical computations indicate ${\rm ind}_2(\beta,\kappa,\sigma;k_4,k_j)>0$ throughout the shaded region, implying spectral instability, while ${\rm ind}_2<0$ in the unshaded region, except along the solid blue curve where ${\rm ind}_2=0$. \protect\\
    \protect\hspace*{11pt} Numerical evidence supports ${\rm ind}_2(\beta,\kappa,\sigma;k_2,k_4) >0$ in the resonance region (R12) of order $2$, bounded from above by the dashed curve $k_{c,2}-k_4(\sigma_{c,2})=2\kappa_1$, except along the solid blue curve where ${\rm ind}_2=0$. Likewise, ${\rm ind}_2(\beta,\kappa,\sigma;k_6,k_4)>0$ in the resonance region (R9) for $N=2$, bounded from below by the dashed curve $k_{c,2}-k_4(\sigma_{c,2})=2\kappa_1$ and from above by the solid curve $\kappa_2=2\kappa_1$, corresponding to Wilton ripples of order $2$, except along the solid blue curve where ${\rm ind}_2=0$. On the other hand, ${\rm ind}_2(\beta,\kappa,\sigma;k_4,k_5)<0$ throughout $S_3$ and in the resonance region (R4) for $N=2$, bounded from below by the curve $\kappa_2=2\kappa_1$, except along the solid blue curve where ${\rm ind}_2=0$. Numerical evaluation indicates that the solid blue curve intersects the $\kappa$-axis at $\kappa=1.849404083750\dots$. }
    \label{figure14}
\end{figure}

\subsubsection*{Numerical results for supercritical waves}

Suppose, on the other hand, $(\beta,\kappa)\in S_1\bigcup S_2$, corresponding to supercritical waves, where $\mu_0(\kappa;\beta)>1$. See the right panel of Figure \ref{figure3}. Section \ref{S1S2region} (R1)-(R15) identify all resonance regions where 
\[
k_j(\sigma)-k_{j'}(\sigma)=N\kappa\quad\text{for some $N\in\mathbb{Z}^+$}\quad\text{for the corresponding $j,j'$}
\]
for some $\sigma>0$, while $k_j(\sigma)-k_\ell(\sigma)\neq N'\kappa$ for each $N'\in\mathbb{Z}$ for each $\ell\neq j,j'$. Here we restrict our attention to $N=1$ and $2$. 

Suppose that Section \ref{sec:N=1} (A1) holds, so that $i\sigma\neq0,\in i\mathbb{R}$ is resonant of order $1$. We numerically evaluate \eqref{def:ind_1} and apply Theorem \ref{thm:N=1} to numerically determine spectral stability or instability. For nonzero resonances of order $2$ (see, for instance, Section \ref{sec:N=2} (A1')), we numerically evaluate \eqref{def:ind_2'} and apply Theorem \ref{thm:N=2} to determine spectral stability or instability. 

Our numerical computations indicate that the zeros of \eqref{def:ind_1}, \eqref{def:ind_2}, and \eqref{def:ind_2'} are not transversal. That means,
\[
{\rm ind}(\beta,\kappa,\sigma;k_j,k_{j'})\text{$\geq 0$ or $\leq 0$} \quad \text{within each resonance region},
\]
with equality along some curves or at some points. Here we do not discuss these zero sets in detail for simplicity of exposition, except for the curves where 
\[
{\rm ind}_2(\beta,\kappa,\sigma;k_j,k_{j'})=0,\quad\text{$j$ or $j'$ is $4$},
\]
which connect continuously to the curve in $S_3$ where ${\rm ind}_2(\beta,\kappa,\sigma;k_4,k_5)=0$. See Figure \ref{figure14}. Throughout the discussion below, ${\rm ind}>0$ and ${\rm ind}<0$ are therefore understood to hold away from these curves or points where ${\rm ind}=0$ unless otherwise specified. 

We summarize our numerical findings. 

\subsubsection*{\rm (R1) $k_1(\sigma)-k_5(\sigma)=N\kappa$ for $N=1,2$} 

Section \ref{S1S2region} shows that the resonance region (R1) of order $1$ is  bounded from above by the curve where $-k_{c,1}-k_5(-\sigma_{c,1})=\kappa$ and from below by the curve $\kappa_2=2\kappa_1$, corresponding to Wilton ripples of order $2$. See the right panel of Figure \ref{figure7}. 
Our numerical evaluation of \eqref{def:ind_1} suggests 
\[
{\rm ind}_1(\beta,\kappa,\sigma;k_1,k_5)<0 \quad \text{for $(\beta,\kappa)$ in the resonance region}
\]
for $\sigma$ corresponding to the resonance. Theorem \ref{thm:N=1} therefore predicts that ${\rm spec}(\mathcal{L}(\eps))$ near $i\sigma$ remains confined to the imaginary axis for $\eps\in\mathbb{R}$ and $|\eps|\ll1$, that is, spectral stability near $i\sigma$.

Moreover, the resonance region (R1) of order $2$ is bounded from above by the curve $-k_{c,1}-k_5(-\sigma_{c,1})=2\kappa$ and from below by the curve $\kappa_2=3\kappa_1$, corresponding to Wilton ripples of order $3$. See the right panel of Figure \ref{figure7}.
Our numerical computations indicate 
\[
{\rm ind}_2(\beta,\kappa,\sigma;k_1,k_5)<0 \quad\text{for $(\beta,\kappa)$ in the resonance region}
\]
for $\sigma$ corresponding to the resonance. Theorem \ref{thm:N=2} then implies spectral stability near $i\sigma$.   

\subsubsection*{\rm (R2) $k_6(\sigma)-k_2(\sigma)=N\kappa$ for $N=1,2$}

Section \ref{S1S2region} shows that the resonance region (R2) for $N=1$ is bounded from above by the curve corresponding to Wilton ripples of order $2$. See the right panel of Figure \ref{figure7}. Our numerical result corroborates ${\rm ind}_1(\beta,\kappa,\sigma;k_6,k_2)<0$ for $(\beta,\kappa)$ in the resonance region for $\sigma$ corresponding to the resonance. Theorem \ref{thm:N=1} then implies spectral stability near $i\sigma$. 

The resonance region (R2) for $N=2$ is bounded from above by the curve for Wilton ripples of order $3$. See the right panel of Figure \ref{figure7}. Numerical evidence supports ${\rm ind}_2(\beta,\kappa,\sigma;k_6,k_2)<0$ for $(\beta$, $\kappa)$ in the resonance region for $\sigma$ corresponding to the resonance. Theorem \ref{thm:N=2} implies spectral stability near $i\sigma$.   
    
\subsubsection*{\rm (R4) $k_4(\sigma)-k_5(\sigma)=2\kappa$} 

Section \ref{S1S2region} shows that the resonance region (R4) for $N=2$ is bounded from above by the Wilton ripple curve of order $2$. See the right panel of Figure \ref{figure7}. Our numerical result confirms ${\rm ind}_2(\beta,\kappa,\sigma;k_4,k_5)<0$ for $(\beta,\kappa)$ in the resonance region for $\sigma$ corresponding to the resonance, 
\[
\text{except along a curve where ${\rm ind}_2(\beta,\kappa,\sigma;k_4,k_5)=0$}.
\]
See Figure \ref{figure14}. Theorem \ref{thm:N=2} then implies spectral stability near $i\sigma$ unless ${\rm ind}_2=0$ along the solid blue curve in Figure \ref{figure14}, where spectral stability would be determined by higher order terms in \eqref{def:Vexp}.

We pause to remark that ${\rm ind}_2(\beta,\kappa,\sigma;k_4,k_5)$ is continuous across the curve where $\mu_0(\kappa;\beta)=1$, separating $S_2$ and $S_3$. This facilitates our numerical investigation. 

\subsubsection*{\rm (R6) $k_6(\sigma)-k_1(\sigma)=N\kappa$ for $N=1,2$} 

Section \ref{S1S2region} discusses the resonance regions (R6). See Figure \ref{figure8} for $N=1,2$. The numerical result confirms ${\rm ind}_1(\beta,\kappa,\sigma;k_6,k_1)>0$ for $(\beta,\kappa)$ in the resonance region (R6) for $N=1$, within $S_2$, for $\sigma$ corresponding to the resonance. Theorem \ref{thm:N=1} therefore predicts that there exists $\lambda \in {\rm spec}(\mathcal{L}(\eps))$ near $i\sigma\neq0,\in i\mathbb{R}$ such that ${\rm Re}\,\lambda \neq 0$ for $\eps\neq 0,\in\mathbb{R}$ and $|\eps|\ll1$, that is, spectral instability near $i\sigma$.

Numerical evidence supports ${\rm ind}_2(\beta,\kappa,\sigma;k_6,k_1)>0$ for $(\beta,\kappa)$ in the resonance region (R6) for $N=2$, in $S_1$, for $\sigma$ corresponding to the resonance. Theorem \ref{thm:N=2} implies spectral instability near $i\sigma$.
    
\subsubsection*{\rm (R7) $k_2(\sigma)-k_3(\sigma)=N\kappa$ for $N=1,2$} 

The numerical result confirms ${\rm ind}_1(\beta,\kappa,\sigma;k_2,k_3)>0$ for $(\beta,\kappa)$ in the resonance region (R7) for $N=1$ (see Figure \ref{figure8}) for $\sigma$ corresponding to the resonance. Theorem \ref{thm:N=1} implies spectral instability near $i\sigma$. Numerical evidence supports ${\rm ind}_2(\beta,\kappa,\sigma;k_2,k_3)>0$ for $(\beta,\kappa)$ in the region (R7) for $N=2$ (see Figure \ref{figure8}) for $\sigma$ corresponding to the resonance. Theorem \ref{thm:N=2} implies spectral instability near $i\sigma$.
    
\subsubsection*{\rm (R9) $k_6(\sigma)-k_4(\sigma)=2\kappa$} 

The numerical result confirms ${\rm ind}_2(\beta,\kappa,\sigma;k_6,k_4)>0$ for $(\beta,\kappa)$ in the region (R9) for $N=2$ (see Figure \ref{figure10}) for $\sigma$ corresponding to the resonance, except along the curve where ${\rm ind}_2=0$. See Figure \ref{figure14}. Theorem \ref{thm:N=2} then implies spectral instability near $i\sigma$ unless ${\rm ind}_2=0$ along the solid blue curve in Figure \ref{figure14}.
    
\subsubsection*{\rm (R12) $k_2(\sigma)-k_4(\sigma)=2\kappa$} 

Section \ref{S1S2region} shows that the region (R12) for $N=2$ is bounded from above by the curve $k_{c,2}-k_4(\sigma_{c,2})=2\kappa_1$. See the left panel of Figure \ref{figure10}. The numerical computations reveal ${\rm ind}_2(\beta,\kappa,\sigma;k_2,k_4)>0$ for $(\beta,\kappa)$ in the resonance region for $\sigma$ corresponding to the resonance, except along the curve where ${\rm ind}_2=0$. See Figure \ref{figure14}. Theorem \ref{thm:N=2} implies spectral instability near $i\sigma$ unless ${\rm ind}_2(k_2,k_4)=0$ along the solid blue curve in Figure \ref{figure14}.
Numerical evaluation reveals 
\[
{\rm ind}_2(0,\kappa,\sigma;k_2,k_4)=0 \quad \text{at $\kappa=1.849404083750\dots$}.
\]
This agrees with the spectral instability result for Stokes' gravity waves near the resonance of order $2$. See, for instance, \cite{HY2023} for further details. 
    
\subsubsection*{\rm (R13) $k_3(\sigma)-k_5(\sigma)=N\kappa$ and $N=1,2$} 

Section \ref{S1S2region} discusses the resonance regions (R13). See Figure \ref{figure11} for $N=1,2$. The numerical result reveals ${\rm ind}_1(\beta,\kappa,\sigma;k_3,k_5)<0$ for $(\beta,\kappa)$ in the resonance region (R13) for $N=1$, for $\sigma$ corresponding to the resonance. Theorem \ref{thm:N=1} implies spectral stability near $i\sigma$. Numerical evidence supports ${\rm ind}_2(\beta,\kappa,\sigma;k_3,k_5)<0$ for $(\beta,\kappa)$ in the resonance region for $N=2$, for $\sigma$ corresponding to the resonance. Theorem \ref{thm:N=2} implies spectral stability near $i\sigma$.

\subsubsection*{\rm (R14) $k_2(\sigma)-k_1(\sigma)=N\kappa$ and $N=1,2$} 

The numerical result confirms ${\rm ind}_1(\beta,\kappa,\sigma;k_2,k_1)>0$ for $(\beta,\kappa)$ in the resonance region (R14) for $N=1$ (see Figure \ref{figure12}) for $\sigma$ corresponding to the resonance. Theorem \ref{thm:N=1} implies spectral instability near $i\sigma$. Moreover, ${\rm ind}_2(\beta,\kappa,\sigma;k_2,k_1)>0$ for $(\beta,\kappa)$ in the resonance region for $N=2$ (see Figure \ref{figure12}) for $\sigma$ corresponding to the resonance. Theorem \ref{thm:N=2} implies spectral instability near $i\sigma$. 

\subsubsection*{\rm (R15) $k_6(\sigma)-k_3(\sigma)=N\kappa$ and $N=1,2$}

The numerical result confirms ${\rm ind}_1(\beta,\kappa,\sigma;k_6,k_3)>0$ for $(\beta,\kappa)$ in the resonance region (R15) for $N=1$ (see Figure \ref{figure13}) for $\sigma$ corresponding to the resonance. Theorem \ref{thm:N=1} implies spectral instability near $i\sigma$. Moreover, ${\rm ind}_2(\beta,\kappa,\sigma;k_6,k_3)>0$ for $(\beta,\kappa)$ in the resonance region for $N=2$ (see Figure \ref{figure13}) for $\sigma$ corresponding to the resonance, and Theorem \ref{thm:N=2} implies spectral instability near $i\sigma$. 

\begin{remark*}\rm 
Recall Definition \ref{def:resonance}. Our numerical computations indicate:
\begin{itemize}
\item[$\bullet$] ${\rm ind}_N(\beta,\kappa,\sigma;k_j,k_{j'})\leq 0$, $N=1,2$, whenever $k_j(\sigma)$ and $k_{j'}(\sigma)$ have the same Krein signature.
\item[$\bullet$] ${\rm ind}_N(\beta,\kappa,\sigma;k_j,k_{j'})\geq 0$, $N=1,2$, whenever $k_j(\sigma)$ and $k_{j'}(\sigma)$ have opposite Krein signatures.
\end{itemize}
Particularly, the computations do not detect spectral instability for the resonances (R1), (R2), (R4), and (R13). This aligns with the prediction based on Krein signature. See Section \ref{resonance_summary} and Remark \ref{rem:Krein}. 

On the other hand, ${\rm ind}\geq 0$ predicts spectral instability unless ${\rm ind}=0$ along some curves or at some points. This is consistent with other stability criteria (see, for instance, \cite{creedon_deconinck_trichtchenko_2022,Noble_2023,berti2025,JRSY25}) for which spectral instability is determined provided an appropriate index function does not vanish.     
\end{remark*}

\subsection{Repeated resonances of order \texorpdfstring{$1$}{Lg}}\label{sec:sigma N=1}

Recall Section \ref{sec:unimodal}, and suppose $\mu_0(\kappa;\beta)>1$, corresponding to supercritical capillary-gravity waves. 
Lemma \ref{lem:eps=0} establishes that $\mathbf{L}(-i\sigma_{c,1})$ has a repeated eigenvalue $-ik_{c,1}=ik_j(-\sigma_{c,1})$, $j=1,3$, and the corresponding generalized eigenfunctions satisfy
\[
(\mathbf{L}(-i\sigma_{c,1})+ik_{c,1})\boldsymbol{\phi}_{1}(-\sigma_{c,1})=\boldsymbol{0},\quad
(\mathbf{L}(-i\sigma_{c,1})+ik_{c,1})\boldsymbol{\phi}_{3}(-\sigma_{c,1})=\boldsymbol{\phi}_{1}(-\sigma_{c,1}).
\]
Similarly, $\mathbf{L}(i\sigma_{c,2})$ has a repeated eigenvalue $ik_{c,2}=ik_j(\sigma_{c,2})$, $j=2,6$, with the generalized eigenfunctions satisfying
\[
(\mathbf{L}(i\sigma_{c,2})-ik_{c,2})\boldsymbol{\phi}_{6}(\sigma_{c,2})=\boldsymbol{0},\quad(\mathbf{L}(i\sigma_{c,2})-ik_{c,2})\boldsymbol{\phi}_{2}(\sigma_{c,2})=\boldsymbol{\phi}_{6}(\sigma_{c,2}).
\]
Proceeding along the lines of the argument in Section \ref{sec:proj}, we show that $\mathbf{L}^\dag(-i\sigma_{c,1})$ has a repeated eigenvalue $ik_{c,1}$, and $\boldsymbol{\psi}_1(-\sigma_{c,1})$ (see \eqref{def:adj}) and $\boldsymbol{\psi}_3(-\sigma_{c,1})$ (see \eqref{def:adj0'}) are the corresponding generalized eigenfunctions. Similarly, $\mathbf{L}^\dag(i\sigma_{c,2})$ has a repeated eigenvalue $-ik_{c,2}$, with the generalized eigenfunctions $\boldsymbol{\psi}_6(\sigma_{c,2})$ (see \eqref{def:adj}) and $\boldsymbol{\psi}_2(\sigma_{c,2})$ (see \eqref{def:adj0'}). The formulae for $\boldsymbol{\psi}_1(-\sigma_{c,1})$,  $\boldsymbol{\psi}_3(-\sigma_{c,1})$, $\boldsymbol{\psi}_2(\sigma_{c,2})$, and $\boldsymbol{\psi}_6(\sigma_{c,2})$ are omitted. Recall Definition \ref{def:proj}. Utilizing these generalized eigenfunctions, we define $\boldsymbol{\Pi}(-\sigma_{c,1})$ and $\boldsymbol{\Pi}(\sigma_{c,2})$ in the same manner as in \eqref{def:Pi0} and \eqref{def:Pi;high}. 

Section \ref{S1S2region} (R16)-(R19) demonstrate that $-i\sigma_{c,1}$ and $i\sigma_{c,2}$ are repeated resonances. In this subsection and Section \ref{sec:sigma N=2}, we address spectral stability and instability near these resonances. We assume
\begin{itemize}
\item[(B1)] $(-1)^p k_{c,p}-k_{j'}((-1)^p\sigma_{c,p})=N\kappa$ for $N=1$ or $-1$ for $p=1$ or $2$ and for some $j'$, where $k_{j'}((-1)^p\sigma_{c,p})$ is as in Definition \ref{def:resonance}. 
\end{itemize}
See Section \ref{S1S2region} (R16)-(R18) for further details. We emphasize that such repeated resonance occurs along a curve in the $(\beta,\kappa)$-plane, so that $\beta$ is determined by $\kappa$. 

\begin{notation*}
Throughout the remainder of this subsection, we write 
\[
\text{$\sigma_c$ instead of $(-1)^p\sigma_{c,p}$ and $k_c$ instead of $(-1)^p k_{c,p}$, $p=1$ or $2$,}
\]
for simplicity of notation unless otherwise specified. 
\end{notation*}

After reordering $Y_0(\sigma_{c})$, in Definition \ref{def:proj}, we assume 
\begin{itemize}
\item[(B2)] $Y_0(\sigma_{c})=\{ \boldsymbol{\phi}_{j_1}(\sigma_{c}),\boldsymbol{\phi}_{j_2}(\sigma_{c}),\boldsymbol{\phi}_{j'}(\sigma_{c}),\dots \}$,
where
\[
(\mathbf{L}(i\sigma_{c})-ik_{c})\boldsymbol{\phi}_{j_1}(\sigma_{c})=\boldsymbol{0}\quad\text{and}\quad
(\mathbf{L}(i\sigma_{c})-ik_{c})\boldsymbol{\phi}_{j_2}(\sigma_{c})=\boldsymbol{\phi}_{j_1}(\sigma_c),
\]
and $j'$ is in (B1).
\end{itemize}
For example, for the resonance (R16) of order $1$, where $-k_{c,1}-k_5(-\sigma_{c,1})=\kappa$, 
\[
Y_0(-\sigma_{c,1})=\{\boldsymbol{\phi}_1(-\sigma_{c,1}),\boldsymbol{\phi}_3(-\sigma_{c,1}),\boldsymbol{\phi}_5(-\sigma_{c,1}),\boldsymbol{\phi}_4(-\sigma_{c,1}) \}.
\]

We begin by evaluating $\mathbf{V}^{(m,n)}(T;\sigma_{c})$, in \eqref{def:Vexp}, for $m+n=0$ and $1$, where $T=2\pi/\kappa$ is the period. We assume $k_j(\sigma_c)-k_\ell(\sigma_c)\neq N\kappa$ for any $N\in\mathbb{Z}$ for any $\ell\neq j$, where $j=j_1, j_2$ or $j'$, so that it suffices to concentrate on the $3\times 3$ leading principal submatrices of \eqref{def:V(m,n)}. We pause to remark that $\mathbf{V}^{(m,n)}(T;\sigma_c)$, $m,n=0,1,2,\dots$, are $4\times 4$ or $6\times 6$ matrices.

\begin{notation*}
We relabel the entries of \eqref{def:V(m,n)} and write
\[
\mathbf{V}^{(m,n)}(x;\sigma_{c})=(w^{(m,n)}_{jk}(x;\sigma_{c})),
\quad \text{$j,k=1,2,\dots,6$ or $1,2,\dots,4$},
\]
depending on $\sigma_c$. For example, for the resonance (R16) of order $1$,
\[
(v^{(m,n)}_{jk}(x;-\sigma_{c,1}))_{j,k=1,3,5,4}=(w^{(m,n)}_{jk}(x;-\sigma_{c,1}))_{j,k=1,2,3,4}.
\]
\end{notation*}

\begin{lemma}\label{lem:V(m,n)c N=1}
Assume {\rm (B1)} and {\rm (B2)}. Then
\begin{align*}
&\mathbf{V}^{(0,0)}(T;\sigma_c)=e^{i k_cT}\begin{pmatrix}
1 & T & 0 & \cdots \\
0 & 1 & 0 & \cdots \\
0 & 0 & 1 & \cdots \\
\vdots & \vdots & \vdots & \ddots \end{pmatrix}, \\
&\mathbf{V}^{(1,0)}(T;\sigma_c)=\begin{pmatrix}
w^{(1,0)}_{11}&w^{(1,0)}_{12}&0 &\cdots\\
w^{(1,0)}_{21}&w^{(1,0)}_{22}&0 &\cdots \\
0&w^{(1,0)}_{32}&w^{(1,0)}_{33} &\cdots \\
\vdots & \vdots & \vdots & \ddots
\end{pmatrix}, \\
&\mathbf{V}^{(0,1)}(T;\sigma_c)=\begin{pmatrix}
0&w^{(0,1)}_{12}&w^{(0,1)}_{13} &\cdots\\
0&w^{(0,1)}_{22}&w^{(0,1)}_{23} &\cdots\\
w^{(0,1)}_{31}&w^{(0,1)}_{32}&0 &\cdots \\
\vdots & \vdots & \vdots & \ddots 
\end{pmatrix},
\end{align*}
where $w^{(1,0)}_{11}, w^{(1,0)}_{12}, w^{(1,0)}_{21}, w^{(1,0)}_{22}$, $w^{(1,0)}_{32}, w^{(1,0)}_{33}$ and $w^{(0,1)}_{12}, w^{(0,1)}_{13}, w^{(0,1)}_{22}, w^{(0,1)}_{23}$, $w^{(0,1)}_{31}, w^{(0,1)}_{32}$ are not necessarily zero. 
Furthermore, 
\[
w_{21}^{(1,0)},w_{22}^{(0,1)},w_{31}^{(0,1)}\in ie^{i k_cT}\mathbb{R}\quad\text{and}\quad 
w_{33}^{(1,0)},w_{23}^{(0,1)}\in e^{i k_cT}\mathbb{R}.
\]
\end{lemma}

\begin{proof}
We follow the lines of the argument in the proofs of Lemmas \ref{lem:V(m,n) N=1} and \ref{lem:V(m,n) N=2}. See also the proof of \cite[Lemma 6.3]{HY2023} for an analogous calculation for Stokes' gravity waves.

One main difference from Lemmas \ref{lem:V(m,n) N=1}, \ref{lem:V(m,n) N=2}, and \ref{lem:V(m,n) N=3} is that $\mathbf{V}^{(0,0)}(T;\sigma_c)$ contains a Jordan block because of the repeated eigenvalue. Indeed, a direct calculation reveals that 
\[
\mathbf{v}_{j_2}(x;\sigma_c)=xe^{ik_cx}\boldsymbol{\phi}_{j_1}(\sigma_c)+e^{ik_cx}\boldsymbol{\phi}_{j_2}(\sigma_c)
\]
solves the analogue of \eqref{eqn:v_k(0)}-\eqref{def:f(0)} and \eqref{cond:V(0)}, where $k=j_2$ and $m,n=0$, after replacing $k_j(0)$ with $k_c$ and $\sigma=0$ with $\sigma_c$. Consequently, the second column of $\mathbf{V}^{(0,0)}(x;\sigma_c)$ becomes $\begin{pmatrix}xe^{ik_cx}\\e^{ik_cx}\\0 \\ \vdots \end{pmatrix}$, rather than the columns in \eqref{def:V(0,0)sigma}. See also \eqref{def:V(0,0)0}.

Moreover, the second columns of $\mathbf{V}^{(1,0)}(x;\sigma_c)$ and $\mathbf{V}^{(0,1)}(x;\sigma_c)$ are not necessarily zero because they involve $xe^{ik_cx}\boldsymbol{\phi}_{j_1}(\sigma_c)$. The formulae for $\mathbf{V}^{(1,0)}(x;\sigma_c)$ and $\mathbf{V}^{(0,1)}(x;\sigma_c)$ are excessively long and therefore omitted. 
\end{proof}

We proceed to a power series expansion of $D(i\sigma_c+\delta,k_c+\gamma,\eps)$, in \eqref{def:Delta}, about $(\delta,\gamma,\eps)=(0,0,0)$. Substituting the result of Lemma \ref{lem:V(m,n)c N=1} into \eqref{def:Vexp} and \eqref{def:V(m,n)}, we make a straightforward calculation to arrive at
\begin{align*}
&D(i\sigma_c+\delta,k_c+\gamma,\eps)\\
&=d^{(2,0,0)}\delta^2+d^{(0,0,2)}\eps^2+d^{(1,1,0)}\delta\gamma+o((|\delta|+|\gamma|+|\eps|)^2)
\end{align*}
as $\delta, \gamma, \eps\to 0$, where 
\ba \label{def:d(l,m,n)c N=1}
d^{(2,0,0)}&=-Tw_{21}^{(1,0)}w_{33}^{(1,0)}\prod_{\ell\neq j_1,j_2,j'}(e^{iT(k_c + k_{\ell})} - e^{2iTk_c}),\\
d^{(0,0,2)}&=Tw_{23}^{(0,1)}w_{31}^{(0,1)}\prod_{\ell\neq j_1,j_2,j'}(e^{iT(k_c + k_{\ell})} - e^{2iTk_c}),\\
d^{(1,1,0)}&=iT^2w_{21}^{(1,0)}e^{iTk_c}\prod_{\ell\neq j_1,j_2,j'}(e^{iT(k_c + k_{\ell})} - e^{2iTk_c}).
\ea 
The Weierstrass preparation theorem then asserts 
\[
D(i\sigma_c+\delta,k_c+\gamma,\eps)=W(\delta,\gamma,\eps)h(\delta,\gamma,\eps)
\]
for $\delta\in\mathbb{C}$, $\gamma,\eps\in\mathbb{R}$ and  $|\delta|,|\gamma|,|\eps|\ll1$, where 
\ba \label{def:W(l,m,n)c N=1}
W(\delta,\gamma,\eps)=\delta^2+\left(\frac{d^{(1,1,0)}}{d^{(2,0,0)}}\gamma+o(|\gamma|+|\eps|)\right)\delta
+\frac{d^{(0,0,2)}}{d^{(2,0,0)}}\eps^2+o((|\gamma|+|\eps|)^2)
\ea
as $\gamma, \eps\to 0$, while $h(\delta,\gamma,\eps)$ is analytic at $(0,0,0)$ and $h(0,0,0)\neq0$. We adapt the argument from Sections \ref{sec:N=1} and \ref{sec:N=2} to derive $\Delta(\gamma,\eps)$ (see \eqref{def:disc1} for $\sigma\neq \sigma_c$), the discriminant of \eqref{def:W(l,m,n)c N=1}, and apply Lemma \ref{lem:weierstrass2} to the Weierstrass polynomial associated with $\Delta(\gamma,\eps)$ about $(\gamma,\eps)=(0,0)$, to ultimately derive $\Delta_2(\eps)$  (see \eqref{def:disc2}  for $\sigma\neq \sigma_c$), its discriminant, whose sign determines spectral stability or instability. 

We state the main result of this subsection. 

\begin{theorem}[Repeated resonances of order $1$]\label{thm:sigma N=1}

Let $\beta>0$ and $\kappa>0$. Suppose that a steady unimodal wave solution exists to \eqref{stationary-periodic}-\eqref{eqn:stokes exp} for $\eps\neq 0,\in\mathbb{R}$ and $|\eps|\ll1$, satisfying \eqref{def:stokes1}. Assume {\rm (B1)}, {\rm (B2)}, and \eqref{A:transversality}. Recall Section \ref{sec:L}, and consider ${\rm spec}(\mathcal{L}(\eps))$ near $i(-1)^p\sigma_{c,p} \in i\mathbb{R}$, $p=1$ or $2$. We define 
\be \label{def:ind_3}
{\rm ind}_3(\kappa,(-1)^p\sigma_{c,p};k_{j'})=\frac{w^{(0,1)}_{23}w^{(0,1)}_{31}}{w_{21}^{(1,0)}w_{33}^{(1,0)}},
\ee 
where the right side is as in Lemma \ref{lem:V(m,n)c N=1}. We pause to remark that $\beta$ is determined by $\kappa$. 

\begin{itemize}
\item[$\bullet$] If ${\rm ind}_3(\kappa,(-1)^p\sigma_{c,p};k_{j'})>0$ then there exists $\lambda \in {\rm spec}(\mathcal{L}(\eps))$ near $i(-1)^p\sigma_{c,p}$ such that ${\rm Re}\,\lambda\neq0$.

\item[$\bullet$] If ${\rm ind}_3<0$, on the other hand, ${\rm spec}(\mathcal{L}(\eps))$ near $i(-1)^p\sigma_{c,p}$ remains confined to the imaginary axis.

\item[$\bullet$] If ${\rm ind}_3=0$ then spectral stability near $i(-1)^p\sigma_{c,p}$ would be determined by higher order terms in \eqref{def:Vexp}. 
\end{itemize}
\end{theorem}

Therefore, a steady unimodal capillary-gravity wave of sufficiently small amplitude in water of unit depth is spectrally unstable near a repeated resonance of order $1$ provided that \eqref{def:ind_3} is positive. We emphasize that \eqref{def:ind_3} is an explicit symbolic expression in $\kappa$ (and, hence, $\beta(\kappa)$), $\sigma$, and $k_{j'}(\sigma)$, whose formula is omitted due to its excessive length.

The proof of Theorem \ref{thm:sigma N=1} mirrors the proofs of Theorems \ref{thm:N=1} and \ref{thm:N=2}, and is therefore omitted.

\begin{corollary}[Unstable spectrum near repeated resonance of order $1$]${}$\label{cor:sigma N=1}

\noindent Recall Theorem \ref{thm:sigma N=1}. If ${\rm ind}_3>0$ then $\lambda \in {\rm spec}(\mathcal{L}(\eps))$ near $i(-1)^p\sigma_{c,p}\in i\mathbb{R}$, $p=1$ or $2$, satisfies
\be \label{circle:eps}
({\rm Re}\,\lambda)^2+({\rm Im}\,\lambda-(-1)^p \sigma_{c,p})^2={\rm ind}_3\eps^2
\ee
to order $\eps$ as $\eps\to0$. In other words, the unstable spectrum forms a circle to leading order, whose center is at $i(-1)^p\sigma_{c,p}$ and the radius is $O(\eps)$ for $\eps\in\mathbb{R}$ and $|\eps|\ll1$. 
\end{corollary}

\begin{proof}
We write $\delta=\delta_r+i\delta_i$, where $\delta_r,\delta_i\in\mathbb{R}$. The roots of \eqref{def:W(l,m,n)c N=1} satisfy 
\[
(\delta_r+i\delta_i)^2+\frac{d^{(1,1,0)}}{d^{(2,0,0)}}\gamma(\delta_r+i\delta_i)+\frac{d^{(0,0,2)}}{d^{(2,0,0)}}\eps^2=0
\]
to leading order, where \eqref{eqn:aj sign N=1} implies
$$
\frac{d^{(1,1,0)}}{d^{(2,0,0)}}\in i\mathbb{R}  \quad \text{and}\quad \frac{d^{(0,0,2)}}{d^{(2,0,0)}}\in\mathbb{R}.
$$
Taking the imaginary part, we arrive at $\gamma=-\frac{2id^{(2,0,0)}}{d^{(1,1,0)}}\delta_i$ 
to leading order. Inserting this, let $\lambda= i(-1)^p\sigma_{c,p}+\delta$, and we arrive at \eqref{circle:eps} to leading order. This completes the proof.
\end{proof}

\subsection{Repeated resonances of order \texorpdfstring{$2$}{Lg}}\label{sec:sigma N=2}

We turn to repeated resonances of order $2$. For concreteness, we assume
\begin{itemize}
\item[(B1')] $k_{c,2}-k_4(\sigma_{c,2})=2\kappa$, where $k_{c,2}$ and $k_4(\sigma_{c,2})$ are in Lemma \ref{lem:eps=0}. 
\end{itemize}
See Section \ref{S1S2region} (R19) for further details. We emphasize that such resonance occurs along a curve in the $(\beta,\kappa)$-plane (see the left panel of Figure \ref{figure10}), so that $\beta$ is determined by $\kappa$. The analysis of this subsection extends without essential modification to other repeated resonances of order $2$ at $-i\sigma_{c,1}$ or $i\sigma_{c,2}$. See Section \ref{sec:numerics sigma} for examples. 

Lemma \ref{lem:eps=0} establishes that $\mathbf{L}(i\sigma_{c,2})$ has a repeated eigenvalue $ik_{c,2}=ik_j(\sigma_{c,2})$, $j=6,2$, and $\boldsymbol{\phi}_6(\sigma_{c,2})$ (see \eqref{def:eigen}) and $\boldsymbol{\phi}_2(\sigma_{c,2})$ (see \eqref{def:eigen'}) are the corresponding generalized eigenfunctions, satisfying 
\[
(\mathbf{L}(i\sigma_{c,2})-ik_{c,2})\boldsymbol{\phi}_6(\sigma_{c,2})=\boldsymbol{0}
\quad\text{and}\quad
(\mathbf{L}(i\sigma_{c,2})-ik_{c,2})\boldsymbol{\phi}_2(\sigma_{c,2})=\boldsymbol{\phi}_6(\sigma_{c,2}).
\]
We select the ordered basis of $Y(\sigma_{c,2})$, in Definition \ref{def:proj}, as
\begin{itemize}
\item[(B2')] $Y_0(\sigma_{c,2})=\{\boldsymbol{\phi}_6(\sigma_{c,2}),\boldsymbol{\phi}_2(\sigma_{c,2}),\boldsymbol{\phi}_4(\sigma_{c,2}),\boldsymbol{\phi}_5(\sigma_{c,2})\}$.
\end{itemize}

Proceeding as in Section \ref{sec:sigma N=1}, we evaluate $\mathbf{V}^{(m,n)}(T;\sigma_{c,2})$, in \eqref{def:Vexp}, for $m+n=0$ and $1$, where $T=2\pi/\kappa$ is the period. It suffices to concentrate on the $3\times 3$ leading principal submatrices of \eqref{def:V(m,n)}.

\begin{notation*}
We relabel the entries of \eqref{def:V(m,n)} and write
\[
(v^{(m,n)}_{jk}(x;\sigma_{c,2}))_{j,k=6,2,4,5}=(w^{(m,n)}_{jk}(x;\sigma_{c,2}))_{j,k=1,2,3,4}.
\]
\end{notation*}

\begin{lemma}\label{lem:V(m,n)c N=2}
Assume {\rm (B1')} and {\rm (B2')}. Then
\begin{align*}
&\mathbf{V}^{(0,0)}(T;\sigma_{c,2})=e^{ik_{c,2}T}\begin{pmatrix}1&T&0& *\\0&1&0& *\\0&0&1 & * \\ *&*&*&*\end{pmatrix}, \\
&\mathbf{V}^{(1,0)}(T;\sigma_{c,2})=\begin{pmatrix}w^{(1,0)}_{11}&w^{(1,0)}_{12}&0&*\\w^{(1,0)}_{21}&w^{(1,0)}_{22}&0&*\\0&w^{(1,0)}_{32}&w^{(1,0)}_{33}&* \\ *&*&*&*\end{pmatrix},\\
&\mathbf{V}^{(0,1)}(T;\sigma_{c,2})=\begin{pmatrix}0&w^{(0,1)}_{12}&0&*\\0&w^{(0,1)}_{22}&0&*\\0&w^{(0,1)}_{32}&0&*\\ *&*&*&*\end{pmatrix},
\end{align*}
where $w^{(1,0)}_{11}, w^{(1,0)}_{12}, w^{(1,0)}_{21}, w^{(1,0)}_{22}, w^{(1,0)}_{32}, w^{(1,0)}_{33}$ and $w^{(0,1)}_{12}, w^{(0,1)}_{22}, w^{(0,1)}_{32}$ are not necessarily zero. Our analysis does not involve the entries marked with an asterisk.
\end{lemma}

\begin{proof}
The proof mirrors that of Lemma \ref{lem:V(m,n)c N=1}. We pause to remark that the first and third columns of $\mathbf{V}^{(0,1)}(x;\sigma_{c,2})$ take the form $\begin{pmatrix} 0 \\ 0 \\ 0 \\ *\end{pmatrix}$ because the resonance in (B1') is of higher order than the resonance in (B1). The formulae for $\mathbf{V}^{(1,0)}(T;\sigma_{c,2})$ and $\mathbf{V}^{(0,1)}(T;\sigma_{c,2})$ are excessively long and therefore omitted.
\end{proof}

Following the lines of the argument in Section \ref{sec:sigma N=1}, we substitute the result of Lemma \ref{lem:V(m,n)c N=2} into \eqref{def:Delta}-\eqref{def:V(m,n)}, and make a straightforward calculation to compute the leading order terms in a power series expansion of $D(i\sigma_{c,2}+\delta,k_{c,2}+\gamma,\eps)$ about $(\delta,\gamma,\eps)=(0,0,0)$. Below we state the result.

\begin{lemma}\label{lem:D(sigma_c2)} 
Assume {\rm (B1')} and {\rm (B2')}. Then 
\begin{align*}
&D(i\sigma_{c,2}+\delta,k_{c,2}+\gamma,\eps)\\
&=d^{(2,0,0)}\delta^2+d^{(1,1,0)}\delta\gamma+d^{(1,0,2)}\delta\eps^2+d^{(0,1,2)}\gamma\eps^2+d^{(0,0,4)}\eps^4\\
&\quad+o((|\delta|+|\gamma|+|\eps^2|)^2)
\end{align*}
as $\delta,\gamma,\eps\to 0$, where 
\begin{align*} 
d^{(2,0,0)}=&-Tw_{21}^{(1,0)}w_{33}^{(1,0)}(e^{iT(k_{c,2}+k_5)} - e^{2iTk_{c,2}}),\\
d^{(1,1,0)}=&iT^2w_{21}^{(1,0)}e^{iTk_{c,2}}(e^{iT(k_{c,2}+k_5)} - e^{2iTk_{c,2}}),\\
d^{(1,0,2)}=&Te^{iTk_{c,2}}(w_{21}^{(1,0)}w_{34}^{(0,1)}w_{43}^{(0,1)} + w_{24}^{(0,1)}w_{33}^{(1,0)}w_{41}^{(0,1)}\\
&\qquad\qquad - w_{21}^{(0,2)}w_{33}^{(1,0)}e^{iTk_5} - w_{21}^{(1,0)}w_{33}^{(0,2)}e^{iTk_5} \\
&\qquad\qquad+ w_{21}^{(0,2)}w_{33}^{(1,0)}e^{iTk_{c,2}} + w_{21}^{(1,0)}w_{33}^{(0,2)}e^{iTk_{c,2}}),\\
d^{(0,1,2)}=&-iT^2e^{2iTk_{c,2}}(w_{21}^{(0,2)}e^{iTk_{c,2}} - w_{21}^{(0,2)}e^{iTk_5} + w_{24}^{(0,1)}w_{41}^{(0,1)}),\\
d^{(0,0,4)}=&Te^{iTk_{c,2}}(w_{21}^{(0,2)}w_{34}^{(0,1)}w_{43}^{(0,1)} - w_{23}^{(0,2)}w_{34}^{(0,1)}w_{41}^{(0,1)} \\
&\qquad\qquad- w_{24}^{(0,1)}w_{31}^{(0,2)}w_{43}^{(0,1)} + w_{24}^{(0,1)}w_{33}^{(0,2)}w_{41}^{(0,1)} \\
&\qquad\qquad- w_{21}^{(0,2)}w_{33}^{(0,2)}e^{iTk_5}  + w_{23}^{(0,2)}w_{31}^{(0,2)}e^{iTk_5}\\
&\qquad\qquad+ w_{21}^{(0,2)}w_{33}^{(0,2)}e^{iTk_{c,2}} - w_{23}^{(0,2)}w_{31}^{(0,2)}e^{iTk_{c,2}}),
\end{align*}
and the right sides are in Lemma \ref{lem:V(m,n)c N=2}.
\end{lemma}

We then adapt the argument from Section \ref{sec:sigma N=1} (see also Sections \ref{sec:N=1} and \ref{sec:N=2}), applying the Weierstrass preparation theorem (see, particularly, Lemma \ref{lem:weierstrass2}), and make an explicit calculation to derive $\Delta_2(\eps)$ (see \eqref{def:disc2} for $\sigma\neq -\sigma_{c,1},\sigma_{c,2}$) whose sign determines spectral stability or instability.

We state the main result of this subsection.

\begin{theorem}[Repeated resonances of order $2$]\label{thm:sigma N=2}

Let $\beta>0$ and $\kappa>0$. Suppose that a steady unimodal wave solution exists to \eqref{stationary-periodic}-\eqref{eqn:stokes exp} for $\eps\neq 0,\in\mathbb{R}$ and $|\eps|\ll1$, satisfying \eqref{def:stokes1}. Assume {\rm (B1')}, {\rm (B2')}, and \eqref{A:transversality}. Recall Section \ref{sec:L}, and consider ${\rm spec}(\mathcal{L}(\eps))$ near $i\sigma_{c,2} \in i\mathbb{R}$. We define
\ba \label{def:ind_4}
&{\rm ind}_4(\kappa,\sigma_{c,2};k_4)\\&=-\frac{d^{(2,0,0)}(d^{(0,1,2)})^2 - d^{(1,0,2)}d^{(0,1,2)}d^{(1,1,0)} + d^{(0,0,4)}(d^{(1,1,0)})^2}{(d^{(1,1,0)})^2d^{(2,0,0)}},
\ea
where the right side is as in Lemmas \ref{lem:V(m,n)c N=2} and  \ref{lem:D(sigma_c2)}. We pause to remark that $\beta$ is determined by $\kappa$ so that $k_{c,2}-k_4(\sigma_{c,2})=2\kappa$. 

\begin{itemize}
\item[$\bullet$] If ${\rm ind}_4(\kappa,\sigma_{c,2};k_4)>0$ then there exists $\lambda \in {\rm spec}(\mathcal{L}(\eps))$ near $i\sigma_{c,2}$ such that ${\rm Re}\,\lambda\neq 0$.

\item[$\bullet$] If ${\rm ind}_4<0$, on the other hand, ${\rm spec}(\mathcal{L}(\eps))\subset i\mathbb{R}$ near $i\sigma_{c,2}$.

\item[$\bullet$] If ${\rm ind}_4=0$ then spectral stability near $i\sigma_{c,2}$ would be determined by higher order terms in \eqref{def:Vexp}. 
\end{itemize}
\end{theorem}

Therefore, a steady unimodal capillary-gravity wave of sufficiently small amplitude in water of unit depth is spectrally unstable near a repeated resonance of order $2$ provided that \eqref{def:ind_4} is positive, where the right side of \eqref{def:ind_4} can be defined in the same manner as in Lemmas \ref{lem:V(m,n)c N=2} and  \ref{lem:D(sigma_c2)}.
We remark that \eqref{def:ind_4} is an explicit symbolic expression in $\kappa$ (and, hence, $\beta(\kappa)$), $\sigma$, and $k_{j'}(\sigma)$, whose formula is omitted due to its excessive length.

The proof of Theorem \ref{thm:sigma N=2} mirrors the proofs of Theorems \ref{thm:N=1} and \ref{thm:N=2}, and is therefore omitted.

\begin{corollary}[Unstable spectrum near repeated resonance of order $2$]${}$\label{cor:sigma N=2} 

\noindent Recall Theorem \ref{thm:sigma N=2}. If ${\rm ind}_4>0$ then $\lambda\in {\rm spec}(\mathcal{L}(\eps))$ near $i\sigma_{c,2}\in i\mathbb{R}$ satisfies
\[
({\rm Re}\,\lambda)^2+\left({\rm Im}\,\lambda-\sigma_{c,2}-\frac{d^{(0,1,2)}}{d^{(1,1,0)}}\eps^2\right)^2={\rm ind}_4\eps^4
\]
to order $\eps^2$ as $\eps\to 0$. In other words, the unstable spectrum forms a circle to leading order, whose center is displaced from $i\sigma_{c,2}$ by $O(\eps^2)$ along the imaginary axis, and the radius is $O(\eps^2)$ for $\eps\in\mathbb{R}$ and $|\eps|\ll1$. 
\end{corollary}

The proof of Corollary \ref{cor:sigma N=2} mirrors the proofs of Corollaries \ref{cor:N=1}, \ref{cor:N=2}, and \ref{cor:sigma N=1}, and is omitted. 

\begin{remark*}\rm 
Comparing Corollaries \ref{cor:sigma N=1} and \ref{cor:sigma N=2}, we observe that the unstable spectrum forms a circle centered at $i(-1)^p\sigma_{c,p}$, $p=1$ or $2$, for resonances of order $1$, while the center is shifted along the imaginary axis by a distance of $O(\eps^2)$ for resonances of order $2$. 
\end{remark*}

\begin{remark*}\rm 
Comparing Corollaries \ref{cor:N=1}, \ref{cor:N=2} versus \ref{cor:sigma N=1}, \ref{cor:sigma N=2}, we observe that the unstable spectrum forms an ellipse for $\sigma \neq -\sigma_{c,1}, \sigma_{c,2}$ while a circle for $\sigma=-\sigma_{c,1}$ or $\sigma_{c,2}$.
\end{remark*}

\subsection{Numerical studies for repeated resonances}\label{sec:numerics sigma}

We conclude this section by numerically investigating the results of Sections \ref{sec:sigma N=1} and \ref{sec:sigma N=2}. 

Section \ref{S1S2region} (R16)-(R19) classify all repeated resonances. In this subsection, we restrict our attention to resonances of orders $1$ and $2$.  We pause to remark that resonance (R19) occurs only for order $\geq 2$ in the region where $-\sigma_{c,1}<\sigma_{c,2}$. 

Suppose that Section \ref{sec:sigma N=1} (B1) holds. That is, $i(-1)^p\sigma_{c,p}$, $p=1$ or $2$, is a repeated resonance of order $1$. We numerically evaluate \eqref{def:ind_3} and apply Theorem \ref{thm:sigma N=1} to numerically determine spectral stability or instability near $i(-1)^p\sigma_{c,p}$. Suppose, instead, 
\[
(-1)^p k_{c,p}-k_{j'}((-1)^p\sigma_{c,p})=N\kappa\quad \text{for $N=\pm2$ for $p=1$ or $2$,}
\]
for some $j'$. We then numerically evaluate ${\rm ind}_4(\kappa,(-1)^p\sigma_{c,p};k_{j'})$, where ${\rm ind}_4$ can be defined in the same manner as in \eqref{def:ind_4} after replacing $\sigma_{c,2}$ with $(-1)^p\sigma_{c,p}$, $k_{c,2}$ with $(-1)^pk_{c,p}$, and the subscript $4$ with $j'$. We correspondingly apply Theorem \ref{thm:sigma N=2} to numerically determine spectral stability or instability. 

We emphasize that each repeated resonance occurs along a curve in the $(\beta,\kappa)$-plane, where $\beta$ is determined by $\kappa$. Throughout this subsection, the numerical computations are carried out for $\kappa\in (0,8)$. 

We summarize our numerical findings. 

\subsubsection*{\rm (R16) $-k_{c,1}-k_5(-\sigma_{c,1})=N\kappa$ for $N=1,2$}

Section \ref{S1S2region} shows that along the curve where $-k_{c,1}-k_5(-\sigma_{c,1})=\kappa$ (see the right panel of Figure \ref{figure7}), 
\[
k_1(-\sigma_{c,1})-k_5(-\sigma_{c,1})=\kappa\quad\text{and}\quad
k_3(-\sigma_{c,1})-k_5(-\sigma_{c,1})=\kappa.
\]
Our numerical evaluation of \eqref{def:ind_3} suggests ${\rm ind}_3(\kappa,-\sigma_{c,1};k_5)<0$ for $\kappa\in(0,\infty)$ along the resonance curve (R16) for $N=1$. Theorem \ref{thm:sigma N=1} therefore predicts that ${\rm spec}(\mathcal{L}(\eps))$ near $-i\sigma_{c,1}$ remains confined to the imaginary axis for $\eps\in\mathbb{R}$ and $|\eps|\ll1$, that is, spectral stability near $-i\sigma_{c,1}$. 

Along the curve $-k_{c,1}-k_5(-\sigma_{c,1})=2\kappa$ (see the right panel of Figure \ref{figure7}), 
\[
k_1(-\sigma_{c,1})-k_5(-\sigma_{c,1})=2\kappa\quad\text{and}\quad
k_3(-\sigma_{c,1})-k_5(-\sigma_{c,1})=2\kappa.
\]
Our numerical computations indicate ${\rm ind}_4(\kappa,-\sigma_{c,1};k_5)<0$ for $\kappa\in(0,\infty)$ along the resonance curve (R16) for $N=2$. Theorem \ref{thm:sigma N=2} then implies spectral stability near $-i\sigma_{c,1}$. 

\begin{remark*}\rm 
We do not numerically detect spectral instability near the resonance (R16). This is compatible with the prediction based on Krein signature. See Remark \ref{rem:Krein}. 
\end{remark*}

\subsubsection*{\rm (R17) $k_6(-\sigma_{c,1})+k_{c,1}=N\kappa$ or $k_{c,2}-k_1(\sigma_{c,2})=N\kappa$ for $N=1,2$} 

Our numerical result corroborates ${\rm ind}_3(\kappa,-\sigma_{c,1};k_6)>0$ for $\kappa\in(0,\infty)$ along the curve $k_6(-\sigma_{c,1})+k_{c,1}=\kappa$, where $-i\sigma_{c,1}$ is a repeated resonance of order $1$. See the left panel of Figure \ref{figure8}. Theorem \ref{thm:sigma N=1} then predicts that there exists $\lambda\in {\rm spec}(\mathcal{L}(\eps))$ near $-i\sigma_{c,1}$ such that ${\rm Re}\,\lambda \neq0$ for $\eps\neq 0,\in\mathbb{R}$ and $|\eps|\ll1$, implying spectral instability near $i\sigma$. Numerical evidence supports ${\rm ind}_4(\kappa, \sigma_{c,1};k_6)>0$ for $\kappa\in(0,\infty)$ along the curve $k_6(-\sigma_{c,1})+k_{c,1}=2\kappa$, where $-i\sigma_{c,1}$ is a repeated resonance of order $2$. See the right panel of Figure \ref{figure8}. Theorem \ref{thm:sigma N=2} then implies spectral instability near $-i\sigma_{c,1}$. 

Our numerical result confirms ${\rm ind}_3(\kappa,\sigma_{c,2};k_1)>0$ for $\kappa\in(0,\infty)$ along the resonance curve $k_{c,2}-k_1(\sigma_{c,2})=\kappa$ (see the left panel of Figure \ref{figure8}), and Theorem \ref{thm:sigma N=1} implies spectral instability near $i\sigma_{c,2}\in i\mathbb{R}$. Numerical evidence supports ${\rm ind}_4(\kappa,\sigma_{c,2};k_1)>0$ for $\kappa\in(0,\infty)$ along the resonance curve $k_{c,2}-k_1(\sigma_{c,2})=2\kappa$ (see the right panel of Figure \ref{figure8}), and Theorem \ref{thm:sigma N=2} implies spectral instability near $i\sigma_{c,2}$.

\subsubsection*{\rm (R18) $k_2(-\sigma_{c,1})+k_{c,1}=N\kappa$ or $k_{c,2}-k_3(\sigma_{c,2})=N\kappa$ for $N=1,2$} 

The numerical result confirms ${\rm ind}_3(\kappa,-\sigma_{c,1};k_2)>0$ for $\kappa\in(0,\infty)$ along the curve $k_2(-\sigma_{c,1})+k_{c,1}=\kappa$, where $-i\sigma_{c,1}$ is a repeated resonance of order $1$. See the right panel of Figure \ref{figure9}. Theorem \ref{thm:sigma N=1} then implies spectral instability near $-i\sigma_{c,1}$. Numerical evidence supports ${\rm ind}_4(\kappa,-\sigma_{c,1};k_2)>0$ for $\kappa\in(0,\infty)$ along the resonance curve $k_2(-\sigma_{c,1})+k_{c,1}=2\kappa$ (see the left panel of Figure \ref{figure9}), and Theorem \ref{thm:sigma N=2} implies spectral instability near $-i\sigma_{c,1}$.

The numerical result confirms ${\rm ind}_3(\kappa,\sigma_{c,2};k_3)>0$ for $\kappa\in(0,\infty)$ along the curve $k_{c,2}-k_3(\sigma_{c,2})=\kappa$, where $i\sigma_{c,2}$ is a repeated resonance of order $1$ (see the right panel of Figure \ref{figure9}), and Theorem \ref{thm:sigma N=1} implies spectral instability near $i\sigma_{c,2}$. Moreover, ${\rm ind}_4(\kappa, \sigma_{c,2};k_3)>0$ along the resonance curve $k_{c,2}-k_3(\sigma_{c,2})=2\kappa$ (see the left panel of Figure \ref{figure9}), and Theorem \ref{thm:sigma N=2} implies spectral instability near $i\sigma_{c,2}$. 

\subsubsection*{\rm (R19) $k_{c,2}-k_4(\sigma_{c,2})=N\kappa$ for $N=2$}  
Section \ref{S1S2region} shows that $k_{c,2}-k_4(\sigma_{c,2})=2\kappa$ along the resonance curve (R19) of order $2$. See the left panel of Figure \ref{figure10}. Numerical evidence supports ${\rm ind}_4(\kappa,\sigma_{c,2};k_4)>0$ along the resonance curve. Theorem \ref{thm:sigma N=2} therefore predicts spectral instability near $i\sigma_{c,2}$.

\begin{remark*}\rm 
Recall Section \ref{sec:numerics N=2}. We observe that the numerical instability result in the resonance region (R6) extends continuously up to the boundary curve along which resonance (R17) occurs. See Figure \ref{figure8}. Similarly, the numerical instability result in the resonance region (R7) extends continuously to its boundary, which serves as the resonance curve for (R18). See Figure \ref{figure9}. The numerical instability result in the resonance regions (R9) and (R12) persist along the boundary $k_{c,2}-k_4(\sigma_{c,2})=2\kappa_1$, which is the resonance curve for (R19). See Figure \ref{figure10}.
\end{remark*}

The top right panel of Figure \ref{figure1} graphically summarizes our numerical findings of spectral stability and instability for steady unimodal capillary-gravity waves near nonzero resonances of orders $1$ and $2$, namely resonances in Section \ref{resonances}.

\newpage

\section{Modulational instability of unimodal waves}\label{sec:modulational}

Let $\beta>0$ and $\kappa>0$. Recall Section \ref{sec:unimodal} and we continue to assume that 
there exist steady unimodal capillary-gravity waves of sufficiently small amplitude, solving \eqref{stationary-periodic}-\eqref{eqn:stokes exp} for $\eps\neq0,\in\mathbb{R}$ and $|\eps|\ll1$, and satisfying \eqref{def:stokes1}. Section \ref{sec:wilton instability} treats bimodal Wilton ripples, satisfying either \eqref{def:wilton1 M=2} or \eqref{def:wilton1 M=3}.

It is evident from \eqref{def:mu0} and \eqref{eqn:sigma} that $i0$ is resonant in the sense of Definition \ref{def:resonance}. Theorem \ref{thm:I res} then implies that there may exist $\lambda \in {\rm spec}(\mathcal{L}(\eps))$ near $i0$ such that ${\rm Re}\,\lambda \neq 0$, giving rise to spectral instability. 

Lemma \ref{lem:dispersion} establishes $D(i0,k_j(0),0)=0$ for each $k_j(0)$ in Lemma \ref{lem:eps=0}, $j=6,2,4,5$ or $j=6,2,4,3,1,5$, depending on whether $\mu_0(\kappa;\beta)<1$ or $>1$. It therefore suffices to investigate the roots of 
\[
D(i0+\delta,k_j(0)+\gamma,\eps) \text{ for $\delta\in\mathbb{C}$ and $|\delta|\ll1$ for $\gamma, \eps\in\mathbb{R}$ and $|\gamma|,|\eps|\ll1$}.
\]
Since \eqref{def:Delta} is jointly analytic in its arguments, we make a power series expansion of $D(i0+\delta,k_j(0)+\gamma,\eps)$ about $(\delta,\gamma,\eps)=(0,0,0)$, and examine the roots of its leading order terms in $\delta$, $\gamma$, and $\eps$.

\begin{notation*}
Throughout this section, $\sigma=0$ is suppressed for simplicity of notation unless otherwise specified.
\end{notation*}

\subsection{The modulational instability index for subcritical waves}\label{sec:modulational subcritical}

Suppose $0<\mu_0(\kappa;\beta)<1$, corresponding to subcritical capillary-gravity waves. 

\begin{notation*}
Recall Lemma \ref{lem:eps=0} and we relabel the eigenvalues and eigenfunctions of $\mathbf{L}(i0)$ as
\[
k_j=(-1)^j\kappa,\quad j=1,2,\quad\text{and}\quad k_j=0,\quad j=3,4.
\]
Correspondingly, $\boldsymbol{\phi}_j$, $j=1,2$, are in \eqref{def:eigen0}, and $\boldsymbol{\phi}_3$ and $\boldsymbol{\phi}_4$ are the latter and the former of \eqref{def:eigen0'}, respectively. We relabel the eigenfunctions of $\mathbf{L}^\dag(i0)$ analogously.

Additionally, we write \eqref{def:V(m,n)} as
\[
\mathbf{V}^{(m,n)}(x)=(w^{(m,n)}_{jk}(x))_{j,k=1,2,3,4}. 
\]
For example, 
\begin{equation}\label{def:V(0,0)0'}
(w^{(0,0)}_{jk}(x))_{j,k=1,2,3,4}=\begin{pmatrix}
e^{-i\kappa x} & 0 & 0 & 0  \\ 0 & e^{i\kappa x} & 0 & 0 \\ 0 & 0 & 1 & 0\\ 0 & 0 & x & 1\end{pmatrix}.
\end{equation}
Compare \eqref{def:V(0,0)0}. 
\end{notation*}

We begin by gathering preliminary results.

\begin{lemma}\label{lem:a4=0}
If $m=0$ then \eqref{def:V(m,n)} satisfies $w^{(0,n)}_{j4}(T)=0$ for $j=1,2,3,4$ and $n=1$, $2$, $\dots$.
\end{lemma}

\begin{proof}
Suppose $\delta=0$. Clearly, $\phi=const.$ and $\eta=u=z=0$ solve \eqref{eqn:spec} for $\lambda=0$. Recall \eqref{def:varphi} and, correspondingly, 
\[
\mathbf{u}(x):=\frac{f_1(0;\eps)}{f_1(x;\eps)}\begin{pmatrix}1\\0\\0\\0\end{pmatrix}=\frac{f_1(0;\eps)}{f_1(x;\eps)}\boldsymbol{\phi}_4
\]
solves \eqref{eqn:LB} for $\lambda=0$, where 
$f_1(x;\eps)$ is in \eqref{def:f1(eps)}. Since $\boldsymbol{\Pi}\mathbf{u}(x)=\mathbf{u}(x)$ (see Definition \ref{def:proj}), clearly, $\mathbf{u}(x)$ solves \eqref{eqn:A} for $\sigma,\delta=0$, satisfying $\mathbf{u}(0)=\boldsymbol{\phi}_4$. Recall Section \ref{sec:expansion_monodromy} and, particularly, 
\[
\mathbf{V}(x;\sigma,\delta,\eps)=\begin{pmatrix} \mathbf{v}_1 & \mathbf{v}_2 & \mathbf{v}_3 & \mathbf{v}_4\end{pmatrix}(x;\sigma,\delta,\eps),
\]
where $\mathbf{v}_k$, $k=1,2,3,4$, take the form \eqref{def:v_k}. We observe $\mathbf{v}_4(x;0,0,\eps)=\mathbf{u}(x)$ and, hence, $\mathbf{v}_4(T;0,0,\eps)=\mathbf{u}(T)=\boldsymbol{\phi}_4$ by periodicity. On the other hand, since 
$$
\sum_{n=0}^\infty w^{(0,n)}_{j4}(T)\eps^n=\langle\mathbf{v}_4(T;0,0,\eps),\boldsymbol{\psi}_j\rangle, \quad j=1,2,3,4,
$$
and since the right side is independent of $\eps$, the assertion follows.
\end{proof}

\begin{corollary}[The periodic Evans function for $\delta=0$]\label{delta_0pk}
If $\delta=0$ then \eqref{def:Delta} satisfies $D(0,K\kappa,\eps)=0$ for $K\in\mathbb{Z}$ for $\eps\in\mathbb{R}$ and $|\eps|\ll1$.
\end{corollary}

\begin{proof}
Lemma \ref{lem:a4=0} and \eqref{cond:V(0)} imply that the fourth column of $\mathbf{I}-\mathbf{V}(T;0,0,\eps)$ vanishes and, hence, $D(0,K\kappa,\eps)=\det(\mathbf{I}-\mathbf{V}(T;0,0,\eps))=0$. 
\end{proof}


We proceed along the lines of the argument in Sections \ref{sec:N=1} and \ref{sec:N=2}. See also \cite[Section 5]{HY2023} for Stokes' gravity waves.

Recall Section \ref{sec:expansion_monodromy} and we evaluate $\mathbf{V}^{(m,n)}(T)$ for $m+n=0,1,2$. To this end, we solve \eqref{eqn:v_k(0)}-\eqref{def:f(0)}, \eqref{cond:V(0)}, and \eqref{eqn:w_k}-\eqref{eqn:w_k bd}, $k=1,2,3,4$, after relabeling $k_j$, $\boldsymbol{\phi}_j$ and $\boldsymbol{\psi}_j$, $j=1,2,3,4$, and utilizing the result in Appendix \ref{app:Bexp}, where $\phi_1$, $\phi_2$, $\eta_1$, $\eta_2$, $\mu_1$, and $\mu_2$ are collected in Section \ref{sec:unimodal}. We then make a straightforward calculation of the solution. 

We summarize the result. 

\begin{lemma}\label{lem:V(m,n)0}
Assume $0<\mu_0(\kappa;\beta)<1$. Then,
\begin{align*}
&\mathbf{V}^{(1,0)}(T)=\begin{pmatrix} w^{(1,0)}_{11}
 & 0 & w^{(1,0)}_{13} & 0 \\ 0 & w^{(1,0)}_{11} & -w^{(1,0)}_{13} & 0 \\ 0 & 0 & w^{(1,0)}_{33} & 0\\ w^{(1,0)}_{41}& w^{(1,0)}_{41}& Tw^{(1,0)}_{33}& w^{(1,0)}_{33} \end{pmatrix},\\ 
&\mathbf{V}^{(0,1)}(T)=\begin{pmatrix} 0 & 0 & w^{(0,1)}_{13}&0\\ 0 & 0 & w^{(0,1)}_{13}&0\\0&0&0&0\\w^{(0,1)}_{41}&-w^{(0,1)}_{41}&*&0\end{pmatrix}, \\ 
&\mathbf{V}^{(2,0)}(T)=\begin{pmatrix} w_{11}^{(2,0)}& -w_{21}^{(2,0)} & * & w_{14}^{(2,0)}  \\ w_{21}^{(2,0)} & (w_{11}^{(2,0)})^*& * &-w_{14}^{(2,0)} \\ 
w_{31}^{(2,0)}& w_{31}^{(2,0)} & w_{33}^{(2,0)} & w_{34}^{(2,0)} \\ * & * & * & w_{33}^{(2,0)} \end{pmatrix},\\
&\mathbf{V}^{(1,1)}(T)=\begin{pmatrix} w_{11}^{(1,1)} & w_{21}^{(1,1)} & * & w_{14}^{(1,1)}\\ 
w_{21}^{(1,1)} & w_{11}^{(1,1)} & * & w_{14}^{(1,1)}\\
w_{31}^{(1,1)} & -w_{31}^{(1,1)} & * & 0 \\ * & * & * & * \end{pmatrix},\\ 
&\mathbf{V}^{(0,2)}(T)=\begin{pmatrix} w_{11}^{(0,2)} & -w_{11}^{(0,2)} & * & 0\\ 
w_{11}^{(0,2)} & -w_{11}^{(0,2)} & * & 0\\
0 & 0 & * & 0 \\ * & * & * & 0 \end{pmatrix}.
\end{align*}
The formulae for $w^{(1,0)}_{11}, w^{(1,0)}_{13}, w^{(1,0)}_{33}, w^{(1,0)}_{41}, w^{(0,1)}_{13}, w^{(0,1)}_{41}, w^{(2,0)}_{11}, w^{(2,0)}_{34}, w^{(1,1)}_{31}$, $w^{(1,1)}_{14}$ and $w^{(0,2)}_{11}$ are collected in Appendix \ref{app:V(m,n) unimodal}. Note that $w_{22}^{(2,0)}$ is the complex conjugate of $w_{11}^{(2,0)}$. The expressions for $w_{33}^{(2,0)}$ and $w_{44}^{(2,0)}$ are excessively long and are not presented here. We emphasize, however, $w_{33}^{(2,0)}=w_{44}^{(2,0)}$. Our analysis does not involve the entries marked with an asterisk, whence their formulae are omitted. 

Additionally,
\begin{align*}
&Tw^{(2,0)}_{21}=-w^{(1,0)}_{13}w^{(1,0)}_{41},\\
&Tw_{14}^{(2,0)}=-w_{13}^{(1,0)}(w_{11}^{(1,0)}-w_{33}^{(1,0)}),\\ 
&Tw_{31}^{(2,0)}=-w_{41}^{(1,0)}(w_{11}^{(1,0)} - w_{33}^{(1,0)}),\\
&Tw_{11}^{(1,1)}=(w_{13}^{(0,1)}w_{41}^{(1,0)} + w_{13}^{(1,0)}w_{41}^{(0,1)}),\\
&Tw_{21}^{(1,1)}=(w_{13}^{(0,1)}w_{41}^{(1,0)} - w_{13}^{(1,0)}w_{41}^{(0,1)}).
\end{align*} 
\end{lemma}

See \cite[Section 5.1]{HY2023} for the corresponding result for Stokes waves. 

Recall Sections \ref{sec:reduction} and \ref{sec:expansion_monodromy}, and we proceed to a power series expansion of  $D(i0+\delta,k_j+\gamma,\eps)$, $j=1,2,3,4$, about $(\delta,\gamma,\eps)=(0,0,0)$. 
Substituting the result of Lemma \ref{lem:V(m,n)0} along with \eqref{def:V(0,0)0'} into \eqref{def:Delta}, \eqref{def:Vexp}, and \eqref{def:V(m,n)}, we make a straightforward calculation\footnote{We utilize the identities among $w^{(m,n)}_{jk}$'s in Lemma \ref{lem:V(m,n)0}, to simplify the resulting expressions.} to arrive at
\ba \label{eqn:evans0}
D(\delta,K\kappa+\gamma,\eps)=&\sum_{\ell=0}^4d^{(\ell,4-\ell,0)}\delta^{\ell}\gamma^{4-\ell}+\sum_{\ell=0}^3d^{(\ell,3-\ell,2)}\delta^{\ell}\gamma^{3-\ell}\eps^2\\
&+o((|\delta|+|\gamma|)^3(|\delta|+|\gamma|+|\eps|^2))
\ea
as $\delta,\gamma,\eps\to 0$, where $K=0,\pm1$. Here $d^{(\ell,m,n)}$, for either $\ell,m=0,1,\dots,4$ and $n=0$, or $\ell,m=0,1,\dots,3$ and $n=2$, can be expressed in terms of $w^{(m,n)}_{jk}$, in Lemma \ref{lem:V(m,n)0}, whose formulae are in Appendix \ref{app:V(m,n) unimodal}. Particularly, 
\begin{align*}
d^{(4,0,0)}=&(w_{11}^{(1,0)})^2((w_{33}^{(1,0)})^2 - Tw_{34}^{(2,0)})\\
=&64\pi^4\kappa^{-4}{\sh(2\kappa)}^2\sh(\kappa)(\beta \kappa ^2+1)^3\\
&\cdot(\sh(\kappa)-\kappa \ch(\kappa)+\beta \kappa ^2\sh(\kappa))^{-1}\\
&\cdot(2\kappa -\sh(2\kappa)+2\beta \kappa ^3+\beta \kappa ^2\sh(2\kappa))^{-2}\neq 0.
\end{align*}
The formulae for $d^{(\ell,m,n)}$ for the other values of $(\ell,m,n)$ are omitted. We emphasize that $w_{jk}^{(m,n)}$ and, hence, $d^{(\ell,m,n)}$ depend on $\kappa=2\pi/T$ and $\beta$. We suppress the dependence for brevity unless otherwise specified. See \cite[Section 5.2]{HY2023} for the corresponding result for Stokes waves. See also \eqref{def:d(l,m,n) N=1}, \eqref{def:d(l,m,n)c N=1}, and Lemma \ref{lem:D(sigma_c2)} for analogous results for nonzero resonances. 

The Weierstrass preparation theorem asserts 
\[
D(\delta,\gamma,\eps)=W(\delta,\gamma,\eps)h(\delta,\gamma,\eps)
\]
for $\delta\in\mathbb{C}$, $\gamma,\eps\in\mathbb{R}$ and $|\delta|,|\gamma|,|\eps|\ll1$, where 
\be \label{def:W}
W(\delta,\gamma,\eps)=\delta^4+a_3(\gamma,\eps)\delta^3+a_2(\gamma,\eps)\delta^2+a_1(\gamma,\eps)\delta+a_0(\gamma,\eps)
\ee 
is a Weierstrass polynomial, $a_p(\gamma,\eps)$, $p=0,1,2,3$, are analytic at $(0,0)$, and $a_p(0,0)=0$, $p=0,1,2,3$, while $h(\delta,\gamma,\eps)$ is analytic at $(0,0,0)$ and $h(0,0,0)=d^{(4,0,0)}\neq 0$. We emphasize that the roots of \eqref{eqn:evans0} coincide with the roots of \eqref{def:W}, counting multiplicity, for $\gamma,\eps\in\mathbb{R}$ and $|\gamma|,|\eps|\ll1$, denoted by 
\[
\delta_j(\gamma,\eps), \quad j=1,2,3,4.
\]
Lemma \ref{lem:Wsymm} implies
\be\label{g_eqs}
a_3(\gamma,\eps),a_1(\gamma,\eps)\in i\mathbb{R} \quad \text{and}\quad a_2(\gamma,\eps),a_0(\gamma,\eps)\in \mathbb{R}. 
\ee

Suppose $\eps=0$. The roots of \eqref{def:W} recover ${\rm spec}(\mathcal{L}(0))$ near $i0$. More precisely, $\delta_j(\gamma,0)=i\sigma(k_j+\gamma)$, $j=1,2,3,4$, for $\gamma\in\mathbb{R}$ and $|\gamma|\ll1$, where $\sigma$ is in \eqref{eqn:sigma} and $k_j$ in Lemma \ref{lem:eps=0} after reindexing. Since \eqref{def:sigma} is analytic at $k_j$, we may expand
\[
\delta_j(\gamma,0)=\delta^{(1,0)}_j\gamma+\delta^{(2,0)}_j\gamma^2+o(|\gamma|^2)\quad \text{as $\gamma\to 0$}, \quad j=1,2,3,4,
\]
where $\delta^{(m,0)}_j$, $j=1,2,3,4$ and $m=1,2,\dots$, can be determined from the Taylor expansions of \eqref{def:sigma} about $k_j$, $j=1,2,3,4$. 

Suppose $\eps \neq0,\in\mathbb{R}$ and $|\eps|\ll1$. Corollary \ref{delta_0pk} establishes that $0$ is a root of $D(\cdot,0,\eps)$. Accordingly, we seek an expansion of the form
\be \label{def:lambda(gamma,eps)}
\delta_j(\gamma,\eps)=\delta^{(1,0)}_j\gamma+\delta^{(2,0)}_j\gamma^2+\delta^{(1,1)}_j\gamma\eps+o(|\gamma|^2+|\gamma||\eps|) \quad \text{as $\gamma,\eps\to 0$},
\ee 
$j=1,2,3,4$, where $\delta^{(1,1)}_j$, $j=1,2,3,4$, are to be determined in terms of $d^{(\ell,m,n)}$ and, hence, $w_{jk}^{(m,n)}$. Importantly, $\delta^{(1,0)}_j$ and $\delta^{(2,0)}_j$, $j=1,2,3,4$, can be expressed in terms of $w_{jk}^{(m,n)}$, which must coincide with those obtained from the Taylor expansions of \eqref{def:sigma} about $k_j$, $j=1,2,3,4$. We pause to remark that \cite[Theorem 1.1]{hsiao2025} and \cite[Proposition 5.14]{sun2025} justify the asymptotic expansion \eqref{def:lambda(gamma,eps)} in the regime where 
\[
0\ll |\gamma|\ll |\eps|\ll1.
\]

Substituting \eqref{def:lambda(gamma,eps)} into \eqref{eqn:evans0} and collecting terms of order $\gamma^m\eps^n$, $m,n=0,1,2,\dots$, 
we make a straightforward calculation to show that the leading order contribution is of order $\gamma^4$, whose coefficient is
\begin{equation}\label{eqn:alpha10}
-(-w_{11}^{(1,0)}\delta^{(1,0)}_j+iT)^2
((Tw_{34}^{(2,0)}-{w_{33}^{(1,0)}}^2)(\delta^{(1,0)}_j)^2+2iTw_{33}^{(1,0)}\delta^{(1,0)}_j+T^2),
\end{equation}
$j=1,2,3,4$, where $w_{11}^{(1,0)}$, $w_{33}^{(1,0)}$, and $w_{34}^{(2,0)}$ are in Appendix \ref{app:V(m,n) unimodal}. If \eqref{def:lambda(gamma,eps)} makes the roots of $D(\delta,\gamma,\eps)$ for $\gamma,\eps\in\mathbb{R}$ and $|\gamma|,|\eps|\ll1$, then \eqref{eqn:alpha10} necessarily vanishes. Consequently, 
\be\label{def:alpha10'}
\delta^{(1,0)}_j=\frac{iT}{w_{11}^{(1,0)}}\quad \text{or}\quad 
\frac{iTw_{33}^{(1,0)}\pm\sqrt{-T^3w^{(2,0)}_{34}}}{(w_{33}^{(1,0)})^2-Tw_{34}^{(2,0)}},
\ee
$j=1,2,3,4$. An inspection of the result of Appendix \ref{app:V(m,n) unimodal} reveals that $w_{11}^{(1,0)}$ and $w_{33}^{(1,0)}$ are real and $w_{34}^{(2,0)}>0$ and, hence, the right sides of \eqref{def:alpha10'} are purely imaginary. On the other hand, \eqref{def:alpha10'} must agree with the coefficients of the $\gamma$ term in the Taylor expansions of \eqref{def:sigma} about $k_j$, $j=1,2,3,4$. Therefore we conclude  
\begin{equation}\label{def:alpha10}
\delta^{(1,0)}_j=\frac{iT}{w_{11}^{(1,0)}},\, j=1,2
\quad\text{and}\quad
\delta^{(1,0)}_j=\frac{iTw_{33}^{(1,0)}\pm\sqrt{-T^3w^{(2,0)}_{34}}}{(w_{33}^{(1,0)})^2-Tw_{34}^{(2,0)}},\, j=3,4.
\end{equation}
In other words, the former of \eqref{def:alpha10} corresponds to the Taylor expansion of \eqref{def:sigma} about $k_j=\mp\kappa$, $j=1,2$, while the latter corresponds to the Taylor expansion of \eqref{def:sigma} about $k_j=0$, $j=3,4$. For the remainder of this subsection, we restrict our attention to $j=1,2$.

Substituting \eqref{def:lambda(gamma,eps)} and the former equation of \eqref{def:alpha10} into \eqref{eqn:evans0}, we make a straightforward calculation to show that the coefficient of the $\gamma^5$ term vanishes. Proceeding to the $\gamma^6$ term, after some algebra, we arrive at 
\be\label{def:alpha20}
\delta^{(2,0)}_j=\pm\frac{T^2((w_{11}^{(1,0)})^2-2w_{11}^{(2,0)})+2Tw_{13}^{(1,0)}w_{41}^{(1,0)}}{2(w_{11}^{(1,0)})^3},\quad j=1,2,
\ee
where $w_{11}^{(1,0)}$, $w_{13}^{(1,0)}$, $w_{41}^{(1,0)}$, and $w_{11}^{(2,0)}$ are in Appendix \ref{app:V(m,n) unimodal}. We pause to remark that $\delta^{(2,0)}_j$, $j=1,2$, are purely imaginary because $(w_{11}^{(1,0)})^2$ offsets the real part of $2w_{11}^{(2,0)}$. The $\pm$ signs account for the opposite convexities of \eqref{def:sigma} at $k_j$, $j=1,2$. See Lemma \ref{lem:sigmaS3}.

Substituting \eqref{def:lambda(gamma,eps)}, the former equation of \eqref{def:alpha10}, and \eqref{def:alpha20} into \eqref{eqn:evans0}, we make a straightforward calculation to show that the coefficient of the $\gamma^3\eps^2$ term vanishes, and the coefficient of the $\gamma^4\eps^2$ term is
\ba \label{eqn:f}
T^2(\delta_j^{(1,1)})^2\omega_1-T^3\omega_2,
\ea 
$j=1,2$, where
\ba \label{g1g2}
\omega_1=&Tw_{34}^{(2,0)} - (w_{11}^{(1,0)}-w_{33}^{(1,0)})^2,\\
\omega_2=&(w_{11}^{(1,0)})^{-4}(-T  (w^{(1,0)}_{11})^2- 2  w^{(1,0)}_{13}  w^{(1,0)}_{41} + 2 T  w^{(2,0)}_{11} )\\
&\cdot((w^{(1,0)}_{11}-w^{(1,0)}_{33})^2w^{(0,2)}_{11}+w^{(0,1)}_{13}w^{(0,1)}_{41}w^{(2,0)}_{34}\\
&\quad +w^{(1,0)}_{11}w^{(0,1)}_{13}w^{(1,1)}_{31}+w^{(1,0)}_{11}w^{(0,1)}_{41}w^{(1,1)}_{14}\\
&\quad -w^{(1,0)}_{33}w^{(0,1)}_{13}w^{(1,1)}_{31}-w^{(1,0)}_{33}w^{(0,1)}_{41}w^{(1,1)}_{14}\\
&\quad -T w^{(0,2)}_{11}w^{(2,0)}_{34}+Tw^{(1,1)}_{14}w^{(1,1)}_{31}),
\ea 
for $w^{(m,n)}_{jk}$ in Appendix \ref{app:V(m,n) unimodal}. Furthermore, a straightforward calculation reveals 
\ba\label{g_i}
\omega_1(\beta,\kappa)&=\frac{4 \pi^2 (e^{2\kappa} - 1) (\beta \kappa^2+1)}{\kappa^2q(\beta,\kappa)}{\rm ind}_{5,1}(\beta,\kappa),\\
\omega_2(\beta,\kappa)&=\frac{\kappa \pi }{256 e^{4\kappa}\sh(2\kappa) (e^{2\kappa} + 1) (\beta \kappa^2+1)^2 q(\beta,\kappa)}\frac{{\rm ind}_{5,2}{\rm ind}_{5,3}}{{\rm ind}_{5,4}}(\beta,\kappa),\\
\ea
where
\begin{align*}
q(\beta,\kappa)=&(4 \kappa e^{2\kappa} - \beta \kappa^2 - e^{4\kappa} + 4 \beta \kappa^3 e^{2\kappa} + \beta \kappa^2 e^{4\kappa} + 1)^2\\
&\cdot (\kappa - e^{2\kappa} + \beta \kappa^2 + \kappa e^{2\kappa} - \beta \kappa^2 e^{2\kappa} + 1)
\end{align*}
and ${\rm ind}_{5,\ell}$, $\ell=1,2,3,4$, are collected in Appendix \ref{app:ind}, whose signs determine spectral stability or instability.

We state the main result of this subsection.

\begin{theorem}[Modulational instability for unimodal wave solutions]${}$\label{thm:modulation}

\noindent Let $\beta>0$ and $\kappa>0$. Suppose that a steady unimodal wave solution exists to \eqref{stationary-periodic}-\eqref{eqn:stokes exp} for $\eps\neq 0,\in\mathbb{R}$ and $|\eps|\ll1$, satisfying \eqref{def:stokes1}. Assume $0<\mu_0(\kappa;\beta)<1$. Recall Section \ref{sec:L}, and consider ${\rm spec}(\mathcal{L}(\eps))$ near $i0$.

Then, there exists $\lambda \in {\rm spec}(\mathcal{L}(\eps))$ such that $|\lambda|\ll1$ and ${\rm Re}\,\lambda\neq 0$, implying modulational instability, provided 
\be \label{def:ind_5}
{\rm ind}_5(\beta,\kappa):=({\rm ind}_{5,1}{\rm ind}_{5,2}{\rm ind}_{5,3}{\rm ind}_{5,4})(\beta,\kappa)>0,
\ee
where ${\rm ind}_{5,\ell}$, $\ell=1,2,3,4$, are in Appendix \ref{app:ind}.
\end{theorem}

Therefore, a steady unimodal capillary-gravity wave of sufficiently small amplitude in water of unit depth, satisfying $0<\mu_0(\kappa;\beta)<1$, is modulationally unstable provided that \eqref{def:ind_5} holds. We emphasize that \eqref{def:ind_5} consists of functions of $\beta$ and $\kappa$ whose signs can be rigorously established via validated numerics. 

\begin{proof}
Recall that \eqref{def:lambda(gamma,eps)} is rigorously justified \cite{hsiao2025, sun2025} and makes the roots of \eqref{eqn:evans0} when $0\ll|\gamma|\ll|\eps|\ll1$. Upon substitution, where $\delta^{(1,0)}_j$ and $\delta^{(2,0)}_j$, $j=1,2$, are in the former equation of \eqref{def:alpha10} and \eqref{def:alpha20}, respectively, since \eqref{eqn:f} necessarily vanishes, consequently, 
\begin{equation}\label{def:alpha11}
\delta^{(1,1)}_j=\pm\sqrt{\frac{T\omega_2}{\omega_1}}, 
\end{equation}
where $\omega_1$ and $\omega_2$ are in \eqref{g1g2} and \eqref{g_i}. Importantly, $\delta^{(1,1)}_j\in\mathbb{R}$ implies modulational instability. An inspection of \eqref{g_i} reveals that the sign of $\frac{\omega_2}{\omega_1}$ coincides with the sign of $\dfrac{{\rm ind}_{5,2}{\rm ind}_{5,3}}{{\rm ind}_{5,1}{\rm ind}_{5,4}}$ and, hence, the sign of ${\rm ind}_5(\beta,\kappa)$, in \eqref{def:ind_5}. 
This completes the proof. 
\end{proof}

\subsection{Modulational instability of supercritical waves}\label{sec:modulational supercritical}

Suppose, on the other hand, $\mu_0(\kappa;\beta)>1$, corresponding to supercritical capillary-gravity waves. One can proceed along the lines of the argument in Section \ref{sec:modulational subcritical} to derive a modulational instability index. One main difference from the subcritical wave case is that $\mathbf{V}^{(m,n)}(x;0)$, $m,n=0,1,2,\dots$, are $6\times 6$ matrices. Nonetheless, Lemma \ref{lem:a4=0} remains valid for the columns of $\mathbf{V}^{(0,n)}(x;0)$, $n=0,1,2,\dots$, associated with the former of \eqref{def:eigen0''} and, correspondingly, Corollary \ref{delta_0pk} continues to hold. Additionally, Lemma \ref{lem:V(m,n)0} remains valid for the $4\times 4$ leading principal submatrices of $\mathbf{V}^{(m,n)}(T;0)$ for $m+n=0,1,2$. We omit the details. 

Notably, ${\rm ind}_5(\beta,\kappa)$, in \eqref{def:ind_5}, is continuous across the curve in the $(\beta,\kappa)$-plane where $\mu_0(\kappa;\beta)=1$, separating the regions for supercritical and subcritical waves. Indeed, Theorem \ref{thm:modulation} extends to supercritical steady unimodal capillary-gravity waves, so that the sign of \eqref{def:ind_5} determines modulational instability.

An inspection of the result in Appendix \ref{app:ind} reveals that ${\rm ind}_{5,\ell}(\beta,\kappa)$, $\ell=1,2,3,4$, are polynomials in $\beta$ of degree $2$ for $\ell=1,2$, degree $4$ for $\ell=3$, and degree $1$ for $\ell=4$. Furthermore, an explicit calculation shows: 

\begin{itemize}
\item[$\bullet$] ${\rm ind}_{5,1}(\cdot,\kappa)$ has one positive root, denoted by $\beta_4(\kappa)$, and one negative root for each $\kappa\in(0,\infty)$. 

\item[$\bullet$] ${\rm ind}_{5,2}(\cdot,\kappa)$ has one positive root, denoted by $\beta_2(\kappa)$, and one negative root  for each $\kappa\in(0,\infty)$. 
    
\item[$\bullet$] ${\rm ind}_{5,3}(\cdot,\kappa)$ has two real roots for each $\kappa\in (0,\infty)$, denoted by $\beta_1(\kappa)$ and $\beta_5(\kappa)$ such that $\beta_1(\kappa)<\beta_5(\kappa)$. Additionally, it has two negative real roots for $\kappa\in (0,0.32\ldots)$ while it has one pair of complex (conjugate) roots for $\kappa \in (0.32\dots,\infty)$. The latter two complex roots do not play any role in determining modulational instability. Furthermore,  $\beta_5(\kappa)>0$ for each $\kappa\in (0,\infty)$ whereas $\beta_1(\kappa)>0$ if $\kappa \in (0,\kappa_*)$ while $\beta_1(\kappa)<0$ if $\kappa \in (\kappa_*,\infty)$, where $\kappa_*=1.3627\ldots$ is the critical dimensionless wave number for the Benjamin--Feir instability of Stokes' gravity waves. Specifically, a $2\pi/\kappa$-periodic Stokes wave of sufficiently small amplitude in water of unit depth is modulationally unstable if $\kappa>\kappa_*$ \cite{BF;BF, Whitham;BF}. See also \cite{BM;BF, HY2023, Berti2023}.
    
\item[$\bullet$] ${\rm ind}_{5,4}(\cdot,\kappa)$ has one positive root, denoted by $\beta_3(\kappa)$. Interestingly, ${\rm ind}_{5,4}=0$ if and only if \eqref{def:beta_W} holds for $M=2$. In other words, the curve $\beta=\beta_3(\kappa)$ corresponds to Wilton ripples of order $2$. 
\end{itemize}

Figure \ref{figure15} illustrates the region for modulational instability, where ${\rm ind}_5(\beta,\kappa)$, in \eqref{def:ind_5}, is positive, bounded by the curves $\beta=\beta_\ell(\kappa)$, $\ell=1,2,3,4,5$. Our result is in excellent agreement with, for instance, the result in \cite{djordjevic_redekopp_1977} obtained through formal asymptotic analysis.

\begin{figure}[htbp]
    \centering
    \includegraphics[scale=0.3]{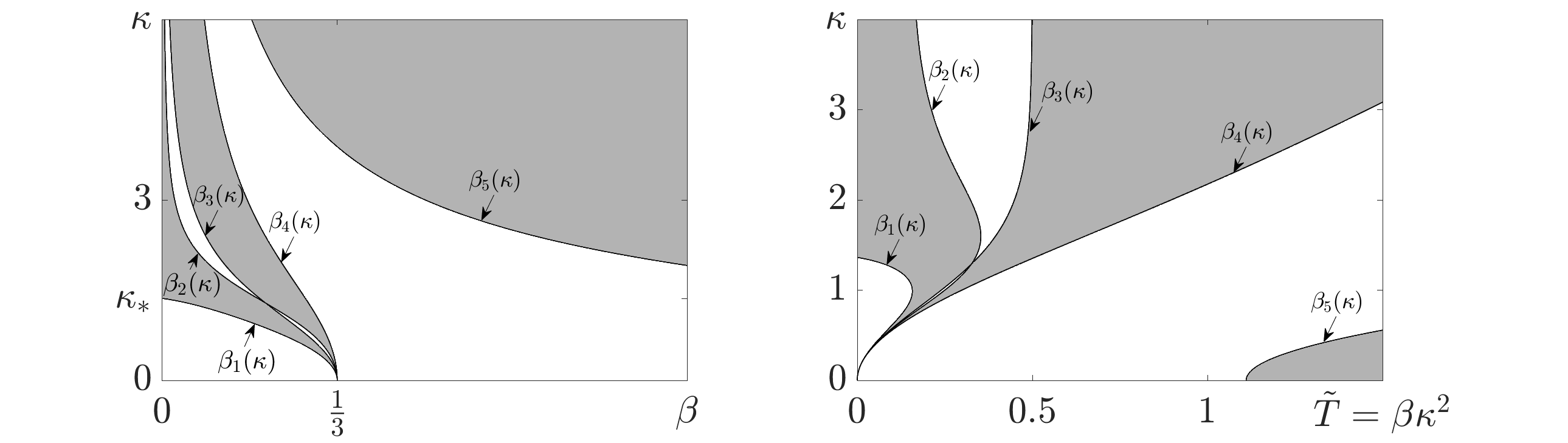}
    \caption{The shaded region means modulational instability of steady unimodal capillary-gravity waves of sufficiently small amplitude in water of unit depth. \protect\\
    \protect\hspace*{11pt} Left: $\kappa$ versus $\beta$. The curves $\beta=\beta_2(\kappa)$ and $\beta=\beta_3(\kappa)$ intersect at $\kappa=1.28523\dots$. Rigorous numerics reveals that the curve $\beta=\beta_1(\kappa)$ intersects the $\kappa$-axis at $\kappa=1.3627\dots$. Right: $\kappa$ versus $\beta\kappa^2$. The result agrees with, for instance, \cite[Figure 1]{djordjevic_redekopp_1977} obtained through formal asymptotic analysis.}
    \label{figure15}
\end{figure}

\begin{remark*}\rm
A straightforward calculation reveals
\begin{align*}
{\rm ind}_{5,3}(0,\kappa)=-64 e^{8\kappa}(&8\ch(2\kappa)+24\kappa \sh(2\kappa)+2\kappa\sh(4\kappa)\\
&+19\ch(2\kappa)^2-8\ch(2\kappa)^3-10\cosh(2\kappa)^4\\
&-8\kappa^2\ch(2\kappa)^2-28\kappa^2+8\ch(2\kappa)^3\sh(2\kappa)-9),
\end{align*}
which coincides with the Benjamin--Feir instability index for Stokes' gravity waves. See, for instance, \cite[eq. (5.34)]{HY2023}. Therefore, although \eqref{def:u} becomes singular when $\beta=0$, our modulational instability index remains well defined and valid in the zero surface-tension limit.
\end{remark*}

\newpage

\section{Spectral stability and instability of Wilton ripples}\label{sec:wilton instability}

Recall Section \ref{sec:bimodal}. Suppose that \eqref{A:wilton} holds, so that there exist bimodal Wilton ripples of order $M$ for sufficiently small amplitude, solving \eqref{stationary-periodic}-\eqref{eqn:stokes exp} for $\eps\neq0,\in\mathbb{R}$ and $|\eps|\ll1$, and satisfying either \eqref{def:wilton1 M=2} or \eqref{def:wilton1 M=3}. We address their spectral stability and instability. 

Throughout this section, we assume that $i\sigma \in i\mathbb{R}$ is resonant in the sense of Definition \ref{def:resonance}. That is, 
\[
k_j(\sigma)-k_{j'}(\sigma)=N\kappa \quad\text{for some $N\in\mathbb{Z}^+$}\quad\text{for some $j$ and $j'$}.
\]
Lemma \ref{lem:dispersion} establishes $D(i\sigma,k_j(\sigma),0)=0$. We proceed along the lines of the argument in Sections \ref{sec:high freq instability} and \ref{sec:modulational}, investigating the roots of
\[
\text{$D(i\sigma+\delta,k_j(\sigma)+\gamma,\eps)$ for $\delta\in\mathbb{C}$ and $|\delta|\ll1$ for $\gamma, \eps\in\mathbb{R}$ and $|\gamma|,|\eps|\ll1$.}
\]
The developments in Sections \ref{sec:eps=0} and \ref{sec:proj} apply without modification to steady unimodal capillary-gravity waves and bimodal Wilton ripples. One main difference when implementing Section \ref{sec:expansion_monodromy} is that \eqref{def:Bexp} involves the Stokes expansion for Wilton ripples, in Section \ref{sec:bimodal} and Appendix \ref{app:Wilton exp}, rather than that for unimodal waves in Section \ref{sec:unimodal}. 

We emphasize that Wilton ripples of order $M$ take place along the curve $\beta=\beta_W(\kappa;M)$, in \eqref{def:beta_W}, so that $\beta$ is determined by $\kappa$. Section \ref{S1S2region} classifies all resonances along these curves. Rather than examining each resonance individually, we focus on representatives that illustrate the essential features of our spectral analysis. 

\begin{notation*}
Throughout this section, we may write, for instance, $k_j$ and $k_{j'}$ instead of $k_j(\sigma)$ and $k_{j'}(\sigma)$ where there is no confusion.
\end{notation*}

\subsection{Wilton ripples of order \texorpdfstring{$2$}{} and resonances of order \texorpdfstring{$2$}{}}\label{sec:M=2 N=2}

Suppose that \eqref{A:wilton} holds for $M=2$, corresponding to Wilton ripples of order $2$. We assume 
\begin{itemize}
\item[(C1)] $k_j(\sigma)-k_{j'}(\sigma)=2\kappa$ for some $\sigma>0$ for some $j$ and $j'$, where $k_j(\sigma)$ and $k_{j'}(\sigma)$ are in Definition \ref{def:resonance}. 
\end{itemize}
That is, $i\sigma\neq0,\in i\mathbb{R}$ is resonant of order $2$. It follows from the result in Section \ref{S1S2region} that the curve $\beta=\beta_W(\kappa;2)$ in \eqref{def:beta_W} or, equivalently, $\kappa_2=2\kappa_1$, corresponding to Wilton ripples of order $2$, intersects transversely some resonance regions of order $2$, where (C1) holds for the corresponding $j$ and $j'$. See Section \ref{sec:numerics M=2 N=2} for further details. 

After reordering $Y_0(\sigma)$, in Definition \ref{def:proj}, if necessary, we assume 
\begin{itemize}
\item[(C2)] $\boldsymbol{\phi}_j(\sigma)$ and $\boldsymbol{\phi}_{j'}(\sigma)$ are the first and second elements of $Y_0(\sigma)$, where $j$ and $j'$ are in (C1).
\end{itemize}
Additionally, we assume
\begin{itemize}
\item[(C3)] $k_j(\sigma)-k_\ell(\sigma)\neq N\kappa$ for any $N\in\mathbb{Z}$ for any $\ell \neq j,j'$ for which $\boldsymbol{\phi}_\ell(\sigma)\in Y_0(\sigma)$.
\end{itemize}
That is, $k_j(\sigma)$ (or $k_{j'}(\sigma)$) is not resonant with $k_\ell(\sigma)$, $\ell\neq j,j'$.

We begin by evaluating $\mathbf{V}^{(m,n)}(T;\sigma)$, in \eqref{def:Vexp}, for $m+n=0$ and $1$, where $T=2\pi/\kappa$ is the period. It suffices to concentrate on the $2\times 2$ leading principal submatrices of \eqref{def:V(m,n)}. We pause to remark that $\mathbf{V}^{(m,n)}(T;\sigma)$, $m,n=0,1,2,\dots$, are $4\times 4$ or $6\times 6$ matrices.

\begin{lemma}\label{lem:V(m,n) M=2 N=2}
Assume {\rm (C1)}, {\rm (C2)}, and {\rm (C3)}. Then
\begin{align*}
&\mathbf{V}^{(0,0)}(T;\sigma)=\begin{pmatrix}
e^{ik_jT} & 0 & \cdots  \\
0 & e^{ik_{j'}T} & \cdots \\
\vdots & \vdots & \ddots \end{pmatrix}, \\
&\mathbf{V}^{(1,0)}(T;\sigma)=\begin{pmatrix}
v^{(1,0)}_{jj} & 0 & \cdots  \\
0 & v^{(1,0)}_{j'j'} & \cdots \\
\vdots & \vdots & \ddots \end{pmatrix}, \\
&\mathbf{V}^{(0,1)}(T;\sigma)=\begin{pmatrix}
v^{(0,1)}_{jj} & v^{(0,1)}_{jj'} & \cdots \\
v^{(0,1)}_{j'j} & v^{(0,1)}_{j'j'} & \cdots \\
\vdots & \vdots & \ddots \end{pmatrix},
\end{align*}
where $v^{(1,0)}_{jj}, v^{(1,0)}_{j'j'}\neq 0$ are in \eqref{def:V(1,0)sigma}, and $v^{(0,1)}_{jj}$, $v^{(0,1)}_{jj'}$, $v^{(0,1)}_{j'j}$, and $v^{(0,1)}_{j'j'}$ are not necessarily zero. Furthermore, $v^{(0,1)}_{jj'}$ and $v^{(0,1)}_{j'j}$ depend linearly on $\alpha$, where $\alpha$ is, for instance, in \eqref{def:wilton1 M=2}. See also \eqref{def:alpha M=2}.
\end{lemma}

\begin{proof}
We follow the same argument as in the proof of Lemma \ref{lem:V(m,n) N=1}. 

See \eqref{def:V(0,0)sigma} and \eqref{def:V(1,0)sigma} for $\mathbf{V}^{(0,0)}(T;\sigma)$ and $\mathbf{V}^{(1,0)}(T;\sigma)$, respectively. We write the solution to \eqref{eqn:v_k(0):sigma}-\eqref{def:f(sigma)} and \eqref{cond:V(0)}, $m=0$ and $n=1$, as:
\begin{align*}
&v^{(0,1)}_{jj}=e^{ik_{j}T}\Big\langle\int_0^T\mathbf{B}^{(0,1)}(x;\sigma)\boldsymbol{\phi}_{j}(\sigma)~dx,\boldsymbol{\psi}_{j}(\sigma)\Big\rangle_Y,
\intertext{and $v^{(0,1)}_{j'j'}$ by replacing $j$ with $j'$;}
&v^{(0,1)}_{jj'}=e^{ik_{j}T}\Big\langle\int_0^Te^{i(k_{j'}-k_j)x}\mathbf{B}^{(0,1)}(x;\sigma)\boldsymbol{\phi}_{j'}(\sigma)~dx,\boldsymbol{\psi}_{j}(\sigma)\Big\rangle_Y,
\end{align*}
and $v^{(0,1)}_{j'j}$ by interchanging $j$ and $j'$. Here, $\mathbf{B}^{(0,1)}(x;\sigma)$ is in Appendix \ref{app:Bexp}, where $\phi_1$, $\eta_1$, and $\mu_1$ are in \eqref{def:wilton1 M=2} and \eqref{def:alpha M=2}, $\boldsymbol{\phi}_j(\sigma)$ and $\boldsymbol{\phi}_{j'}(\sigma)$ are in Lemma \ref{lem:eps=0}, $\boldsymbol{\psi}_j(\sigma)$ and $\boldsymbol{\psi}_{j'}(\sigma)$ in Section \ref{sec:proj}, and $\langle \cdot \,, \cdot \rangle_Y$ in \eqref{def:inner}. 

Upon inspection, $\mathbf{B}^{(0,1)}(x;\sigma)$ depends linearly on $\s(\kappa x)$, $\c(\kappa x)$, $\s(2\kappa x)$, $\c(2\kappa x)$, together with a constant, with respect to $x$. Consequently, $v^{(0,1)}_{jj}$ and $v^{(0,1)}_{j'j'}$ are not necessarily zero. We emphasize that it depends on a constant because $\mu_1\neq 0$ (see \eqref{def:alpha M=2}). This is in contrast to Lemma \ref{lem:V(m,n) N=1} for steady unimodal capillary-gravity waves, where $\mu_1=0$ and, hence, $v^{(0,1)}_{jj}, v^{(0,1)}_{j'j'}=0$. Moreover, since 
\[
\int_0^T e^{\pm 2i\kappa x}\s(2\kappa x)~dx, \int_0^T e^{\pm 2i\kappa x}\c(2\kappa x)~dx\neq 0,
\]
$v^{(0,1)}_{jj'}$ and $v^{(0,1)}_{j'j}$ are not necessarily zero. Indeed, $v^{(0,1)}_{jj'}$, $v^{(0,1)}_{j'j} \in i\alpha e^{ik_jT}\mathbb{R}$. The formulae for $v^{(0,1)}_{jj}$, $v^{(0,1)}_{jj'}$, $v^{(0,1)}_{j'j}$, and $v^{(0,1)}_{j'j'}$ are omitted due to their excessive length.
\end{proof}

We proceed to a power series expansion of $D(i\sigma+\delta,k_{j}(\sigma)+\gamma,\eps)$ about $(\delta,\gamma,\eps)=(0,0,0)$. Substituting the result of Lemma \ref{lem:V(m,n) M=2 N=2} into \eqref{def:Delta}, \eqref{def:Vexp}, and \eqref{def:V(m,n)}, we make a straightforward calculation to arrive at
\begin{align*}  
D&(i\sigma+\delta,k_{j}(\sigma)+\gamma,\eps)\\
=&d^{(2,0,0)}\delta^2+d^{(0,2,0)}\gamma^2+d^{(0,0,2)}\eps^2 \\
&+d^{(1,1,0)}\delta\gamma+d^{(1,0,1)}\delta\eps+d^{(0,1,1)}\gamma\eps+o((|\delta|+|\gamma|+|\eps|)^2)
\end{align*}
as $\delta,\gamma,\eps\to0$, where $d^{(\ell,m,n)}$ can be expressed in terms of the entries of $\mathbf{V}^{(m,n)}(T;\sigma)$, in Lemma \ref{lem:V(m,n) M=2 N=2}, for $m+n=0,1$. We omit their formulae. We then adapt the argument from Section \ref{sec:N=1} and, after explicit calculations, \eqref{def:disc2} becomes 
\begin{equation}\label{eqn:disc2 M=2 N=2}
\Delta_2(\eps)=\frac{-16}{h(0,0)}\frac{v_{jj'}^{(0,1)}v_{j'j}^{(0,1)}}{v_{jj}^{(1,0)}v_{j'j'}^{(1,0)}}\eps^2+O(\eps^3)
\end{equation}
as $\eps\to 0$, where $v_{jj'}^{(0,1)}$, $v_{j'j}^{(0,1)}$, $v_{jj}^{(1,0)}$, and $v_{j'j'}^{(1,0)}$ are in Lemma \ref{lem:V(m,n) M=2 N=2}, and $h(0,0)<0$. It is noteworthy that $v_{jj'}^{(0,1)}v_{j'j}^{(0,1)} \in \alpha^2\mathbb{R}$, where $\alpha$ is, for instance, in \eqref{def:wilton1 M=2}, while $v_{jj}^{(1,0)}v_{j'j'}^{(1,0)}$ is independent of $\alpha$. Consequently, after normalization by a positive constant, the leading order term in \eqref{eqn:disc2 M=2 N=2} is independent of $\alpha$, whose sign can determine spectral stability or instability. 

We state the main result of this subsection.

\begin{theorem}[Wilton ripples of order $2$ and resonance of order $2$]${}$\label{thm:M=2 N=2}

\noindent Let $\kappa>0$. Suppose that there exists a bimodal Wilton-ripple solution of order $2$ to \eqref{stationary-periodic}-\eqref{eqn:stokes exp} for $\eps\neq 0,\in\mathbb{R}$ and $|\eps|\ll1$, satisfying \eqref{def:wilton1 M=2}. Assume {\rm (C1)}, {\rm (C2)}, {\rm (C3)}, and \eqref{A:transversality}. Recall Section \ref{sec:L}, and consider ${\rm spec}(\mathcal{L}(\eps))$ near $i\sigma \neq 0,\in i\mathbb{R}$. We define
\be \label{def:ind_6}
{\rm ind}_6(\kappa,\sigma;k_{j},k_{j'})=\frac{v_{jj'}^{(0,1)}v_{j'j}^{(0,1)}}{\alpha^2v_{jj}^{(1,0)}v_{j'j'}^{(1,0)}},
\ee
where the right side is as in Lemma \ref{lem:V(m,n) M=2 N=2}, independent of $\alpha$. 

\begin{itemize}
\item[$\bullet$]If ${\rm ind}_6(\kappa,\sigma;k_j,k_{j'})>0$ then there exists $\lambda \in {\rm spec}(\mathcal{L}(\eps))$ near $i\sigma \in i\mathbb{R}$ such that ${\rm Re}\,\lambda \neq 0$ . 
\item[$\bullet$]If ${\rm ind}_6<0$, on the other hand, ${\rm spec}(\mathcal{L}(\eps))$ near $i\sigma$ remains confined to the imaginary axis.
\item[$\bullet$]If ${\rm ind}_6=0$ then spectral stability near $i\sigma$ would be determined by higher order terms in \eqref{def:Vexp}. 
\end{itemize}

\end{theorem}

Therefore, Wilton ripples of order $2$ for sufficiently small amplitude in water of unit depth are spectrally unstable near a nonzero resonance of order $2$ provided that \eqref{def:ind_6} is positive. We remark that \eqref{def:ind_6} is an explicit symbolic expression in $\kappa$ (and, hence, $\beta_W(\kappa;2)$), $\sigma$, $k_j(\sigma)$, and $k_{j'}(\sigma)$ (but independent of $\alpha$), whose formula is omitted due to its excessive length.  

The proof of Theorem \ref{thm:M=2 N=2} mirrors the proof of Theorem \ref{thm:N=1}, and is therefore omitted. 

\begin{corollary}[Unstable spectrum near resonance of order $2$]\label{cor:M=2 N=2} 

Recall Theorem \ref{thm:M=2 N=2}. If ${\rm ind}_6>0$ then $\lambda\in {\rm spec}(\mathcal{L}(\eps))$ near $i\sigma\neq 0,\in i\mathbb{R}$ satisfies 
\[
({\rm Re}\,\lambda)^2+\left(\frac{v^{(1,0)}_{jj}-v^{(1,0)}_{j'j'}}{v^{(1,0)}_{jj}+v^{(1,0)}_{j'j'}}\right)^2\left({\rm Im}\,\lambda-\sigma-\frac{v_{jj}^{(0,1)} -v_{j'j'}^{(0,1)}}{v_{jj}^{(1,0)} - v_{j'j'}^{(1,0)}}\eps\right)^2={\rm ind}_6\alpha^2\eps^2
\]
up to order $\eps$ as $\eps\to 0$. In other words, the unstable spectrum forms an ellipse to leading order, whose center is displaced from $i\sigma$ by $O(\eps)$ along the imaginary axis and the semiaxes are $O(\eps)$ for $\eps\in\mathbb{R}$ and $|\eps|\ll1$.
\end{corollary}

The proof of Corollary \ref{cor:M=2 N=2} is the same as the proof of Corollary \ref{cor:N=1}, and is omitted.

Lemma \ref{lem:disc2even} establishes that ${\rm spec}(\mathcal{L}(\eps))$ remains invariant under $\eps\mapsto -\eps$ for Wilton ripples of orders $2M+1$, $M\in\mathbb{Z}^+$. This no longer holds for Wilton ripples of order $2$. Indeed, the unstable spectrum of $\mathcal{L}(-\eps)$ forms an ellipse at leading order, but its center is displaced in the direction opposite to the direction in Corollary \ref{cor:M=2 N=2}. 

\subsection{Numerical studies}\label{sec:numerics M=2 N=2}

We numerically investigate the result of Section \ref{sec:M=2 N=2}. Observe that Wilton ripples of order $2$ occur along the curve $\beta=\beta_W(\kappa;2)$ in \eqref{def:beta_W}, that is, the solid curve labeled $\kappa_2=2\kappa_1$ in the right panel of Figure \ref{figure7}. It follows from the result in Section \ref{S1S2region} that the curve intersects transversely the resonance regions (R1), (R7), (R13), (R14), and (R15) for order $2$, where (C1) holds for the corresponding $j$ and $j'$. We numerically evaluate \eqref{def:ind_6} and apply Theorem \ref{thm:M=2 N=2} to numerically determine spectral stability or instability. 

We summarize our numerical findings.

\subsubsection*{\rm (R1) $k_1(\sigma)-k_5(\sigma)=2\kappa$} If $\kappa \in (2.811980\dots,\infty)$, then the curve $\beta=\beta_W(\kappa;2)$, corresponding to Wilton ripples of order $2$, lies in the resonance region (R1) for $N=2$.  
Our numerical evaluation of \eqref{def:ind_6} suggests 
\[
{\rm ind}_6(\kappa,\sigma;k_1,k_5)<0
\]
for $\kappa$ throughout the interval and for $\sigma\in(0,-\sigma_{c,1})$ corresponding to the resonance. Theorem \ref{thm:M=2 N=2} therefore predicts that ${\rm spec}(\mathcal{L}(\eps))$ near $i\sigma \neq 0,\in i\mathbb{R}$ remains confined to the imaginary axis for $\eps\in\mathbb{R}$ and $|\eps|\ll1$, implying spectral stability near $i\sigma$.

\subsubsection*{\rm (R13) $k_3(\sigma)-k_5(\sigma)=2\kappa$} 

If $\kappa\in (1.662761\ldots,2.811980\ldots)$, then the corresponding segment of the curve $\beta=\beta_W(\kappa;2)$ lies in the resonance region (R13) for $N=2$. Our numerical computations indicate ${\rm ind}_6(\kappa,\sigma;k_3,k_5)<0$ for $\kappa$ in the interval for $\sigma\in (0,-\sigma_{c,1})$ and corresponding to the resonance. Theorem \ref{thm:M=2 N=2} then implies spectral stability near $i\sigma$.

\subsubsection*{\rm (R15) $k_6(\sigma)-k_3(\sigma)=2\kappa$}  

If $\kappa\in(1.370646\ldots,1.662761\dots)$, then the segment of the curve $\beta=\beta_W(\kappa;2)$ lies in the resonance region (R15) for $N=2$. Our numerical result indicates ${\rm ind}_6(\kappa,\sigma;k_6,k_3)>0$ for $\kappa$ in the interval for $\sigma\in(0,-\sigma_{c,1})$ corresponding to the resonance. Theorem \ref{thm:M=2 N=2} therefore predicts that there exists $\lambda \in {\rm spec}(\mathcal{L}(\eps))$ near $i\sigma$ such that ${\rm Re}\,\lambda \neq0$ for $\eps\neq 0,\in \mathbb{R}$ and $|\eps|\ll1$, implying spectral instability near $i\sigma$.

\subsubsection*{\rm (R7) $k_2(\sigma)-k_3(\sigma)=2\kappa$} 
If $\kappa\in(1.361869\ldots,1.370646\ldots)$, then the curve $\beta=\beta_W(\kappa;2)$ lies in the resonance region  (R7) for $N=2$. The numerical result confirms ${\rm ind}_6(\kappa,\sigma;k_2,k_3)>0$ for $\kappa$ in the interval for $\sigma\in (0,\min(-\sigma_{c,1},\sigma_{c,2}))$ corresponding to the resonance, that is, spectral instability near $i\sigma$.

\subsubsection*{\rm (R14) $k_2(\sigma)-k_1(\sigma)=2\kappa$}  
If $\kappa\in(1.285234\ldots,1.361869\ldots)$, then the curve $\beta=\beta_W(\kappa;2)$ lies in the resonance region  (R14) for $N=2$. Numerical evidence supports ${\rm ind}_6(\kappa,\sigma;k_2,k_1)>0$ for $\kappa$ in the interval for $\sigma\in (0,\sigma_{c,2})$ corresponding to the resonance, that is, spectral instability near $i\sigma$.

\begin{remark*}\rm 
We do not numerically detect spectral instability near the resonances (R1) and (R13). This is compatible with the prediction based on Krein signature. See Section \ref{resonance_summary} and Remark \ref{rem:Krein} for further discussion.
\end{remark*}

\subsection{Wilton ripples of orders \texorpdfstring{$M\geq 3$}{Lg} and resonances of orders \texorpdfstring{$1$}{}, \texorpdfstring{$M$}{}}\label{sec:M=3 N=1,M} 

Suppose that \eqref{A:wilton} holds for $M\geq 3, \in\mathbb{Z}^+$. We assume 
\begin{itemize}
\item[(C1')] $k_j(\sigma)-k_{j'}(\sigma)=N\kappa$ for some $\sigma>0$ for some $N\in\mathbb{Z}^+$ and $j$, $j'$, where $k_j(\sigma)$ and $k_{j'}(\sigma)$ are in Definition \ref{def:resonance}. 
\end{itemize} 
Recall Definition \ref{def:proj} and we assume
\begin{itemize}
\item[(C2')] $Y_0(\sigma)=\{ \boldsymbol{\phi}_j(\sigma), \boldsymbol{\phi}_{j'}(\sigma),\dots \}$, where $j$ and $j'$ are in (C1'), and $k_j(\sigma)-k_\ell(\sigma)\neq N'\kappa$ for any $N'\in\mathbb{Z}$ for any $\ell\neq j,j'$ for which $ \boldsymbol{\phi}_\ell(\sigma)\in Y_0(\sigma)$,
\end{itemize}
so that it suffices to concentrate on the $2\times 2$ leading principal submatrices of \eqref{def:V(m,n)}. We pause to remark that $\mathbf{V}^{(m,n)}(T;\sigma)$, $m,n=0,1,2,\dots$ are square matrices of order $4$ or $6$.

\begin{lemma}\label{lem:V(m,n) M=3 N=1,M}
Assume {\rm (C1')} and {\rm (C2')}. Then
\begin{align*}
&\mathbf{V}^{(0,0)}(T;\sigma)=\begin{pmatrix}
e^{ik_jT} & 0 & \cdots  \\
0 & e^{ik_{j'}T} & \cdots \\
\vdots & \vdots & \ddots \end{pmatrix}, \\
&\mathbf{V}^{(1,0)}(T;\sigma)=\begin{pmatrix}
v^{(1,0)}_{jj} & 0 & \cdots  \\
0 & v^{(1,0)}_{j'j'} & \cdots \\
\vdots & \vdots & \ddots \end{pmatrix}, \\
&\mathbf{V}^{(0,1)}(T;\sigma)=\begin{pmatrix}
0 & v^{(0,1)}_{jj'} & \cdots \\
v^{(0,1)}_{j'j} & 0 & \cdots \\
\vdots & \vdots & \ddots \end{pmatrix},
\end{align*}
where $v^{(1,0)}_{jj}, v^{(1,0)}_{j'j'}\neq 0$ are in \eqref{def:V(1,0)sigma}, and $v^{(0,1)}_{jj'}$ and $v^{(0,1)}_{j'j}$ are not necessarily zero. Furthermore,
\[
v^{(0,1)}_{jj'}, v^{(0,1)}_{j'j}\,\begin{cases}
\text{are independent of $\alpha$} & \text{if $N=1$}\\
\text{depend linearly on $\alpha$}&\text{if $N=M$},
\end{cases}
\]
where $\alpha$ is, for instance, in \eqref{def:wilton1 M=3}. See also \eqref{eqn:alpha M=3} for $M=3$.
\end{lemma}

\begin{proof}
We follow the same argument as in the proofs of Lemmas \ref{lem:V(m,n) N=1} and \ref{lem:V(m,n) M=2 N=2}. An inspection of Appendix \ref{app:Bexp} and Section \ref{sec:bimodal} reveals that $\mathbf{B}^{(0,1)}(x;\sigma)$ depends linearly on $\s(\kappa x)$, $\c(\kappa x)$, $\s(M\kappa x)$, and $\c(M\kappa x)$, with respect to $x$, and does not depend on a constant because $\mu_1=0$. Consequently, $v^{(0,1)}_{jj}, v^{(0,1)}_{j'j'}=0$. On the other hand, $v^{(0,1)}_{jj'}$ and $v^{(0,1)}_{j'j}$ are not necessarily zero if $N=1$ or $M$. The formulae for $v^{(0,1)}_{jj'}$ and $v^{(0,1)}_{j'j}$ are omitted due to their excessive length.
\end{proof}

See Lemma \ref{lem:V(m,n) M=3 N=1,M2} if $N\neq 1,M$.

In the remainder of this subsection, we assume (C1') for $N=1$ or $M$. 
Substituting the result of Lemma \ref{lem:V(m,n) M=3 N=1,M} into \eqref{def:Delta}, \eqref{def:Vexp}, and \eqref{def:V(m,n)}, we make a straightforward calculation to arrive at
\begin{align*}  
D&(i\sigma+\delta,k_{j}(\sigma)+\gamma,\eps)\\
=&d^{(2,0,0)}\delta^2+d^{(0,2,0)}\gamma^2+d^{(0,0,2)}\eps^2+d^{(1,1,0)}\delta\gamma+o((|\delta|+|\gamma|+|\eps|)^2)
\end{align*}
as $\delta,\gamma,\eps\to0$, where $d^{(\ell,m,n)}$ can be expressed in terms of the entries of $\mathbf{V}^{(m,n)}(T;\sigma)$, in Lemma \ref{lem:V(m,n) M=3 N=1,M}, for $m+n=0$ and $1$. We omit their formulae. We emphasize that the right side is the same as that of \eqref{eqn:D(l,m,n) N=1}. We then adapt the argument from Section \ref{sec:N=1} and make an explicit calculation to derive a spectral stability index whose sign determines stability or instability. 

We state the main results of this subsection.

\begin{theorem}[Wilton ripples of order $\geq 3$ and resonance of order $1$]${}$\label{thm:M=3 N=1}

\noindent Let $\kappa>0$. Suppose that there exists a bimodal Wilton-ripple solution of order $M\geq 3,\in \mathbb{Z}^+$ to \eqref{stationary-periodic}-\eqref{eqn:stokes exp} for $\eps\neq 0,\in\mathbb{R}$ and $|\eps|\ll1$, satisfying \eqref{def:wilton1 M=3}. Assume {\rm (C1')} for $N=1$, {\rm (C2')}, and \eqref{A:transversality}. Recall Section \ref{sec:L}, and consider ${\rm spec}(\mathcal{L}(\eps))$ near $i\sigma \neq 0,\in i\mathbb{R}$. Recall \eqref{def:ind_1}, and consider
\be \label{def:ind_7}
{\rm ind}_1(\beta_W(\kappa;M),\kappa,\sigma;k_{j},k_{j'})
=\frac{v^{(0,1)}_{jj'}v^{(0,1)}_{j'j}}{v_{jj}^{(1,0)}v_{j'j'}^{(1,0)}},
\ee  
where the right side is as in Lemma \ref{lem:V(m,n) M=3 N=1,M}. See also Lemma \ref{lem:V(m,n) N=1}.

\begin{itemize}
\item[$\bullet$] If ${\rm ind}_1(\beta_W(\kappa;M),\kappa,\sigma;k_{j},k_{j'})>0$ then there exists $\lambda \in {\rm spec}(\mathcal{L}(\eps))$ near $i\sigma \in i\mathbb{R}$ such that ${\rm Re}\,\lambda \neq0$. 

\item[$\bullet$] If ${\rm ind}_1<0$, on the other hand, ${\rm spec}(\mathcal{L}(\eps))\subset i\mathbb{R}$ near $i\sigma$.

\item[$\bullet$] If ${\rm ind}_1=0$ then spectral stability near $i\sigma$ would be determined by higher order terms in \eqref{def:Vexp}.
\end{itemize}

\end{theorem}

Therefore, Wilton ripples of order $\geq 3$ for sufficiently small amplitude in water of unit depth are spectrally unstable near a nonzero resonance of order $1$ provided that \eqref{def:ind_7} is positive. The formula for \eqref{def:ind_7} is omitted due to its excessive length.

The proof of Theorem \ref{thm:M=3 N=1} is the same as the proof of Theorem \ref{thm:N=1}, and is therefore omitted. 
Notably, the analysis does not involve $\alpha$, which is, for instance, in \eqref{def:wilton1 M=3}. 

Additionally, Corollary \ref{cor:N=1} extends without essential modification to Wilton ripples of order $\geq 3$ and a nonzero resonance of order $1$. 

\begin{theorem}[Wilton ripples of order $M$ and resonances of order $M$]${}$\label{thm:M=3 N=M}

\noindent Let $\kappa>0$, and suppose that there exists a bimodal Wilton-ripple solution of order $M\geq 3,\in \mathbb{Z}^+$ to \eqref{stationary-periodic}-\eqref{eqn:stokes exp} for $\eps\neq 0,\in\mathbb{R}$ and $|\eps|\ll1$, satisfying \eqref{def:wilton1 M=3}. Assume {\rm (C1')} for $N=M$, {\rm (C2')}, and \eqref{A:transversality}. Consider ${\rm spec}(\mathcal{L}(\eps))$ near $i\sigma \neq 0,\in i\mathbb{R}$. Then, the conclusions of Theorems \ref{thm:M=2 N=2} and \ref{thm:M=3 N=1} remain valid for
\begin{equation}\label{def:ind_7'}
{\rm ind}_7(\kappa,\sigma;k_{j},k_{j'})
=\frac{v^{(0,1)}_{jj'}v^{(0,1)}_{j'j}}{\alpha^2v_{jj}^{(1,0)}v_{j'j'}^{(1,0)}},
\end{equation}
where the right side is in Lemma \ref{lem:V(m,n) M=3 N=1,M}, independent of $\alpha$.
\end{theorem}

The proof of Theorem \ref{thm:M=3 N=M} mirrors the proof of Theorem \ref{thm:N=1}, and is therefore omitted. See also Section \ref{sec:M=2 N=2} and, particularly, Theorem \ref{thm:M=2 N=2}. We pause to remark that we obtain \eqref{eqn:disc2 M=2 N=2}, where $v^{(0,1)}_{jj'}v^{(0,1)}_{j'j} \in \alpha^2\mathbb{R}$ while $v^{(1,0)}_{jj}v^{(1,0)}_{j'j'}$ is independent of $\alpha$, for instance, in \eqref{def:wilton1 M=3}. Consequently, after normalization by a positive constant, the leading order term in \eqref{eqn:disc2 M=2 N=2} defines a spectral stability index, which is independent of $\alpha$. This is particularly useful because the computation of $\alpha$ becomes increasingly unwieldy for higher values of $M$. See Section \ref{sec:bimodal} for further discussion.

\begin{corollary}[Unstable spectrum near resonance of order $M$]\label{cor:N=3 M=N} 

Recall Theorem \ref{thm:M=3 N=M}. If ${\rm ind}_7>0$ then $\lambda \in {\rm spec}(\mathcal{L}(\eps))$ near $i\sigma\neq0,\in i\mathbb{R}$ satisfies
\[
({\rm Re}\lambda)^2+\left(\frac{v^{(1,0)}_{jj}-v^{(1,0)}_{j'j'}}{v^{(1,0)}_{jj}+v^{(1,0)}_{j'j'}}\right)^2({\rm Im}\lambda-\sigma)^2={\rm ind}_7\alpha^2\,\eps^2
\]
up to order $\eps$ as $\eps\to0$. In other words, the unstable spectrum forms an ellipse to leading order, whose center is $i\sigma$ and the semiaxes are $O(\eps)$ for $\eps\in\mathbb{R}$ and $|\eps|\ll1$.
\end{corollary}

The proof of Corollary \ref{cor:N=3 M=N} is the same as the proof of Corollary \ref{cor:N=1}, and is omitted.

\subsubsection*{Numerical studies}

We conclude this subsection by numerically investigating Theorems \ref{thm:M=3 N=1} and \ref{thm:M=3 N=M}.

Consider Wilton ripples of order $M\geq 3,\in\mathbb{Z}^+$ and resonances of order $1$. It follows from the result in Section \ref{S1S2region} that the curve $\kappa_2=M\kappa_1$, or, equivalently, $\beta=\beta_W(\kappa;M)$, in \eqref{def:beta_W}, corresponding to Wilton ripples of order $M$, lies entirely in the resonance region (R2) of order $1$, where 
\[
k_6(\sigma)-k_2(\sigma)=\kappa\quad\text{for some $\sigma\in(0,\infty)$},
\]
and the curve does not intersect other resonance regions of order $1$. See the right panel of Figure \ref{figure7}. Our numerical evaluation of \eqref{def:ind_7} suggests 
\[
{\rm ind}_1(\beta_W(\kappa;M),\kappa,\sigma;k_6,k_2)<0
\]
for $\kappa\in(0,\infty)$ and for $\sigma$ corresponding to the resonance. Theorem \ref{thm:M=3 N=1} therefore predicts that ${\rm spec}(\mathcal{L}(\eps))$ near $i\sigma\neq 0,\in i\mathbb{R}$ remains confined to $i\mathbb{R}$ for $\eps\in\mathbb{R}$ and $|\eps|\ll1$, that is, spectral stability near $i\sigma$.

Considering Wilton ripples of order $M\geq 3,\in\mathbb{Z}^+$ and resonances of order $M$, it follows from the result in Section \ref{S1S2region} that the curve $\kappa_2=M\kappa_1$ or, equivalently, $\beta=\beta_W(\kappa;M)$, intersects transversely the resonance regions (R1), (R6), (R13), and (R15) for $N=M$, where (C1') holds for the corresponding $j$ and $j'$. We numerically evaluate \eqref{def:ind_7'} and apply Theorem \ref{thm:M=3 N=M} to numerically determine spectral stability or instability. We pause to remark that \eqref{def:ind_7'} is independent of $\alpha$, for instance, in \eqref{def:wilton1 M=3}.

\medskip

Below we summarize our numerical findings for $M=3$.

\subsubsection*{\rm (R1) $k_1(\sigma)-k_5(\sigma)=3\kappa$} 

If $\kappa\in(2.995857\ldots,\infty)$, then the curve $\beta=\beta_W(\kappa;3)$, corresponding to Wilton ripples of order $3$, lies in the region (R1) for $N=3$. Our numerical computation indicates ${\rm ind}_7(\kappa,\sigma;k_1,k_5)<0$ for $\kappa$ in the interval for $\sigma \in (0,-\sigma_{c,1})$ corresponding to the resonance. Theorem \ref{thm:M=3 N=M} then predicts spectral stability near $i\sigma$.

\subsubsection*{\rm (R13) $k_3(\sigma)-k_5(\sigma)=3\kappa$} If $\kappa\in (1.894994\ldots,2.995857\ldots)$, then the curve $\beta=\beta_W(\kappa;3)$ lies in the resonance region (R13) for $N=3$. Our numerical result confirms ${\rm ind}_7(\kappa,\sigma;k_3,k_5)<0$ for $\kappa$ in the interval for $\sigma \in (0,-\sigma_{c,1})$ corresponding to the resonance, implying spectral stability near $i\sigma$.

\subsubsection*{\rm (R6) $k_6(\sigma)-k_1(\sigma)=3\kappa$} If $\kappa\in(2.238016\ldots,\infty)$, then the curve $\beta=\beta_W(\kappa;3)$ lies in the resonance region (R6) for $N=3$. Numerical evidence supports ${\rm ind}_7(\kappa,\sigma;k_6,k_1)>0$ for $\kappa$ in the interval for $\sigma \in (0,-\sigma_{c,1})$ corresponding to the resonance. Theorem \ref{thm:M=3 N=M} then predicts spectral instability near $i\sigma$.

\subsubsection*{\rm (R15) $k_6(\sigma)-k_3(\sigma)=3\kappa$} If $\kappa\in(1.864984\ldots, 2.238016\ldots)$, then the curve $\beta=\beta_W(\kappa;3)$ lies in the resonance region (R15) for $N=3$. Numerical evidence supports ${\rm ind}_7(\kappa,\sigma;k_6,k_3)>0$ for $\kappa$ in the interval for $\sigma \in $ $ (0,\min(-\sigma_{c,1},\sigma_{c,2}))$ corresponding to the resonance, that is, spectral instability near $i\sigma$.

\medskip

We proceed to $M=N=4$.

\subsubsection*{\rm (R1) $k_1(\sigma)-k_5(\sigma)=4\kappa$} If $\kappa\in(3.101532\ldots,\infty)$, then the curve $\beta=\beta_W(\kappa;4)$, corresponding to Wilton ripples for $M=4$, lies in the resonance region (R1) for $N=4$. The numerical computation indicates ${\rm ind}_7(\kappa,\sigma;k_1,k_5)<0$ for $\kappa$ in the interval for $\sigma \in (0,-\sigma_{c,1})$ corresponding to the resonance. Theorem \ref{thm:M=3 N=M} predicts spectral stability near $i\sigma$.

\subsubsection*{\rm (R13) $k_3(\sigma)-k_5(\sigma)=4\kappa$} If $\kappa\in(1.975278\ldots,3.101532\ldots)$, then the curve $\beta=\beta_W(\kappa;4)$ lies in the resonance region (R13) for $N=4$. The numerical result confirms ${\rm ind}_7(\kappa,\sigma;k_3,k_5)<0$ for $\kappa$ in the interval for $\sigma \in (0,-\sigma_{c,1})$ corresponding to the resonance, implying spectral stability near $i\sigma$.

\subsubsection*{\rm (R6) $k_6(\sigma)-k_1(\sigma)=4\kappa$} If $\kappa\in(2.538180\ldots,\infty)$, then the curve $\beta=\beta_W(\kappa;4)$ lies in the resonance region (R6) for $N=4$. Numerical evidence supports ${\rm ind}_7(\kappa,\sigma;k_6,k_1)>0$ for $\kappa$ in the interval for $\sigma \in (0,-\sigma_{c,1})$ corresponding to the  resonance, implying spectral instability near $i\sigma$.

\subsubsection*{\rm (R15) $k_6(\sigma)-k_3(\sigma)=4\kappa$} If $\kappa\in(1.963115\ldots, 2.538180\ldots)$, then the curve $\beta=\beta_W(\kappa;4)$ lies in the resonance region (R15) for $N=4$. Numerical evidence supports ${\rm ind}_7(\kappa,\sigma;k_6,k_3)>0$ for $\kappa$ in the interval for $\sigma \in $ $(0,\min(-\sigma_{c,1},\sigma_{c,2}))$ corresponding to the resonance, that is, spectral instability near $i\sigma$.

\medskip

We continue to $M=N=5$. 

\subsubsection*{\rm (R1) $k_1(\sigma)-k_5(\sigma)=5\kappa$} If $\kappa\in(3.173682\ldots,\infty)$, then the Wilton-ripple curve $\beta=\beta_W(\kappa;5)$ lies in the resonance region (R1) for $N=5$. The numerical computation reveals ${\rm ind}_7(\kappa,\sigma;k_1,k_5)<0$ for $\kappa$ in the interval for $\sigma \in (0,-\sigma_{c,1})$ corresponding to the resonance, implying spectral stability near $i\sigma$.

\subsubsection*{\rm (R13) $k_3(\sigma)-k_5(\sigma)=5\kappa$} If $\kappa\in (2.018839\ldots,3.173682\ldots)$, then the curve $\beta=\beta_W(\kappa;5)$ lies in the resonance region (R13) for $N=5$. The numerical result confirms ${\rm ind}_7(\kappa,\sigma;k_3,k_5)<0$ for $\kappa$ in the interval for $\sigma \in (0,-\sigma_{c,1})$ for the resonance, implying spectral stability near $i\sigma$.

\subsubsection*{\rm (R6) $k_6(\sigma)-k_1(\sigma)=5\kappa$} If $\kappa\in(2.719131\ldots,\infty)$, then the curve $\beta=\beta_W(\kappa;5)$ lies in the resonance region (R6) for $N=5$. Numerical evidence supports ${\rm ind}_7(\kappa,\sigma;k_6,k_1)>0$ for $\kappa$ in the interval for $\sigma \in (0,-\sigma_{c,1})$ for the resonance, implying spectral instability near $i\sigma$.

\subsubsection*{\rm (R15) $k_6(\sigma)-k_3(\sigma)=5\kappa$} If $\kappa\in(2.012186\ldots,2.719131\ldots)$, then the curve $\beta=\beta_W(\kappa;5)$ lies in the resonance region (R15) for $N=5$. Numerical evidence supports ${\rm ind}_7(\kappa,\sigma;k_6,k_3)>0$ for $\kappa$ in the interval for $\sigma \in$ $(0,\min(-\sigma_{c,1},\sigma_{c,2}))$ for the resonance, that is, spectral instability near $i\sigma$.

\begin{remark*}\rm 
We do not numerically detect instability near the resonances (R1) and (R13) for orders $3$, $4$, and $5$, which is consistent with the prediction based on Krein signature. See Remark \ref{rem:Krein}.
\end{remark*}

\subsection{Wilton ripples of order \texorpdfstring{$M\geq 3$}{} and resonances of order \texorpdfstring{$\neq 1,M$}{}}\label{sec:M=3 1<N<M} 

We continue to assume that \eqref{A:wilton} holds for $M\geq 3,\in\mathbb{Z}^+$. Throughout this subsection, we assume that Section \ref{sec:M=3 N=1,M} (C1') holds for $N\neq 1,M$, and (C2'). It suffices to concentrate on the $2\times 2$ leading principal submatrices of \eqref{def:V(m,n)}. We relabel the eigenvalues and eigenfunctions of $\mathbf{L}(i\sigma)$ and $\mathbf{L}^\dag(i\sigma)$ and, correspondingly, the entries of \eqref{def:V(m,n)}, so that
\[
\mathbf{V}^{(m,n)}(T;\sigma)=\begin{pmatrix}
w^{(m,n)}_{11} & w^{(m,n)}_{12} & \cdots \\
w^{(m,n)}_{21} & w^{(m,n)}_{22} & \cdots \\
\vdots & \vdots & \ddots \end{pmatrix}.
\]
The statements of Lemma \ref{lem:V(m,n) M=3 N=1,M} then become 
\begin{equation}\label{eqn:V(m,n) M=3 N=1,M}
\mathbf{V}^{(0,0)}(T;\sigma)=\begin{pmatrix}
e^{ik_1T} & 0 & \cdots  \\
0 & e^{ik_{2}T} & \cdots \\
\vdots & \vdots & \ddots \end{pmatrix}, \,\,
\mathbf{V}^{(1,0)}(T;\sigma)=\begin{pmatrix}
w^{(1,0)}_{11} & 0 & \cdots  \\
0 & w^{(1,0)}_{22} & \cdots \\
\vdots & \vdots & \ddots \end{pmatrix},
\end{equation}
where $w^{(1,0)}_{11}, w^{(1,0)}_{22}\neq 0$ are in \eqref{def:V(1,0)sigma}. 

\begin{lemma}\label{lem:V(m,n) M=3 N=1,M2}
Assume {\rm (C1')} for $N\neq 1,M$, and {\rm (C2')}. Then 
\[
\mathbf{V}^{(0,1)}(T;\sigma)=\begin{pmatrix}
0 & 0 & \cdots  \\
0 & 0 & \cdots \\
\vdots & \vdots & \ddots \end{pmatrix} \quad\text{and}\quad 
\mathbf{V}^{(0,2)}(T;\sigma)=\begin{pmatrix}
w^{(0,2)}_{11} & w^{(0,2)}_{12} & \cdots \\
w^{(0,2)}_{21} & w^{(0,2)}_{22} & \cdots \\
\vdots & \vdots & \ddots \end{pmatrix},
\]
where $w^{(0,2)}_{11}$ and $w^{(0,2)}_{22}$ are not necessarily zero, while $w^{(0,2)}_{12}$ and $w^{(0,2)}_{21}$ are not necessarily zero if $N=2,M\pm1,2M$, and zero otherwise.
\end{lemma}

\begin{proof}
We follow the lines of the argument in the proofs of \cite[Lemmas 6.3 and 6.4]{HY2023}. 
Throughout the proof, $\sigma$ is suppressed for the sake of notation. Recall the shorthand notation $\s(\cdot)=\sin(\cdot)$ and $\c(\cdot)=\cos(\cdot)$.

A straightforward calculation reveals 
\ba\label{eqn:a01 M=3}
w_{jj'}^{(0,1)}(x)=&C_{1}e^{i(k_{j'}+\kappa)x}+C_{2}e^{i(k_{j'}-\kappa)x}
+\alpha C_{3}e^{i(k_{j'}+M\kappa)x}+\alpha C_{4}e^{i(k_{j'}-M\kappa)x}\\
&-(C_{1}+C_{2}+\alpha C_{3}+\alpha C_{4})e^{ik_{j}x},
\ea
$j,j'=1,2$, for some constants $C_1,\ldots, C_4>0$, independent of $\alpha$. Clearly, $w^{(0,1)}_{jj'}(T)=0$, $j,j'=1,2$.

We write the solution to \eqref{eqn:v_k(0):sigma}-\eqref{def:f(sigma)} and \eqref{cond:V(0)}, $m=0$ and $n=2$, as
\begin{align*} 
w_{jj'}^{(0,2)}
=&e^{ik_{j}T}\Big\langle\int_0^Te^{i(k_{j'}-k_{j})x}\mathbf{B}^{(0,2)}(x)\boldsymbol{\phi}_{j'}~dx,\boldsymbol{\psi}_{j}\Big\rangle\\
&+e^{ik_{j}T}\Big\langle\int_0^T\mathbf{B}^{(0,1)}(x)e^{-ik_{j}x}\Big(\mathbf{w}^{(0,1)}_{j'}(x)+\sum_{\ell=1,2}w_{\ell j'}^{(0,1)}(x)\boldsymbol{\phi}_{\ell}\Big)~dx,\boldsymbol{\psi}_{j}\Big\rangle,
\end{align*}
$j,j'=1,2$. Here $\mathbf{B}^{(0,1)}(x)$ and $\mathbf{B}^{(0,2)}(x)$ are in Appendix \ref{app:Bexp}, where $\phi_1$, $\eta_1$, $\mu_1$, $\phi_2$, $\eta_2$, and $\mu_2$ are collected in Section \ref{sec:bimodal} and Appendix \ref{app:Wilton exp}. Moreover, $\mathbf{w}_{j'}^{(0,1)}(x)$ solves \eqref{eqn:w_k}-\eqref{eqn:w_k bd}, whose formulae are omitted. We pause to remark that $\mathbf{w}_{j'}^{(0,1)}(x)$ takes the form \eqref{eqn:a01 M=3} with respect to $x$. Additionally, $\boldsymbol{\phi}_j$ and $\boldsymbol{\phi}_{j'}$ are in Lemma \ref{lem:eps=0}, $\boldsymbol{\psi}_j$ and $\boldsymbol{\psi}_{j'}$ in Section \ref{sec:proj}, and $\langle \cdot \,, \cdot \rangle$ in \eqref{def:inner}. 

Suppose $j=j'$. An inspection of Section \ref{sec:bimodal} and Appendices \ref{app:Wilton exp} and \ref{app:Bexp} reveals that $\mathbf{B}^{(0,2)}(x)\boldsymbol{\phi}_{j}$ depends linearly on
\[
1, \quad \begin{Bmatrix}\s \\ \c \end{Bmatrix}(2\kappa x), \quad \begin{Bmatrix}\s \\ \c \end{Bmatrix}(2M\kappa x), \quad
\begin{Bmatrix}\s \\ \c \end{Bmatrix}((1\pm M)\kappa x),
\]
with respect to $x$. Upon inspection, $\mathbf{B}^{(0,1)}(x)e^{-ik_{j}x}\big(\mathbf{w}^{(0,1)}_{j}(x)+\sum w_{\ell j}^{(0,1)}(x)\boldsymbol{\phi}_{\ell}\big)$ is a linear combination of
\begin{gather*}
1, \quad \begin{Bmatrix}\s \\ \c \end{Bmatrix}(\kappa x),\quad 
\begin{Bmatrix}\s \\ \c \end{Bmatrix}(2\kappa x),\quad
\begin{Bmatrix}\s \\ \c \end{Bmatrix}(M\kappa x), \quad
\begin{Bmatrix}\s \\ \c \end{Bmatrix}(2M\kappa x),\\ 
\begin{Bmatrix}\s \\ \c \end{Bmatrix}((1\pm M)\kappa x),\quad
\begin{Bmatrix}\s \\ \c \end{Bmatrix}((1\pm N)\kappa x), \quad
\begin{Bmatrix}\s \\ \c \end{Bmatrix}((M\pm N)\kappa x),
\end{gather*}
with respect to $x$. 
Since $\int^T_0 dx\neq0$, consequently, 
$w^{(0,2)}_{11}$ and $w^{(0,2)}_{22}$ are not necessarily zero. We omit their formulae.

Suppose $j=1$ and $j'=2$. An inspection of Section \ref{sec:bimodal} and Appendices \ref{app:Wilton exp} and \ref{app:Bexp} reveals that $e^{-iN\kappa x}\mathbf{B}^{(0,2)}(x)\boldsymbol{\phi}_{2}$ is a linear combination of
\begin{gather*}
\begin{Bmatrix}\s \\ \c \end{Bmatrix}(N\kappa x),\quad 
\begin{Bmatrix}\s \\ \c \end{Bmatrix}((2\pm N)\kappa x),\quad 
\begin{Bmatrix}\s \\ \c \end{Bmatrix}((2M\pm N)\kappa x),\\
\begin{Bmatrix}\s \\ \c \end{Bmatrix}((1+ M\pm N)\kappa x),\quad 
\begin{Bmatrix}\s \\ \c \end{Bmatrix}((1- M\pm N)\kappa x),
\end{gather*}
with respect to $x$. Upon inspection, $\mathbf{B}^{(0,1)}(x)e^{-ik_{1}x}(\mathbf{w}^{(0,1)}_2(x)+\sum w_{\ell 2}^{(0,1)}(x)\boldsymbol{\phi}_{\ell})$ is a linear combination of
\begin{gather*}
\begin{Bmatrix}\s \\ \c \end{Bmatrix}(\kappa x),\quad 
\begin{Bmatrix}\s \\ \c \end{Bmatrix}(M\kappa x),\quad
\begin{Bmatrix}\s \\ \c \end{Bmatrix}(N\kappa x),\quad
\begin{Bmatrix}\s \\ \c \end{Bmatrix}((M\pm N)\kappa x),\\
\begin{Bmatrix}\s \\ \c \end{Bmatrix}((1\pm N)\kappa x),\quad 
\begin{Bmatrix}\s \\ \c \end{Bmatrix}((2\pm N)\kappa x),\quad
\begin{Bmatrix}\s \\ \c \end{Bmatrix}((2M\pm N)\kappa x),\\
\begin{Bmatrix}\s \\ \c \end{Bmatrix}((1-M\pm N)\kappa x),\quad 
\begin{Bmatrix}\s \\ \c \end{Bmatrix}((1+M\pm N)\kappa x),
\end{gather*}
with respect to $x$. Recall $N\neq 1,M$.
\begin{itemize}
\item[$\bullet$] If $N=2$ then $\int^T_0 \c((2-N)\kappa x)~dx=T\neq 0$.
\item[$\bullet$] If $N=2M$ then $\int^T_0 \c((2M-N)\kappa x)~dx=T\neq 0$.
\item[$\bullet$] If $N=M-1$ then $\int^T_0 \c((1-M+N)\kappa x)~dx=T\neq 0$.
\item[$\bullet$] If $N=M+1$ then $\int^T_0\c((1+M-N)\kappa x)~dx=T\neq 0$.
\end{itemize}
Consequently, 
\[
\int^T_0 e^{-iN\kappa x}\mathbf{B}^{(0,2)}(x)\boldsymbol{\phi}_{2}~dx
+\int^T_0 \mathbf{B}^{(0,1)}(x)e^{-ik_{1}x}(\mathbf{w}^{(0,1)}_2(x)+\sum w_{\ell 2}^{(0,1)}(x)\boldsymbol{\phi}_{\ell})~dx
\]
is not necessarily zero when $N=2$, $2M$, or $M\pm 1$, and zero otherwise. Suppose, on the other hand, $j=2$ and $j'=1$, and we can treat analogously. We omit the details. Therefore, $w^{(0,2)}_{12}$ and $w^{(0,2)}_{21}$ are not necessarily zero if $N=2$, $2M$, or $M\pm 1$, and zero otherwise. We omit their formulae.
\end{proof}

Substituting the result of Lemma \ref{lem:V(m,n) M=3 N=1,M2} into \eqref{def:Delta}, \eqref{def:Vexp} and \eqref{def:V(m,n)}, we make a straightforward calculation to arrive at
\begin{align*}
D&(i\sigma+\delta,k_{j}(\sigma)+\gamma,\eps)\\
=&d^{(2,0,0)}\delta^2+d^{(0,2,0)}\gamma^2+d^{(1,1,0)}\delta\gamma\\
&+d^{(1,0,2)}\eps^2\delta+d^{(0,1,2)}\gamma\eps^2+d^{(0,0,4)}\eps^4+o((|\delta|+|\gamma|+|\eps|^2)^2)
\end{align*}
as $\delta,\gamma,\eps\to 0$, where $d^{(\ell,m,n)}$ can be expressed in terms of the entries of $\mathbf{V}^{(m,n)}(T;\sigma)$, in Lemma \ref{lem:V(m,n) M=3 N=1,M2}, for $m+n=0,1,2$. We omit their formulae. We then follow the same argument as in Sections \ref{sec:N=1}, \ref{sec:N=2}, \ref{sec:M=2 N=2}, and \ref{sec:M=3 N=1,M}, and make an explicit calculation to derive a spectral stability index whose sign determines stability or instability. 

We state the main result of this subsection. 

\begin{theorem}[Wilton ripples of order $M$ and resonances of orders $\neq 1,M$]\label{thm:M=3 1<N<M}
${}$

\noindent Let $\kappa>0$. Suppose that there exists a bimodal Wilton-ripple solution of order $M\geq 3,\in \mathbb{Z}^+$ to \eqref{stationary-periodic}-\eqref{eqn:stokes exp} for $\eps\neq 0,\in\mathbb{R}$ and $|\eps|\ll1$, satisfying \eqref{def:wilton1 M=3}. Assume Section \ref{sec:M=3 N=1,M} {\rm (C1')} for $N\neq 1,M$, {\rm (C2')}, and \eqref{A:transversality}. Recall Section \ref{sec:L}, and consider ${\rm spec}(\mathcal{L}(\eps))$ near $i\sigma \neq 0,\in i\mathbb{R}$. We define
\be \label{def:ind_8}
{\rm ind}_8(\kappa,\sigma,\alpha;k_j,k_{j'})=\frac{w^{(0,2)}_{12}w^{(0,2)}_{21}}{w_{11}^{(1,0)}w_{22}^{(1,0)}},
\ee 
where the right side is in \eqref{eqn:V(m,n) M=3 N=1,M} and Lemma \ref{lem:V(m,n) M=3 N=1,M2}, depending on $\alpha$.

\begin{itemize}
\item[$\bullet$] If ${\rm ind}_8(\kappa,\sigma,\alpha;k_j,k_{j'})>0$ then there exists $\lambda \in {\rm spec}(\mathcal{L}(\eps))$ near $i\sigma \in i\mathbb{R}$ such that ${\rm Re}\,\lambda \neq 0$. 

\item[$\bullet$] If ${\rm ind}_8<0$, on the other hand, ${\rm spec}(\mathcal{L}(\eps))\subset i\mathbb{R}$ near $i\sigma$.

\item[$\bullet$] If ${\rm ind}_8=0$ then spectral stability near $i\sigma$ would be determined by higher order terms in \eqref{def:Vexp}.
\end{itemize}
\end{theorem}

The proof of Theorem \ref{thm:M=3 1<N<M} mirrors the proofs of Theorems \ref{thm:N=1} and \ref{thm:N=2}, and is therefore omitted. It is noteworthy that \eqref{def:ind_8} takes the same expression as \eqref{def:ind_2}. An important difference, however, is that the right side of \eqref{def:ind_8} depends on $\alpha$ whose explicit computation becomes increasingly unwieldy for higher values of $M$. See Section \ref{sec:bimodal} for further discussion.

\begin{corollary}\label{cor:ind8}

Recall Lemma \ref{lem:V(m,n) M=3 N=1,M2} and Theorem \ref{thm:M=3 1<N<M}. If $N=2,M\pm1$, or $2M$, then $w^{(0,2)}_{12}$ and $w^{(0,2)}_{21}$ are not necessarily zero, and the sign of \eqref{def:ind_8} can determine spectral stability or instability. If not, $w^{(0,2)}_{12}, w^{(0,2)}_{21}=0$ and, correspondingly, ${\rm ind}_8=0$, whence spectral stability would be determined by higher order terms in \eqref{def:Vexp}.
\end{corollary}

\begin{corollary}[Unstable spectrum near resonance of order $N$]${}$\label{cor:M=3 1<N<M} 

\noindent Recall Theorem \ref{thm:M=3 1<N<M}. If ${\rm ind}_8>0$ then, to leading order as $\eps\to 0$, $\lambda\in {\rm spec}(\mathcal{L}(\eps))$ near $i\sigma\neq 0,\in i\mathbb{R}$ forms an ellipse whose center is displaced from $i\sigma$ by $O(\eps^2)$ along the imaginary axis and whose semiaxes are $O(\eps^2)$.
\end{corollary}

The proof of Corollary \ref{cor:M=3 1<N<M} is the same as that of Corollary \ref{cor:N=2}, and is omitted.

\subsubsection*{Numerical studies} 

We conclude this subsection by numerically investigating Theorem \ref{thm:M=3 1<N<M}. Since \eqref{def:ind_8} depends on $\alpha$, whose explicit computation becomes increasingly unwieldy for higher values of $M$, we restrict our attention to $M=3$. Recall Corollary \ref{cor:ind8}, and we restrict our attention to $N=2$, $4$, and $6$.

It follows from the result of Section \ref{S1S2region} that the curve $\kappa_2=3\kappa_1$ or, equivalently, $\beta=\beta_W(\kappa;3)$, corresponding to Wilton ripples of order $3$, intersects transversely the resonance regions (R7), (R9), (R12), and (R14) for $N=2$, and the region (R4) for $N=4$ and $6$, where Section \ref{sec:M=3 N=1,M} (C1') holds for the corresponding $j$ and $j'$. 

Recall that \eqref{eqn:alpha M=3} has three distinct real roots for each $\kappa\in (0,\infty)$, denoted by
\[
\alpha_1(\kappa)<\alpha_2(\kappa)<0<\alpha_3(\kappa).
\]
Correspondingly, there are three distinct families of Wilton ripples of order $3$. We numerically evaluate \eqref{def:ind_8} at $\alpha=\alpha_\ell(\kappa)$, $\ell=1,2,3$, and apply Theorem \ref{thm:M=3 1<N<M} to numerically determine spectral stability or instability.  

We summarize our numerical findings for $N=2$.

\subsubsection*{\rm (R7) $k_2(\sigma)-k_3(\sigma)=2\kappa$} 
If $\kappa\in(2.707944\ldots,\infty)$ then the curve $\beta=\beta_W(\kappa;3)$ lies in the resonance region (R7) for $N=2$. Our numerical computation indicates ${\rm ind}_8(\kappa,\sigma,\alpha_\ell(\kappa); k_2,k_3)>0$ for $\kappa$ in the interval for $\sigma\in (0,\min(-\sigma_{c,1},\sigma_{c,2}))$ corresponding to the resonance for $\ell=1,2,3$. Theorem \ref{thm:M=3 1<N<M} then predicts spectral instability near $i\sigma$.

\subsubsection*{\rm (R14) $k_2(\sigma)-k_1(\sigma)=2\kappa$}  
If $\kappa\in(2.080480\ldots,2.707944\ldots)$ then the curve $\beta=\beta_W(\kappa;3)$ lies in the resonance region (R14) for $N=2$. Numerical evidence supports ${\rm ind}_8(\kappa,\sigma,\alpha_\ell; k_2,k_1)>0$ for $\kappa$ in the interval for $\sigma \in (0,\min(-\sigma_{c,1},\sigma_{c,2}))$ corresponding to the resonance for $\ell=1,2,3$, implying spectral instability near $i\sigma$.

\subsubsection*{\rm (R9) $k_6(\sigma)-k_4(\sigma)=2\kappa$} If $\kappa\in(1.651632\ldots,\infty)$ then the curve $\beta=\beta_W(\kappa;3)$ lies in the resonance region (R9) for $N=2$. Numerical evidence supports ${\rm ind}_8(\kappa,\sigma,\alpha_\ell; k_6,k_4)>0$ for $\kappa$ in the interval for $\sigma \in (0,\sigma_{c,2})$ corresponding to the resonance for $\ell=1,2,3$, implying spectral instability near $i\sigma$.

\subsubsection*{\rm (R12) $k_2(\sigma)-k_4(\sigma)=2\kappa$} If $\kappa\in(0,1.651632\ldots)$ then the curve $\beta=\beta_W(\kappa;3)$ is in the resonance region (R12) for $N=2$. Numerical evidence supports ${\rm ind}_8(\kappa,\sigma,\alpha_\ell; k_2,k_4)>0$ for $\kappa$ in the interval for $\sigma \in (0,\sigma_{c,2})$ corresponding to the resonance for $\ell=1,2,3$, implying spectral instability near $i\sigma$.

\medskip

It is straightforward to verify that the curve $\beta=\beta_W(\kappa;3)$ lies entirely in the resonance regions (R4) for $N=4$ and $6$. Our numerical evaluation of \eqref{def:ind_8} suggests 
\[
{\rm ind}_8(\kappa,\sigma,\alpha_\ell(\kappa); k_4,k_5)<0
\]
for $\kappa\in(0,\infty)$ and for $\sigma$ corresponding to the resonance for $\ell=1,2,3$. Theorem \ref{thm:M=3 1<N<M} then predicts spectral stability near $i\sigma$.

\begin{remark*}\rm 
We do not numerically detect instability near the resonance (R4) for orders $4$ and $6$. This is consistent with the prediction based on Krein signature. See Remark \ref{rem:Krein}.
\end{remark*}

\subsection{Modulational instability of Wilton ripples}\label{sec:wilton modulational}

We continue to assume \eqref{A:wilton}, so that there exist bimodal Wilton ripples of order $M$ for sufficiently small amplitude, solving \eqref{stationary-periodic}-\eqref{eqn:stokes exp} for $\eps\neq0,\in\mathbb{R}$ and $|\eps|\ll1$, and satisfying either \eqref{def:wilton1 M=2} or \eqref{def:wilton1 M=3}. We are interested in their spectral instability near $i0\in \mathbb{C}$, namely, modulational instability.

\begin{notation*}
Throughout this subsection, $\sigma=0$ is suppressed for simplicity of notation unless otherwise specified.
\end{notation*}

Proceeding along the lines of the argument as in Section \ref{sec:modulational}, we investigate the roots of 
\[
\text{$D(i0+\delta,k_j+\gamma,\eps)$ for $\delta\in\mathbb{C}$ and $|\delta|\ll1$ for $\gamma, \eps\in\mathbb{R}$ and $|\gamma|,|\eps|\ll1$,}
\]
where $k_j$, $j=1,2,\dots,6$, are in Lemma \ref{lem:eps=0}. Particularly, Lemma \ref{lem:a4=0} and Corollary \ref{delta_0pk} remain valid for steady unimodal capillary-gravity waves and bimodal Wilton ripples. One main difference is that \eqref{def:Bexp} and, hence, \eqref{def:V(m,n)} are based on the Stokes expansion of Wilton ripples in Section \ref{sec:bimodal} and Appendix \ref{app:Wilton exp}, rather than Section \ref{sec:unimodal}. 

\begin{notation*}
Recall Lemma \ref{lem:eps=0} and we relabel the eigenvalues and eigenfunctions of $\mathbf{L}(i0)$ as
\[
k_j=(-1)^j\kappa, \quad j=1,2,\quad\text{and}\quad k_j=0, \quad j=3,4,
\]
$k_j=(-1)^jM\kappa$, $j=5,6$. Correspondingly, $\boldsymbol{\phi}_j$, $j=1,2,5,6$, are in \eqref{def:eigen0}, and $\boldsymbol{\phi}_3$ and $\boldsymbol{\phi}_4$ are the latter and the former of \eqref{def:eigen0''}, respectively. We relabel the eigenfunctions of $\mathbf{L}^\dag(i0)$ analogously.  

Additionally, we write \eqref{def:V(m,n)} as
\[
\mathbf{V}^{(m,n)}(x)=(w^{(m,n)}_{jk}(x))_{j,k=1,2,\dots,6}.
\]
\end{notation*}

\subsubsection{Wilton ripples of order $\geq 3$}\label{sec:modulational M=3}

Suppose $M\geq 3,\in \mathbb{Z}^+$. 

Recall Section \ref{sec:expansion_monodromy} and we begin by evaluating $\mathbf{V}^{(m,n)}(T)$ for $m+n=0,1,2$. We make the same calculation as in Section \ref{sec:modulational subcritical}, solving \eqref{eqn:v_k(0)}-\eqref{def:f(0)}, \eqref{cond:V(0)}, and \eqref{eqn:w_k}-\eqref{eqn:w_k bd}, $k=1,2,\dots, 6$, after relabeling $k_j$, $\boldsymbol{\phi}_j$, and $\boldsymbol{\psi}_j$, $j=1,2,\dots,6$, and utilizing the results of Section \ref{sec:bimodal} and Appendices \ref{app:Wilton exp} and \ref{app:Bexp}, rather than Section \ref{sec:unimodal}. 

Notably, $\mu_1=0$ for unimodal capillary-gravity waves and bimodal Wilton ripples of order $\geq 3$. Accordingly, the analysis in this subsection parallels that of Section \ref{sec:modulational}. 
See Section \ref{sec:modulational M=2} for $M=2$ for which $\mu_1\neq 0$.

Below we summarize the result.

\begin{lemma}\label{lem:V(m,n)0 M=3} 
Assume that \eqref{A:wilton} holds for $M\geq 3,\in \mathbb{Z}^+$. Then
\begin{align*}
&\mathbf{V}^{(0,0)}(T)=\begin{pmatrix} 
1&0&0&0 &0&0\\ 0&1&0&0&0&0\\0&0&1&0&0&0\\0&0&T&1&0&0\\0&0&0&0&1&0\\0&0&0&0&0&1\end{pmatrix},\\ 
&\mathbf{V}^{(1,0)}(T)=\begin{pmatrix} w^{(1,0)}_{11}
 & 0 & w^{(1,0)}_{13} & 0&0&0 \\ 0 & w^{(1,0)}_{11} & -w^{(1,0)}_{13} & 0 &0&0\\ 0 & 0 & w^{(1,0)}_{33} & 0&0&0\\ w^{(1,0)}_{41}&w^{(1,0)}_{41}&Tw^{(1,0)}_{33}&w^{(1,0)}_{33} &w^{(1,0)}_{45}&w^{(1,0)}_{45}\\0&0&w^{(1,0)}_{53}&0&w^{(1,0)}_{55}&0\\0&0&-w^{(1,0)}_{53}&0&0&w^{(1,0)}_{55}\end{pmatrix}, \\ 
&\mathbf{V}^{(0,1)}(T)=\begin{pmatrix} 
0&0&w^{(0,1)}_{13}&0 &0&0\\ 0&0&w^{(0,1)}_{13}&0&0&0\\0&0&0&0&0&0\\w^{(0,1)}_{41}&-w^{(0,1)}_{41}&w^{(0,1)}_{43}&0&w^{(0,1)}_{45}&-w^{(0,1)}_{45}\\0&0&w^{(0,1)}_{53}&0&0&0\\0&0&w^{(0,1)}_{53}&0&0&0\end{pmatrix}, \\ 
&\mathbf{V}^{(2,0)}(T)
=\begin{pmatrix} w_{11}^{(2,0)}& w_{12}^{(2,0)} & w_{13}^{(2,0)} &w_{14}^{(2,0)} &w_{15}^{(2,0)}&w_{16}^{(2,0)} \\ -w_{12}^{(2,0)} & (w_{11}^{(2,0)})^*& (w_{13}^{(2,0)})^* &-w_{14}^{(2,0)} &-w_{16}^{(2,0)}&-w_{15}^{(2,0)}\\ w_{31}^{(2,0)}&w_{31}^{(2,0)} & w_{33}^{(2,0)} & w_{34}^{(2,0)}&w_{35}^{(2,0)} &w_{35}^{(2,0)}\\ w_{41}^{(2,0)} & (w_{41}^{(2,0)})^* & w_{43}^{(2,0)} & w_{33}^{(2,0)}&w_{45}^{(2,0)}&(w_{45}^{(2,0)})^*\\ w_{51}^{(2,0)}&w_{52}^{(2,0)}&w_{53}^{(2,0)}&w_{54}^{(2,0)}&w_{55}^{(2,0)}&w_{56}^{(2,0)}\\ -w_{52}^{(2,0)} & -w_{51}^{(2,0)} & (w_{53}^{(2,0)})^* &-w_{54}^{(2,0)}&-w_{56}^{(2,0)}&(w_{55}^{(2,0)})^* \end{pmatrix}, \\ 
&\mathbf{V}^{(1,1)}(T)=\begin{pmatrix} w_{11}^{(1,1)} & w_{12}^{(1,1)} & *& w_{14}^{(1,1)} &* &*\\
w_{12}^{(1,1)} & w_{11}^{(1,1)} & *& w_{14}^{(1,1)} &* &*\\
w_{31}^{(1,1)} & -w_{31}^{(1,1)} & *& 0 &w_{35}^{(1,1)} &-w_{35}^{(1,1)}\\
* & * & *& * &* &*\\
* & * & *& w_{54}^{(1,1)} &w_{55}^{(1,1)} &w_{56}^{(1,1)}\\
* & * & *& w_{54}^{(1,1)} &w_{56}^{(1,1)} &w_{55}^{(1,1)}\\
\end{pmatrix}.
\end{align*}
The formulae for the entries of $\mathbf{V}^{(1,0)}(T)$, $\mathbf{V}^{(0,1)}(T)$, and $\mathbf{V}^{(2,0)}(T)$ are in Appendix \ref{app:V(m,n) M=3}. We pause to remark that $(w_{jk}^{(m,n)})^*$ means the complex conjugate of $w_{jk}^{(m,n)}$. The formulae for $\mathbf{V}^{(1,1)}(T)$ and  $\mathbf{V}^{(0,2)}(T)$ are prohibitively lengthy and therefore omitted. Our analysis does not involve the entries marked with an asterisk, whence their formulae are omitted. 

Additionally, 
\begin{align*}
&Tw^{(2,0)}_{12}=w^{(1,0)}_{13}w^{(1,0)}_{41},
&&Tw_{14}^{(2,0)}=-w_{13}^{(1,0)}(w_{11}^{(1,0)} - w_{33}^{(1,0)}),\\ 
&Tw_{31}^{(2,0)}=-w_{41}^{(1,0)}(w_{11}^{(1,0)} - w_{33}^{(1,0)}),
&&Tw^{(2,0)}_{56}=w^{(1,0)}_{53}w^{(1,0)}_{45},\\
&Tw_{54}^{(2,0)}=-w_{53}^{(1,0)}(w_{55}^{(1,0)} - w_{33}^{(1,0)}),
&&Tw_{35}^{(2,0)}=-w_{45}^{(1,0)}(w_{55}^{(1,0)} - w_{33}^{(1,0)}),\\
&Tw_{11}^{(1,1)}=w_{13}^{(0,1)}w_{41}^{(1,0)} + w_{13}^{(1,0)}w_{41}^{(0,1)},
&&Tw_{12}^{(1,1)}=w_{13}^{(0,1)}w_{41}^{(1,0)} - w_{13}^{(1,0)}w_{41}^{(0,1)},\\
&Tw_{55}^{(1,1)}=w_{53}^{(0,1)}w_{45}^{(1,0)} + w_{53}^{(1,0)}w_{45}^{(0,1)},
&&Tw_{56}^{(1,1)}=w_{53}^{(0,1)}w_{45}^{(1,0)} - w_{53}^{(1,0)}w_{45}^{(0,1)}.
\end{align*}
\end{lemma}

See Lemma \ref{lem:V(m,n)0} for the corresponding results for steady unimodal capillary-gravity waves. 

Recall Sections \ref{sec:reduction} and \ref{sec:expansion_monodromy}, and we proceed to a power series expansion of $D(i0+\delta,k_j+\gamma,\eps)$, $j=1,2,\dots,6$, about $(\delta,\gamma,\eps)=(0,0,0)$. Substituting the result of Lemma \ref{lem:V(m,n)0 M=3} into \eqref{def:Delta}, \eqref{def:Vexp}, and \eqref{def:V(m,n)}, we make a straightforward calculation to arrive at
\begin{equation}\label{eqn:d(l,m,n) M=3}
D(\delta,\gamma,\eps)=\sum_{\ell=0}^6d^{(\ell,6-\ell,0)}\delta^{\ell}\gamma^{6-\ell}+\sum_{\ell=0}^5d^{(\ell,5-\ell,2)}\delta^{\ell}\gamma^{5-\ell}\eps^2+\cdots
\end{equation}
as $\delta,\gamma,\eps\to 0$, where $d^{(\ell,m,n)}$ can be expressed in terms of $w^{(m,n)}_{jk}$, in Lemma \ref{lem:V(m,n)0 M=3} and Appendix \ref{app:V(m,n) M=3}. We omit their formulae. We emphasize that $w_{jk}^{(m,n)}$ and, hence, $d^{(\ell,m,n)}$ depend on $\kappa=2\pi/T$, with $M$ as a parameter, and, additionally, on $\alpha$, which appears, for instance, in \eqref{def:wilton1 M=3}. See also \eqref{eqn:alpha M=3} for $M=3$. We suppress the dependence for brevity unless otherwise specified. See \eqref{eqn:evans0} for steady unimodal capillary-gravity waves.

We then adapt the argument from Section \ref{sec:modulational subcritical}. The Weierstrass preparation theorem implies that \eqref{eqn:d(l,m,n) M=3} admits six roots for $\gamma,\eps\in\mathbb{R}$ and $|\gamma|,|\eps|\ll1$, counting multiplicity, denoted by 
\[
\delta_j(\gamma,\eps),\quad j=1,2,\dots,6.
\]
Suppose $\eps=0$. It follows from Section \ref{sec:eps=0} that $\delta_j(\gamma,0)=i\sigma(k_j+\gamma)$, $j=1,2,\dots,6$, where $\sigma(k)$ is in \eqref{eqn:sigma}. 
Since \eqref{def:sigma} is analytic at $k_j$, $j=1,2,\dots,6$, we write 
\[
\delta_j(\gamma,0)=\delta^{(1,0)}_j\gamma+\delta^{(2,0)}_j\gamma^2+\cdots,\quad \text{as $\gamma\to 0$},\quad j=1,2,\dots,6,
\]
where $\delta^{(1,0)}_j$ and $\delta^{(2,0)}_j$, $j=1,2,\dots, 6$, can be determined from the Taylor expansions of \eqref{def:sigma} about $k_j$, $j=1,2,\dots,6$. 

Suppose, on the other hand, $\eps\neq0,\in\mathbb{R}$ and $|\eps|\ll1$. Corollary \ref{delta_0pk} ensures $D(0,0,\eps)=0$. We then assume
\ba\label{def:lambda(gamma,eps)_wiltonm}
\delta_j(\gamma,\eps)=
\delta^{(1,0)}_j\gamma+\delta^{(2,0)}_j\gamma^2+\delta^{(1,1)}_j\gamma\eps+\cdots, \quad \text{as $\gamma,\eps\to0$},
\ea
$j=1,2,\dots,6$, where $\delta^{(1,0)}_j$, $\delta^{(2,0)}_j$, and $\delta^{(1,1)}_j$ are to be determined in terms of $d^{(\ell,m,n)}$ and, hence, $w_{jk}^{(m,n)}$. Importantly, $\delta^{(1,0)}_j$ and $\delta^{(2,0)}_j$, $j=1,2,\dots,6$, coincide with the coefficients of the Taylor expansion of \eqref{def:sigma} about $k_j$, $j=1,2,\dots,6$. See \eqref{def:lambda(gamma,eps)} for steady unimodal capillary-gravity waves. 

Substituting \eqref{def:lambda(gamma,eps)_wiltonm} into \eqref{eqn:d(l,m,n) M=3} and collecting terms of order $\gamma^m\eps^n$, $m,n=0,1,2,\dots$, we make a straightforward calculation to show that the leading order contribution is of order $\gamma^6$, whose coefficient is
\ba\label{eqn:alpha10_wiltonm}
- &(-w_{11}^{(1,0)}\delta^{(1,0)}_j+iT)^2 (-w_{55}^{(1,0)}\delta^{(1,0)}_j+iT)^2
\\ &\cdot((Tw_{34}^{(2,0)}-{w_{33}^{(1,0)}}^2)(\delta^{(1,0)}_j)^2+2iTw_{33}^{(1,0)}\delta^{(1,0)}_j+T^2),
\ea 
$j=1,2,\dots,6$, where $w_{11}^{(1,0)}$, $w_{33}^{(1,0)}$, $w_{55}^{(1,0)}$, and $w_{34}^{(2,0)}$ are in Appendix \ref{app:V(m,n) M=3}. If \eqref{def:lambda(gamma,eps)_wiltonm} makes the roots of $D(\cdot,\gamma,\eps)$ for $\gamma,\eps\in\mathbb{R}$ and $|\gamma|,|\eps|\ll1$, then \eqref{eqn:alpha10_wiltonm} must vanish. Consequently,
\ba\label{def:alpha10'_wiltonm}
\delta^{(1,0)}_j=\frac{iT}{w_{11}^{(1,0)}},\quad \text{or}\quad \frac{iT}{w_{55}^{(1,0)}},\quad \text{or}\quad \frac{iTw_{33}^{(1,0)}\pm\sqrt{-T^3w^{(2,0)}_{34}}}{(w_{33}^{(1,0)})^2-Tw_{34}^{(2,0)}},
\ea
$j=1,2,\dots,6$. We make an explicit calculation of $w^{(1,0)}_{11}$, $w^{(1,0)}_{33}$, $w^{(1,0)}_{55}$, and $w^{(2,0)}_{34}$ in Appendix \ref{app:V(m,n) M=3}, and simplify the result taking a similar approach as in Section \ref{sec:bimodal}, to show that the right sides of \eqref{def:alpha10'_wiltonm} are purely imaginary. Since \eqref{def:alpha10'_wiltonm} must agree with the coefficient of the $\gamma$ terms in the Taylor expansions of \eqref{def:sigma} about $k_j$, $j=1,2,\dots,6$, therefore,
\begin{equation}\label{def:alpha10_wiltonm}
\delta^{(1,0)}_j=\frac{iT}{w_{11}^{(1,0)}},\quad j=1,2\quad\text{and}\quad 
\delta^{(1,0)}_j=\frac{iT}{w_{55}^{(1,0)}},\quad j=5,6,
\end{equation} 
$\delta^{(1,0)}_j=\frac{iTw_{33}^{(1,0)}\pm\sqrt{-T^3w^{(2,0)}_{34}}}{(w_{33}^{(1,0)})^2-Tw_{34}^{(2,0)}}$, $j=3,4$. For the remainder of this subsubsection, we restrict our attention to $j=1,2$ or $5,6$.

Substituting \eqref{def:lambda(gamma,eps)_wiltonm} and \eqref{def:alpha10_wiltonm} into \eqref{eqn:d(l,m,n) M=3}, we make a straightforward calculation to show that the coefficient of the $\gamma^7$ term vanishes. Proceeding to the $\gamma^8$ terms, after some algebra, we arrive at
\ba\label{def:alpha20_wiltonm}
\delta^{(2,0)}_j=\begin{cases}
\displaystyle \pm\frac{T^2((w_{11}^{(1,0)})^2-2w_{11}^{(2,0)})+2Tw_{13}^{(1,0)}w_{41}^{(1,0)}}{2(w_{11}^{(1,0)})^3}, &j=1,2,\\
\displaystyle \pm\frac{T^2((w_{55}^{(1,0)})^2-2w_{55}^{(2,0)})+2Tw_{53}^{(1,0)}w_{45}^{(1,0)}}{2(w_{55}^{(1,0)})^3},&j=5,6,
\end{cases}
\ea
where the right side is as in Appendix \ref{app:V(m,n) M=3}. We pause to remark that $\delta^{(2,0)}_j$, $j=1,2$ or $5,6$, are purely imaginary because $(w_{11}^{(1,0)})^2$ offset the real part of $2w_{11}^{(2,0)}$. The $\pm$ signs account for the opposite convexities of \eqref{def:sigma} at $k_j$, $j=1,2$ or $5,6$. See Lemma \ref{lem:sigmaS12}.

Substituting \eqref{def:lambda(gamma,eps)_wiltonm}, \eqref{def:alpha10_wiltonm}, and \eqref{def:alpha20_wiltonm} into \eqref{eqn:d(l,m,n) M=3}, we make a straightforward calculation to show that the coefficient of the $\gamma^6\eps^2$ term is 
\ba \label{eqn:f_wiltonm}
T^4(w^{(1,0)}_{11} - w^{(1,0)}_{55})^2\omega_1(\delta_j^{(1,1)})^2-T^5(w^{(1,0)}_{11} - w^{(1,0)}_{55})^2\omega_2,
\ea 
where
\begin{align}
&\begin{aligned}
\omega_1=\omega_{1,1}(\kappa,\alpha;M):=&(w^{(1,0)}_{11})^{-2}((w^{(1,0)}_{11}-w^{(1,0)}_{33})^2 - T w^{(2,0)}_{34}),\\
\omega_2=\omega_{2,1}(\kappa,\alpha;M):=&(w_{11}^{(1,0)})^{-6}(T(w_{11}^{(1,0)})^2 + 2w_{13}^{(1,0)}w_{41}^{(1,0)} - 2Tw_{11}^{(2,0)})\\
&\cdot (w_{11}^{(1,0)} w_{13}^{(0,1)} w_{31}^{(1,1)} + w_{11}^{(1,0)} w_{41}^{(0,1)} w_{14}^{(1,1)} \\
&\quad-  w_{33}^{(1,0)} w_{13}^{(0,1)} w_{31}^{(1,1)} - w_{33}^{(1,0)} w_{41}^{(0,1)} w_{14}^{(1,1)}\\
&\quad + (w_{11}^{(1,0)}-w_{33}^{(1,0)})^2 w_{11}^{(0,2)} +w_{13}^{(0,1)} w_{41}^{(0,1)} w_{34}^{(2,0)} \\
&\quad - T w_{34}^{(2,0)} w_{11}^{(0,2)} + T w_{14}^{(1,1)} w_{31}^{(1,1)}), \end{aligned}\label{g1g2_wiltonmj12}
\intertext{$j=1,2$, and}
&\begin{aligned}
\omega_1=\omega_{1,M}(\kappa,\alpha;M):=&(w^{(1,0)}_{55})^{-2}
\big((w^{(1,0)}_{33}-w^{(1,0)}_{55})^2-Tw^{(2,0)}_{34}\big),\\
\omega_2=\omega_{2,M}(\kappa,\alpha;M):=&(w_{55}^{(1,0)})^{-6}
(T(w_{55}^{(1,0)})^2-2Tw_{55}^{(2,0)}
+2w_{45}^{(1,0)}w_{53}^{(1,0)})\\
&\cdot(-w_{33}^{(1,0)}w_{45}^{(0,1)}w_{54}^{(1,1)}
-w_{33}^{(1,0)}w_{53}^{(0,1)}w_{35}^{(1,1)}\\
&\quad+w_{55}^{(1,0)}w_{45}^{(0,1)}w_{54}^{(1,1)}
+w_{55}^{(1,0)}w_{53}^{(0,1)}w_{35}^{(1,1)}\\
&\quad+(w_{33}^{(1,0)}-w_{55}^{(1,0)})^2w_{55}^{(0,2)}
+w_{45}^{(0,1)}w_{53}^{(0,1)}w_{34}^{(2,0)}\\
&\quad-Tw_{34}^{(2,0)}w_{55}^{(0,2)}
+Tw_{35}^{(1,1)}w_{54}^{(1,1)}),\end{aligned}\label{g1g2_wiltonmj56}
\end{align}
$j=5,6$, for $w^{(m,n)}_{jk}$ in Appendix \ref{app:V(m,n) M=3}. The signs of $\omega_1$ and $\omega_2$ then determine modulational stability or instability. 

We state the main result of this subsection.

\begin{theorem}[Modulational instability of Wilton ripples of order $\geq 3$]${}$\label{thm:wiltonm}

\noindent Let $\kappa>0$. Suppose that there exists a bimodal Wilton-ripple solution of order $M\geq 3,\in \mathbb{Z}^+$ to \eqref{stationary-periodic}-\eqref{eqn:stokes exp} for $\eps\neq0,\in\mathbb{R}$ and $|\eps|\ll1$, satisfying \eqref{def:wilton1 M=3}. Recall Section \ref{sec:L}, and consider ${\rm spec}(\mathcal{L}(\eps))$ near $i0$. Assume that \eqref{def:lambda(gamma,eps)_wiltonm} holds for the roots of \eqref{eqn:d(l,m,n) M=3} near $(\delta,\gamma,\eps)=(0,0,0)$. 

Then, there exists $\lambda\in{\rm spec}(\mathcal{L}(\eps))$ such that ${\rm Re}\,\lambda \neq0$, implying modulational instability, provided either
\be \label{def:ind_9}
{\rm ind}_9(\kappa,\alpha;M)=\frac{\omega_{2,1}}{\omega_{1,1}}(\kappa,\alpha;M)>0,
\ee
where $\omega_{1,1}$ and $\omega_{2,1}$ are in \eqref{g1g2_wiltonmj12}, or
\be\label{def:ind_10}
{\rm ind}_{10}(\kappa,\alpha;M)
=\frac{\omega_{2,M}}{\omega_{1,M}}(\kappa,\alpha;M)>0,
\ee
where $\omega_{1,M}$ and $\omega_{2,M}$ are in \eqref{g1g2_wiltonmj56}.
\end{theorem}

The proof of Theorem \ref{thm:wiltonm} mirrors that of Theorem \ref{thm:modulation}, and is therefore omitted.

\begin{remark*}\rm
Without a rigorous justification of \eqref{def:lambda(gamma,eps)_wiltonm}, it remains open whether \eqref{def:ind_9} or \eqref{def:ind_10} indeed determines modulational instability. Theorem \ref{thm:wiltonm} is therefore conditional, and its stability criteria are partly formal. 
For unimodal waves, by contrast, \eqref{def:lambda(gamma,eps)} has been rigorously justified \cite{hsiao2025, sun2025}, and the results in Theorem \ref{thm:modulation} are unconditional.
\end{remark*}

An inspection reveals 
\be\label{ind9_reduction}
{\rm ind}_9(\kappa,\alpha;M)=\frac{\omega_2}{\omega_1}(\beta_W(\kappa;M),\kappa),
\ee 
where $\omega_1$ and $\omega_2$ are in \eqref{g1g2} and \eqref{g_i}, respectively, and $\beta_W(\kappa;M)$ in \eqref{def:beta_W}. Particularly, \eqref{def:ind_9} is independent of $\alpha$. Recall the proof of Theorem \ref{thm:modulation}, and ${\rm ind}_5$ in \eqref{def:ind_5} extends continuously across the curve \eqref{def:beta_W} corresponding to Wilton ripples of order $\geq 3$.

Upon inspection, $w_{45}^{(0,1)}$, $w_{53}^{(0,1)}$, $w_{35}^{(1,1)}$, $w_{54}^{(1,1)}$, and $w_{55}^{(0,2)}$ depend polynomially on $\alpha$ through the forms
\begin{align*}
& w_{45}^{(0,1)}=b_{45}^{(0,1)}\alpha+c_{45}^{(0,1)}
&& w_{53}^{(0,1)}=b_{53}^{(0,1)}\alpha, \\
&w_{35}^{(1,1)}= b_{35}^{(1,1)}\alpha+c_{35}^{(1,1)},  
&& w_{54}^{(1,1)}=b_{54}^{(1,1)}\alpha, \\
&w_{55}^{(0,2)}=a_{55}^{(0,2)}\alpha^2+b_{55}^{(0,2)}\alpha+c_{55}^{(0,2)}.
\end{align*}
Upon substitution, after some simplification, we arrive at
\be \label{ind10_reduction}
{\rm ind}_{10}(\kappa,\alpha;M)=a(\kappa;M)\alpha^2+c(\kappa;M),
\ee 
where
\ba \label{ab_terms}
a(\kappa;M)
=&(w^{(1,0)}_{55})^{-4}( (w^{(1,0)}_{33}-w^{(1,0)}_{55})^2 - T w^{(2,0)}_{34})^{-1}\\
&\cdot(T(w_{55}^{(1,0)})^2 - 2Tw_{55}^{(2,0)} + 2w_{45}^{(1,0)}w_{53}^{(1,0)})\\
&\cdot( (w_{33}^{(1,0)}-w_{55}^{(1,0)})^2  a_{55}^{(0,2)}- T w_{34}^{(2,0)} a_{55}^{(0,2)}\\
& \quad +(w_{55}^{(1,0)}- w_{33}^{(1,0)} )b_{35}^{(1,1)} b_{53}^{(0,1)} + w_{34}^{(2,0)} b_{45}^{(0,1)} b_{53}^{(0,1)} \\
&\quad + (w_{55}^{(1,0)}- w_{33}^{(1,0)}) b_{45}^{(0,1)} b_{54}^{(1,1)} + T b_{35}^{(1,1)} b_{54}^{(1,1)}),\\
c(\kappa;M)
=&(w^{(1,0)}_{55})^{-4}(T(w_{55}^{(1,0)})^2 - 2Tw_{55}^{(2,0)} + 2w_{45}^{(1,0)}w_{53}^{(1,0)})c_{55}^{(0,2)} .
\ea 

When $M=3$, $\alpha$ is one of the three real roots of the cubic polynomial \eqref{eqn:alpha M=3}, and \eqref{ind10_reduction} can be evaluated explicitly at each root. When $M\geq4$, on the other hand, computing $\alpha$ explicitly becomes increasingly unwieldy, and it is impractical to evaluate \eqref{ind10_reduction} directly at the roots $\alpha_\ell$, $\ell=1,2,\dots,M$, corresponding to Wilton ripples of order $M$. Nonetheless, the sign of \eqref{ind10_reduction} is independent of $\alpha_\ell$ when $a(\kappa;M)$ and $c(\kappa;M)$ have the same sign. Without explicit knowledge of $\alpha$, this appears to be the strongest conclusion available for ${\rm ind}_{10}$ for Wilton ripples of order $\geq4$.

We conclude this subsection by numerically investigating Theorem \ref{thm:wiltonm}.
Since ${\rm ind}_9(\kappa,\alpha;M)$ is independent of $\alpha$, by \eqref{ind9_reduction}, whose sign coincides with the sign of ${\rm ind}_5(\beta_W(\kappa;M),\kappa)$ in \eqref{def:ind_5}, it suffices to numerically evaluate the latter. Our numerical computation suggests ${\rm ind}_5(\beta_W(\kappa;M),\kappa)>0$. Therefore, if \eqref{def:lambda(gamma,eps)_wiltonm} is valid, Theorem \ref{thm:wiltonm} predicts modulational instability:
\begin{itemize}
    \item[$\bullet$] for $\kappa\in (0.600792\ldots,2.080480\ldots)$ for $M=3$;
    \item[$\bullet$] for $\kappa\in (0.994569\ldots,2.624458\ldots)$ for $M=4$;
    \item[$\bullet$] for $\kappa\in (1.122253\ldots,3.191355\ldots)$ for $M=5$.
\end{itemize}

Turning to ${\rm ind}_{10}(\kappa,\alpha;M)$, suppose $M=3$. Let $\alpha_\ell(\kappa)$, $\ell=1,2,3$, denote the three real roots of \eqref{eqn:alpha M=3} such that $\alpha_1(\kappa)<\alpha_2(\kappa)<0<\alpha_3(\kappa)$. Our numerical computations indicate ${\rm ind}_{10}(\kappa,\alpha_\ell;3)>0$:
\begin{itemize}
    \item[$\bullet$] for $\kappa\in (0,1.422151\ldots) \bigcup (1.894994\ldots,\infty)$ for $\ell=1$;
    \item[$\bullet$] for $\kappa\in (0.318469\ldots,1.569441\ldots)\bigcup (1.894994\ldots,\infty)$ for $\ell=2$;
    \item[$\bullet$] for $\kappa\in (0.360095\ldots,1.669691\ldots)\bigcup(1.894994\ldots,\infty)$ for $\ell=3$.
\end{itemize}
We pause to remark that 
\[
(w^{(1,0)}_{33}- w^{(1,0)}_{55})^2 - T w^{(2,0)}_{34}=0\quad \text{at $\kappa=1.894994\ldots$},
\]
which corresponds to a vertical asymptote of $a(\kappa;3)$ in \eqref{ab_terms} and, hence, ${\rm ind}_{10}(\cdot,\alpha_\ell;3)$.

Let $M\geq4,\in\mathbb{Z}^+$. Numerical evidence supports $c(\kappa;M)>0$ for all $\kappa>0$ for all $M\geq 4$ while $a(\kappa;M)>0$ and, hence, ${\rm ind}_{10}(\kappa,\alpha;M)>0$
\begin{itemize}
    \item[$\bullet$] for $\kappa\in (0,1.260977\ldots)\bigcup (1.975278\ldots,\infty)$ for $M=4$,
    \item[$\bullet$] for $\kappa\in (0,1.203246\ldots)\bigcup (2.018839\ldots,\infty)$ for $M=5$,
    \item[$\bullet$] for $\kappa\in (0,1.234445\ldots)\bigcup (2.047172\ldots,\infty)$ for $M=6$.
\end{itemize}
Additionally,
\[
(w^{(1,0)}_{33}-w^{(1,0)}_{55})^2 - T w^{(2,0)}_{34}=0
\]
at $\kappa=1.975278\ldots$, $2.018839\ldots$, and $2.047172\ldots$, for $M=4$, $5$, and $6$, respectively, corresponding to vertical asymptotes of $a(\kappa;M)$.

\subsubsection{Wilton ripples of order $2$}\label{sec:modulational M=2}

Suppose $M=2$. Recall from Section \ref{sec:bimodal} and, particularly, \eqref{def:alpha M=2} that $\mu_1\neq 0$. Consequently, several entries of $\mathbf{V}^{(0,1)}(T)$ become nonzero, in contrast to $M\geq 3$ in Lemma \ref{lem:V(m,n)0 M=3}. Such difference in the monodromy matrix causes changes in a series expansion of the spectrum. See \eqref{eqn:Wilton exp2}.


\begin{lemma}\label{lem:V(m,n)0 M=2} 
Assume that \eqref{A:wilton} holds for $M=2$. Then
\[
\mathbf{V}^{(0,1)}(T)=\begin{pmatrix} 
w^{(0,1)}_{11}&w^{(0,1)}_{11}&w^{(0,1)}_{13}&0 &w^{(0,1)}_{15}&0\\ -w^{(0,1)}_{11}&-w^{(0,1)}_{11}&w^{(0,1)}_{13}&0&0&-w^{(0,1)}_{15}\\0&0&0&0&0&0\\w^{(0,1)}_{41}&-w^{(0,1)}_{41}&w^{(0,1)}_{43}&0&w^{(0,1)}_{45}&-w^{(0,1)}_{45}\\w^{(0,1)}_{51}&0&w^{(0,1)}_{53}&0&w^{(0,1)}_{55}&0\\0&-w^{(0,1)}_{51}&w^{(0,1)}_{53}&0&0&-w^{(0,1)}_{55}\end{pmatrix}.
\] 
The formulae for $w^{(0,1)}_{11}$ and $w^{(0,1)}_{55}$ are in Appendix \ref{app:V(m,n) M=2}. Particularly, $w^{(0,1)}_{11}, w^{(0,1)}_{55}\neq0$. The formulae for others are omitted. 

The statements of Lemma \ref{lem:V(m,n)0 M=3} for $\mathbf{V}^{(0,0)}(T)$, $\mathbf{V}^{(1,0)}(T)$, and $\mathbf{V}^{(2,0)}(T)$ remain valid for $M=2$. 
\end{lemma}

\begin{proof}
We solve \eqref{eqn:v_k(0)}-\eqref{def:f(0)}, \eqref{cond:V(0)}, and \eqref{eqn:w_k}-\eqref{eqn:w_k bd}, $k=1,2,\dots,6$, after relabeling $k_j$, $\boldsymbol{\phi}_j$ and $\boldsymbol{\psi}_j$, $j=1,2,\dots,6$, and utilizing the result of Appendix \ref{app:Bexp}, where $\phi_1$, $\phi_2$, $\eta_1$, $\eta_2$, $\mu_1$, and $\mu_2$ are in Section \ref{sec:bimodal} and Appendix \ref{app:Wilton exp}. We then make a straightforward calculation to evaluate $\mathbf{V}^{(m,n)}(T)$ for $ m+n=0,1,2$. For example,
\begin{align*}
w^{(0,1)}_{11}=&e^{-i\kappa T}\Big\langle \int_0^T
e^{i\kappa x}\mathbf{B}^{(0,1)}(x)e^{-i\kappa x}\boldsymbol{\phi}_1~dx, 
\boldsymbol{\psi}_1\Big\rangle \\
=&\Big\langle \int_0^T
\mathbf{B}^{(0,1)}(x)\boldsymbol{\phi}_1~dx, 
\boldsymbol{\psi}_1\Big\rangle.
\end{align*}
Here $\mathbf{B}^{(0,1)}(x)$ is in Appendix \ref{app:Bexp}, where $\phi_1$ and $\eta_1$ are in \eqref{def:wilton1 M=2} and $\mu_1\neq 0$ in \eqref{def:alpha M=2}, $\boldsymbol{\phi}_1$ in \eqref{def:eigen0}, and $\boldsymbol{\psi}_1$ in \eqref{def:adj0}, and $\langle\,,\rangle$ in \eqref{def:inner}. Since $\int^T_0~dx\neq0$, it follows immediately that $w_{11}^{(0,1)}$ is not necessarily zero. In fact, an inspection of Appendix \ref{app:V(m,n) M=2} reveals $w^{(0,1)}_{11}, w^{(0,1)}_{55}\neq0$. Other entries can be treated analogously. We omit the details. 
\end{proof}

We proceed along the lines of the argument in Sections \ref{sec:modulational subcritical} and \ref{sec:modulational M=3}.
Substituting the result of Lemmas \ref{lem:V(m,n)0 M=2} and \ref{lem:V(m,n)0 M=3} into \eqref{def:Delta}, \eqref{def:Vexp}, and \eqref{def:V(m,n)}, we make a straightforward calculation to arrive at
\ba \label{eqn:evans_wilton2}
D(\delta,k_j+\gamma,\eps)=&\sum_{\ell=0}^6d^{(\ell,6-\ell,0)}\delta^{\ell}\gamma^{6-\ell}+\sum_{\ell=0}^6d^{(\ell,6-\ell,1)}\delta^{\ell}\gamma^{6-\ell}\eps+\cdots\ea
as $\delta,\gamma,\eps\to 0$, $j=1,2,\dots,6$, where $d^{(\ell,m,n)}$ can be expressed in terms of $w^{(m,n)}_{jk}$, in Lemmas \ref{lem:V(m,n)0 M=2} and \ref{lem:V(m,n)0 M=3}. We omit their formulae. See \eqref{eqn:d(l,m,n) M=3} for $M\geq 3$. See also \eqref{eqn:evans0} for steady unimodal capillary-gravity waves.

Suppose $\eps=0$, and we define $\delta_j(\gamma,0)=i\sigma(k_j+\gamma)$, $j=1,2,\dots,6$, where $\sigma(k)$ is in \eqref{eqn:sigma}. Since \eqref{def:sigma} is analytic at $k_j$, $j=1,2,\dots,6$, we write 
\[
\delta_j(\gamma,0)=\delta^{(1,0)}_j\gamma+\delta^{(2,0)}_j\gamma^2+\cdots\quad\text{as $\gamma\to 0$},\quad j=1,2,\dots,6,
\]
where $\delta^{(1,0)}_j$ and $\delta^{(2,0)}_j$, $j=1,2,\dots,6$, must coincide with those from the Taylor expansions of \eqref{eqn:sigma} about $k_j$, $j=1,2,\dots,6$. 
Suppose $\eps\neq0,\in\mathbb{R}$ and $|\eps|\ll1$. Let $\eps=s\eps_+^2$ for $\eps_+>0$, where $s=1$ or $-1$, and we assume
\[
\delta_j(\gamma,s\eps_+^2)=\delta^{(1,0)}_j\gamma+\delta^{(2,0)}_j\gamma^2+\delta^{(m'/m,n'/n)}_{j,s}\gamma^{m'/m}\eps_+^{2n'/n}+\cdots\quad \text{as $\gamma,\eps_+\to0$},
\] 
$j=1,2$, for some $m'/m,n'/n\in\mathbb{Q}$. In other words, we assume a Puiseux expansion rather than a Taylor expansion. Inserting this into \eqref{eqn:evans_wilton2}, upon inspection, we find $m'/m=1$ and $n'/n=1/2$. Therefore, 
\ba\label{eqn:Wilton exp2}
\delta_j(\gamma,s\eps_+^2)
=\delta^{(1,0)}_j\gamma+\delta^{(2,0)}_j\gamma^2+\delta^{(1,1/2)}_{j,s}\gamma\eps_++\cdots\quad\text{as $\gamma,\eps_+\to0$},
\ea
$j=1,2$. 

Since $\delta^{(1,0)}_j$ and $\delta^{(2,0)}_j$ agree with the coefficients of the Taylor expansions of \eqref{def:sigma} about $k_j$, $j=1,2,\dots,6$, the former equation of \eqref{def:alpha10_wiltonm}, $j=1,2$, and \eqref{def:alpha20_wiltonm} remain valid for $M=2$. Substituting these and \eqref{eqn:Wilton exp2} into \eqref{eqn:evans_wilton2}, we make a straightforward calculation to show that the coefficient of the $\gamma^6\eps_+^2$ term is:
\ba \label{eqn:g12 M=2}
&-T^4(w^{(1,0)}_{11})^{-6}(w^{(1,0)}_{11} - w^{(1,0)}_{55})^2\\
&\quad\cdot(- (w^{(1,0)}_{11})^2 + 2w^{(1,0)}_{11}w^{(1,0)}_{33}- (w^{(1,0)}_{33})^2 + Tw^{(2,0)}_{34})\\
&\quad\cdot((w^{(1,0)}_{11})^4(\delta^{(1,1/2)}_{j,s})^2-sTw^{(0,1)}_{11}(T(w^{(1,0)}_{11})^2- 2Tw^{(2,0)}_{11} + 2w^{(1,0)}_{13}w^{(1,0)}_{41})),
\ea 
$j=1,2$, for $w^{(m,n)}_{jk}$ in Appendix \ref{app:V(m,n) M=2}.

We state the main result of this subsection. 

\begin{theorem}[Modulational instability of Wilton ripples of order $2$]${}$\label{thm:wilton2}

\noindent Let $\kappa>0$. Suppose that there exists a bimodal Wilton-ripple solution of order $2$ to  \eqref{stationary-periodic}-\eqref{eqn:stokes exp} for $\eps\neq0,\in\mathbb{R}$ and $|\eps|\ll1$, satisfying \eqref{def:wilton1 M=2}. Recall Section \ref{sec:L}, and consider ${\rm spec}(\mathcal{L}(\eps))$ near $i0$. Assume that \eqref{eqn:Wilton exp2} holds for the roots of \eqref{eqn:evans_wilton2} near $(\delta,\gamma,\eps)=(0,0,0)$. 
Let
\ba \label{def:ind_11}
{\rm ind}_{11}(\kappa,\alpha)=
&(w^{(1,0)}_{11})^{-3}(T(w^{(1,0)}_{11})^2 - 2Tw^{(2,0)}_{11} + 2w^{(1,0)}_{13}w^{(1,0)}_{41})\\
=&-\frac{\kappa^2(4e^{2\kappa} + e^{4\kappa} + 1)}{1152\alpha \pi(e^{2\kappa} - 1)^2(e^{2\kappa} + 1)^4(e^{4\kappa} + 1)}\\
&\quad \cdot (42e^{2\kappa} + 75e^{4\kappa} - 212e^{6\kappa} + 75e^{8\kappa} + 42e^{10\kappa}- 11e^{12\kappa} \\& 
\qquad+ 144\kappa^2e^{2\kappa} + 144\kappa^2e^{4\kappa} + 144\kappa^2e^{8\kappa} + 144\kappa^2e^{10\kappa}\\& 
\qquad+ 120\kappa e^{2\kappa} + 48\kappa e^{4\kappa} - 48\kappa e^{8\kappa} - 120\kappa e^{10\kappa} - 11).
\ea

Then, there exists $\lambda\in{\rm spec}(\mathcal{L}(\eps))$ such that ${\rm Re}\,\lambda\neq0$, implying modulational instability, provided either ${\rm ind}_{11}(\kappa,\alpha)>0$ and $\eps>0$, or ${\rm ind}_{11}(\kappa,\alpha)<0$ and $\eps<0$.
\end{theorem}

\begin{proof}
If \eqref{eqn:Wilton exp2} is valid and makes the roots of $D(\cdot,\gamma,\eps)$ for $\gamma,\eps\in\mathbb{R}$ and $|\gamma|,|\eps|\ll1$, then \eqref{eqn:g12 M=2}, $j=1,2$, must vanish. A straightforward calculation reveals $(\delta^{(1,1/2)}_{j,s})^2=s\,{\rm ind}_{11}(\kappa,\alpha)T$. Therefore,
\[
(\delta^{(1,1/2)}_{j,s}\eps_+)^2={\rm ind}_{11}(\kappa,\alpha)T\eps>0,
\]
implying modulational instability, if either ${\rm ind}_{11}(\kappa,\alpha)>0$ and $\eps>0$, or ${\rm ind}_{11}(\kappa,\alpha)<0$ and $\eps<0$. This completes the proof.
\end{proof}

\begin{remark*}\rm
If we substitute \eqref{eqn:Wilton exp2} and the latter equation of \eqref{def:alpha10_wiltonm}, $j=5,6$, into \eqref{eqn:evans_wilton2}, instead of the former, $j=1,2$, we arrive at
\begin{align*} 
&-T^4(w^{(1,0)}_{55})^{-6}(w^{(1,0)}_{11} - w^{(1,0)}_{55})^2 \\
&\quad\cdot (- (w^{(1,0)}_{33})^2 + 2w^{(1,0)}_{33}w^{(1,0)}_{55}- (w^{(1,0)}_{55})^2 + Tw^{(2,0)}_{34})\\
&\quad \cdot ((w^{(1,0)}_{55})^4(\delta^{(1,1/2)}_{j,s})^2-sTw^{(0,1)}_{55}(T(w^{(1,0)}_{55})^2 - 2Tw^{(2,0)}_{55} + 2w^{(1,0)}_{53}w^{(1,0)}_{45}))=0,
\end{align*}
$j=5,6$. This would lead to an index
\ba \label{def:ind_12}
{\rm ind}_{12}(\kappa,\alpha)=
&(w^{(1,0)}_{55})^{-3}(T(w^{(1,0)}_{55})^2 - 2Tw^{(2,0)}_{55} + 2w^{(1,0)}_{53}w^{(1,0)}_{45})\\
=&-\frac{\kappa^2 e^{-8\kappa} (4 e^{2\kappa} + e^{4\kappa} + 1)}{4608 \alpha \pi \sh(4\kappa)^2 (e^{4\kappa} + 1)}\\
&\cdot (32 e^{2\kappa} + 64 e^{4\kappa} - 32 e^{6\kappa} - 82 e^{8\kappa} - 32 e^{10\kappa}\\
&\quad + 64 e^{12\kappa} + 32 e^{14\kappa} - 23 e^{16\kappa} + 576 \kappa^2 e^{4\kappa} \\
&\quad - 576 \kappa^2 e^{8\kappa} + 576 \kappa^2 e^{12\kappa}+ 336 \kappa e^{4\kappa}\\
&\quad - 384 \kappa e^{6\kappa} + 384 \kappa e^{10\kappa} - 336 \kappa e^{12\kappa} - 23),
\ea
whose sign would determine modulational instability in a similar manner as in Theorem \ref{thm:wilton2}. 

We have recently learned, via personal communication, of ongoing work on the modulational instability of Wilton ripples of order $2$. This work indicates that \eqref{eqn:Wilton exp2} holds for $j=1,2$, and the resulting computation agrees with the prediction of \eqref{def:ind_11}. However, it indicates that the spectrum associated with $j=5,6$ exhibits a different instability mechanism. It thus remains unclear whether \eqref{def:ind_12} provides a modulational instability criterion. Accordingly, Theorem \ref{thm:wilton2} is stated only for \eqref{def:ind_11}.
\end{remark*}

Recall \eqref{def:alpha M=2}, and let 
$$
\alpha_\pm=\pm(8\ch(2\kappa))^{-1/2}.
$$ 
Clearly, the sign of ${\rm ind}_{11}(\kappa,\alpha_+)$ is opposite to the sign of ${\rm ind}_{11}(\kappa,\alpha_-)$ unless it is zero.
A straightforward calculation reveals
\[
{\rm ind}_{11}(\kappa,\alpha)=-\frac{\kappa^2(4e^{2\kappa} + e^{4\kappa} + 1)^3}{288\alpha \pi(e^{2\kappa} - 1)^2(e^{2\kappa} + 1)^6(e^{4\kappa} + 1)}{\rm ind}_{5,2}(\beta_W(\kappa;2),\kappa),
\]
where ${\rm ind}_{5,2}(\beta,\kappa)$ is in Appendix \ref{app:ind} (see also \eqref{g_i}) and $\beta_W(\kappa;2)$ in \eqref{def:beta_W}. Recall from Section \ref{sec:modulational supercritical} that ${\rm ind}_{5,2}(\beta,\kappa)=0$ along the curve $\beta=\beta_2(\kappa)$ and that Wilton ripples of order $2$ take place along the curve $\beta=\beta_3(\kappa)$. See the left panel of Figure \ref{figure15}. A direct calculation reveals that the two curves intersect at 
\[
(\beta,\kappa)=(0.19678\ldots,1.28523\ldots),
\]
so that ${\rm ind}_{5,2}(\kappa)$ and, hence, ${\rm ind}_{11}(\kappa,\alpha)$ change their sign at $\kappa_*:=1.28523\ldots$.

Under the hypothesis of Theorem \ref{thm:wilton2}, suppose $0<\eps\ll1$. Taking $\alpha=\alpha_+$, if $\kappa>\kappa_*=1.28523\ldots$ then $\delta_j(\gamma,\eps)\in {\rm spec}(\mathcal{L}(\eps))$, $j=1,2$, in \eqref{eqn:Wilton exp2}, leave the imaginary axis at the order of $\gamma\eps^{1/2}$ for $\gamma,\eps\in\mathbb{R}$ and $|\gamma|,|\eps|\ll1$, implying modulational instability. If $\kappa<\kappa_*$, on the other hand, $\delta_j(\gamma,\eps)$, $j=1,2$, remain confined to the imaginary axis near $i0$ up to order $\gamma\eps^{1/2}$ as $\gamma,\eps\to 0$.

Let $\alpha=\alpha_-$ instead. If $\kappa>\kappa_*$ then $\delta_j(\gamma,\eps)$, $j=1,2$, remain along $i\mathbb{R}$ near $i0$ up to order $\gamma\eps^{1/2}$ as $\gamma,\eps\to 0$, whereas if $\kappa<\kappa_*$ then they leave the imaginary axis at order $\gamma\eps^{1/2}$ for $\gamma,\eps\in\mathbb{R}$ and $|\gamma|,|\eps|\ll1$.

For $\eps<0$ and $|\eps|\ll1$, the modulational stability and instability properties at order $\gamma|\eps|^{1/2}$ are reversed.

\newpage

\begin{appendix}

\section{The Stokes expansion for Wilton ripples}\label{app:Wilton exp} 

Recall the shorthand notation:
\begin{align*}
&\s(\cdot)=\sin(\cdot),& &\c(\cdot)=\cos(\cdot), & & \\
&\sh(\cdot)=\sinh(\cdot),& &\ch(\cdot)=\cosh(\cdot),& &\th(\cdot)=\tanh(\cdot).
\end{align*}

\subsubsection*{Wilton ripples of order $2$} We begin by setting $M=2$. At the order of $\eps^2$, we solve \eqref{stationary-periodic}-\eqref{eqn:stokes exp} by the same procedure as in Section \ref{sec:unimodal}, where $\phi_1$ and $\eta_1$ are in \eqref{def:wilton1 M=2} and $\mu_0$ in \eqref{def:mu0}, and make a straightforward calculation to obtain:
\begin{align*}
\phi_2(x,y)=
&-\alpha\kappa(2 + 3 \sh(\kappa)^2)\s(\kappa x)\ch(\kappa y)
-\tfrac{1}{2}\alpha\kappa \sh(2\kappa)y\s(\kappa x)\sh(\kappa y)\\
&+\alpha\kappa \sh(\kappa)y\s(\kappa x)\sh(2\kappa y)
+\tfrac{1}{2}\kappa \sh(\kappa)y\s(2\kappa x)\sh(\kappa y)\\
&+\alpha'\s(2\kappa x)\ch(2\kappa y)
+\tfrac{1}{2}\alpha\kappa \sh(2\kappa)y\s(3\kappa x)\sh(\kappa y)\\
&+\alpha\kappa \sh(\kappa)y\s(3\kappa x)\sh(2\kappa y)
+\alpha^2\kappa \sh(2\kappa)y\s(4\kappa x)\sh(2\kappa y)\\
&-\frac{\alpha\kappa(\ch(2\kappa) + 9\ch(2\kappa)^2 + 8)}{4(\ch(2\kappa) - 1)^2}\s(3\kappa x)\ch(3\kappa y)\\
&-\frac{\alpha^2\kappa(2\ch(4\kappa) - 4\ch(2\kappa) + 5)}{16\sh(\kappa)^4} \s(4\kappa x)\ch(4 \kappa y)\\
\eta_2(x)=&\tfrac{1}{4}\sh(2\kappa)(\kappa + 4\alpha')\c(2\kappa x) \\
&-\frac{3\alpha \kappa(16\ch(2\kappa) + 7\ch(4\kappa) + 2\ch(6\kappa) + 11)}{32\sh(\kappa)^3}\c(3\kappa x)\\
& -\frac{\alpha^2\kappa(4\sh(4\kappa) + \sh(8\kappa))}{8(\ch(2\kappa) - 1)^2}\c(4\kappa x),
\end{align*}
for some $\alpha'$ to be determined.
We pause to remark that after rescaling $\eps$, if necessary, we may assume without loss of generality that $\eta_2$, $\eta_3$, \dots do not contain $\c(\kappa x)$ in the second equation of \eqref{eqn:stokes exp}. 

At the order of $\eps^3$, we solve \eqref{stationary-periodic}-\eqref{eqn:stokes exp} by the same procedure to obtain: 
\begin{align*}
\alpha'=&-\kappa(128\ch(2\kappa)^2\sh(\kappa)^4(5\sh(\kappa)^2 + 2\sh(\kappa)^4 + 3))^{-1} \\
&\quad \cdot (687 \sh(\kappa)^2 + 2363\sh(\kappa)^4 + 4530\sh(\kappa)^6 \\
&\qquad + 4988\sh(\kappa)^8 + 3000\sh(\kappa)^{10} + 768\sh(\kappa)^{12} + 90)
\intertext{and}
\mu_2=&\frac{\kappa^3(\ch(2\kappa) + 2)}{1536\ch(2\kappa)^2\sh(2\kappa)(\ch(2\kappa) - 1)^2(\ch(2\kappa) + 1)(2\ch(2\kappa) + 1)}\\
&\cdot(16398\ch(2\kappa) + 11611\ch(4\kappa) + 6805\ch(6\kappa) + 3002\ch(8\kappa)\\
& \quad + 1213\ch(10\kappa) + 325\ch(12\kappa) + 32\ch(14\kappa) + 8998).
\end{align*}
The formulae for $\phi_3$, $\eta_3$, $q_3$, \dots are omitted.

\subsubsection*{Wilton ripples of order $\geq 3$} 
Turning to $M\geq 3$, at the order of $\eps^2$, we solve \eqref{stationary-periodic}-\eqref{eqn:stokes exp} by the same procedure as for $M=2$, to obtain:
\begin{align*}
\phi_2(x,y)
=&\tfrac{1}{2}\kappa \sh(\kappa)y\s(2\kappa x)\sh(\kappa y)+\tfrac{1}{2}M\alpha ^2\kappa \sh(M\kappa)y\s(2M\kappa x)\sh(M\kappa y)\\
&+\tfrac{1}{2}\alpha\kappa \sh(M\kappa)y\s((M+1)\kappa x)\sh(\kappa y)\\
&+\tfrac{1}{2}M\alpha \kappa \sh(\kappa)y\s((M+1)\kappa x)\sh(M\kappa y)\\
&-\tfrac{1}{2}\alpha\kappa \sh(M\kappa)y\s((M-1)\kappa x)\sh(\kappa y)\\
&+\tfrac{1}{2}M\alpha \kappa \sh(\kappa)y\s((M-1)\kappa x)\sh(M\kappa y)\\
&-\tfrac34\kappa p_2r_2^{-1}\s(2\kappa x)\ch(2\kappa y)\\
&+\tfrac38\alpha ^2\kappa M p_{2M}r_{2M}^{-1}\s(2M\kappa x)\ch(2M\kappa y)\\
&-\tfrac12\alpha\kappa p_{M+1}r_{M+1}^{-1}\s((M+1)\kappa x)\ch((M+1)\kappa y)\\
&-\tfrac12\alpha\kappa p_{M-1}r_{M-1}^{-1}\s((M-1)\kappa x)\ch((M-1)\kappa y),
\intertext{and}
\eta_2(x)
=&f_2g_2^{-1}\c(2\kappa x) -f_{2M}g_{2M}^{-1}\c(2M\kappa x)\\
&\begin{aligned}-2\alpha \kappa(M^2 - 1) (f_{M+1}g_{M+1}^{-1} \c((M + 1)\kappa x)+f_{M-1}g_{M-1}^{-1}\c((M - 1)\kappa x)).\end{aligned}
\end{align*}
Additionally, $q_3=0$. Here,
\begin{align*}
p_2\,\,=\,\,&2\sh(M\kappa) + \ch(2\kappa)\sh(M\kappa)- M^2\sh(M\kappa) - M\sh(2\kappa)\ch(M\kappa), \\
r_2\,\,=\,\,&4\sh(M\kappa) + 2\ch(2\kappa)\sh(M\kappa) - M^2\sh(M\kappa)\\ 
&+ M^2\ch(2\kappa)\sh(M\kappa) - 3M\sh(2\kappa)\ch(M\kappa),\\
p_{2M}\,=\,&(1 - M^2)\sh(\kappa)  - 2M^2\ch(M\kappa)^2\sh(\kappa) 
+ 2M\ch(M\kappa)\sh(M\kappa)\ch(\kappa),\\
r_{2M}\,=\,&M^2\sh(\kappa) - \sh(\kappa) + \ch(M\kappa)^2\sh(\kappa)\\ 
&+ 2M^2\ch(M\kappa)^2\sh(\kappa)- 3M\ch(M\kappa)\sh(M\kappa)\ch(\kappa), \\
p_{M+1}=&3M - 2\ch(2M\kappa) + 3M^3\ch(2\kappa) + 2M^4\ch(2\kappa) \\
&- 3M\ch(2M\kappa) - M^3\ch(2M\kappa) - 3M^3 - 2M^4 + M\ch(2\kappa)\\
& + 2M^3\sh(2\kappa)\sh(2M\kappa) - M\ch(2\kappa)\ch(2M\kappa)\\
& - 2M\sh(2\kappa)\sh(2M\kappa) + M^3\ch(2\kappa)\ch(2M\kappa) + 2\\
r_{M+1}=&3M - 2\ch(2M\kappa) + M^2\ch(2(M + 1)\kappa) + 3M^2\ch(2\kappa) \\
& + 2M^3\ch(2\kappa)  - 3M\ch(2M\kappa)- M\ch(2(M + 1)\kappa)\\
& - M^2\ch(2M\kappa) - 3M^2 - 2M^3 + M\ch(2\kappa) + 2\\
p_{M-1}=&3M + 2\ch(2M\kappa) + 3M^3\ch(2\kappa) - 2M^4\ch(2\kappa) \\
& - 3M\ch(2M\kappa)- M^3\ch(2M\kappa) - 3M^3 + 2M^4 + M\ch(2\kappa) \\
&- 2M^3\sh(2\kappa)\sh(2M\kappa) - M\ch(2\kappa)\ch(2M\kappa) \\
&+ 2M\sh(2\kappa)\sh(2M\kappa) + M^3\ch(2\kappa)\ch(2M\kappa) - 2\\
r_{M-1}=&3M + 2\ch(2M\kappa) - M^2\ch(2(M - 1)\kappa) - 3M^2\ch(2\kappa)\\
& + 2M^3\ch(2\kappa) - 3M\ch(2M\kappa) - Mc(2(M - 1)) \\
&+ M^2\ch(2M\kappa) + 3M^2 - 2M^3 + M\ch(2\kappa) - 2,
\intertext{and}
f_2\,\,=\,\,&\kappa \sh(M\kappa) \ch(\kappa) \sh(\kappa) (M^2 - 1) (2 \ch(\kappa)^2 + 1)\\
g_2\,\,=\,\,&4  \sh(M\kappa) + 8  \sh(M\kappa) \ch(\kappa)^2 - 4 M^2  \sh(M\kappa) \\
&+ 4 M^2  \sh(M\kappa) \ch(\kappa)^2 - 12 M \ch(M\kappa) \ch(\kappa) \sh(\kappa),\\
f_{2M}\,=\,&\alpha^2\kappa M\ch(M\kappa)\sh(M\kappa)\sh(\kappa)(M^2 - 1)(\ch(2M\kappa) + 2),\\
g_{2M}\,=\,&4(M^2\sh(\kappa) - \sh(\kappa) + \ch(M\kappa)^2\sh(\kappa) \\
&+ 2M^2\ch(M\kappa)^2\sh(\kappa)- 3M\ch(M\kappa)\sh(M\kappa)\ch(\kappa)),\\
f_{M+1}=&\ch(M\kappa)^3\sh(\kappa) - \ch(M\kappa)\sh(\kappa)\\
&+ M^2\sh(M\kappa)\ch(\kappa)^3 - M^2\sh(M\kappa)\ch(\kappa)\\
& - 2M\ch(M\kappa)^2\sh(M\kappa)\ch(\kappa)- 2M\ch(M\kappa)\ch(\kappa)^2\sh(\kappa)\\
&+ 2M\ch(M\kappa)^2\sh(M\kappa)\ch(\kappa)^3  + 2M\ch(M\kappa)^3\ch(\kappa)^2\sh(\kappa),\\
g_{M+1}=&3M - 2\ch(2M\kappa) + M^2\ch(2(M + 1)\kappa) + 3M^2\ch(2\kappa) \\
&+ 2M^3\ch(2\kappa) - 3M\ch(2M\kappa) - M\ch(2(M + 1)\kappa) \\
&- M^2\ch(2M\kappa) - 3M^2 - 2M^3 + M\ch(2\kappa) + 2, \\
f_{M-1}=& \ch(M\kappa)^3 \sh(\kappa)- \ch(M\kappa) \sh(\kappa)\\
& -M^2  \sh(M\kappa) \ch(\kappa)^3+M^2  \sh(M\kappa) \ch(\kappa)\\
&-2 M  \ch(M\kappa)^2  \sh(M\kappa) \ch(\kappa)+2 M  \ch(M\kappa) \ch(\kappa)^2 \sh(\kappa)\\
&+2 M  \ch(M\kappa)^2  \sh(M\kappa) \ch(\kappa)^3-2 M  \ch(M\kappa)^3 \ch(\kappa)^2 \sh(\kappa),\\
g_{M-1}=&(3 M+2  \ch(2M\kappa)-M^2 \ch(2(M-1)\kappa)-3 M^2  \ch(2\kappa)\\
&+2 M^3  \ch(2\kappa)-3 M  \ch(2M\kappa)-M \ch(2(M-1)\kappa)\\
&+M^2  \ch(2M\kappa)+3 M^2-2 M^3+M  \ch(2\kappa)-2).
\end{align*}

Let $M\geq 4$. A straightforward calculation reveals
\[
\mu_2=\mu_{2,2}(\kappa;M)\alpha^2+\mu_{2,0}(\kappa;M),
\] 
where $\mu_{2,0}(\kappa;M)$ is in \eqref{def:mu2} and 
\begin{align*}
\mu_{2,2}&(\kappa;M)\\
=&(4 \ch(\kappa) (M^2 - 1)^2 (\ch(\kappa)^2 - 1))^{-1}\\
&\cdot(-2 \kappa^3 M^{12} \sh(M\kappa) \sh(\kappa)^5 (\sh(\kappa)^2 + 1) \\
&\quad-\kappa^3m_{11} M^{11} \ch(M\kappa) \ch(\kappa)(\ch(\kappa)^2 - 1)^2 \\ 
&\quad + \kappa^3 m_{10}M^{10} \sh(M\kappa) \sh(\kappa)^3  -\kappa^3  m_9M^9 \ch(M\kappa)\ch(\kappa)(\ch(\kappa)^2 - 1)\\
&\quad -\kappa^3 m_8M^8 \sh(M\kappa) \sh(\kappa) 
-2 \kappa^3m_7 M^7 \ch(M\kappa) \ch(\kappa) \\
&\quad + \kappa^3m_6 M^6 \sh(M\kappa) \sh(\kappa) 
+ \kappa^3 m_5M^5 \ch(M\kappa) \ch(\kappa) (\ch(M\kappa)^2 - 1) \\
&\quad -\kappa^3m_4 M^4 \sh(M\kappa)^3 \sh(\kappa)
+ \kappa^3m_3 M^3 \ch(M\kappa) \ch(\kappa) (\ch(M\kappa)^2 - 1)^2  \\
&\quad + 2 \kappa^3 M^2\sh(M\kappa)^5 \sh(\kappa)(\sh(M\kappa)^2 + 1))\\
&\cdot (-M^6 (\sh(M\kappa)- 2 \sh(M\kappa) \ch(\kappa)^2 + \sh(M\kappa) \ch(\kappa)^4)\\
&\quad \,+ m_5'M^5 -m_4'M^4 + m_3'M^3+ m_2'M^2 + m_1'M \\
&\quad + \sh(M\kappa)- 2 \ch(M\kappa)^2 \sh(M\kappa) + \ch(M\kappa)^4 \sh(M\kappa))^{-1}.
\end{align*}
Here 
\begin{align*}
m_{11}=&\ch(M\kappa)^2 \ch(\kappa)^2 - 5 \ch(\kappa)^2 + 2 \ch(M\kappa)^4 \ch(\kappa)^2 + \ch(M\kappa)^2 + 1,\\
m_{10}=& 2 \sh(M\kappa)^2 - 12 \sh(M\kappa)^4+ 16 \sh(\kappa)^2 + 15 \sh(\kappa)^4 + 23 \sh(M\kappa)^2 \sh(\kappa)^2 \\
&+ 20 \sh(M\kappa)^2 \sh(\kappa)^4 - 14 \sh(M\kappa)^4 \sh(\kappa)^2 - 2 \sh(M\kappa)^4 \sh(\kappa)^4,\\
m_9=& 16 \ch(\kappa)^2 - 34 \ch(\kappa)^4 - 37 \ch(M\kappa)^2 \ch(\kappa)^2 \\
& + 73 \ch(M\kappa)^2 \ch(\kappa)^4 - 29 \ch(M\kappa)^4 \ch(\kappa)^4 + \ch(M\kappa)^6 \ch(\kappa)^2 \\
&- 2 \ch(M\kappa)^2 + 9 \ch(M\kappa)^4 + \ch(M\kappa)^6 + 2,\\
m_8=&14 \sh(\kappa)^4 - 16 \sh(M\kappa)^6 + 6 \sh(\kappa)^6  + 30 \sh(M\kappa)^2 \sh(\kappa)^2 \\&+ 103 \sh(M\kappa)^2 \sh(\kappa)^4 + 23 \sh(M\kappa)^4 \sh(\kappa)^2+ 61 \sh(M\kappa)^2 \sh(\kappa)^6 \\ &  + 84 \sh(M\kappa)^4 \sh(\kappa)^4 - 31 \sh(M\kappa)^6 \sh(\kappa)^2 \\ &  + 56 \sh(M\kappa)^4 \sh(\kappa)^6 - 14 \sh(M\kappa)^6 \sh(\kappa)^4, \\
m_7=&12 \ch(\kappa)^2 - 17 \ch(\kappa)^4+ \ch(\kappa)^6 + 7 \ch(M\kappa)^2 \ch(\kappa)^4 \\&- 7 \ch(M\kappa)^4 \ch(\kappa)^2 + 17 \ch(M\kappa)^2 \ch(\kappa)^6 - 17 \ch(M\kappa)^6 \ch(\kappa)^2\\& - 22 \ch(M\kappa)^4 \ch(\kappa)^6+ 22 \ch(M\kappa)^6 \ch(\kappa)^4 \\ &- 12 \ch(M\kappa)^2 + 17 \ch(M\kappa)^4 - \ch(M\kappa)^6,\\
m_6=&14 \sh(M\kappa)^4 + 6 \sh(M\kappa)^6 - 16 \sh(\kappa)^6 + 30 \sh(M\kappa)^2 \sh(\kappa)^2 \\
&+ 23 \sh(M\kappa)^2 \sh(\kappa)^4 + 103 \sh(M\kappa)^4 \sh(\kappa)^2 - 31 \sh(M\kappa)^2 \sh(\kappa)^6 \\
&+ 84 \sh(M\kappa)^4 \sh(\kappa)^4 + 61 \sh(M\kappa)^6 \sh(\kappa)^2 \\ &- 14 \sh(M\kappa)^4 \sh(\kappa)^6 + 56 \sh(M\kappa)^6 \sh(\kappa)^4,\\
m_5=&9 \ch(\kappa)^4 - 2 \ch(\kappa)^2 + \ch(\kappa)^6 - 37 \ch(M\kappa)^2 \ch(\kappa)^2 \\ & + 73 \ch(M\kappa)^4 \ch(\kappa)^2 + \ch(M\kappa)^2 \ch(\kappa)^6 \\ &  - 29 \ch(M\kappa)^4 \ch(\kappa)^4 + 16 \ch(M\kappa)^2 - 34 \ch(M\kappa)^4 + 2,\\
m_4=&16 \sh(M\kappa)^2 + 15 \sh(M\kappa)^4 + 2 \sh(\kappa)^2 - 12 \sh(\kappa)^4 + 23 \sh(M\kappa)^2 \sh(\kappa)^2 \\
&- 14 \sh(M\kappa)^2 \sh(\kappa)^4 + 20 \sh(M\kappa)^4 \sh(\kappa)^2 - 2 \sh(M\kappa)^4 \sh(\kappa)^4,\\
m_3=&\ch(\kappa)^2 + \ch(M\kappa)^2 \ch(\kappa)^2 + 2 \ch(M\kappa)^2 \ch(\kappa)^4 - 5 \ch(M\kappa)^2 + 1,
\intertext{and}
m_5'=&2 \ch(M\kappa) \ch(\kappa)^3 \sh(\kappa) + 2 \ch(M\kappa)^3 \ch(\kappa) \sh(\kappa) \\
&- 2 \ch(M\kappa)^3 \ch(\kappa)^3 \sh(\kappa) - 2 \ch(M\kappa) \ch(\kappa) \sh(\kappa), \\
m_4'=&2 \sh(M\kappa) \ch(\kappa)^2 - \ch(M\kappa)^4 \sh(M\kappa) - 3 \sh(M\kappa) \\
& + \sh(M\kappa) \ch(\kappa)^4 + 5 \ch(M\kappa)^2 \sh(M\kappa) \ch(\kappa)^2\\
&- 5 \ch(M\kappa)^2 \sh(M\kappa) \ch(\kappa)^4 + \ch(M\kappa)^4 \sh(M\kappa) \ch(\kappa)^2,\\
m_3'=&4 \ch(M\kappa) \ch(\kappa)^3 \sh(\kappa) - 4 \ch(M\kappa)^3 \ch(\kappa) \sh(\kappa) \\
&+ 4 \ch(M\kappa)^5 \ch(\kappa) \sh(\kappa) - 4 \ch(M\kappa)^3 \ch(\kappa)^3 \sh(\kappa),\\
m_2'=&2 \ch(M\kappa)^2 \sh(M\kappa) - 3 \sh(M\kappa) + \ch(M\kappa)^4 \sh(M\kappa) \\
& - \sh(M\kappa) \ch(\kappa)^4 + 5 \ch(M\kappa)^2 \sh(M\kappa) \ch(\kappa)^2\\& + \ch(M\kappa)^2 \sh(M\kappa) \ch(\kappa)^4 - 5 \ch(M\kappa)^4 \sh(M\kappa) \ch(\kappa)^2,\\
m_1'=& 2\ch(M\kappa) (\ch(M\kappa)^2 - 1)^2 \ch(\kappa) \sh(\kappa) .
\end{align*}

When $M=3$, 
\[
\mu_2=\mu_{2,2}(\kappa;3)\alpha^2+\mu_{2,1}(\kappa;3)\alpha+\mu_{2,0}(\kappa;3),
\] 
where
\begin{align*}
\mu_{2,1}(\kappa;3)=
&\frac{-3\kappa^3\ch(\kappa)^3}{8\sh(\kappa)^5(4\sh(\kappa)^2 + 3)^2(6\sh(\kappa)^2 + 7)}\\
&\cdot (108\ch(\kappa)^6 - 18\ch(\kappa)^4 + 351\ch(\kappa)^8 \\&\quad - 4194\ch(\kappa)^{10} + 8856\ch(\kappa)^{12}).
\end{align*}
The formulae for $\phi_3$, $\eta_3$, $q_3$, $\dots$ are omitted.

\section{Expansion of \texorpdfstring{$\mathbf{B}(x;\sigma,\delta,\epsilon)$}{Lg}}\label{app:Bexp}
 
Recall \eqref{eqn:LB;delta} and \eqref{def:Bexp}. We use the result of Sections \ref{sec:linearize} and \ref{sec:L}, and make a straightforward calculation to obtain
\[
\mathbf{B}^{(1,0)}(x;\sigma)\mathbf{u}=\begin{pmatrix}-\dfrac{y^2z}{2 \beta}\\-\eta\\0\\\dfrac{\varphi(1) + i\sigma\eta}{\mu_0} \end{pmatrix}
\text{ and } 
\mathbf{B}^{(2,0)}(x;\sigma)\mathbf{u}=\begin{pmatrix}0\\0\\0\\ \dfrac{\eta}{2\mu_0}\end{pmatrix},\quad 
\mathbf{u}=\begin{pmatrix} \varphi\\u \\\eta\\  z \end{pmatrix}
\]
where $\mu_0=\mu_0(\kappa;\beta)$ is in \eqref{def:mu0}. Note that, more generally, $\mathbf{B}^{(m,0)}(x;\sigma)$, $m\geq1$, do not involve $x$. 

We write $\mathbf{B}^{(m,n)}(x;\sigma)\mathbf{u}=\begin{pmatrix}
(\mathbf{B}^{(m,n)}\mathbf{u})_1\\
(\mathbf{B}^{(m,n)}\mathbf{u})_2 \\
(\mathbf{B}^{(m,n)}\mathbf{u})_3\\
(\mathbf{B}^{(m,n)}\mathbf{u})_4\end{pmatrix}$, for brevity, and $\mathbf{u}=\begin{pmatrix} \varphi\\u \\\eta\\  z \end{pmatrix}$. We use the result of Sections \ref{sec:Stokes} and \ref{sec:linearize}, and make a straightforward calculation to obtain:
\begin{align*}
(\mathbf{B}^{(0,1)}\mathbf{u})_1=&({\phi_1}_{xx}(1) - 2 {\eta_1}_x)\varphi+2 y {\eta_1}_x\varphi_y+{\eta_1}_x\varphi(0)\\
&+({\phi_1}_x(1) - \eta_1)u-{\eta_1}_{xx}\int_0^yy'u(\cdot,y')~dy'\\
&+\frac{2i\sigma y^2 {\eta_1}_x  - y^2 {\phi_1}_{xy}(1)}{2}\eta \\
&+\frac{2 y {\phi_1}_y - y^2 {\phi_1}_y(1) -i \sigma y^2 {\phi_1}_x(1) }{2 \beta}z,\\
(\mathbf{B}^{(0,1)}\mathbf{u})_2=&({\phi_1}_x(1) + \eta_1)\varphi_{yy}-{\eta_1}_x u\\
&+(i\sigma \eta_1 - {\phi_1}_y(1) + 2 {\phi_1}_{yy})\eta+\frac{y {\phi_1}_{xy}}{\beta}z,\\
(\mathbf{B}^{(0,1)}\mathbf{u})_3=&0,\\
(\mathbf{B}^{(0,1)}\mathbf{u})_4=&-\frac{i\sigma}{\mu_0^2 } (\mu_0{\phi_1}_x(1) -\mu_0 \eta_1+\mu_1)\varphi(1)\\
&+\frac{\mu_0{\phi_1}_x(1)+\mu_1}{\mu_0^2}u(1)+\frac{i\sigma {\eta_1}_x}{\mu_0}\int_0^1yu(\cdot,y)~dy \\
&+\frac{\sigma }{2\mu_0^2 }(i\mu_0{\phi_1}_y(1) - \sigma\mu_0 \eta_1+\mu_1\sigma)\eta
-\frac{{\phi_1}_y(1)}{\beta \mu_0}z,
\intertext{and}
(\mathbf{B}^{(1,1)}\mathbf{u})_1=&y^2 {\eta_1}_x\eta-\frac{y^2 {\phi_1}_x(1)}{2 \beta}z,\\
(\mathbf{B}^{(1,1)}\mathbf{u})_2=&\eta_1\eta,\\
(\mathbf{B}^{(1,1)}\mathbf{u})_3=&0,\\
(\mathbf{B}^{(1,1)}\mathbf{u})_4=&\frac{\mu_0 \eta_1 - \mu_0 {\phi_1}_x(1) - \mu_1}{\mu_0^2}\varphi(1)+\frac{{\eta_1}_x}{\mu_0}\int_0^1yu(\cdot,y)~dy\\
&+\frac{\mu_0 {\phi_1}_y(1) - 2i\sigma\mu_1  + 2i\mu_0 \sigma \eta_1 }{2 \mu_0^2}\eta,
\intertext{where $\phi_1$, $\eta_1$, and $\mu_1$ are in Section \ref{sec:Stokes}. We similarly make a straightforward calculation to obtain:}
(\mathbf{B}^{(0,2)}\mathbf{u})_1
=&(\mathbf{B}^{(0,2)}\mathbf{u})_{1,\varphi}
+(\mathbf{B}^{(0,2)}\mathbf{u})_{1,u}
+(\mathbf{B}^{(0,2)}\mathbf{u})_{1,\eta}
+(\mathbf{B}^{(0,2)}\mathbf{u})_{1,z}\\
(\mathbf{B}^{(0,2)}\mathbf{u})_2
=&( {\phi_2}_x(1) +  \eta_2 -  {\phi_1}_x(1) \eta_1 - 3 {\eta_1}_x^2/2 - \eta_1^2 -  {\phi_1}_y(1) {\eta_1}_x)\varphi_{yy}\\
&+(\eta_1 {\eta_1}_x - {\eta_2}_x)u\\
&+(3 {\phi_1}_y(1) \eta_1 - i\sigma \eta_1^2  - {\phi_2}_y(1)\\
&\quad  + i\sigma \eta_2 - 6 \eta_1 {\phi_1}_{yy} - y {\eta_1}_x {\phi_1}_{xy} + 2 {\phi_2}_{yy})\eta\\
&+(y {\phi_2}_{xy} - y \eta_1 {\phi_1}_{xy} - y^2 {\eta_1}_x {\phi_1}_{yy} - y {\phi_1}_y {\eta_1}_x)\frac{z}{\beta},\\
(\mathbf{B}^{(0,2)}\mathbf{u})_3=&\frac{3{\eta_1}_x^2}{2\beta}z ,\\
(\mathbf{B}^{(0,2)}\mathbf{u})_4
=&(\mathbf{B}^{(0,2)}\mathbf{u})_{4,\varphi}
+(\mathbf{B}^{(0,2)}\mathbf{u})_{4,u}
+(\mathbf{B}^{(0,2)}\mathbf{u})_{4,\eta}
+(\mathbf{B}^{(0,2)}\mathbf{u})_{4,z},
\end{align*}
where $\phi_2$, $\eta_2$, and $\mu_2$ are in Section \ref{sec:Stokes}. Here
\begin{align*}
(\mathbf{B}^{(0,2)}\mathbf{u})_{1,\varphi}=&({\phi_2}_{xx}(1) - {\phi_1}_y(1) {\eta_1}_{xx}  - 3 {\eta_1}_x {\eta_1}_{xx} - 2 {\eta_2}_x \\ 
&\quad - {\phi_1}_{xy}(1) {\eta_1}_x + {\phi_1}_x(1) {\phi_1}_{xx}(1) + 2 \eta_1 {\eta_1}_x)\varphi \\
&+(2 y {\eta_2}_x - 2 y \eta_1 {\eta_1}_x)\varphi_y +({\eta_2}_x - \eta_1 {\eta_1}_x)\varphi(0)\\  
(\mathbf{B}^{(0,2)}\mathbf{u})_{1,u}=&({\phi_2}_x(1) - \eta_2 - {\phi_1}_x(1) \eta_1 - \tfrac32 {\eta_1}_x^2\\
& \quad +  {\phi_1}_x(1)^2 + \eta_1^2 + y^2 {\eta_1}_x^2 - {\phi_1}_y(1) {\eta_1}_x)u\\
&-( {\phi_1}_x(1) {\eta_1}_{xx} + {\eta_2}_{xx})\int_0^yy'u(\cdot,y')~dy'\\
(\mathbf{B}^{(0,2)}\mathbf{u})_{1,\eta}=&(2 \phi_1 {\eta_1}_x - \tfrac12y^2 {\phi_2}_{xy}(1) - 2 {\eta_1}_x \phi_1(\cdot, 0) +i \sigma y^2 {\eta_2}_x \\
&\quad - 3 y {\phi_1}_y {\eta_1}_x- \tfrac12 y^2 {\phi_1}_x(1) {\phi_1}_{xy}(1) + 2y^2 {\phi_1}_y(1) {\eta_1}_x \\
&\quad + y^2 \eta_1 {\phi_1}_{xy}(1) + i\sigma y^2 {\phi_1}_x(1) {\eta_1}_x  - i\sigma y^2 \eta_1 {\eta_1}_x) \eta\\ 
(\mathbf{B}^{(0,2)}\mathbf{u})_{1,z}=&
\Big( y {\phi_2}_y - \tfrac12 y^2 {\phi_2}_y(1) -\tfrac{i}{2}\sigma y^2 {\phi_2}_x(1)-  {\eta_1}_x\int_0^y(y')^2{\phi_1}_{xy
}(\cdot,y') dy' \\
&\quad  +  y {\phi_1}_x(1) {\phi_1}_y - 2 y \eta_1 {\phi_1}_y  -\tfrac{i}{2} \sigma y^2 {\phi_1}_x(1)^2 \\
&\quad -\tfrac12 y^2 {\phi_1}_x(1) {\phi_1}_y(1)+  y^2 {\phi_1}_y(1) \eta_1 + \tfrac{i}{2}\sigma y^2 {\phi_1}_y(1) {\eta_1}_x \Big) \frac{z}{\beta},
\intertext{and}
(\mathbf{B}^{(0,2)}\mathbf{u})_{4,\varphi}=&i(\mu_1^2 \sigma - \mu_0 \mu_2 \sigma + \mu_0 \mu_1 \sigma {\phi_1}_x(1) - \mu_0 \mu_1 \sigma \eta_1\\
&\quad - \mu_0^2\sigma  ({\phi_2}_x(1) - \eta_2 - \tfrac{3}{2}{\eta_1}_x^2 - {\phi_1}_y(1) {\eta_1}_x + {\phi_1}_x(1) \eta_1)) \frac{\varphi(1)}{\mu_0^3}\\
(\mathbf{B}^{(0,2)}\mathbf{u})_{4,u}=&(\mu_0^2 {\phi_2}_x(1) + \mu_0 \mu_2 - \mu_1^2 - \mu_0 \mu_1 {\phi_1}_x(1))\frac{u(1)}{\mu_0^{3}}\\
&+\frac{i\sigma}{\mu_0^{2}}(\mu_0 {\eta_2}_x  - \mu_1  {\eta_1}_x  + \mu_0  \eta_1 {\eta_1}_x)\int_0^1yu(\cdot,y)~dy\\
(\mathbf{B}^{(0,2)}\mathbf{u})_{4,\eta}=
&(\mu_0 \mu_2 \sigma^2 - \mu_1^2 \sigma^2 + i\sigma\mu_0^2  {\phi_2}_y(1)- \mu_0^2 \sigma^2 \eta_2\\ 
&\quad - i\sigma\mu_0^2  {\phi_1}_y(1) \eta_1 - i\sigma\mu_0 \mu_1  {\phi_1}_y(1)  + \mu_0 \mu_1 \sigma^2 \eta_1)\frac{\eta}{2\mu_0^{3}}\\
(\mathbf{B}^{(0,2)}\mathbf{u})_{4,z}=&(\mu_1 {\phi_1}_y(1) - \mu_0 {\phi_2}_y(1) + \mu_0 {\phi_1}_x(1) {\phi_1}_y(1) + \mu_0 {\phi_1}_y(1) \eta_1) \frac{z}{\mu_0^{2}\beta},
\end{align*}

The formulae of $\mathbf{B}^{(m,n)}(x;\sigma)$, $m+n\geq3$, are not included.

\section{Formulae for \texorpdfstring{$\mathbf{V}^{(m,n)}(T;0)$}{Lg} for unimodal waves}\label{app:V(m,n) unimodal} 

Recall the shorthand notation:
\begin{align*}
&\s(\cdot)=\sin(\cdot),& &\c(\cdot)=\cos(\cdot), & & \\
&\sh(\cdot)=\sinh(\cdot),& &\ch(\cdot)=\cosh(\cdot),& &\th(\cdot)=\tanh(\cdot).
\end{align*}

We solve \eqref{def:V(0,0)0}-\eqref{def:f(0)} and \eqref{cond:V(0)}, using the result of Appendix \ref{app:Bexp}, and we make an explicit calculation to obtain:
\begin{align*}
w^{(1,0)}_{11}=&\frac{-4 \pi \sh(2\kappa) (\beta \kappa^2 + 1)}{\kappa (2 \kappa - \sh(2\kappa) + 2 \beta \kappa^3 + \beta \kappa^2 \sh(2\kappa))},\\
w^{(1,0)}_{13}=&\frac{4 i\pi \sh(\kappa) (\beta \kappa^2 + 1)}{\kappa^2 (2 \kappa - \sh(2\kappa) + 2 \beta \kappa^3 + \beta \kappa^2 \sh(2\kappa))},\\
w^{(1,0)}_{33}=&\frac{2 \pi \sh(\kappa) (\beta \kappa^2 + 1)}{\kappa (\sh(\kappa) - \kappa \ch(\kappa) + \beta \kappa^2 \sh(\kappa))},\\
w^{(1,0)}_{41}=&\frac{-2 \pi \ch(\kappa) \sh(\kappa)}{\kappa \sh(\kappa) - \kappa^2 \ch(\kappa) + \beta \kappa^3 \sh(\kappa)},
\intertext{and}
w^{(0,1)}_{13}=&\frac{-2 \pi \sh(\kappa) (\beta \kappa^2 + 1) (2 \sh(\kappa)^2 + \beta \kappa^2 + 3)}{\ch(\kappa) (2 \kappa - \sh(2\kappa) + 2 \beta \kappa^3 + \beta \kappa^2 \sh(2\kappa))},\\
w^{(0,1)}_{41}=&-i\pi (4 \ch(\kappa) - 4 \ch(\kappa)^3 + \kappa \sh(\kappa) + \beta \kappa^3 \sh(\kappa)+ 2 \kappa \ch(\kappa)^2 \sh(\kappa))\\
&\quad\cdot (\sh(\kappa) - \kappa \ch(\kappa) + \beta \kappa^2 \sh(\kappa))^{-1}.
\end{align*}

\begin{align*}
w^{(2,0)}_{11}=&\tfrac12(w^{(1,0)}_{11})^2
+4i\pi \kappa^{-1}\sh(2\kappa) e^{2\kappa} (\beta \kappa^2 + 1)p_{11}^{(2,0)} \\ 
&\qquad\qquad\quad\cdot(\kappa - e^{2\kappa} + \beta \kappa^2 + \kappa e^{2\kappa}  - \beta \kappa^2 e^{2\kappa}+ 1)^{-1} \\
&\qquad\qquad\quad \cdot (4 \kappa e^{2\kappa} - \beta \kappa^2- e^{4\kappa} + 4 \beta \kappa^3 e^{2\kappa} + \beta \kappa^2 e^{4\kappa} + 1)^{-3},
\intertext{where}
p_{11}^{(2,0)}=&-\beta^3 \kappa^5 (e^{2\kappa} - 1)(6 e^{4\kappa} - 3 e^{8\kappa} + 16 \kappa^2 e^{2\kappa} - 16 \kappa^2 e^{4\kappa}\\ 
&\qquad\qquad \qquad\qquad+ 16 \kappa^2 e^{6\kappa} + 24 \kappa e^{2\kappa} - 24 \kappa e^{6\kappa}- 3) \\
&-\beta^2\kappa^3 (3 \kappa - 10 e^{2\kappa} - 20 e^{4\kappa} + 20 e^{6\kappa} + 10 e^{8\kappa} - 10 e^{10\kappa}- 72 \kappa^2 e^{2\kappa}\\
&\qquad \qquad - 16 \kappa^3 e^{2\kappa} + 88 \kappa^2 e^{4\kappa} - 88 \kappa^2 e^{6\kappa} + 72 \kappa^2 e^{8\kappa}- 16 \kappa^3 e^{8\kappa}\\
&\qquad \qquad  - 61 \kappa e^{2\kappa} + 58 \kappa e^{4\kappa} + 58 \kappa e^{6\kappa} - 61 \kappa e^{8\kappa} + 3 \kappa e^{10\kappa} + 10)\\
&-\beta\kappa (6 \kappa - 3 e^{2\kappa} - 6 e^{4\kappa} + 6 e^{6\kappa} + 3 e^{8\kappa} - 3 e^{10\kappa}- 80 \kappa^2 e^{2\kappa} \\
&\qquad \quad - 32 \kappa^3 e^{2\kappa} + 96 \kappa^2 e^{4\kappa} - 96 \kappa^2 e^{6\kappa} + 80 \kappa^2 e^{8\kappa}- 32 \kappa^3 e^{8\kappa} \\
&\qquad\quad- 34 \kappa e^{2\kappa} + 28 \kappa e^{4\kappa} + 28 \kappa e^{6\kappa} - 34 \kappa e^{8\kappa}+ 6 \kappa e^{10\kappa} + 3) \\
&+(e^{2\kappa} - 2 e^{4\kappa} - 2 e^{6\kappa} + e^{8\kappa} + e^{10\kappa} + 16 \kappa^2 e^{2\kappa}\\
&\qquad+ 16 \kappa^2 e^{8\kappa} + 24 \kappa e^{2\kappa} - 40 \kappa e^{4\kappa} + 40 \kappa e^{6\kappa} - 24 \kappa e^{8\kappa} + 1),
\intertext{and}
w^{(2,0)}_{34}=&\frac{2 \pi (e^{4\kappa} - 1) (\beta \kappa^2 + 1)}{(\kappa - e^{2\kappa} + \beta \kappa^2 + \kappa e^{2\kappa} - \beta \kappa^2 e^{2\kappa} + 1)^2},\\
w^{(1,1)}_{31}=&-i\pi \sh(2\kappa) (\beta \kappa^2 + 1) \\
&\quad \cdot (\kappa - e^{2\kappa} + \beta \kappa^2 + \kappa e^{2\kappa}  - \beta \kappa^2 e^{2\kappa} + 1)^{-2} \\
&\quad \cdot(4 \kappa e^{2\kappa} - \beta \kappa^2 - e^{4\kappa} + 4 \beta \kappa^3 e^{2\kappa} + \beta \kappa^2 e^{4\kappa} + 1)^{-1}\\
&\quad \cdot (3 \kappa - 4 e^{2\kappa} + 4 e^{6\kappa} - 2 e^{8\kappa} - 2 \kappa^2 e^{2\kappa} + 8 \kappa^3 e^{4\kappa} \\
&\qquad+ 2 \kappa^2 e^{6\kappa} + 6 \beta \kappa^2 + \beta \kappa^3 - 8 \kappa e^{2\kappa} + 10 \kappa e^{4\kappa}\\
&\qquad - 8 \kappa e^{6\kappa} + 3 \kappa e^{8\kappa} - 12 \beta \kappa^2 e^{2\kappa} - 8 \beta \kappa^3 e^{2\kappa} \\
&\qquad- 6 \beta \kappa^4 e^{2\kappa} + 14 \beta \kappa^3 e^{4\kappa} + 12 \beta \kappa^2 e^{6\kappa}- 8 \beta \kappa^3 e^{6\kappa} \\
&\qquad+ 8 \beta \kappa^5 e^{4\kappa} - 6 \beta \kappa^2 e^{8\kappa} + 6 \beta \kappa^4 e^{6\kappa}  + \beta \kappa^3 e^{8\kappa} + 2),
\intertext{and}
w^{(1,1)}_{14}=&2 \pi (e^{2\kappa} - 1) (\beta \kappa^2 + 1)^2 (e^{4\kappa} + 2 \kappa e^{2\kappa} - 1) \\
&\cdot (\kappa - e^{2\kappa} - e^{4\kappa} + e^{6\kappa} + 4 \kappa^2 e^{2\kappa} + 4 \kappa^2 e^{4\kappa} - \beta \kappa^3\\
&\quad - \beta^2 \kappa^4 + 5 \kappa e^{2\kappa} - 5 \kappa e^{4\kappa} - \kappa e^{6\kappa} + \beta^2 \kappa^4 e^{2\kappa} + 4 \beta^2 \kappa^5 e^{2\kappa}\\
&\quad+ \beta^2 \kappa^4 e^{4\kappa} - 4 \beta^2 \kappa^5 e^{4\kappa} - \beta^2 \kappa^4 e^{6\kappa} + 7 \beta \kappa^3 e^{2\kappa}\\
&\quad+ 4 \beta \kappa^4 e^{2\kappa} - 7 \beta \kappa^3 e^{4\kappa} + 4 \beta \kappa^4 e^{4\kappa}  + \beta \kappa^3 e^{6\kappa} + 1)^{-1},\\
w^{(0,2)}_{11}=&-i\kappa \pi (e^{2\kappa} - 1)z_{11}^{(0,2)}\\
&\cdot\big(32 \kappa^6 e^{2\kappa} \beta^3 (e^{4\kappa} + 4 \kappa e^{2\kappa} - 1) (3 e^{2\kappa} - 3 e^{4\kappa} - e^{6\kappa} + 1)\\
&\quad +16 \kappa^4 e^{2\kappa}\beta^2 (2 e^{4\kappa} - 3 e^{2\kappa} - 2 \kappa + 2 e^{6\kappa} - 3 e^{8\kappa} \\
&\hspace*{77pt}+ e^{10\kappa}  + 8 \kappa^2 e^{2\kappa}+ 40 \kappa^2 e^{4\kappa} + 40 \kappa^2 e^{6\kappa} + 8 \kappa^2 e^{8\kappa}\\
&\hspace*{77pt}+ 2 \kappa e^{2\kappa} + 52 \kappa e^{4\kappa} - 52 \kappa e^{6\kappa}- 2 \kappa e^{8\kappa}+ 2 \kappa e^{10\kappa} + 1)\\
&\quad+16 \kappa^2 e^{2\kappa} \beta(3 \kappa + 6 e^{2\kappa} - 8 e^{4\kappa} - 8 e^{6\kappa} + 6 e^{8\kappa} + 2 e^{10\kappa} \\
&\hspace*{71pt}+ 4 \kappa^2 e^{2\kappa}+ 44 \kappa^2 e^{4\kappa} + 44 \kappa^2 e^{6\kappa} + 4 \kappa^2 e^{8\kappa} + 9 \kappa e^{2\kappa}\\
&\hspace*{71pt}+ 54 \kappa e^{4\kappa}  - 54 \kappa e^{6\kappa} - 9 \kappa e^{8\kappa}- 3 \kappa e^{10\kappa} + 2)\\
&-16 e^{2\kappa} (e^{2\kappa} - 1)^2 (\kappa - e^{2\kappa} - e^{4\kappa} + e^{6\kappa} + 4 \kappa^2 e^{2\kappa} \\&\hspace*{88pt}+ 4 \kappa^2 e^{4\kappa} + 5 \kappa e^{2\kappa} - 5 \kappa e^{4\kappa} - \kappa e^{6\kappa} + 1)\big)^{-1},
\intertext{where}
z_{11}^{(0,2)}=&2 \kappa^6 \beta^3(4 e^{2\kappa} + e^{4\kappa} + 1) \\
&\cdot(\kappa - 32 e^{2\kappa} + 64 e^{4\kappa} - 32 e^{6\kappa} + 32 \kappa^2 e^{4\kappa} - 6 \kappa e^{2\kappa} + 6 \kappa e^{6\kappa} - \kappa e^{8\kappa})\\
&+\kappa^4 \beta^2(1216 e^{6\kappa} - 224 e^{2\kappa} - 352 e^{4\kappa} - 13 \kappa - 352 e^{8\kappa} \\&\qquad\qquad- 224 e^{10\kappa} - 32 e^{12\kappa} + 64 \kappa^2 e^{2\kappa} + 430 \kappa^2 e^{4\kappa} + 1120 \kappa^2 e^{6\kappa} \\&\qquad\qquad+ 430 \kappa^2 e^{8\kappa} + 64 \kappa^2 e^{10\kappa} + 2 \kappa^2 e^{12\kappa} + 2 \kappa^2 - 72 \kappa e^{2\kappa} \\&\qquad\qquad- 153 \kappa e^{4\kappa} + 153 \kappa e^{8\kappa} + 72 \kappa e^{10\kappa} + 13 \kappa e^{12\kappa} - 32)\\
&+\kappa^2 \beta(\kappa - 160 e^{2\kappa} - 368 e^{4\kappa} + 1088 e^{6\kappa} - 368 e^{8\kappa}\\&\qquad\quad- 160 e^{10\kappa} - 16 e^{12\kappa} + 62 \kappa^2 e^{2\kappa} + 479 \kappa^2 e^{4\kappa} + 1220 \kappa^2 e^{6\kappa} \\&\qquad\quad+ 479 \kappa^2 e^{8\kappa}+ 62 \kappa^2 e^{10\kappa} + \kappa^2 e^{12\kappa} + \kappa^2- 36 \kappa e^{2\kappa} \\&\qquad\quad - 75 \kappa e^{4\kappa} + 75 \kappa e^{8\kappa} + 36 \kappa e^{10\kappa} - \kappa e^{12\kappa} - 16)\\
&+16 \kappa - 144 e^{4\kappa} + 256 e^{6\kappa}- 144 e^{8\kappa} + 16 e^{12\kappa} + 16 \kappa^2 e^{2\kappa}\\
&+ 104 \kappa^2 e^{4\kappa} + 320 \kappa^2 e^{6\kappa} + 104 \kappa^2 e^{8\kappa} + 16 \kappa^2 e^{10\kappa} + 8 \kappa^2 e^{12\kappa} \\
&+ 8 \kappa^2 + 32 \kappa e^{2\kappa} + 32 \kappa e^{4\kappa} - 32 \kappa e^{8\kappa} - 32 \kappa e^{10\kappa} - 16 \kappa e^{12\kappa} + 16.
\end{align*}

\section{Modulational instability indices for unimodal waves}\label{app:ind}

Substituting the results of Appendix \ref{app:V(m,n) unimodal} into \eqref{g1g2}, after simplification, we obtain:
\begin{align*}
{\rm ind}_{5,1}
    =&\kappa^4\beta^2 (3 e^{4\kappa} + 4 \kappa e^{2\kappa} - 3)^2\\
    &+2 \kappa^2\beta (2 \kappa - 6 e^{4\kappa} + 3 e^{8\kappa}  + 16 \kappa^2 e^{4\kappa} \\ &\qquad\qquad- 12 \kappa e^{2\kappa}+ 12 \kappa e^{6\kappa} - 2 \kappa e^{8\kappa} + 3)\\
    &+4 \kappa - 2 e^{4\kappa} + e^{8\kappa} 
    + 16 \kappa^2 e^{4\kappa} - 4 \kappa e^{8\kappa} + 1,\\
{\rm ind}_{5,2}
    =&\kappa^4 \beta^2(6 e^{4\kappa} - 3 e^{8\kappa} + 16 \kappa^2 e^{2\kappa} - 16 \kappa^2 e^{4\kappa}\\ &\qquad\quad + 16 \kappa^2 e^{6\kappa} + 24 \kappa e^{2\kappa} - 24 \kappa e^{6\kappa}- 3)\\
    &+2 \kappa^2 \beta (6 e^{4\kappa} - 3 e^{8\kappa} + 16 \kappa^2 e^{2\kappa} - 16 \kappa^2 e^{4\kappa} \\ &\qquad \qquad + 16 \kappa^2 e^{6\kappa} + 16 \kappa e^{2\kappa} - 16 \kappa e^{6\kappa} - 3)\\
    &+e^{8\kappa} - 2 e^{4\kappa} + 16 \kappa^2 e^{2\kappa} - 16 \kappa^2 e^{4\kappa} + 16 \kappa^2 e^{6\kappa} + 8 \kappa e^{2\kappa} - 8 \kappa e^{6\kappa} + 1,\\
{\rm ind}_{5,3}
    =&2 \kappa^8 (4 e^{2\kappa} + e^{4\kappa} + 1) a_{5,3}^{(4)}\beta^4
    +\kappa^6 a_{5,3}^{(3)}\beta^3
    +4 \kappa^4 a_{5,3}^{(2)}\beta^2+\kappa^2 a_{5,3}^{(1)}\beta+a_{5,3}^{(0)},\\
{\rm ind}_{5,4}
    =&-2 \kappa^2 (4 e^{2\kappa} + e^{4\kappa} + 1) \beta +(e^{2\kappa} - 1)^2,
\end{align*}
where
\begin{align*}
a_{5,3}^{(4)}=&42 e^{2\kappa} - 9 e^{4\kappa} - 84 e^{6\kappa} - 9 e^{8\kappa} + 42 e^{10\kappa} + 9 e^{12\kappa} + 16 \kappa^2 e^{4\kappa} \\&+ 416 \kappa^2 e^{6\kappa} + 16 \kappa^2 e^{8\kappa} - 24 \kappa e^{2\kappa} - 240 \kappa e^{4\kappa}+ 240 \kappa e^{8\kappa} + 24 \kappa e^{10\kappa} + 9,\\
a_{5,3}^{(3)}=&8 \kappa - 188 e^{2\kappa} - 28 e^{4\kappa} + 188 e^{6\kappa}+ 142 e^{8\kappa} + 188 e^{10\kappa}\\
&- 28 e^{12\kappa} - 188 e^{14\kappa}  - 43 e^{16\kappa} + 80 \kappa^2 e^{4\kappa}+ 2880 \kappa^2 e^{6\kappa} \\
&+ 12512 \kappa^2 e^{8\kappa} + 2880 \kappa^2 e^{10\kappa}+ 80 \kappa^2 e^{12\kappa} + 168 \kappa e^{2\kappa} - 112 \kappa e^{4\kappa}\\& - 3384 \kappa e^{6\kappa} + 3384 \kappa e^{10\kappa} + 112 \kappa e^{12\kappa} - 168 \kappa e^{14\kappa} - 8 \kappa e^{16\kappa} - 43,\\
a_{5,3}^{(2)}=&3 \kappa - 135 e^{2\kappa} - 314 e^{4\kappa} + 135 e^{6\kappa} + 636 e^{8\kappa}\\& + 135 e^{10\kappa} - 314 e^{12\kappa} - 135 e^{14\kappa} - 4 e^{16\kappa} + 48 \kappa^2 e^{4\kappa} + 720 \kappa^2 e^{6\kappa}\\& + 4224 \kappa^2 e^{8\kappa} + 720 \kappa^2 e^{10\kappa} + 48 \kappa^2 e^{12\kappa} + 66 \kappa e^{2\kappa} + 498 \kappa e^{4\kappa}\\& + 138 \kappa e^{6\kappa} - 138 \kappa e^{10\kappa} - 498 \kappa e^{12\kappa} - 66 \kappa e^{14\kappa} - 3 \kappa e^{16\kappa} - 4,\\
a_{5,3}^{(1)}=&36 \kappa - 132 e^{2\kappa} - 892 e^{4\kappa} + 132 e^{6\kappa} + 1686 e^{8\kappa} + 132 e^{10\kappa} \\&- 892 e^{12\kappa} - 132 e^{14\kappa} + 49 e^{16\kappa} + 272 \kappa^2 e^{4\kappa} + 960 \kappa^2 e^{6\kappa} + 9824 \kappa^2 e^{8\kappa} \\&+ 960 \kappa^2 e^{10\kappa} + 272 \kappa^2 e^{12\kappa} + 48 \kappa e^{2\kappa} + 1560 \kappa e^{4\kappa} + 2736 \kappa e^{6\kappa} \\&- 2736 \kappa e^{10\kappa} - 1560 \kappa e^{12\kappa} - 48 \kappa e^{14\kappa} - 36 \kappa e^{16\kappa} + 49,\\
a_{5,3}^{(0)}=&32 \kappa+ 64 e^{2\kappa} - 144 e^{4\kappa} - 64 e^{6\kappa} + 208 e^{8\kappa} - 64 e^{10\kappa} - 144 e^{12\kappa} \\&+ 64 e^{14\kappa} + 40 e^{16\kappa} + 128 \kappa^2 e^{4\kappa} + 2048 \kappa^2 e^{8\kappa} + 128 \kappa^2 e^{12\kappa} + 128 \kappa e^{4\kappa} \\&+ 768 \kappa e^{6\kappa} - 768 \kappa e^{10\kappa} - 128 \kappa e^{12\kappa} - 32 \kappa e^{16\kappa} + 40.
\end{align*}

\section{\texorpdfstring{$\mathbf{V}^{(m,n)}(T;0)$}{Lg} for Wilton ripples of order \texorpdfstring{$\geq3$}{Lg}}\label{app:V(m,n) M=3}

The formulae for $\mathbf{V}^{(1,0)}(T;0)$ and $\mathbf{V}^{(2,0)}(T;0)$ in Appendix \ref{app:V(m,n) unimodal} remain valid for bimodal Wilton ripples. A straightforward calculation reveals:
\begin{align*}
w_{45}^{(1,0)}&=-2\pi\sh(M\kappa)(\ch(M\kappa)\sh(\kappa) - M\sh(M\kappa)\ch(\kappa))\\
&\quad\cdot(\kappa(1-M^2)\sh(M\kappa)\sh(\kappa)  \\
&\qquad+ \kappa^2 M^2\sh(M\kappa)\ch(\kappa) - \kappa^2 M\ch(M\kappa)\sh(\kappa))^{-1},\\
w_{53}^{(1,0)}&=2i\pi \sh(M\kappa) \sh(\kappa) (M^2 - 1) \\
&\quad \cdot (\kappa^2 M (\kappa M^3 \sh(\kappa) - \kappa M \sh(\kappa) + \ch(M\kappa) \sh(M\kappa) \sh(\kappa)\\
&\qquad - 2 M \sh(M\kappa)^2 \ch(\kappa) + M^2 \ch(M\kappa) \sh(M\kappa) \sh(\kappa)))^{-1},\\
w_{55}^{(1,0)}&=-4 \pi \ch(M\kappa) \sh(M\kappa) \sh(\kappa) (M^2 - 1)\\
&\quad\cdot (\kappa (2 M \ch(\kappa) - \kappa M \sh(\kappa) + \kappa M^3 \sh(\kappa) + \ch(M\kappa) \sh(M\kappa) \sh(\kappa)\\&\qquad - 2 M \ch(M\kappa)^2 \ch(\kappa)+ M^2 \ch(M\kappa) \sh(M\kappa) \sh(\kappa)))^{-1},
\intertext{and}
w_{13}^{(0,1)}
=&-\pi (M^2-1) M^{-1} z_{13}^{(0,1)}\\
&\quad\cdot (\ch(M\kappa) \sh(\kappa)-M \sh(M\kappa)  \ch(\kappa))^{-1} \\
&\quad \cdot (\kappa \sh(M\kappa)-\kappa M^2 \sh(M\kappa)-2 M \ch(M\kappa) \sh(\kappa)^2\\&\qquad+\sh(M\kappa)  \ch(\kappa) \sh(\kappa)+M^2 \sh(M\kappa)  \ch(\kappa) \sh(\kappa))^{-1},\\
w_{41}^{(0,1)}
=&4 i\pi \alpha M  \sh(\kappa) (M \ch(\kappa)(1- \ch(M\kappa)^2)+\sh(M\kappa)  \sh(\kappa) \ch(M\kappa))\\
&\cdot (\sh(M\kappa)  \sh(\kappa)-M^2 \sh(M\kappa)  \sh(\kappa)\\
&\quad +\kappa M^2 \sh(M\kappa) \ch(\kappa)-\kappa M \ch(M\kappa)  \sh(\kappa))^{-1}\\&
-i\pi  \sh(\kappa) z_{41}^{(0,1)}(\sh(M\kappa) \sh(\kappa)-M^2 \sh(M\kappa)  \sh(\kappa)\\
&\hspace*{77pt}+\kappa M^2 \sh(M\kappa) \ch(\kappa)-\kappa M \ch(M\kappa)  \sh(\kappa)\big)^{-1},\\
w_{43}^{(0,1)}
=&-2\pi \alpha z_{43}^{(0,1)}-2 \pi (\sh(\kappa)-\kappa  \ch(\kappa))\kappa^{-1}\\
&-2 \pi \sh(\kappa) \kappa^{-1}(\sh(M\kappa) \sh(\kappa)-M^2 \sh(M\kappa) \sh(\kappa)\\
&\hspace*{77pt}-\kappa M^2 \sh(M\kappa)  \ch(\kappa)+\kappa M \ch(M\kappa) \sh(\kappa))\\
&\quad\cdot (\sh(M\kappa) \sh(\kappa)-M^2 \sh(M\kappa) \sh(\kappa)\\&\qquad+\kappa M^2 \sh(M\kappa)  \ch(\kappa)-\kappa M \ch(M\kappa) \sh(\kappa))^{-1},\\
w_{45}^{(0,1)}
=&4 i\pi M \sh(M\kappa)  \sh(\kappa) (\ch(M\kappa)  \sh(\kappa)-M \sh(M\kappa) \ch(\kappa))\\&\cdot(\sh(M\kappa)  \sh(\kappa)-M^2 \sh(M\kappa)  \sh(\kappa)\\ &\quad +\kappa M^2 \sh(M\kappa) \ch(\kappa)-\kappa M \ch(M\kappa)  \sh(\kappa))^{-1}\\&-i\pi \alpha M z_{45}^{(0,1)},\\
w_{53}^{(0,1)}
=&\alpha M \pi (M^2-1)(\ch(M\kappa) \sh(\kappa)-M \sh(M\kappa)  \ch(\kappa))^{-1} z_{53}^{(0,1)}\\&\cdot (\kappa M^3 \sh(\kappa)-\kappa M \sh(\kappa)+\ch(M\kappa) \sh(M\kappa) \sh(\kappa)\\&\quad-2 M \sh(M\kappa)^2  \ch(\kappa)+M^2 \ch(M\kappa) \sh(M\kappa) \sh(2\kappa))^{-1},
\end{align*}
where 
\begin{align*}
z_{13}^{(0,1)}=&\sh(\kappa)-M^2 \sh(\kappa)-\ch(M\kappa)^2 \sh(\kappa)+M^2 \ch(M\kappa)^2 \sh(\kappa)\\&-2 M^2  \ch(\kappa)^2 \sh(\kappa)+2 M^2 \ch(M\kappa)^2  \ch(\kappa)^2 \sh(\kappa)\\&-2 M \ch(M\kappa) \sh(M\kappa)  \ch(\kappa)^3+2 M \ch(M\kappa) \sh(M\kappa)  \ch(\kappa),\\
z_{41}^{(0,1)}=&4 M \ch(M\kappa)+\kappa \sh(M\kappa)-\kappa M^2 \sh(M\kappa)\\&-4 M \ch(M\kappa) \ch(\kappa)^2-2 \kappa M^2 \sh(M\kappa) \ch(\kappa)^2\\
&+4 M^2 \sh(M\kappa) \ch(\kappa)  \sh(\kappa)+2 \kappa M \ch(M\kappa) \ch(\kappa)  \sh(\kappa),\\
z_{43}^{(0,1)}=&(\sh(M\kappa)-\kappa M \ch(M\kappa))\kappa^{-1}+ \sh(M\kappa) (\sh(M\kappa) \sh(\kappa)\\&-M^2 \sh(M\kappa) \sh(\kappa)-\kappa M^2 \sh(M\kappa)  \ch(\kappa)+\kappa M \ch(M\kappa) \sh(\kappa))\\
&\cdot\kappa^{-1} (\sh(M\kappa) \sh(\kappa)-M^2 \sh(M\kappa) \sh(\kappa)\\
&\qquad\quad+\kappa M^2 \sh(M\kappa)  \ch(\kappa)-\kappa M \ch(M\kappa) \sh(\kappa))^{-1}\\
z_{45}^{(0,1)}=&(4 \ch(M\kappa)  \sh(\kappa)-4 \ch(M\kappa)^3  \sh(\kappa)\\&+4 M \sh(M\kappa)^3 \ch(\kappa)+2 \kappa M^2 \ch(M\kappa) \ch(\kappa)\\&+2 \kappa M \sh(M\kappa)^3  \sh(\kappa)-\kappa M^3 \sh(M\kappa)  \sh(\kappa)\\&-2 \kappa M^2 \ch(M\kappa)^3 \ch(\kappa)+3 \kappa M \sh(M\kappa)  \sh(\kappa))\\
&\cdot (\sh(M\kappa)  \sh(\kappa)-M^2 \sh(M\kappa)  \sh(\kappa)\\&\quad+\kappa M^2 \sh(M\kappa) \ch(\kappa)-\kappa M \ch(M\kappa)  \sh(\kappa))^{-1} \\
z_{53}^{(0,1)}=&\sh(M\kappa)+2 \ch(M\kappa)^2 \sh(M\kappa)-\sh(M\kappa)  \ch(\kappa)^2\\&-M^2 \sh(M\kappa)+M^2 \sh(M\kappa)  \ch(\kappa)^2-2 \ch(M\kappa)^2 \sh(M\kappa)  \ch(\kappa)^2\\&-2 M \ch(M\kappa)  \ch(\kappa) \sh(\kappa)+2 M \ch(M\kappa)^3  \ch(\kappa) \sh(\kappa)
\end{align*}

\section{\texorpdfstring{$\mathbf{V}^{(m,n)}(T;0)$}{Lg} for Wilton ripples of order \texorpdfstring{$2$}{Lg} }\label{app:V(m,n) M=2}

A straightforward calculation reveals 
\begin{align*}
w_{11}^{(0,1)}&=-\frac{3i\kappa\pi\sh(2\kappa)(5e^{2\kappa} + 5e^{4\kappa} + e^{6\kappa} + 1)}{4\alpha \ch(2\kappa)(9e^{2\kappa} - 9e^{4\kappa} - e^{6\kappa} + 12\kappa e^{2\kappa} + 12\kappa e^{4\kappa} + 1)}
\intertext{and}
w_{55}^{(0,1)}&=\frac{3i\kappa\pi(4e^{2\kappa} - 4e^{6\kappa} - e^{8\kappa} + 1)}{2\alpha (8e^{2\kappa} - 8e^{6\kappa} + e^{8\kappa} + 24\kappa e^{4\kappa} - 1)}.
\end{align*}

\end{appendix}

\newpage

\bibliographystyle{amsplain}
\bibliography{wwbib}

\providecommand{\bysame}{\leavevmode\hbox to3em{\hrulefill}\thinspace}
\providecommand{\MR}{\relax\ifhmode\unskip\space\fi MR }
\providecommand{\MRhref}[2]{%
  \href{http://www.ams.org/mathscinet-getitem?mr=#1}{#2}
}
\providecommand{\href}[2]{#2}
\begin{thebibliography}{10}

\bibitem{AS1979}
Mark~J. Ablowitz and Harvey Segur, \emph{On the evolution of packets of water
  waves}, J. Fluid Mech. \textbf{92} (1979), no.~4, 691--715. \MR{544892}

\bibitem{doi:10.1029/JB073i020p06545}
Richard Barakat and Agnes Houston, \emph{Nonlinear periodic capillary-gravity
  waves on a fluid of finite depth}, Journal of Geophysical Research
  (1896-1977) \textbf{73} (1968), no.~20, 6545--6554.

\bibitem{Benjamin;BF}
T.~Brooke Benjamin, \emph{Instability of periodic wavetrains in nonlinear
  dispersive systems}, Proc. R. Soc. Lond. A \textbf{299} (1967), no.~1456,
  59--76.

\bibitem{BF;BF}
T.~Brooke Benjamin and J.~E. Feir, \emph{The disintegration of wave trains on
  deep water {P}art 1. {T}heory}, J. Fluid Mech. \textbf{27} (1967), no.~3,
  417--430.

\bibitem{Benney1976}
D.~J. Benney, \emph{Significant interactions between small and large scale
  surface waves}, Studies in Appl. Math. \textbf{55} (1976), no.~2, 93--106.
  \MR{452076}

\bibitem{berti2025}
Massimiliano Berti, Livia Corsi, Alberto Maspero, and Paolo Ventura,
  \emph{Infinitely many isolas of modulational instability for {S}tokes waves},
  2025.

\bibitem{Berti2022}
Massimiliano Berti, Alberto Maspero, and Paolo Ventura, \emph{Full description
  of {B}enjamin-{F}eir instability of {S}tokes waves in deep water}, Invent.
  Math. \textbf{230} (2022), no.~2, 651--711. \MR{4493325}

\bibitem{Berti2023}
\bysame, \emph{Benjamin-{F}eir instability of {S}tokes waves in finite depth},
  Arch. Ration. Mech. Anal. \textbf{247} (2023), no.~5, Paper No. 91, 54.
  \MR{4632837}

\bibitem{berti2024isolamodulationalinstabilitystokes}
Massimiliano Berti, Alberto Maspero, and Paolo Ventura, \emph{First isola of
  modulational instability of {S}tokes waves in deep water}, 2024.

\bibitem{BM;BF}
Thomas~J. Bridges and Alexander Mielke, \emph{A proof of the {B}enjamin-{F}eir
  instability}, Arch. Rational Mech. Anal. \textbf{133} (1995), no.~2,
  145--198. \MR{1367360}

\bibitem{MR1720395}
B.~Buffoni and M.~D. Groves, \emph{A multiplicity result for solitary
  gravity-capillary waves in deep water via critical-point theory}, Arch.
  Ration. Mech. Anal. \textbf{146} (1999), no.~3, 183--220. \MR{1720395}

\bibitem{BGT}
B.~Buffoni, M.~D. Groves, and J.~F. Toland, \emph{A plethora of solitary
  gravity-capillary water waves with nearly critical {B}ond and {F}roude
  numbers}, Philos. Trans. Roy. Soc. London Ser. A \textbf{354} (1996),
  no.~1707, 575--607. \MR{1378603}

\bibitem{CSaffmam1979}
B.~Chen and P.~G. Saffman, \emph{Steady gravity-capillary waves on deep
  water—1. weakly nonlinear waves}, Studies in Applied Mathematics
  \textbf{60} (1979), no.~3, 183--210.

\bibitem{Chicone}
Carmen Chicone, \emph{Ordinary differential equations with applications}, Texts
  in Applied Mathematics, vol.~34, Springer-Verlag, New York, 1999.
  \MR{1707333}

\bibitem{creedon_deconinck_trichtchenko_2022}
Ryan~P. Creedon, Bernard Deconinck, and Olga Trichtchenko, \emph{High-frequency
  instabilities of {S}tokes waves}, J. Fluid Mech. \textbf{937} (2022), Paper
  No. A24, 32. \MR{4386807}

\bibitem{DO}
Bernard Deconinck and Katie Oliveras, \emph{The instability of periodic surface
  gravity waves}, J. Fluid Mech. \textbf{675} (2011), 141--167. \MR{2801039}

\bibitem{DECONINCK201497}
Bernard Deconinck and Olga Trichtchenko, \emph{Stability of periodic gravity
  waves in the presence of surface tension}, Eur. J. Mech. B Fluids \textbf{46}
  (2014), 97--108. \MR{3200404}

\bibitem{djordjevic_redekopp_1977}
V.~D. Djordjevic and L.~G. Redekopp, \emph{On two-dimensional packets of
  capillary-gravity waves}, Journal of Fluid Mechanics \textbf{79} (1977),
  no.~4, 703--714.

\bibitem{FK}
Marc Francius and Christian Kharif, \emph{Three-dimensional instabilities of
  periodic gravity waves in shallow water}, J. Fluid Mech. \textbf{561} (2006),
  417--437. \MR{2266187}

\bibitem{Gardner;evans1}
R.~A. Gardner, \emph{On the structure of the spectra of periodic travelling
  waves}, J. Math. Pures Appl. (9) \textbf{72} (1993), no.~5, 415--439.
  \MR{1239098}

\bibitem{Gardner;evans2}
Robert~A. Gardner, \emph{Spectral analysis of long wavelength periodic waves
  and applications}, J. Reine Angew. Math. \textbf{491} (1997), 149--181.
  \MR{1476091}

\bibitem{GM}
M.~D. Groves and A.~Mielke, \emph{A spatial dynamics approach to
  three-dimensional gravity-capillary steady water waves}, Proc. Roy. Soc.
  Edinburgh Sect. A \textbf{131} (2001), no.~1, 83--136. \MR{1820296}

\bibitem{HS;cg-solitary}
Mariana Haragus and Arnd Scheel, \emph{Finite-wavelength stability of
  capillary-gravity solitary waves}, Comm. Math. Phys. \textbf{225} (2002),
  no.~3, 487--521. \MR{1888871}

\bibitem{Hogan2}
S.~J. Hogan, \emph{Some effects of surface tension on steep water waves. {II}},
  J. Fluid Mech. \textbf{96} (1980), no.~3, 417--445. \MR{565514}

\bibitem{Hogan3}
\bysame, \emph{Some effects of surface tension on steep water waves. {III}}, J.
  Fluid Mech. \textbf{110} (1981), 381--410. \MR{633452}

\bibitem{Hogan1985}
S.~J. Hogan, \emph{The fourth-order evolution equation for deep-water
  gravity-capillary waves}, Proceedings of the Royal Society of London. Series
  A, Mathematical and Physical Sciences \textbf{402} (1985), no.~1823,
  359--372.

\bibitem{Hormander}
Lars H\"ormander, \emph{An introduction to complex analysis in several
  variables}, third ed., North-Holland Mathematical Library, vol.~7,
  North-Holland Publishing Co., Amsterdam, 1990. \MR{1045639}

\bibitem{HHY2026b}
Ting-Yang Hsiao, Vera~Mikyoung Hur, and Zhao Yang, \emph{Spectral instability
  of stokes waves in constant vorticity flows in the infinite depth}, In
  preparation, 2026.

\bibitem{HHY2026a}
\bysame, \emph{Unstable stokes waves in constant vorticity flows}, In
  preparation, 2026.

\bibitem{hsiao2025}
Ting-Yang Hsiao and Alberto Maspero, \emph{Full {B}enjamin--{F}eir instability
  of capillary-gravity {S}tokes waves in finite depth}, 2025.

\bibitem{HY2023}
Vera~Mikyoung Hur and Zhao Yang, \emph{Unstable {S}tokes waves}, Arch. Ration.
  Mech. Anal. \textbf{247} (2023), no.~4, Paper No. 62, 59. \MR{4600217}

\bibitem{ik1992}
G\'erard Iooss and Klaus Kirchg\"assner, \emph{Water waves for small surface
  tension: an approach via normal form}, Proc. Roy. Soc. Edinburgh Sect. A
  \textbf{122} (1992), no.~3-4, 267--299. \MR{1200201}

\bibitem{JRSY25}
Ziang Jiao, L.~Miguel Rodrigues, Changzhen Sun, and Zhao Yang,
  \emph{Small-amplitude finite-depth {S}tokes waves are transversally
  unstable}, Comm. Math. Phys. \textbf{406} (2025), no.~10, Paper No. 255, 38.
  \MR{4952108}

\bibitem{jones_1989}
M.~C.~W. Jones, \emph{Small amplitude capillary-gravity waves in a channel of
  finite depth}, Glasgow Math. J. \textbf{31} (1989), no.~2, 141--160.
  \MR{997809}

\bibitem{Jones1992}
\bysame, \emph{Nonlinear stability of resonant capillary-gravity waves}, Wave
  Motion \textbf{15} (1992), no.~3, 267--283. \MR{1158291}

\bibitem{Jones1986}
Mark Jones and John Toland, \emph{Symmetry and the bifurcation of
  capillary-gravity waves}, Arch. Rational Mech. Anal. \textbf{96} (1986),
  no.~1, 29--53. \MR{853974}

\bibitem{KK2013}
Todd Kapitula and Keith Promislow, \emph{Spectral and dynamical stability of
  nonlinear waves}, Applied Mathematical Sciences, vol. 185, Springer, New
  York, 2013, With a foreword by Christopher K. R. T. Jones. \MR{3100266}

\bibitem{doi:10.1098/rspa.1978.0080}
Michael~Selwyn Longuet-Higgins, \emph{The instabilities of gravity waves of
  finite amplitude in deep water i. superharmonics}, Proceedings of the Royal
  Society of London. A. Mathematical and Physical Sciences \textbf{360} (1978),
  no.~1703, 471--488.

\bibitem{doi:10.1098/rspa.1978.0081}
\bysame, \emph{The instabilities of gravity waves of finite amplitude in deep
  water ii. subharmonics}, Proceedings of the Royal Society of London. A.
  Mathematical and Physical Sciences \textbf{360} (1978), no.~1703, 489--505.

\bibitem{Ma1978}
Yan~Chow Ma, \emph{The complete solution of the long-wave\thinspace -\thinspace
  short-wave resonance equations}, Stud. Appl. Math. \textbf{59} (1978), no.~3,
  201--221. \MR{521795}

\bibitem{MR1979}
Yan~Chow Ma and L.~G. Redekopp, \emph{Some solutions pertaining to the resonant
  interaction of long and short waves}, Phys. Fluids \textbf{22} (1979),
  no.~10, 1872--1876. \MR{545229}

\bibitem{MS}
R.~S. MacKay and P.~G. Saffman, \emph{Stability of water waves}, Proc. Roy.
  Soc. London Ser. A \textbf{406} (1986), no.~1830, 115--125. \MR{853684}

\bibitem{McGoldrick_1970}
L.~F. McGoldrick, \emph{On wilton’s ripples: a special case of resonant
  interactions}, Journal of Fluid Mechanics \textbf{42} (1970), no.~1,
  193–200.

\bibitem{McLean;infinite2}
J.~W. McLean, Y.~C. Ma, D.~U. Martin, P.~G. Saffman, and H.~C. Yuen,
  \emph{Three-dimensional instability of finite-amplitude water waves}, Phys.
  Rev. Lett. \textbf{46} (1981), no.~13, 817--820. \MR{608380}

\bibitem{McLean;finite-depth}
John~W. McLean, \emph{Instabilities of finite-amplitude gravity waves on water
  of finite depth}, Journal of Fluid Mechanics \textbf{114} (1982), 331--341.

\bibitem{mclean_1982}
\bysame, \emph{Instabilities of finite-amplitude water waves}, Journal of Fluid
  Mechanics \textbf{114} (1982), 315--330.

\bibitem{Mielke;reduction}
Alexander Mielke, \emph{Reduction of quasilinear elliptic equations in
  cylindrical domains with applications}, Math. Methods Appl. Sci. \textbf{10}
  (1988), no.~1, 51--66. \MR{929221}

\bibitem{Nayfeh1970}
Ali~Hasan Nayfeh, \emph{Finite amplitude surface waves in a liquid layer},
  Journal of Fluid Mechanics \textbf{40} (1970), no.~4, 671–684.

\bibitem{NS2023}
Huy~Q. Nguyen and Walter~A. Strauss, \emph{Proof of modulational instability of
  stokes waves in deep water}, Communications on Pure and Applied Mathematics
  \textbf{76} (2023), no.~5, 1035--1084.

\bibitem{NR;anal}
David~P. Nicholls and Fernando Reitich, \emph{On analyticity of travelling
  water waves}, Proc. R. Soc. Lond. Ser. A Math. Phys. Eng. Sci. \textbf{461}
  (2005), no.~2057, 1283--1309. \MR{2147749}

\bibitem{Noble_2023}
Pascal Noble, Luis~Miguel Rodrigues, and Changzhen Sun, \emph{Spectral
  instability of small-amplitude periodic waves of the electronic
  {E}uler-{P}oisson system}, Nonlinearity \textbf{36} (2023), no.~9,
  4615--4640. \MR{4619244}

\bibitem{OS;evans}
Myunghyun Oh and Bj\"{o}rn Sandstede, \emph{Evans functions for periodic waves
  on infinite cylindrical domains}, J. Differential Equations \textbf{248}
  (2010), no.~3, 544--555. \MR{2557905}

\bibitem{Reeder1981part1}
John Reeder and Marvin Shinbrot, \emph{On {W}ilton ripples. {I}. {F}ormal
  derivation of the phenomenon}, Wave Motion \textbf{3} (1981), no.~2,
  115--135. \MR{609316}

\bibitem{Reeder1981part2}
\bysame, \emph{On {W}ilton ripples. {II}. {R}igorous results}, Arch. Rational
  Mech. Anal. \textbf{77} (1981), no.~4, 321--347. \MR{642551}

\bibitem{Ru99a}
{S.M.} Rump, \emph{{INTLAB - INTerval LABoratory}}, {Developments~in~Reliable
  Computing} (Tibor Csendes, ed.), Kluwer Academic Publishers, Dordrecht, 1999,
  \url{http://www.tuhh.de/ti3/rump/}, pp.~77--104.

\bibitem{Stokes1847}
George~Gabriel Stokes, \emph{On the theory of oscillatory waves}, Trans. Camb.
  Philos. Soc. \textbf{8} (1847), 441--455.

\bibitem{sun2025}
Changzhen Sun and Erik Wahl\'en, \emph{On the spectral stability of the
  periodic capillary-gravity waves}, 2025.

\bibitem{10.2307/2397696}
J.~F. Toland and M.~C.~W. Jones, \emph{The bifurcation and secondary
  bifurcation of capillary-gravity waves}, Proc. Roy. Soc. London Ser. A
  \textbf{399} (1985), no.~1817, 391--417. \MR{799114}

\bibitem{TDW16}
Olga Trichtchenko, Bernard Deconinck, and Jon Wilkening, \emph{The instability
  of {W}ilton ripples}, Wave Motion \textbf{66} (2016), 147--155. \MR{3529601}

\bibitem{vanden-broeck_2010}
Jean-Marc Vanden-Broeck, \emph{Gravity--capillary free-surface flows},
  Cambridge Monographs on Mechanics, Cambridge University Press, 2010.

\bibitem{Whitham;BF}
G.~B. Whitham, \emph{Non-linear dispersion of water waves}, J. Fluid Mech.
  \textbf{27} (1967), 399--412. \MR{208903}

\bibitem{doi:10.1080/14786440508635350}
J.~R. Wilton, \emph{Lxxii. on ripples}, The London, Edinburgh, and Dublin
  Philosophical Magazine and Journal of Science \textbf{29} (1915), no.~173,
  688--700.

\bibitem{ZM1987}
Jun Zhang and W.~K. Melville, \emph{Three-dimensional instabilities of
  nonlinear gravity-capillary waves}, Journal of Fluid Mechanics \textbf{174}
  (1987), 187–208.

\end{thebibliography}

\end{document}